\documentclass[phd,tocprelim]{cornell}
\usepackage{graphicx}

\usepackage{graphics}
\usepackage{moreverb}
\usepackage{epsfig}
\usepackage{hangcaption}
\usepackage{palatino}
\usepackage[colorlinks, linkcolor=blue]{hyperref}

\usepackage{amsthm,amsmath,amsfonts}
\usepackage{txfonts}
\usepackage[normalem]{ulem}
\usepackage{color,graphicx}
\usepackage{float}
\DeclareMathOperator{\argmin}{argmin}
\usepackage{centernot}
\usepackage{enumerate}
\newcommand{\R}{\mathbb{R}}

\newcommand{\Z}{\mathbb{Z}}
\newcommand{\I}{\mathbb{I}}
\newcommand{\E}{\mathbb{E}}

\newcommand{\X}{\mathcal{X}}
\newcommand{\M}{\mathcal{M}}
\newcommand{\A}{\mathcal{A}}
\newcommand{\F}{\mathcal{F}}
\newcommand{\V}{\mathcal{V}}
\newcommand{\B}{\mathcal{B}}
\newcommand{\T}{\mathcal{T}}
\newcommand{\LL}{\mathcal{L}}
\newcommand{\J}{\mathcal{J}}
\newcommand{\Prob}{\mathbb{P}}
\usepackage{graphicx}
\graphicspath{{images/}{../images/}}
\providecommand{\abs}[1]{\lvert#1\rvert}
\providecommand{\argmin}{\text{argmin}}

\providecommand{\norm}[1]{\lVert#1\rVert}
\newtheorem{lemma}{Lemma}
\newtheorem{theorem}{Theorem}
\newtheorem{corollary}{Corollary}
\newtheorem{remark}{Remark}
\theoremstyle{definition}
\numberwithin{equation}{section}

 \usepackage{thmtools, thm-restate}

\newtheorem{example}{Example}
\usepackage{tikz}
 \usetikzlibrary{arrows}

\usepackage{subfig}
\usepackage{float}
\usepackage[ruled,vlined,linesnumbered]{algorithm2e}
\usepackage{xcolor}
\usepackage{authblk}

\usepackage{quoting}
\usepackage{graphicx}
\usepackage{tikz}
\usetikzlibrary{positioning}
\usetikzlibrary{calc,fit}
\usepgflibrary{shapes.geometric}
\definecolor{processblue}{cmyk}{0.96,0,0,0}
\tikzset{
  server/.style={circle, inner sep=0.0mm, minimum width=.8cm,
    draw=green,fill=green!10,thick},
  ellipseserver/.style={ellipse, inner sep=0.0mm, minimum width=.8cm,
    draw=green,fill=green!10,thick},
  buffer/.style={rectangle, rounded corners=3pt,
    inner sep=0.0mm, minimum width=.9cm, minimum height=.6cm,
    draw=orange,fill=blue!10,thick},
  vbuffer/.style={rectangle,rounded corners=3pt,
     inner sep=0.0mm,  minimum width=.6cm, minimum height=.8cm,
     draw=orange,fill=blue!10,thick},
   routing/.style={circle,inner sep=0pt,minimum size=.5cm,
     draw=red,minimum width=.1cm, fill=red!30},
   state/.style={inner sep=1.0mm, rounded corners=2pt,
    draw=green,fill=green!10,thick},
  dot/.style={circle, inner sep=0.0mm, minimum width=.6cm,
    draw=blue,fill=blue!10,thick},
  task/.style={circle, inner sep=0.0mm, minimum width=.6cm,
    draw=blue,fill=blue!10,thick},
  data/.style={rectangle, inner sep=0.0mm, minimum width=.5cm,
    minimum height=.4cm,
    draw=red,fill=red!10,thick}
}

\tikzset{
  pics/openrectangle/.style n args={3}{
    code = { %
      \pgfmathsetmacro\x{.5}
      \pgfmathsetmacro\y{4}
      \pgfmathsetmacro\z{.55}
      \coordinate (SN) at #1;
      \coordinate (SS) at ($(0,-.3)+#1$);
      \draw[thick,red] ([shift={(-\x,\y)}]SN)--([shift={(-\x, \z)}]SN)
      --([shift={(\x,\z)}]SN)--([shift={(\x,\y)}]SN);

      \foreach \i/\t in #2
      {
        \node[task] (\t) at ([shift={(0,\i)}]SN)  {\t};
      }

      \foreach \i/\d in #3
      {
        \node[data] (\d) at ([shift={(0,-\i)}]SS)  {C\d};
      }
    }
  }
}

\tikzset{
  pics/rack/.style n args={3}{
    code = { %
      \pgfmathsetmacro\x{1.7}
      \pgfmathsetmacro\y{4}
      \pgfmathsetmacro\z{-2}
      \pgfmathsetmacro\zz{-2.5}
      \coordinate (S) at ($#1!.5!#2$);
      \draw[thick,blue] ([shift={(-\x,\y)}]S)--([shift={(-\x, \z)}]S)
      --([shift={(\x,\z)}]S)--([shift={(\x,\y)}]S);
      \node at ([shift={(0,\zz)}]S) {Rack #3};
    }
  }
}

\renewcommand{\caption}[1]{\singlespacing\hangcaption{#1}\normalspacing}

\title {Processing Network Controls   via Deep Reinforcement Learning}
\author {Mark Gluzman}
\conferraldate {May}{2022}
\degreefield {Ph.D.}
\copyrightholder{Mark Gluzman}
\copyrightyear{2022}

\begin{document}

\maketitle
\makecopyright

\begin{abstract}

Novel   advanced policy gradient (APG) algorithms, such as proximal policy optimization (PPO), trust region policy
optimization,   and their variations, have become the dominant reinforcement learning (RL) algorithms because of their ease of implementation and good practical performance. This dissertation is concerned with theoretical justification and practical application of the APG algorithms for solving processing network control optimization problems.

Processing network control problems are typically formulated as Markov decision process (MDP) or semi-Markov decision process (SMDP) problems  that have several unconventional for RL  features: infinite state spaces,
unbounded costs,  long-run average cost objectives. Policy improvement bounds   play a crucial role in
the theoretical justification of the APG algorithms. In this thesis we refine existing bounds for MDPs with finite state spaces and prove novel policy improvement bounds for classes of MDPs and SMDPs used to model processing network operations.

We consider two examples of processing network control problems  and customize the PPO algorithm to solve them. First, we   consider   parallel-server and  multiclass queueing networks controls.  Second, we consider the drivers repositioning problem in a ride-hailing service system. For both examples the PPO algorithm with auxiliary modifications   consistently generates control policies   that outperform state-of-art heuristics.

\end{abstract}

\begin{biosketch}
Mark Gluzman received his Bachelors degree in System Analysis from National Technical University of Ukraine "Kyiv Polytechnic Institute"   in 2015. Mark received his Masters degree in Applied Mathematics from Columbia University in 2016.
\end{biosketch}

\begin{dedication}
Dedicated to a professor who inspired me to push the limit,\\ Yuri Bogdanski (1949-2021).
\end{dedication}

\begin{acknowledgements}

I want to express my special gratitude to my advisor,   Jim Dai, for making me a mature researcher and  person. He always was supportive and confident  about my intellectual and professional capabilities even when I  had some doubts about them myself.

Thank you to my committee members and coauthors, Alex Vladimirsky, Shane Henderson, Pengyi Shi, Jacob Scott, and Aurora Feng,
  for their time, helpful guidance,
and patience while working with me.

I am grateful to all my teachers who helped me to find my way in life and  get to the graduate school. Especially, I am thankful to my high school math and physics teachers, Oleg Nagel, Alexey Akimov, Alexandr Kuzmitsky; my mentor and advisor at the undergraduate level, Pavlo Kasyanov.

I want to thank all professors and staff members of ORIE and CAM departments at Cornell as well as of School of Data Science at the at the Chinese University of Hong Kong, Shenzhen.  My special gratitude to my friend, Chang Cao, who did everything he could to make my days in Shenzhen comfortable and enjoyable. Thanks to my Ithaca friends who made my grad school journey more memorable.

Finally, I thank to my parents, Nelya and Alexander Gluzman, who were always there to talk and mentally support me.

This work was supported by the Cornell Graduate School, School of Data Science and Shenzhen Resesearch Institute for Big Data at CUHK-Shenzhen, the National
Science Foundation through grant CMMI-1537795, and through teaching assistantships in
Cornell’s Departments of Mathematics, Operations Research and Information Engineering.

\end{acknowledgements}

\contentspage
\tablelistpage
\figurelistpage

\normalspacing \setcounter{page}{1} \pagenumbering{arabic}
\pagestyle{cornell} \addtolength{\parskip}{0.5\baselineskip}

\chapter{Introduction}

 Policy iteration is a classic dynamic programming method that is used to find an optimal policy of a Markov decision process (MDP) problem \cite{Puterman2005}. The policy iteration method computes  the exact state-action value function of a current policy  at the beginning of  each iteration, and then creates a new deterministic  policy that at each state selects an action with the largest   state-action value. In practice,  computation of the exact state-action value function  is only feasible  for MDPs with known underlying models,  also called model-based MDPs,   and with  moderate-size state and action spaces.

A \emph{reinforcement learning (RL) problem} often refers to a
(\emph{model-free})  MDP problem in which the underlying model
governing the dynamics is not known, but sequences of data (actions,
states, and rewards), called as episodes in this thesis, can be observed
under a given policy \cite{Sutton2018}.  One way to   solve a  RL problem is to compute approximate estimates of the exact state-action values. Unfortunately, the exact dynamic programming methods, such as the policy iteration, may suffer from  significant policy degradation, if inexact state-action values are directly used for  greedy policy updates. Moreover,   it is practically impossible to estimate values for each state-action pairs of an MDP with large state and/or action spaces.

Remarkably, it has been demonstrated that \emph{RL
  algorithms} designed for solving RL problems can successfully overcome the
curse of dimensionality in both the model-free    and model-based MDP
problems. Three factors are the keys to the success. First,  Monte Carlo sampling method  is used   to approximately evaluate
expectations. The sampling
method also naturally supports   the exploration  needed in RL algorithms.
Second,  some mechanism preventing  drastic policy changes is incorporated into the course of learning.  Third,  a parametric, low-dimensional representation of a
value function and/or a policy can be used.  In recent years, various \textit{deep} RL algorithms that
use neural networks as an architecture for value function approximation and policy
parametrization have shown state-of-art results \cite{Bellemare2013, Mnih2015,
  Silver2017, OpenAI2019a}.

In this thesis, we focus on  advanced policy gradient (APG) deep RL algorithms, such as proximal policy optimization (PPO)  \cite{Schulman2017}, trust region policy
optimization (TRPO) \cite{Schulman2015} and their variations.
Advanced policy gradient algorithms are iterative.  Each iteration a new policy is obtained   by minimizing a certain surrogate objective function that also regulates the size of allowed changes to the current policy. These step sizes are theoretically  defined by \textit{policy improvement bounds} on the difference of infinite-horizon discounted or long-run average cost returns.  Policy improvement bounds dictate the magnitude of policy changes that can guarantee monotonic improvement each policy iteration.

The authors of  the APG algorithms designed them to solve finite state space MDP problems with  the episodic and discounted
formulations. Inspired by applications in the \textit{stochastic processing networks} domain,  we explore how the APG algotithms can be generalized for more classes of control problems: MDPs with long-run average cost objectives,  MDPs with countable state spaces, semi-Markov decision processes (SMDPs).

 Stochastic processing networks is a broad class of mathematical models that are used to  represent operations of service systems, industrial processes, computing and communication digital systems, see \cite{DaiHarrison2020}. These models are characterized by having  capacity constrained
processing resources and being operated to satisfy the needs of externally generated jobs.

In this thesis we study how to generalize the use  of the APG algorithms for processing network controls from theoretical and practical perspectives. We refine existing policy improvement bounds for MDPs with finite state spaces,  derive novel bounds for MDPs with countable state spaces and for SMDPs. These new results are foundations for theoretical justification of  the use  of the APG algorithms for processing network control problems and beyond. We extend the theoretical
framework of   the APG algorithms for  MDP problems with countable state spaces and long-run average cost objectives.

 We customize and test   PPO  for several examples of processing networks:  multiclass queueing networks, parallel-server system, ride-hailing transportation system.  Each system has required additional auxiliary modifications to the original PPO algorithm to obtain state-of-art results.

For multiclass queueing networks and  parallel-server system we combine and incorporate three variance reduction techniques   to improve estimation of the relative value function.  First, we use a discounted relative value function as
an approximation of the relative value function. Second, we propose regenerative simulation to estimate
the discounted relative value function. Finally, we incorporate the approximating martingale-process
method, first proposed in \cite{Henderson2002}, into the regenerative estimator. We also suggest automatic adjustment of policy and value neural network architectures to the size of a multiclass queueing network. Moreover,  we propose a novel proportionally randomized  policy as an initial policy when PPO is applied for multiclass queueuing networks. 
The resulting PPO algorithm is tested on a parallel-server system and large-size multiclass queueing networks. The algorithm consistently generates control policies that outperform state-of-art heuristics in literature in a variety of load conditions from light to heavy traffic. These policies are demonstrated to be near-optimal when the optimal policy can be computed.

We consider a ride-hailing order dispatching and drivers repositioning model proposed in \cite{Feng2020}. In \cite{Feng2020} the author reformulated a ride-hailing service optimization problem from \cite{Braverman2019} as an MDP problem and   suggested to solve it using  PPO. Due to scalability issues caused by the large action space, the original PPO could not be applied directly and a special actions decomposition technique was used. In this thesis, we analyze the role of this    special actions decomposition in the PPO scalability and argue why PPO continues to be theoretically justified. We also conduct additional numerical experiments to test the scalability of the proposed PPO algorithm and to verify the importance of  the empty-car routing in achieving higher driver-passenger matching rate.

\section{Outline of dissertation}\label{sec:outline}

In Chapter \ref{ch:1} we consider  queueing network control optimization problems.  A conventional setup for  such
problems is an MDP that has three features: infinite state space,
unbounded cost function, and long-run average cost objective.
  We extend the theoretical framework of APG algorithms for such MDP problems. The resulting PPO algorithm is tested on a parallel-server system and large-size multiclass queueing networks. A key role of variance reduction techniques in estimating the relative
value function is discussed.

In Chapter \ref{ch:ride_hailing} we consider a ride-hailing service optimization problem. We justify the use of PPO algorithm to solve MDPs with incorporated actions decomposition. The proposed PPO is tested on a model of a  large-size transportation network. A role of  the empty-car routing     is discussed.

In Chapter \ref{ch:2} we focus on policy improvement bounds.
We refine the existing bound for MDPs with finite state spaces and propose novel  bounds for MDPs with countable state spaces and for SMDPs. These new policy improvement bounds are obtained by introducing a one-norm ergodicity coefficient.  Various obtained  bounds on the  one-norm ergodicity coefficient  help to uncover its dependency on the underlying system dynamics.

\section{Notation}\label{sec:notation}

   The set of real numbers is denoted by $\R$. The sets of nonnegative
   integers, nonnegative real numbers  are denoted by $\Z_+$, $\R_+$, respectively.
We use $\X$ to denote finite or countable discrete metric space.

For a vector $a$ and a matrix $A$,
   $a^T$ and $A^T$ denote their transposes.
For a vector $a$ from space $\X$, we use the following vector norms: $\|a\|_1 :=\sum\limits_{x\in \X}\abs{a(x)}$ and  $\|a\|_\infty:=\max\limits_{x\in \X}|a(x)|$. For a matrix $A$ from space $\X\times\X$, we define the following induced operator norms:
\begin{align*}
\|A\|_1 :&= \sup\limits_{x\in \X : \|x\|_1\neq 0} \frac{\|Ax\|_1}{\|x\|_1} = \max\limits_{y\in \X}\sum\limits_{x\in \X}\abs{A(x,y)}
\end{align*}
 and $\norm{A}_\infty := \max\limits_{x\in \X}\sum\limits_{y\in \X}\abs{A(x,y)}$. We note that $\norm{A^T}_1 = \|A\|_\infty$ and $\|A^T\|_\infty = \|A\|_1$.

\chapter{Queueing Network Controls via Deep Reinforcement Learning}\label{ch:1}

For more than 30 years, one of the most difficult problems in applied
probability and operations research is to find a scalable algorithm
for approximately solving the optimal control of stochastic processing
networks, particularly when they are heavily loaded. These optimal
control problems have many important applications including healthcare
\cite{Dai2019a} and communications networks \cite{SrikYing2014, Luong2019},
data centers \cite{McKeown1999,Maguluri2012}, and manufacturing systems
\cite{Perkins1989, Kumar1993}. Stochastic processing networks are a broad class of
models that were advanced in \cite{Harr2000} and \cite{Harr2002} and
recently recapitulated and extended in \cite{DaiHarrison2020}.

In this chapter, we demonstrate that a class of \emph{deep reinforcement
  learning} algorithms known as proximal policy optimization (PPO),
 generalized from \cite{Schulman2015, Schulman2017} to our
  setting, can generate control policies that consistently beat the
performance of all state-of-arts control policies known in the
literature. The superior performances of our control policies appear
to be robust as stochastic processing networks and their load
conditions vary, with little or no problem-specific configurations of
the algorithms.

Multiclass queueing networks (MQNs) are a special class of stochastic
processing networks.  They were introduced in \cite{Harr1988} and have
 been studied intensively  for more than 30 years for performance
analysis and controls; see, for example, \cite{Harrison1993,
  Kumar1994, Bertsimas1994, Bramson1998, Will1998,
  Bertsimas2011, Bertsimas2015, Chen1999, Henderson2003,
  Veatch2015}.  Our paper focuses primarily on MQNs  with long-run average  cost objectives for two reasons.
First, these stochastic control problems  are
notoriously difficult due to the size of the state space, particularly
in heavy traffic. Second,  a large body of research has motivated the development of  various algorithms and control policies that are based on
either heavy traffic asymptotic analysis or heuristics. See, for
example, fluid policies \cite{Chen1993}, BIGSTEP policies
\cite{Harr1996}, affine shift policies \cite{Meyn1997},
discrete-review policies \cite{Maglaras2000, Ata2005}, tracking
policies \cite{Bauerle2001}, and ``robust fluid''
policies~\cite{Bertsimas2015} in the former group and
\cite{Lu1994,Kumar1996} in the latter group.  We
demonstrate that our algorithms outperform the state-of-art algorithms
in \cite{Bertsimas2015}. In this section, we will also consider an
$N$-model that belongs to the family of parallel-server systems,
another special class of stochastic processing networks.  Unlike an MQN
in which routing probabilities of jobs are fixed, a  parallel-server
system allows dynamic routing of jobs in order to achieve
load-balancing among service stations.  We demonstrate that our
algorithm achieves near optimal performance in this setting, again
with little special configuration of  them.

For the queueing networks with Poisson arrival and exponential service
time distribution, the control problems can be modeled within the
framework of Markov decision processes (MDPs) \cite{Puterman2005} via
uniformization \cite{Serfozo1979}.  A typical algorithm for solving an MDP is via \emph{policy
  iteration} or \emph{value iteration}.  However, in our setting, the corresponding MDP suffers
from the usual curse of dimensionality: there are a large number of
job classes, and the buffer capacity for each class is assumed to be
infinite. Even with a truncation of the buffer capacity either as a
mathematical convenience or a practical control technique, the
resulting state space is huge when the network is large and heavily
loaded.

In recent years the Proximal Policy Optimization (PPO)
algorithm \cite{Schulman2017} has become a default algorithm \cite{Schulman2017a} for control optimization  in  new  challenging environments including robotics \cite{Akkaya2019}, multiplayer video games \cite{Vinyals2019, OpenAI2019a}, neural network architecture engineering \cite{Zoph2018}, molecular design \cite{Simm2020}.
In this section we extend the PPO algorithm to MDP problems with \emph{unbounded cost} functions and
\emph{long-run average cost} objectives. The original PPO algorithm
\cite{Schulman2017} was proposed for problems with \emph{bounded cost}
function and \emph{infinite-horizon discounted} objective. It was
based on the trust region policy optimization (TRPO) algorithm developed
in \cite{Schulman2015}. We use two separate feedforward neural
networks, one for parametrization of the control policy and the other
for value function approximation, a setup common to \emph{actor-critic}
algorithms \cite{Mnih2015}.  We propose an approximating martingale-process
(AMP) method for variance reduction to estimate policy value functions
and show that the AMP estimation  speeds up convergence of the PPO algorithm in the model-based
setting.  We provide a set of instructions for implementing the PPO
algorithm specifically for MQNs. The instructions include the choice of initial stable
randomized policy, methods for improving   policy value function estimation,  architecture of the value and policy neural
networks, and the choice of hyperparameters.  The proposed instructions can be potentially adapted to other advanced policy optimization RL algorithms, e.g. TRPO. Given the success of the PPO in various domains and its ease of use   we  focus on the PPO algorithm in this chapter to illustrate efficiency of the deep RL framework for queueing control optimization.

The actor-critic methods can be considered as a combination of value-based and policy based methods.  In a
\textit{value-based}  approximate dynamic programming (ADP) algorithm, we may assume a low-dimensional
approximation of the optimal value function (e.g.  the optimal value function is a linear combination of known
\textit{features} \cite{DeFarias2003a, Ramirez-Hernandez2007,
  Abbasi_Yadkori2014}).   Value-based ADP algorithm has dominated
in the stochastic control of queueing networks literature;
see, for example, \cite{Chen1999, MoalKumaVanR2008,
  Chenetal2009, Veatch2015}.  These algorithms however have not achieved
robust empirical success for a wide class of control problems. It is
now known that the optimal value function might have a complex
structure which is hard to decompose on features, especially if
decisions have effect over a long horizon \cite{Lehnert2018}.

In a \textit{policy-based} ADP algorithm, we aim to learn the
optimal policy directly. Policy gradient algorithms are used to optimize the objective value within a
parameterized family of policies via gradient descent; see, for example,
\cite{Marbach2001, Paschalidis2004, Peters2008}. Although they are particularly effective for problems with high-dimensional action
space, it might be difficult to  reliably estimate the gradient of
the value under the current policy.  A direct sample-based estimation
typically suffers from high variance in gradient estimation \cite[Section 3]{Peters2008}, \cite[Section 5]{Marbach2001}, \cite[Section 1.1.2]{Baxter2001}.
Thus, actor-critic methods have been proposed \cite{Konda2003} to estimate the value function and use it
as a baseline and bootstrap for gradient direction approximation. The
actor-critic method with Boltzmann parametrization of policies and
linear approximation of the value functions has been applied for
parallel-server system control in \cite{Bhatnagar2012}.
 The standard policy gradient methods typically perform one gradient update per data sample which yields poor data
efficiency, and robustness, and an attempt to use a finite batch of samples to estimate the gradient and perform multiple steps of optimization ``empirically ... leads to destructively large policy
updates'' \cite{Schulman2017}. In \cite{Schulman2017}, the authors also note that the deep Q-learning algorithm \cite{Mnih2015} ``fails on many simple problems''.

 In \cite{Schulman2015,Schulman2017}, the authors propose
  \textit{``advanced policy gradient''} methods to overcome the
  aforementioned problems by designing novel objective functions that constrain
the magnitude of policy updates to avoid performance collapse caused by
large changes in the policy.  In \cite{Schulman2015} the authors prove
that minimizing a certain surrogate objective function guarantees
decreasing the expected discounted cost. Unfortunately,  their
theoretically justified step-sizes of policy updates cannot be
computed from available information for the RL algorithm.  Trust
Region Policy Optimization (TRPO) \cite{Schulman2015} has been
proposed as a practical method to search for step-sizes of policy
updates, and  Proximal Policy Optimization (PPO) method \cite{Schulman2017} has been proposed to compute
these step-sizes based on a clipped, ``proximal'' objective function.

We summarize the major contributions of our study:
\begin{enumerate}
\item In Section \ref{sec:countable} we theoretically justify that the advanced policy gradient algorithms can
  be applied for long-run average cost MDP problems with countable state spaces and unbounded
  cost-to-go functions. We show that starting from a stable policy it
  is possible to improve long-run average performance with
  sufficiently small changes to the initial policy.

\item In Section \ref{sec:AMP} we discuss a new way to estimate relative value
  and advantage functions if transition probabilities are known. We
  adopt the approximating martingale-process method
  \cite{Henderson2002} which, to the best of our knowledge, has not
  been used in simulation-based approximate
    policy improvement setting.
  \item  In Section \ref{sec:ge} we introduce a biased estimator
      of the relative value function through discounting the future
      costs. We interpret the discounting as the modification to the
      transition dynamics that shortens the regenerative cycles.
      We propose a regenerative estimator of the   discounted relative value function.

 The discounting  combined with the AMP method and regenerative simulation significantly reduces the variance of the relative value function estimation at the cost of a tolerable bias. The use of the proposed variance reduction techniques  speeds up the learning process of the PPO algorithm that we demonstrate by  computational experiments in Section
  \ref{sec:cc}.

\item In Section   \ref{sec:experiments} we conduct extensive computational experiments for multiclass queueing networks and parallel
  servers systems.  We propose to choose architectures of neural
  networks automatically as the size of a queueing network varies. We
  demonstrate the effectiveness of these choices as well as other
  hyperparameter choices such as  the  learning rate used in gradient
  decent. We demonstrate that the performance of control policies resulting
  from the proposed PPO algorithm outperforms other heuristics.

   \end{enumerate}

This chapter is based on the research presented  in \cite{DaiGluzman2021}.

\section{Control of multiclass queueing networks}\label{sec:MQN}

In this section we formulate the   control problems
for multiclass processing networks. We first give a   formulation for the criss-cross network, which serves as an example, and then give a
formulation for a general multiclass queueing network.

\subsection{The criss-cross network}
The criss-cross network has been studied in
\cite{Harrison1990},  \cite{Avram1995}, and
  \cite{Martins1996} among others. The network  depicted in Figure
  \ref{fig:cc}, which is taken from Figure~1.2 in \cite{DaiHarrison2020}, consists of two stations that process three classes
of jobs. Each job class has its own dedicated buffer where jobs wait
to be processed. All buffers are assumed to have an infinite capacity.

\begin{figure}[H]
  \centering
\begin{tikzpicture}  [scale=1.0,inner sep=2.0mm]
\node (T0) at (14,4) {};
\node[server] (S1) [left=1in of T0] {$S_1$};
\node[buffer] (B1) [left=.1in of S1] {$B_1$};
\node[buffer] (B2) [right=.5in of S1] {$B_2$};
\node[server] (S2) [right=.1in of B2] {$S_2$};
\node[vbuffer] (B3) [above= .1in of S1] {$B_3$};
\node (source1) [left= .5in of B1] {};
\node (sink2) [right= .5in of S2] {};
\node (sink3) [below= .3in of S1] {};
\node (source2) [above= .3in of B3] {};

\draw[thick,->] (B1) -- (S1);
\draw[thick,->] (B3) -- (S1);
\draw[thick,->] (S1) -- (B2);
\draw[thick,->] (B2) -- (S2);

\draw[thick,->] (source1) --(B1) node [above, align=left, near start]{class 1\\ arrivals};
\draw[thick,->] (source2) --(B3) node [right, align=left, near start]{class 3\\ arrivals};
\draw[thick,->] (S2) --(sink2) node [above, align=left, near end]{\mbox{\ \ } class 2 \\  \mbox{\ \ }  departures};
\draw[thick,->] (S1) --(sink3) node [left, align=left,  near end]{\ class 3 \\ \  departures};
 \end{tikzpicture}
  \caption{The criss-cross network.}
  \label{fig:cc}
\end{figure}
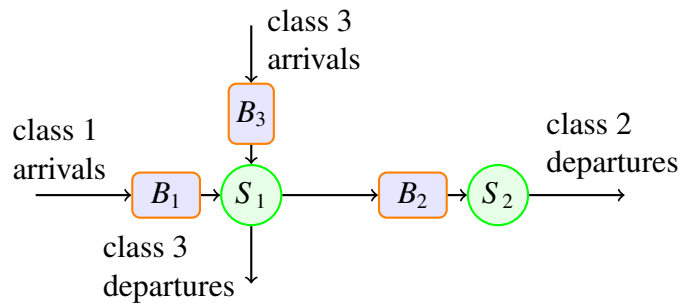

We assume that the jobs of class 1 and class 3 arrive to the
system following Poisson processes with rates $\lambda_1$ and
 $\lambda_3$, respectively. Server 1 processes both classes one job at a time. After being processed class 1 jobs become  class 2 jobs and wait in buffer $2$ for server 2 to process. Class 2 and class 3
jobs   leave the system after their processings
are completed.  We assume that the processing times for class j jobs are i.i.d.,  having exponential
distribution with mean $m_j$, $j=1, 2, 3$. We denote $\mu_j := 1/m_j$ as
the service rate of class $j$ jobs. We assume that the following load conditions are satisfied:
\begin{align}\label{eq:load_cc}
\lambda_1m_1+\lambda_3m_3<1  ~~\text{and}~~ \lambda_1m_2<1.
\end{align}
Again, we assume each server processes one job at a time. Therefore,
processor sharing among multiple jobs is not allowed for each
server. Jobs within a buffer are processed in the first-in--first-out
order.  A service policy dictates the order in which jobs from
different classes are processed. (See below for a precise definition of a stationary Markov policy.) We assume a decision
time occurs when a new job arrives  or a service is
completed.  For concreteness, we adopt a preemptive service policy --- suppose the service policy dictates that server $1$
  processes a class $3$ job next while it is in the middle of processing
  a class $1$ job, the server preempts the unfinished class $1$ job to start the processing of the leading class $3$ job from buffer $3$. Due to the memoryless property of  an exponential distribution, it does not matter whether the preempted job keeps its remaining processing time or is assigned a new service time sampled from the original exponential distribution.

  Under any service policy,
at each decision time, the system manager needs to simultaneously choose  action $a_1$ from set $\{0, 1, 3\}$ for server $1$ and  action $a_2$ from the set $\{0, 2\}$ for server 2; for server $1$ to choose action $j$, $j=1, 3$, means that   server 1 processes a class $j$ job next (if buffer $j$ is non-empty), and to choose action $0$ means   that server 1 idles. Similarly,
for server $2$ to choose action $2$ means that server $2$ processes a class $2$ job, and to choose action $0$ means  server $2$ idles.  Each server is allowed to choose  action $0$ even if there are waiting jobs at the associated buffers.
Therefore, our service policies are not necessarily non-idling. We define the action set as $\A =\left\{ (a_1, a_2)\in  \{0, 1, 3\}\times \{0, 2\}\right\}$.

The service policy  is assumed to be \textit{randomized}.  By
randomized, we mean each server takes a random action sampled from  a
certain distribution on the action set.  For a set $A$, we use
$\mathcal{P}(A)$ to denote the set of probability distributions on
$A$. Therefore, for a pair $(p_1,p_2)\in \mathcal{P}(\{0,1,3\})\times \mathcal{P}(\{0,2\})$, server $1$ takes a random action sampled from distribution $p_1$ and server $2$ takes  a random action sampled from distribution $p_2$.  For notational convenience, we note that a pair has a one-to-one correspondence to a vector $u$ in the following set
\begin{align*}
  \mathcal{U}=\left\{u= \left(u_1, u_2, u_3\right)\in \R^3_+: u_1+u_3\le 1 \text{ and } u_2\le 1\right\},
\end{align*}
where $p_1 = (1-u_1-u_3, u_1, u_3)\in \mathcal{P}(\{0, 1, 3\})$ is a
probability distribution on the action set $\{0, 1, 3\}$ and
$p_2 = (1-u_2, u_2)\in \mathcal{P}(\{0, 2\})$ is a probability
distribution on the action set $\{0,2\}$. Throughout this subsection, we use
$u\in \mathcal{U}$ to denote  pair $(p_1, p_2)$.

To define a randomized stationary Markovian service policy, let $x_j(t)$ be the number of class $j$ jobs (including possibly one job in service) in the system  at
time $t$, $j=1,2, 3$.  Then $x(t) = \left(x_1(t), x_2(t), x_3(t)\right )$ is
  the vector of jobcounts at time $t$. Clearly, $x(t)\in \Z_+^3$. By convention, we assume the sample path of $\{x(t), t\ge 0\}$ is right continuous, which implies that when $t$ is a decision time (triggered by an external arrival  or a service completion), $x(t)$ has taken into account the  arriving job or completed job at time $t$.

By a randomized stationary Markovian service policy we denote a map
\begin{align*}
  \pi: \Z_+^3 \to \mathcal{U}.
\end{align*}
Given this map $\pi$, at each decision time $t$, the system manager
observes jobcounts $x=x(t)$, computes $\pi(x)\in \mathcal{U}$ and the
corresponding pair
$(p_1(x), p_2(x))\in \mathcal{P}(\{0, 1, 3\})\times \mathcal{P}(\{0,
2\})$. Then server $1$ chooses a random action sampled from $p_1(x)$ and
server $2$ chooses a random action sampled from $p_2(x)$. When each
distribution is concentrated on a single action, the corresponding policy is
a \emph{deterministic} stationary Markovian service policy. Hereafter, we use the term  stationary Markovian service policies to mean  randomized
policies, which include deterministic service policies as special cases.

Under a stationary Markovian policy $\pi$, $\{x(t), t\ge 0 \}$ is a
continuous time Markov chain (CTMC). Hereafter, we call  jobcount
vector $x(t)$ the state at time $t$, and we denote the state space as
$\X = \Z^3_+$.  The objective of our optimal control problem is to find
a stationary Markovian policy that minimizes the long-run average
number of jobs in the network:
\begin{equation}\label{co}
\inf_{\pi} \lim_{T\rightarrow \infty} \frac{1}{T} {\E}_\pi \int\limits_0^T\Big(x_1(t)+x_2(t)+x_3(t)\Big)dt.
\end{equation}
 Because the interarrival and service times are exponentially
  distributed, the optimal control problem (\ref{co}) fits the
  framework of semi-Markov decision process (SMDP). See, for example,
  \cite[Chapter 11]{Puterman2005}.  Indeed, one can easily verify that
   at each state $x$, taking action $a$, the distribution of the
  time interval until next decision time is exponential with rate
  $\beta(x,a)$ to be specified below. We use $P(y|x,a)$ to denote the
  transition probabilities of the embedded Markov decision process,
  where $y\in \X$ is a state at the next decision time. For any
  state $x\in \X$ and any action $a\in \A$, the following transition
  probabilities always hold
\begin{align}
   P\big((x_1+1, x_2, x_3)|x, a\big)=\frac{\lambda_1}{\beta(x,a)}, \quad
   P\big((x_1, x_2, x_3+1)|x,a \big)=\frac{\lambda_3}{\beta(x,a)}. \label{eq:arrivals}
\end{align}
In the following, we specify $\beta(x,a)$ for each $(x,a)\in \X\times \A$ and additional
transition probabilities.  For action $a=(1,2)$ and state
$x=(x_1, x_2, x_2)$ with $x_1\ge 1$ and $x_2\ge 1$,
$\beta(x,a)=\lambda_1+\lambda_3+\mu_1+\mu_2$,
\begin{align*}
  &  P\big((x_1-1, x_2+1, x_3)|x,a \big)=\frac{\mu_1}{\beta(x,a)}, \quad
    P\big((x_1, x_2-1, x_3)|x,a\big)=\frac{\mu_2}{\beta(x,a)};
\end{align*}
for action $a=(3,2)$ and state $x=(x_1, x_2, x_2)$ with $x_3\ge 1$ and
$x_2\ge 1$, $\beta(x,a)=\lambda_1+\lambda_3+\mu_3+\mu_2$,
\begin{align*}
    P\big((x_1, x_2, x_3-1)|x,a \big)=\frac{\mu_3}{\beta(x,a)}, \quad
    P\big((x_1, x_2-1, x_3)|x,a\big)=\frac{\mu_2}{\beta(x,a)};
\end{align*}
for action $a=(0,2)$ and state $x=(x_1,x_2, x_3)$ with $x_2\ge 1$, $\beta(x,a)=\lambda_1+\lambda_3+\mu_2$,
\begin{align*}
    P\big((x_1, x_2-1, x_3)|x,a\big)=\frac{\mu_2}{\beta(x,a)};
\end{align*}
for action $a=(1,0)$ and state $x=(x_1, x_2, x_3)$ with $x_1\ge 1$, $\beta(x,a)=\lambda_1+\lambda_3+\mu_1$,
\begin{align*}
   P\big((x_1-1, x_2+1, x_3)|x,a\big)=\frac{\mu_1}{\beta(x,a)};
\end{align*}
for action $a=(3,0)$ and state $x=(x_1, x_2, x_3)$ with $x_3\ge 1$, $\beta(x,a)=\lambda_1+\lambda_3+\mu_3$,
\begin{align*}
  P\big((x_1, x_2, x_3-1)|x,a\big)=\frac{\mu_3}{\beta(x,a)};
\end{align*}
for action $a=(0,0)$ and state $x=(x_1, x_2, x_3)$,
$\beta(x,a)=\lambda_1+\lambda_3$. Also, $x_2=0$ implies that
$a_2=0$, $x_1=0$ implies that $a_1\neq 1$,  and  $x_3=0$ implies that $a_1\neq 3$.

Because the time between state transitions are exponentially
distributed, we   adopt the method of uniformization for solving the SMDP;
see, for example,  \cite{Serfozo1979} and \cite[Chapter 11]{Puterman2005}.
We denote
\begin{align}
  \label{eq:B}
B=\lambda_1+\lambda_3+\mu_1+\mu_2+\mu_3.
\end{align}
 For the new control
problem under uniformization, the decision times are determined by the
arrival times of a Poisson process with (uniform) rate $B$ that is independent of the underlying state. Given current state $x\in \X$
and action $a\in \A$, new transition probabilities into $y\in \X$ are given by
\begin{align}\label{eq:unif}
\tilde P(y|x, a) =
\begin{cases}
P(y|x, a) \beta(x, a) / B \quad \text{ if } x\neq y,\\
1 - \beta(x, a) / B  \quad \text{otherwise.}
\end{cases}
\end{align}
The transition probabilities $\tilde P$ in (\ref{eq:unif}) will define
a (discrete time) MDP. The  objective is given by
\begin{align}\label{eq:co1}
 \eta_\pi:= \inf\limits_{\pi} \lim\limits_{N\rightarrow \infty}\frac{1}{N}\E_{\pi}\left[ \sum\limits_{k=0}^{N-1}\left(x_1^{(k)}+x_2^{(k)}+x_3^{(k)}\right)\right],
\end{align}
where $\pi$ belongs to the family of stationary Markov policies,
and $x^{(k)} = \left(x_1^{(k)}, x_2^{(k)}, x_3^{(k)}\right)$ is the state (vector of jobcounts)
at the time of the $k$th decision (in the uniformized framework).
Under a stationary Markov policy $\pi$, $\{x^{(k)}:k=0, 1, 2, \ldots\}$ is a discrete time Markov chain (DTMC).

The existence of a stationary Markovian policy $\pi^*$ that minimizes
(\ref{eq:co1}) follows from \cite[Theorem 4.3]{Meyn1997} if the load conditions (\ref{eq:load_cc}) are satisfied. Under a mild
condition on $\pi^*$, which can be shown to be satisfied following the
argument in \cite[Theorem 4.3]{Meyn1997}, the policy $\pi^*$ is an
optimal Markovian stationary policy for (\ref{co}) \cite[Theorem
2.1]{Beutler1987}.  Moreover, under policy $\pi^*$, the objective in
(\ref{co}) is equal to  that in (\ref{eq:co1}); see
\cite[Theorem 3.6]{Beutler1987}.

\subsection{General formulation of  a  multiclass queueing network control problem}

We consider a multiclass queueing network that has $L$ stations  and $ J$  job classes.   For
notational convenience we denote $\LL = \{1, ..., L\}$ as  the set
  of stations, $\J = \{1, ...,J \}$ as  the  set of
job classes.
Each station has a single server that
processes jobs from the job classes that belong   to the station.  Each job class belongs to one  station. We use
$\ell= s(j)\in \LL$ to denote the station that class $j$ belongs to.
We assume the function $s: \J \to \LL$ satisfies $s(\J)=\LL$.
Jobs arrive externally  and are processed sequentially at various stations, moving from one class to  the next after each processing step until they exit the network.
Upon arrival if a class $j$ job finds the  associated server busy, the job waits in the corresponding buffer $j$.  We assume that every buffer has an infinite capacity.
For each station $\ell\in \LL$, we define
  \begin{align}\label{def:servers}
 \mathcal{B}(\ell) := \{j\in \J:~ s(j)  = \ell\}
 \end{align} as the set of job classes to be processed by server $\ell$.

Class $j\in \J$ jobs arrive externally to buffer $j$ following a Poisson
 process with rate $\lambda_j$; when $\lambda_j=0$, there are no external arrivals into buffer $j$.
   Class $j$ jobs are processed
 by server $s(j)$ following the service policy as specified below.
 We assume the service times for class $j$ jobs are i.i.d.  having exponential distribution with  mean $1/\mu_j$. Class $j$
 job, after being processed by server $s(j)$, becomes class $k\in \J$ job with probability $r_{jk}$ and leaves the network with probability
  $1 - \sum\limits_{k=1}^K r_{jk}$. We define $J\times J$ matrix $R := (r_{jk})_{j,k=1,..., J}$ as the routing matrix. We assume that the network is open, meaning that $I-R$ is invertible.
We let vector $q= (q_1, q_2, ..., q_J)^T$ satisfy the system of linear
equations
\begin{equation}\label{eq:traffic}
  q = \lambda +R^T q.
\end{equation}
Equation (\ref{eq:traffic}) is known as the traffic equation, and it has a unique solution under the open network assumption.
  For each class $j \in \J$,  $q_j$ is interpreted to be the total arrival rate
  into buffer $j$, considering both the  external arrivals
  and internal arrivals from service completions at stations.
We define \textit{the load} $\rho_\ell$ of station $\ell\in\LL$ as
\begin{equation*}\rho_\ell := \sum\limits_{j\in \B(\ell)} \frac{q_j}{\mu_j}.\end{equation*}
We assume that
\begin{align}\label{eq:load}
\rho_\ell<1 \quad \text{for each  station }\ell\in \LL.
\end{align}

We let $x(t) = \left( x_1(t), .., x_J(t) \right)$ be the vector of jobcounts at time $t$. A decision time occurs when a new job arrives at the system or a service is completed.   Under any service policy, at each decision time, the system manager needs to simultaneously choose an action for each server $\ell\in \LL$.
For each server $\ell\in \LL$ the system manager selects an action from set $\B(\ell)\cup \{0\}$: action $j\in \B(\ell)$ means that  the system manager gives priority to  job class $j$ at station $\ell$; action $0$ means server $\ell$ idles until the next decision time.

We define  set
\begin{align}\label{eq:setP}
  \mathcal{U}=\left\{u= \left(u_1, u_2,..., u_J\right)\in \R^J_+:     \sum\limits_{j\in \B(\ell)} u_j\le1 \text{ for each }\ell\in \LL\right\}.
\end{align}

For each station $\ell\in \LL$ vector $u\in \mathcal{U}$ defines a probability distribution $u_\ell$ on the action set  $\B(\ell)\cup \{0\}$: probability of action $j$ is equal to $u_j$ for $j\in \B(\ell)$, and  probability of action $0$ is equal to $\Big(1 - \sum\limits_{j\in \B(\ell)} u_j\Big)$.
We define a randomized stationary Markovian service policy as a map from a set of  jobcount vectors into set $\mathcal{U}$ defined in (\ref{eq:setP}):
\begin{align*}
\pi:\Z_+^J\rightarrow \mathcal{U}.
\end{align*}

Given this map $\pi$, at each decision time $t$, the system manager
observes jobcounts $x(t)$, chooses $\pi(x)\in\mathcal{U}$, and based on $\pi(x)$ computes probability distribution $u_\ell$ for each $\ell\in \LL$.
Then  the system manager  independently samples  one action from $u_\ell$  for each server $\ell\in \LL$.

 The objective is to find a stationary Markovian policy  $\pi$ that   minimizes the  long-run average number of jobs in the network:
\begin{align}\label{eq:obj3}
\inf\limits_{\pi} \lim\limits_{T\rightarrow \infty} \frac{1}{T} \E_\pi \int\limits_0^T\left(\sum\limits_{j=1}^J x_j(t)\right)dt.
\end{align}

Under a stationary Markovian policy $\pi$, we adopt the method of uniformization to obtain a uniformized discrete time Markov chain (DTMC) $\{x^{(k)}:k=0, 1, \ldots \}$. We abuse the notation and denote a system state as $x^{(k)} = \left( x_1^{( k)}, x_2^{(k)}, \dotsc, x_J^{(k)} \right)$ after $k$ transitions of the DTMC.

In this chapter, we   develop  algorithms to approximately solve the following (discrete-time) MDP problem:
\begin{align}\label{eq:obj4}
 \eta_\pi := \inf\limits_\pi \lim\limits_{N\rightarrow \infty} \frac{1}{N} \E_{\pi}\left[ \sum\limits_{k=0}^{N-1}\sum\limits_{j=1}^J x_j^{(k)} \right].
\end{align}
\begin{remark}

It has been proved in \cite{Meyn1997} that the MDP (\ref{eq:obj4}) has an optimal policy
that satisfies the conditions in \cite[Theorem 3.6]{Beutler1987} if the ``fluid limit model'' under
some policy is $L_2$-stable.  Under the load condition (\ref{eq:load}),
conditions in \cite{Meyn1997} can be verified as follows. First, we adopt the randomized
version of the head-of-line static processor sharing (HLSPS) as defined in
\cite[Section 4.6]{DaiHarrison2020}. We apply this  randomized  policy to the
discrete-time MDP to obtain the resulting DTMC.  The fluid limit path
of this DTMC can be shown to satisfy the fluid model defined in  \cite[Definition 8.17]{DaiHarrison2020} following a procedure that is similar to,
but much simpler than, the proof of \cite[Theorem 12.24]{DaiHarrison2020}. Finally, \cite[Theorem 8.18]{DaiHarrison2020} shows the fluid model is stable, which is
stronger than the $L_2$-stability needed.

\end{remark}

\section{Reinforcement learning approach for queueing network control}\label{sec:countable}


 Originally, policy gradient algorithms have been developed  to find
optimal policies which optimize the finite horizon total cost or
infinite horizon discounted total cost objectives. For
  stochastic processing networks and their applications, it is often
   useful  to optimize the long-run average cost. In this section
we develop a version of the  Proximal Policy Optimization algorithm  for the long-run average cost objective.  See Section \ref{sec:experiments}, which demonstrates the effectiveness of our proposed PPO for finding the  near optimal control policies for stochastic processing networks.

\subsection{Positive recurrence and $\V$-uniform ergodicity}\label{sec:MC}
As discussed in Section~\ref{sec:MQN}, operating under a fixed
randomized stationary control policy, the dynamics of a stochastic
processing network is a DTMC. We restrict policies so that the
resulting  DTMCs are  irreducible and aperiodic. Such a DTMC does not
always have a stationary distribution.  When the DTMC does not have a
stationary distribution, the long-run average cost of the corresponding policy is not well-defined, leading to necessarily poor performance. It is well known that an irreducible DTMC has a unique stationary distribution if and only if it is positive recurrent.
Hereafter, when the DTMC is positive recurrent, we call
the corresponding control policy \emph{stable}. Otherwise, we call it
\emph{unstable.}

A sufficient condition for an irreducible DTMC to be positive recurrent is
the Foster-Lyapunov drift condition.  The drift condition
(\ref{eq:drift}) in the following lemma is stronger than the classic
Foster-Lyapunov drift condition.
For a proof of the lemma, see  Theorem 11.3.4 and Theorem 14.3.7 in \cite{Meyn2009}.

\begin{lemma}\label{lem:drift}
Consider an irreducible Markov chain on a countable state space $\X$ with a transition matrix $P$ on $\X\times \X$.  Assume   there exists a vector $\V:\X\rightarrow [1, \infty)$ such that the following drift condition holds for some constants  $\varepsilon\in (0,1)$ and $b\geq0$, and a finite subset $C\subset \X$:
\begin{align}\label{eq:drift}
\sum\limits_{y\in \X}P(y|x)\V(y)\leq \varepsilon  \V(x) + b \I_{C}(x), \quad \text{for each }x\in \X,
\end{align}
where $\I_{C}(x)=1$ if $x\in C$ and $\I_{C}(x)=0$ otherwise. Here,
$P(y|x)=P(x,y)$ is the transition probability from state $x\in \X$ to
state $y\in \X$. Then
(a)  the Markov chain with the transition matrix  $P$ is positive recurrent with a unique stationary distribution  $d$;  and  (b)
$d^T\V <\infty,
$
where  for any function $f: \X \rightarrow \R$  we define $d^T f$ as
\begin{align*}
d^Tf:=\sum\limits_{x\in \X} d(x)f(x).
\end{align*}

\end{lemma}

Vector $\V$ in the drift condition (\ref{eq:drift}) is called a
\textit{Lyapunov function} for the Markov chain.  We define $\V$-norm of a vector $\nu$ on $\X$ as
\begin{align}\label{eq:Vnorm}
 \|\nu\|_{\infty, \V} := \sup\limits_{x\in \X} \frac{| \nu(x)|}{\V(x)},
\end{align}
where $\V:\X\rightarrow [1, \infty)$.
For any matrix $M$ on $\X\times \X$, its induced $\V$-norm is defined to be
 \begin{align}\label{eq:norn_equv}
   \|M\|_\V^{ }: &= \sup\limits_{\nu:\|\nu\|_{\infty, \V}=1} \|M\nu\|_{\infty, \V}\nonumber\\
&= \sup\limits_{x\in \X} \frac{1}{\V(x)}\sum\limits_{y\in \X}|M(x, y)| \V(y).
\end{align}
The proof of equality (\ref{eq:norn_equv}) can be found in Lemma \ref{lem:norms} in Appendix \ref{sec:proofs}.

An irreducible, aperiodic Markov chain with transition matrix $P$ is
called \textit{$\V$-uniformly ergodic} if
\begin{align*}\|P^n - \Pi\|_\V^{ }\rightarrow 0 \text{ as } n\rightarrow
  \infty,\end{align*} where every row of $ \Pi$  is equal  to the
stationary distribution $d$, i.e. $\Pi(x, y): = d(y),$ for any
$x, y\in \X$.  The drift condition (\ref{eq:drift}) is sufficient and
necessary for an irreducible, aperiodic Markov chain to be
$\V$-uniformly ergodic \cite[Theorem 16.0.1]{Meyn2009}.  For an
irreducible, aperiodic Markov chain that satisfies (\ref{eq:drift}),
for any $g:\X\to \R$ with $\abs{g(x)}\le \V(x)$ for $x\in \X$, there
exist constants $R<\infty$ and $r<1$ such that
\begin{align}\label{eq:geo}
\left| \sum\limits_{y\in \X} P^n(y|x) g(y) - d^T g  \right| \leq R\V(x) r^n
\end{align}
for any $x\in \X$ and $n\geq 0$; see   \cite[Theorem 15.4.1]{Meyn2009}.

\subsection{Poisson equation}\label{sec:PO}
For an irreducible DTMC on state space $\X$ (possibly infinite) with transition matrix $P$, we
assume that there exists a Lyapunov function $\V$ satisfying (\ref{eq:drift}).
For any cost function $g:\X \rightarrow \R$ satisfying $\abs{g(x)}\le \V(x)$ for each $x\in \X$, it follows from Lemma~\ref{lem:drift} that $d^T |g|<\infty$.
Lemma~\ref{lem:poisson_sol} below asserts that the following equation
has a solution $h:\X\to\R$:
\begin{align}\label{eq:Poisson}
g(x) - d^T g + \sum\limits_{y\in \X}P(y|x) h(y) - h(x) =0 \quad \text{ for each }x\in \X.
\end{align}
Equation (\ref{eq:Poisson}) is called  a \textit{Poisson equation} of the Markov chain with transition matrix $P$, stationary distribution $d$, and cost function $g $. Function $h:\X\rightarrow \R$ that satisfies (\ref{eq:Poisson}) is called a \textit{solution} to the Poisson equation.  The solution is unique up to a constant shift, namely,
if $h_1$ and $h_2$ are two solutions to Poisson equation (\ref{eq:Poisson}) with $d^T(|h_1|+|h_2|)<\infty$, then there exists a constant $b\in \R$ such that $h_1(x) = h_2(x) +b$ for each $x\in \X$, see \cite[Proposition 17.4.1]{Meyn2009}.

A solution $h$ to the Poisson equation is called a \textit{fundamental solution} if   $ d^T h =0.$    The proof of the following lemma is provided in \cite[Proposition A.3.11]{Meyn2007}.

\begin{lemma}\label{lem:poisson_sol}

Consider a $\V$-uniformly ergodic  Markov chain with transition matrix $P$ and the stationary  distribution $d$.
 For any cost function $g:\X \rightarrow \R$ satisfying $|g|\leq \V$,  Poisson equation (\ref{eq:Poisson}) admits a fundamental  solution
\begin{align}\label{eq:h}
h^{(f)}(x) : = \E \left[\sum\limits_{k=0}^\infty \left(g(x^{(k)}) - d^Tg\right)~|~x^{(0)} = x\right] ~\text{for each }x\in \X,
\end{align}
where $x^{(k)}$ is the state of the Markov chain after $k$ timesteps.

\end{lemma}

We  define   \textit{fundamental matrix} $Z$ of the Markov chain with transition kernel $P$ as 
 \begin{align}\label{eq:Zs}Z :=\sum\limits_{k=0}^\infty \left( P-\Pi \right)^k. \end{align}
 It follows from  \cite[Theorem 16.1.2]{Meyn2009} that the series (\ref{eq:Zs}) converges in $\V$-norm and, moreover,  $\|Z\|^{ }_\V<\infty$.
Then, it is easy to see that fundamental matrix $Z$ is an inverse matrix of $(I-P+\Pi)$, i.e. $Z(I-P+\Pi) = (I-P+\Pi)Z = I$.
See Appendix Section \ref{sec:proofs} for the proof of the following lemma.
 \begin{lemma}\label{lem:Zeq}
Fundamental matrix $Z$   maps any cost function $|g|\leq \V$ into a corresponding fundamental solution $h^{(f)}$ defined by (\ref{eq:h}):
 \begin{align}h^{(f)} = Z\left(g -(d^T g) e\right) ,\end{align}
 where $e = (1, 1,.., 1, ...)^T$ is the unit vector.
 \end{lemma}

 \begin{remark}
  Consider matrices $A, B,C$ on $\X\times \X$.
The   associativity property,
\begin{align*}
 ABC =(AB)C= A(BC),
 \end{align*}
does not always hold for matrices defined on  a  countable state space; see a counterexample  in \cite[Section 1.1]{Kemeny1976}. However, if  $\|A\|^{ }_\V<\infty, \|B\|^{ }_\V<\infty, \|C\|^{ }_\V<\infty$ then matrices $A, B, C$ associate, see \cite[Lemma 2.1]{Jiang2017}.
Hence,  there is no ambiguity in the definition of the fundamental matrix (\ref{eq:Zs}):
\begin{align*}
\left( P-\Pi \right)^k =\left( P-\Pi \right) \left( P-\Pi \right)^{k-1}=\left( P-\Pi \right)^{k-1}\left( P-\Pi \right),~~\text{for }k\geq 1,
 \end{align*}
where $\|P-\Pi\|^{ }_\V<\infty$ holds due to the drift condition (\ref{eq:drift}).
\end{remark}

 \subsection{Improvement guarantee for average cost objective}\label{sec:TRPOforAC}

 We consider an  MDP problem with a countable state space $\X$, finite action space $\A$, one-step cost function $g(x)$, and  transition function $P(\cdot|x, a)$. For each  state-action pair $(x, a)$, we assume  that the chain can transit to a finite number of distinguished states, i.e. set $\{y\in \X: P(y| x, a)>0\}$ is finite for each $(x, a)\in \X\times \A.$

 We consider  $\Theta\subset \R^m$ for some integer $m>0$ and
  $\Theta$ is open.  With every $\phi\in \Theta,$ we associate a
randomized Markovian policy $\pi_{\phi}$, which at any state
$x\in \X$ chooses action $a\in \A$ with probability
$\pi_\phi(a|x)$. Under   policy $\pi_{\phi}$, the corresponding DTMC has transition matrix $P_{\phi}$ given by
\begin{align*} P_\phi(x, y) = \sum\limits_{a\in \A} \pi_\phi(a|x)P(y|x, a) \text{ for } x, y\in \X.
\end{align*}
For each $\phi\in \Theta$ we assume that the resulting Markov chain with transition probabilities $P_{\phi}$ is irreducible and aperiodic.

We  assume there exists $\phi \in \Theta$ such that the drift
condition (\ref{eq:drift}) is satisfied for the transition matrix
$P_{\phi}$ with a Lyapunov function $\V:\X\rightarrow [1, \infty).$ By
Lemma \ref{lem:poisson_sol} the corresponding fundamental matrix
$Z_{\phi}$ is well-defined. The following lemma says that if $P_{\phi}$ is positive
recurrent and $P_{\theta}$ is ``close'' to $P_{\phi}$, then $P_{ \theta}$ is
also positive recurrent.
See Appendix Section \ref{sec:proofs} for the proof.

\begin{lemma}\label{lem:st}
  Fix a $\phi \in \Theta$.
  We assume  that drift condition (\ref{eq:drift}) holds for $P_{\phi}$. Let some $  \theta\in \Theta$ satisfies,
\begin{align*}
\| (P_{  \theta} - P_{\phi}) Z_{\phi}\|_\V^{ } < 1 , \end{align*}
then the Markov chain with transition matrix $P_{  \theta}$ has a unique stationary distribution $d_{  \theta}$.
\end{lemma}

We assume  that drift condition (\ref{eq:drift}) holds for $P_{\phi}$.
For any cost function $|g|\le \V$, we denote
the corresponding fundamental solution to the Poisson equation as
$h_{\phi}$ and the long-run average cost
\begin{align}\label{eq:ac}
\eta_{\phi} := d_\phi^Tg  = \sum\limits_{x\in \X} d_\phi(x) g(x).
\end{align}

 The following theorem provides a bound on the difference of long-run average performance of policies $\pi_{\phi}$ and $\pi_{\theta}.$ See  Appendix Section \ref{sec:proofs} for the proof.

\begin{theorem}
\label{thm:main}
Suppose that the Markov chain with transition matrix $P_{\phi}$ is  an irreducible  chain such that the drift condition (\ref{eq:drift}) holds for some function $\V\geq1$ and the cost function satisfies $|g|\leq\V$.

 For any $\theta\in \Theta$ such that
\begin{align}\label{eq:D}
 D_{\theta,\phi} : = \|  (P_{\theta} - P_{\phi}) Z_{\phi}\|_\V^{ } < 1
\end{align}
the difference of long-run average costs of policies $\pi_{\phi}$ and $\pi_{\theta}$ is bounded by:
   \begin{align}\label{eq:ineq2}
\eta_{\theta} - \eta_{\phi} ~\leq~ &N_1(\theta, \phi)+  N_2(\theta, \phi),
\end{align}
where $N_1(\theta, \phi)$, $N_2(\theta, \phi)$ are finite and equal to
\begin{align}
N_1(\theta, \phi) &:= d_{\phi}^T( g -\eta_{\phi} e   +P_{\theta}h_{\phi} - h_{\phi} ), \label{eq:M1} \\
N_2(\theta, \phi) &:= \frac{ D_{\theta,\phi}^2}{1- D_{\theta, \phi}}  \left\| g - \eta_{\phi} e  \right\|^{ }_{\infty, \V}(d_{\phi}^T\V). \label{eq:M2}
\end{align}

\end{theorem}

It follows from Theorem \ref{thm:main} that the negativity of the right side of inequality (\ref{eq:ineq2}) guarantees   that policy $\pi_{\theta}$  yields an improved performance compared  with the initial policy $\pi_{\phi}.$
Since
\begin{align}\label{eq:min0}
 \min\limits_{\theta \in \Theta: ~D^{ }_{\theta, \phi}<1}[N_1(\theta, \phi) +N_2(\theta, \phi)] \leq N_1(\phi, \phi) +N_2(\phi, \phi) = 0,
\end{align}
we want to find  $\theta=\theta^*$:
\begin{align}
  \label{eq:argmin}
  \theta^* =\argmin \limits_{\theta \in \Theta: ~D_{\theta, \phi}<1}[N_1(\theta, \phi) +N_2(\theta, \phi)]
\end{align}
 to achieve the maximum improvement in the upper bound  (\ref{eq:ineq2}).    In the setting
of finite horizon and infinite discounted RL problems,
\cite{Kakade2002, Schulman2015} propose to  fix the maximum change
between policies $\pi_{\theta}$ and $\pi_{\phi}$ by bounding  the
$N_2(\theta, \phi)$ term and  to minimize  $N_1(\theta, \phi)$. Below, we discuss the motivation for developing the PPO algorithm proposed in Section \ref{sec:ppo},
leading to a practical algorithm to approximately solve optimization
(\ref{eq:argmin}).

 It follows from property (\ref{eq:min0}) that solution $\theta^*$ to ({\ref{eq:argmin}}) leads to policy $\pi_{\theta^*}$ which performance, at least, as good as performance of policy $\pi_\phi$, i.e. $\eta_{\theta^*}\leq \eta_\phi$.  

It is an open problem if a \textit{strict} improvement can be guaranteed for any suboptimal policy $\pi_\phi$ solving ({\ref{eq:argmin}}). 
Our conjecture is  there exists  a constant $C_\phi>0$ independent of $\theta$ such that for any $\phi,\theta\in \Theta$:
\begin{align}\label{eq:lower_b}
|N_1(\theta, \phi)|\geq C_\phi D_{\theta, \phi}. 
\end{align}
Since $N_2(\theta, \phi) = O(D^2_{\theta, \phi})$, this conjecture implies that, if $N_1({\theta, \phi})<0$ and $D_{\theta, \phi}$ is small enough,
$\pi_{\theta}$ is a strict improvement over $\pi_{\phi}$. While bound (\ref{eq:lower_b}) has not been verified, we note that the opposite bound holds
\begin{align*}
|N_1(\theta, \phi)|:&= \left|d_{\phi}^T( g -\eta_{\phi} e   +P_{\theta}h_{\phi} - h_{\phi} ) \right|\nonumber\\
 &\leq (d_\phi^T\V)\| g - \eta_{\phi} e  +P_{\theta}h_{\phi} - h_{\phi } \|^{ }_{\infty, \V}\nonumber\\
& =  (d_\phi^T\V)\| (P_{\theta} - P_{\phi})h_{\phi} \|^{ }_{\infty, \V}\nonumber\\
&= (d_\phi^T\V)\| (P_{\theta} - P_{\phi}) Z_{\phi}\left (g- \eta_{\phi} e\right)\|^{ }_{\infty, \V}\nonumber\\
&\leq (d_\phi^T\V )\| g -\eta_{\phi} e \|_{\infty, \V}^{ } D_{\theta, \phi}.
\end{align*}

 Lemma \ref{lem:policies}  shows
that the distance $D_{\theta, \phi}$ can be controlled by the
probability ratio
\begin{align}
  \label{eq:ratio}
 r_{\theta, \phi}(a|x): =\frac{\pi_{\theta}(a|x)}{ \pi_{\phi}(a|x)}
\end{align}
between the two policies.
\begin{lemma}\label{lem:policies}

Suppose that the Markov chain with transition matrix $P_{\phi}$ is  an irreducible  chain such that the drift condition (\ref{eq:drift}) holds for some function $\V\geq1$ and the cost function satisfies $|g|\leq\V$. Then for any $\theta\in \Theta$
   \begin{align*}
D_{\theta, \phi} \leq   \|Z_{\phi}\|^{ }_\V\sup\limits_{x\in \X}    \sum\limits_{a\in \A} \left| r_{\theta, \phi}(a|x)-  1 \right|    G_{\phi}(x, a),
\end{align*}
where $G_{\phi}(x, a): =   \frac{1}{\V(x)} \sum\limits_{y\in \X}  \pi_{\phi}(a|x) P(y|x, a) \V(y)$.
\end{lemma}

 Lemma \ref{lem:policies} implies that
$D_{\theta, \phi}$ is small when the ration  $r_{\theta, \phi}(a|x)$  in (\ref{eq:ratio}) is close to 1
for each state-action pair $(x, a)$. Note that
$r_{\theta, \phi}(a|x) = 1$ and $D_{\theta, \phi}=0$ when
$\theta=\phi.$ See Appendix Section \ref{sec:proofs} for the proof.

\subsection{Proximal policy optimization}\label{sec:ppo}

We rewrite the first term of the right-hand side of (\ref{eq:ineq2}) as:
\begin{align}\label{eq:MA}
N_1(\theta, \phi) & =  d_{\phi}^T(g -\eta_{\phi} e   +P_{\theta}h_{\phi} - h_{\phi} ) \nonumber \\
 & =  \underset{\substack{ x\sim d_{\phi}\\ a\sim \pi_{\theta}(\cdot|x) \\ y\sim P(\cdot|x, a)}  }{\E} \left[    g(x ) -\eta_{\phi} e   + h_{\phi}(y) - h_{\phi}(x)   \right]  \nonumber\\                  &
                    =  \underset{\substack{ x\sim d_{\phi}\\ a\sim \pi_{\theta}(\cdot|x)  }  }{\E} A_{\phi} (x, a)\nonumber\\
                  & =  \underset{\substack{ x\sim d_{\phi}\\ a\sim \pi_{\phi}(\cdot|x)  }  }{\E} \left[ \frac{\pi_{\theta}(a|x)}{\pi_{\phi}(a|x)}A_{ \phi} (x, a) \right]
 =  \underset{\substack{ x\sim d_{\phi}\\ a\sim \pi_{\phi}(\cdot|x)  }  }{\E}\Big[  r_{\theta, \phi}(a|x)A_{\phi} (x, a)\Big],
\end{align}
where we define an advantage function $A_{\phi}:\X\times \A\rightarrow \R$ of policy $\pi_{\phi}, $ $\phi\in \Theta$ as:
\begin{align}\label{eq:A}
	A_{\phi}(x, a): =   \underset{\substack{ y\sim P(\cdot|x, a)}  }{\E} \left[g(x) - \eta_{\phi} e + h_{\phi} (y) -h_{\phi} (x) \right].
\end{align}

Equation (\ref{eq:MA}) implies that if we want to minimize $N_1(\theta, \phi)$, then the ratio $r_{\theta, \phi}(a|x)$ should be minimized (w.r.t. $\theta$) when $A_{\phi}(x,a)>0$, and maximized when $A_{\phi}(x,a)<0$  for each $x\in \X$ .

The end of Section~\ref{sec:TRPOforAC} suggests
that we should strive to (a)
\begin{align}\label{eq:minM1}
\text{minimize } \quad N_1(\theta, \phi),
\end{align}
w.r.t. $\theta\in \Theta$ and (b) keep the ratio $r_{\theta, \phi}(a|x)$ in (\ref{eq:ratio})
close to 1. In \cite{Schulman2017} the authors propose to minimize
(w.r.t. $\theta\in \Theta$) the following clipped surrogate objective
\begin{align}\label{eq:PO}
L(\theta, \phi):=\underset{\substack{ x\sim d_{\phi}\\ a\sim \pi_{\phi}(\cdot|x)  }  }{\E}   \max \left[  r_{\theta, \phi}(a|x) A_{\phi} (x, a) ,  ~ \text{clip} (r_{\theta, \phi}(a|x),  1-\epsilon, 1+\epsilon)  A_{\phi} (x, a)  \right],
\end{align}
where  $\epsilon\in(0, 1)$ is a hyperparameter, and clipping function is defined as
\begin{align*}
  \text{clip}(c,  1-\epsilon, 1+\epsilon):=
\begin{cases}
1-\epsilon,~\text{if } c<1-\epsilon,\\
c, ~\text{if } c\in [1-\epsilon, 1+\epsilon],\\
 1+\epsilon,~\text{otherwise.}
\end{cases}
\end{align*}

 In \cite{Schulman2017} the authors  coined the term,
proximal policy optimization (PPO), for their algorithm, and
demonstrated its ease of implementation and its ability to find good control
policies.

The objective term
$\text{clip} (r_{\theta, \phi}(a|x), 1-\epsilon, 1+\epsilon) A_{\phi}
(x, a) $ in (\ref{eq:PO}) prevents changes to the policy that
move $ r_{\theta, \phi}(a|x) $ far from 1. Then the objective function
(\ref{eq:PO}) is a upper bound (i.e. a pessimistic bound) on the
unclipped objective (\ref{eq:minM1}).  Thus, an improvement on the
objective (\ref{eq:PO}) translates to an improvement on
$N_1(\theta, \phi)$ only when $\theta\in \Theta$ satisfies
$r_{\theta, \phi}\in (1-\epsilon, 1+\epsilon)$. The alternative
heuristics proposed in \cite{Schulman2015, Wang2016, Schulman2017, Wu2017} to solve
optimization problem (\ref{eq:argmin}): each defines a loss function that controls $N_2(\theta, \phi)$ and
minimizes the $N_1(\theta, \phi)$ term.
Following \cite{Schulman2017}, we use loss function
(\ref{eq:PO})  because of its implementation simplicity.

To compute objective function in (\ref{eq:PO}) we first evaluate the expectation and precompute advantage functions in (\ref{eq:PO}). We assume that an approximation $\hat A_\phi:\X\times \A\rightarrow \R$ of the advantage function (\ref{eq:A}) is available and focus on estimating the objective from simulations.
See Section \ref{sec:M1} below for estimating $\hat A_\phi$.



 Given an episode with length $N$ generated  under policy $\pi_\phi$  we compute  the advantage function estimates $\hat A_{\phi}\left(x^{({k})}, a^{({k})}\right)$   at the observed state-action pairs:
\begin{align*}
D^{(0:N-1)}: = \Big \{\left(  x^{(0)}, a^{(0)}, \hat A_{\phi}(x^{(0)}, a^{(0)})  \right),~
&\left (   x^{(1)}, a^{(1)}, \hat A_{\phi}(x^{(1)}, a^{(1)} )\right) ,\cdots,\\
&\left(x^{({N-1})}, a^{({N-1})}, \hat A_{\phi}(x^{({N-1})}, a^{({N-1})}) \right)\Big\},
   \end{align*}
   and estimate the loss function (\ref{eq:PO}) as a sample average over the state-action pairs from the episode:
 \begin{align}\label{eq:popt}
   \hat L\left(\theta, \phi,D^{(0:N-1)} \right) = \sum\limits_{ k=0}^{N-1} \max\Big[ &\frac{\pi_{\theta}\left( a^{({k})}| x^{({k})} \right)}{\pi_{\phi}\left( a^{({k})}|x^{({k})} \right)}  \hat A_{\phi}\left(x^{(k)}, a^{(k)}\right) ,\\
&\text{clip}\left(\frac{\pi_{\theta}\left(a^{(k)}| x^{(k)}\right)}{\pi_{\phi}\left(a^{(k)}|x^{(k)}\right)},     1-\epsilon, 1+\epsilon \right  ) \hat  A_{\phi}\left(x^{(k)}, a^{(k)}\right)  \Big].\nonumber
 \end{align}
 In theory, one long episode  under policy $\pi_\phi$  starting from any initial state $x^{(0)}$
 is sufficient because the following SLLN for Markov chains holds:
 with probability $1$,
 \begin{align*}
   \lim_{N\to\infty} \frac{1}{N}   \hat L\left(\theta, \phi,D^{(0:N-1)} \right)
   = L(\theta, \phi).
 \end{align*}

\section{Advantage function estimation}\label{sec:M1}

The computation of objective function (\ref{eq:PO}) relies on the availability of  an estimate of  advantage function
$A_{\phi} (x, a)$ in (\ref{eq:A}).  We assume our MDP model is
known. So the expectation on (\ref{eq:A}) can be computed
exactly, and we can perform the computation in a timely manner. In this section, we explain how to
estimate $h_\phi$, a solution to the Poisson equation
\eqref{eq:Poisson} with $P=P_\phi$ and
$d=d_\phi$.


To compute expectation $A_\phi\left(x^{(k)}, a^{(k)}\right)$ in (\ref{eq:A}) for a given
state-action pair $\left(x^{(k)},a^{(k)}\right)$, we need to evaluate $h_\phi(y)$ for each $y$ that
is reachable from $x^{(k)}$. This requires one to estimate
$h_\phi(y)$ for some states $y$ that have not been visited in the simulation. Our strategy is to
use Monte Carlo method to estimate $h_\phi(y)$ at a  selected
subset of $y$'s, and then use an approximator $f_\psi(y)$ to
replace $h_\phi(y)$ for an arbitrary $y\in \X$. The latter is standard
in deep learning. Therefore, we focus on finding a good estimator
$
  \hat h(y)
$
for $h_\phi(y)$.

\subsection{Regenerative estimation}

 Lemma~\ref{lem:poisson_sol} provides a representation of the fundamental solution
\eqref{eq:h} for $h_\phi$. Unfortunately, the known unbiased Monte Carlo estimators of the fundamental solution rely on obtaining  samples from the stationary distribution of the Markov chain \cite[Section 5.1]{Cooper2003}.

We define the following solution to the Poisson equation (\ref{eq:Poisson}).
\begin{lemma}\label{lem:poisson_sol2}

Consider the $\V$-uniformly ergodic  Markov chain with transition matrix $P$ and the stationary  distribution $d$. Let $x^*\in X$ be an arbitrary state of the positive recurrent Markov chain.
 For any cost function $g:\X \rightarrow \R$ such that $|g|\leq \V$, the Poisson's equation (\ref{eq:Poisson}) admits  a  solution
\begin{align}\label{eq:h2}
h^{(x^*)}(x) : = \E \left[\sum\limits_{k=0}^{\sigma(x^*)-1} \left(g(x^{(k)}) - d^Tg\right)\Big|~x^{(0)} = x\right] ~\text{for each }x\in \X,
\end{align}
where $\sigma(x^*) = \min\left\{k>0~|~x^{(k)} = x^*\right\}$ is the first future time when state $x^*$ is visited.
 Furthermore, the solution has a finite $\V$-norm: $\|h^{(x^*)}\|_{\infty, \V}<\infty.$
\end{lemma}
See \cite[Proposition A.3.1]{Meyn2007} for the proof. Here, we refer to state $x^*$ as a \textit{regeneration state}, and to the times $\sigma(x^*)$ when the regeneration state is visited   as \textit{regeneration times}.

The value of  advantage function (\ref{eq:A}) does not depend on a particular choice of a solution of the Poisson equation (\ref{eq:Poisson}) since if $h_1$ and $h_2$ are two solutions  such that  $d^T(|h_1|+|h_2|)<\infty$, then there exists a constant $b\in \R$ such that $h_1(x) = h_2(x) +b$ for each $x\in \X$,  \cite[Proposition 17.4.1]{Meyn2009}. Therefore, we use representation (\ref{eq:h2})  for $h_\phi$ in computing (\ref{eq:A}).

We assume that an episode consisting  of $N$ regenerative cycles
\begin{align*}
  \left\{x^{(0)}, x^{(1)},\cdots, x^{(\sigma_1)},\cdots,   x^{(\sigma_{N}-1)}   \right \}
\end{align*}
has been generated  under policy $\pi_\phi$, where $x^{(0)}=x^*.$

We compute an estimate of the long-run average cost based on   $N$ regenerative cycles  as
\begin{align}\label{es_av}
\hat{\eta}_\phi  : =\frac{1}{\sigma(N)}  \sum\limits_{k=0}^{\sigma(N)-1} g(x^{(k)}),
\end{align}
where $\sigma(n)$ is the $n$th time when  regeneration state $x^*$ is visited.
Next, we consider an arbitrary state $x^{(k)}$ from the generated episode.  We define a one-replication estimate of the solution  to the Poisson equation (\ref{eq:h2}) for a state $x^{(k)}$ visited at time $k$  as:
\begin{align}\label{eq:es1}
\hat h_k: =  \sum\limits_{t=k}^{\sigma_k-1} \left(g(x^{(t) }) - \hat \eta_\phi  \right) ,
\end{align}
where $\sigma_k = \min\left\{t>k~|~x^{(t)}=x^*\right\}$  is the first time when the regeneration state $x^*$ is visited after time $k$.
We note that the one-replication estimate (\ref{eq:es1}) is computed for every timestep.
   The estimator (\ref{eq:es1}) was    proposed in  \cite[Section 5.3]{Cooper2003}.

We use   function $f_\psi:\X\rightarrow \R$ from  a family of function approximators $\{f_\psi, \psi\in \Psi \}$ to represent function $h_{  \phi}$ and choose function $f_\psi$  from $\{f_\psi, \psi\in \Psi \}$ to minimize the mean square distance to the one-replication estimates  $\{\hat  h_k\}_{k=0}^{\sigma(N)-1}:$
\begin{align}\label{eq:Vappr}
\psi^* = \arg\min\limits_{\psi \in \Psi} \sum\limits_{k=0}^{\sigma(N)-1} \left ( f_{\psi}(x^{(k)}) - \hat h_k   \right)^2.
\end{align}

With available function approximation $f_{\psi^*}$ for $h_\phi$,  we estimate the advantage function (\ref{eq:A})  as:
 \begin{align}\label{eq:Aes1}
 	\hat A_{\phi}(x^{(k)}, a^{(k)}): =   g(x^{(k)}) - \hat\eta_\phi + \sum\limits_{y\in \X} P\left(y|x^{(k)}, a^{(k)}\right) f_{\psi^*} (y) -  f_{\psi^*}(x^{(k)}).
 \end{align}

  We assume that $Q$ episodes, $Q\geq 1$, can be simulated in parallel, and each of $q=1, \dotsc, Q$ (parallel) actors collect an episode
\begin{align}\label{eq:episodes}
\left\{x^{(0, q)}, a^{(0, q)}, x^{(1 , q)}, a^{(1, q)}, \cdots, x^{(k , q)}, a^{(k, q)},\cdots, x^{(\sigma^q(N)-1, q)}, a^{(\sigma^q(N)-1, q)}\right\}
\end{align}
with $N$ regenerative cycles, where $\sigma^q(N)$  is the $N$th regeneration time in the simulation of $q$th actor and $x^{(0, q)} = x^*$ for each $q=1, \dotsc, Q$.
 Given  the episodes (\ref{eq:episodes})  generated under policy $\pi_\phi$,  we compute  the advantage function estimates $\hat A_{\phi}(x^{(k, q)}, a^{(k, q)})$   by (\ref{eq:Aes1}):
\begin{align*}
D^{ (0: \sigma^q(N)-1)_{q=1}^Q } = \Big\{\Big( x^{(0, q)}, a^{(0, q)},&  \hat A_{\phi}(x^{(0, q)}, a^{(0, q)}) \Big), \cdots, \\
 &\Big( x^{(\sigma^q(N)-1, q)}, a^{(\sigma^q(N)-1, q)}, \hat A_{\phi}(x^{(\sigma^q(N)-1, q)}, a^{(\sigma^q(N)-1, q)})  \Big)\Big\}_{q=1} ^{Q }.
\end{align*}
  We estimate the loss function (\ref{eq:PO})  as a sample average over these $\sum\limits_{q=1}^Q \sigma^q(N)$  data-points:
   \begin{align}\label{eq:popt}
    \hat L\left(\theta, \phi, D^{(1:Q), (0:\sigma^q(N)-1) }\right)  =\sum\limits_{q=1}^Q \sum\limits_{k=0}^{\sigma^q(N)-1} &\max\Big[ \frac{\pi_{\theta}(a^{(k, q)}| x^{(k, q)})}{\pi_{\phi}(a^{(k, q)}|x^{(k, q)})}  \hat A_{\phi}(x^{(k, q)}, a^{(k, q)}) ,\\ &\text{clip}\left(\frac{\pi_{\theta}(a^{(k,q)}| x^{(k, q)})}{\pi_{\phi}(a^{(k, q)}|x^{(k, q)})},     1-\epsilon, 1+\epsilon  \right) \hat A _{\phi}(x^{(k, q)}, a^{(k,q)})  \Big] \nonumber
   \end{align}
  Optimization of the loss function yields a new policy for the next iteration, see Algorithm \ref{alg1}.

\begin{algorithm}
\SetAlgoLined
\KwResult{policy $\pi_{\theta_I}$ }
 Initialize  policy $\pi_{\theta_0}$ \;
 \For{ policy iteration $i= 0, 1, \dotsc, I-1$}{
  \For{ actor  $q= 1, 2, \dotsc, Q$}{
  Run policy $\pi_{\theta_{i}}$  until it reaches $N$th regeneration time on $\sigma^q(N)$ step: collect an episode
   $\left\{x^{(0, q)}, a^{(0, q)}, x^{(1 , q)}, a^{(1, q)}, \cdots, x^{(\sigma^q(N)-1, q)}, a^{(\sigma^q(N)-1, q)}, x^{(\sigma^q(N), q)}\right\} $\;
  }

  Compute the average cost estimate $\hat\eta_{\theta_i}$ by (\ref{es_av})\  (utilizing $Q$ episodes) ;

        Compute  $\hat h_{k, q}$,  the estimate of $h_{\theta_i}(x^{(k, q)})$, by (\ref{eq:es1})  for each $q = 1,\dotsc , Q$, $k=0, \dotsc, \sigma^q(N)-1$\;
        Update $\psi_{i}: = \psi$, where $\psi\in \Psi$ minimizes $  \sum\limits_{q=1}^Q \sum\limits_{k=0}^{\sigma^q(N)-1 } \left ( f_{\psi}(x^{(k, q)}) -\hat h_{k, q}   \right)^2$ following (\ref{eq:Vappr}) \;
  Estimate the advantage functions $\hat A_{\theta_{i}}\left(x^{(k, q)}, a^{(k, q)}\right)$ using (\ref{eq:Aes1}) for each $q = 1,\dotsc, Q$, $k=0,\dotsc, \sigma^q(N)-1$:
 \begin{align*}
D^{(1:Q), (0:\sigma^q(N)-1) } = \left\{\left( x^{(0, q)}, a^{(0, q)}, \hat A_{0,q}\right), \cdots,  \left( x^{(\sigma^q(N)-1, q)}, a^{(\sigma^q(N)-1, q)}, \hat A_{\sigma^q(N)-1,q}\right)\right\}_{q=1} ^{Q }.
\end{align*}\\
 Minimize the surrogate objective function w.r.t. $\theta\in \Theta$:
 \begin{align*}
    \hat L\left(\theta, \theta_i, D^{(1:Q), (0:\sigma^q(N)-1) }\right)  =\sum\limits_{q=1}^Q &\sum\limits_{k=0}^{\sigma^q(N)-1} \max\Big[ \frac{\pi_{\theta}(a^{(k, q)}| x^{(k, q)})}{\pi_{\theta_{i}}(a^{(k, q)}|x^{(k, q)})}  \hat A_{\theta_{i}}(x^{(k, q)}, a^{(k, q)}) ,\\ &\text{clip}\left(\frac{\pi_{\theta}(a^{(k,q)}| x^{(k, q)})}{\pi_{\theta_{i}}(a^{(k, q)}|x^{(k, q)})},     1-\epsilon, 1+\epsilon  \right) \hat A _{\theta_{i}}(x^{(k, q)}, a^{(k,q)})  \Big]
   \end{align*}\\
 Update $\theta_{i+1}: = \theta$.
 }
 \caption{Base proximal policy optimization algorithm  for long-run average cost problems}\label{alg1}
\end{algorithm}

%

%
%

In practice a naive (standard) Monte Carlo estimator  (\ref{eq:es1}) fails to improve in PPO policy iteration
  Algorithm~\ref{alg1}  because of the large variance (i.e. the estimator is unreliable).  Therefore, we progressively develop a sequence of estimators in the next subsections. We end this section with   two remarks.

\begin{remark}
For any state $x\in \X$ the one-replication estimate (\ref{eq:es1}) is computed each time  the state is visited (the \textit{every-visit} Monte-Carlo method). It is also possible to implement  a \textit{first-visit}  Monte Carlo method which implies that
a one-replication estimate is computed when state $x$ is visited for the first time within a cycle
 and that the next visits to state $x$ within the same cycle are ignored. See \cite[Section 5.1]{Sutton2018} for more details of  every-visit and first-visit Monte-Carlo methods.
\end{remark}
\begin{remark}
  In regression problem (\ref{eq:Vappr}), each data point
    $(x^{(k)}, \hat h_k)$ is used  in the quadratic loss function, despite
    that many of the $x^{(k)}$'s represent the same state. It is possible to  restrict that
    \emph{only distinct} $x^{(k)}$'s are used in the loss function, with
    corresponding $\hat h_k$'s properly averaged. It turns out that this
    new optimization problem yields the same optimal solution as the one in (\ref{eq:Vappr}).
The equivalence of the optimization problems follows from the fact that for an arbitrary sequence of real numbers $a_1, \dotsc,a_n\in \R$:
\begin{align*}
 \arg\min_{x\in B}\left[ \sum\limits_{i=1}^n (x-a_i)^2 \right]= \arg\min_{x\in B}\left[ \left( x-\frac{1}{n}\sum\limits_{i=1}^n a_i\right)^2\right],
\end{align*}
where $B$ is an arbitrary subset of $\R$.
 \end{remark}

\subsection{Approximating martingale-process method}\label{sec:AMP}



Estimator (\ref{eq:es1}) of  the solution to the Poisson equation suffers from the high variance when the regenerative cycles are long  (i.e. the estimator is a sum of many random terms $g(x^{(k)}) - \eta_{\phi}$).  In this section we explain how to  decrease  the variance by reducing the magnitude of summands in (\ref{eq:es1}) if  an approximation $\zeta$  of the solution to Poisson's equation $ h_\phi$ is available.

We assume  an episode $\left\{x^{(0)}, a^{(1)}, x^{(1)}, a^{(2)}, \cdots, x^{(K-1)}, a^{(K-1)}, x^{(\sigma(N))}\right\}$ has been generated under policy $\pi_\phi$. From the definition of a solution to the Poisson equation (\ref{eq:Poisson}):
\begin{align*}
g(x^{(k) }) - \eta_{\phi} =  h_\phi(x^{(k)})-\sum\limits_{y\in \X}P_\phi (y| x^{(k)} ) h_\phi(y) \text{ for each state } x^{(k)}  \text{ in the simulated episode.}
 \end{align*}
If the approximation $\zeta$ is sufficiently close to $ h_\phi$, then the correlation between
\begin{align*}
g(x^{(k) }) -  \hat\eta_{\phi}  \quad \text{ and }\quad  \zeta(x^{(k)})-\sum\limits_{y\in X} P_\phi\left(y| x^{(k)}\right) \zeta(y)
\end{align*}
is positive and we can use the control variate  to reduce the variance. This idea gives rise to the approximating martingale-process (AMP) method proposed in \cite{Henderson2002}; also see \cite{Andradottir1993}.

Following \cite[Proposition 7]{Henderson2002}, for some approximation $\zeta$ such that   $d_\phi^T \zeta<\infty$ and $\zeta(x^*)=0$,
we consider the martingale process starting from an arbitrary state $x^{(k)}$  until the first regeneration time:
\begin{align}\label{eq:M}
M_{\sigma_k} (x^{(k)})=\zeta(x^{(k)}) +\sum\limits_{t=k}^{\sigma_k-1} \left[\sum\limits_{y\in \X} P_\phi\left(y|x^{(t)}\right)\zeta(y)  - \zeta(x^{(t)})\right],
\end{align}
where $\sigma_k = \min\left\{t>k~|~x^{(t)}=x^*\right\}$  is the first time when the regeneration state $x^*$ is visited after time $k$.
The martingale process (\ref{eq:M}) has zero expectation $\E M_n = 0$ for all $n\geq 0$; therefore we use it as   a control variate to define a new estimator. 
Adding $M_{\sigma_k}$ to estimator (\ref{eq:es1}) we get the AMP estimator of the solution to the Poisson equation:
\begin{align}\label{eq:es2}
\hat h_\phi^{AMP(\zeta)} (x^{(k)} )&: =\zeta(x^{(k)}) +  \sum\limits_{t=k}^{\sigma_k-1} \left(g(x^{(t) }) - \hat\eta_{\phi} +\sum\limits_{y\in \X} P_\phi\left(y|x^{(t)}\right)\zeta(y) -\zeta(x^{(t)})  \right) .
\end{align}
We  assume that the estimation of the average cost is accurate (i.e. $ \hat \eta_{\phi}  = \eta_{\phi}$). In this case estimator (\ref{eq:es2}) has zero variance if the approximation is exact $\zeta = h_\phi$.



Now we want to replace the standard regenerative estimator (\ref{eq:es1})  used in line 7 of Algorithm \ref{alg1} with AMP estimator (\ref{eq:es2}). As the approximation $\zeta$ needed in (\ref{eq:es2}), we use   $f_{\psi_{i-1}}$ that approximates a solution to the Poisson equation corresponding to previous policy $\pi_{\theta_{i-1}}$. In line 7  of Algorithm \ref{alg1} we replace $\hat h (x^{(k)})$ with the estimates $\hat h^{AMP(f_{\psi_{i-1}} )} (x^{(k)})$ that are computed by (\ref{eq:es2}).

 \begin{algorithm}
\SetAlgoLined
\KwResult{policy $\pi_{\theta_I}$ }
 Initialize  policy $\pi_{\theta_0}$ and value function $f_{\psi_{-1}}\equiv 0$ approximators \;
 \For{ policy iteration $i= 0, 1, \dotsc, I-1$}{
  \For{ actor  $q= 1, 2, \dotsc, Q$}{
  Run policy $\pi_{\theta_{i}}$  until it reaches $N$th regeneration time on $\sigma^q(N)$ step: collect an episode
   $\left\{x^{(0, q)}, a^{(0, q)}, x^{(1 , q)}, a^{(1, q)}, \cdots, x^{(\sigma^q( N)-1, q)}, a^{(\sigma^q(N)-1, q)}, x^{(\sigma^q(N), q)}\right\} $\;
  }

  Compute the average cost estimate $\hat\eta_{\theta_i}$ by (\ref{es_av})\;
        Compute  $\hat h^{AMP(f_{\psi_{i-1}})}_{k, q}$,  the estimate of $h_{\theta_i}(x^{(k, q)})$, by (\ref{eq:es2})  for each $q = 1, \dotsc, Q$, $k=0, \dotsc, \sigma^q(N)-1$\;
        Update $\psi_{i}: = \psi$, where $\psi\in\Psi$ minimizes $  \sum\limits_{q=1}^Q \sum\limits_{k=0}^{\sigma^q(N)-1 } \left ( f_{\psi}(x^{(k, q)}) -\hat h^{AMP(f_{\psi_{i-1}})}_{k, q}   \right)^2$ following (\ref{eq:Vappr})\;
  Estimate the advantage functions $\hat A_{\theta_{i}}(x^{(k, q)}, a^{(k, q)})$ using (\ref{eq:Aes1}) for each $q = 1, \dotsc, Q$, $k=0,\dotsc, \sigma^q(N)-1$:
 \begin{align*}
D^{ (0:\sigma^q(N)-1)_{q=1}^Q } = \left\{\left( x^{(0, q)}, a^{(0, q)}, \hat A_{0,q}\right), \cdots,  \left( x^{(\sigma^q(N)-1, q)}, a^{(\sigma^q(N)-1, q)}, \hat A_{\sigma^q(N)-1,q}\right)\right\}_{q=1} ^{Q }.
\end{align*}\\
 Minimize the surrogate objective function w.r.t. $\theta\in \Theta$:
 \begin{align*}
  \hat L\left(\theta, \theta_i, D^{ (0:\sigma^q(N)-1)_{q=1}^Q}\right) =\sum\limits_{q=1}^Q &\sum\limits_{k=0}^{\sigma^q(N)-1} \max\Big[ \frac{\pi_{\theta}(a^{(k, q)}| x^{(k, q)})}{\pi_{\theta_{i}}(a^{(k, q)}|x^{(k, q)})}  \hat A_{\theta_{i}}(x^{(k, q)}, a^{(k, q)}) ,\\ &\text{clip}\left(\frac{\pi_{\theta}(a^{(k,q)}| x^{(k, q)})}{\pi_{\theta_{i}}(a^{(k, q)}|x^{(k, q)})},     1-\epsilon, 1+\epsilon  \right) \hat A _{\theta_{i}}(x^{(k, q)}, a^{(k,q)})  \Big]
   \end{align*}\\
 Update $\theta_{i+1}: = \theta$.

 }
 \caption{Proximal policy optimization with AMP method}\label{alg1amp}
\end{algorithm}

\subsection{Variance reduction through discounting}\label{sec:ge}

 Unless an approximation $\zeta$ is exact, each term in  the summation in (\ref{eq:es2}) is random with nonzero variance.  When the expected length of a regeneration cycle is large, the   cumulative variance  of estimator (\ref{eq:es2})   can be devastating.

In this subsection, we describe a commonly used solution:  introduce a forgetting   factor $\gamma\in (0,1)$
 to discount the future relative costs,  \cite{Jaakkola1994, Baxter2001, Marbach2001, Kakade2001, Thomas2014, Schulman2016}. 

We let
\begin{align}\label{eq:r}
r(x^*):= (1-\gamma)\E \left[\sum\limits_{t=0}^\infty \gamma^t g(x^{(t)}) ~|~x^{(0)}=x^*\right]
\end{align}
be a \textit{present discounted value} at state $x^*$; the term
``present discounted value'' was proposed in \cite[Section
11.2]{Wagner1975}. We define the \textit{regenerative
  discounted relative value function}   as:
\begin{align}\label{eq:Vreg}
V^{(\gamma)}(x): = \E \left[ \sum\limits_{t=0}^{\sigma(x^*)-1} \gamma^{t}  \left(g(x^{(t)}) -  r(x^*) \right) ~\Big| ~x^{(0)} = x\right] \text{ for each }x\in \X,
\end{align}
where $x^{(k)}$ is the state of the Markov chain with transition matrix
$P$ at time $k$, $x^*$ is the prespecified regeneration
state, and $\gamma\in (0,1]$ is a discount factor.  We note that
$V^{(\gamma)}(x^*) = 0$ by definition.  It follows from
  \cite[Corollary 8.2.5.]{Puterman2005} that under the drift condition,
  $r(x^*)\to \eta$ as $\gamma\uparrow 1$.  Furthermore, by
  Lemma \ref{lem:disc}, for each $x\in \X$,
  \begin{align*}
   V^{(\gamma)}(x) \to h^{(x^*)}(x)  \text{ as } \gamma\uparrow 1,
  \end{align*}
  where $h^{(x^*)}$ is a solution to the Poisson equation given in (\ref{eq:h2}).

  The proof of Lemma \ref{lem:disc} can be found in Appendix Section \ref{sec:disc}.

  \begin{lemma}\label{lem:disc}
  We consider  irreducible, aperiodic   Markov chain  with transition matrix $P$ that satisfies drift condition (\ref{eq:drift}).
We let (\ref{eq:Vreg})  be a regenerative
  discounted relative value function  for discount factor $\gamma$ and one-step cost function $g$,  such that  $|g(x)|\leq\V(x)$ for each $x\in \X$. We let $h^{(x^*)} $ be a solution of the Poisson equation (\ref{eq:Poisson}) defined by (\ref{eq:h2}).

 Then for some constants $R<\infty$ and $r\in(0,1)$  we have
\begin{align*}
\left|V^{(\gamma)}(x) - h^{(x^*)}(x) \right|\leq \frac{rR(1-\gamma)}{(1-r)(1-\gamma r)}(\V(x)+\V(x^*))
\end{align*}
for each $x\in \X.$
\end{lemma}

  We let
\begin{align}\label{eq:VdiscVarEst}
 \hat V^{(\gamma )}(x):=  \sum\limits_{k=0}^{\sigma(x^*)-1} \gamma^k\left(g (x^{(k)} )- r(x^*) \right )
\end{align}
   where $x^{(0)} = x$ and $x^{(k)}$ is the $k$th step of the Markov chain with transition matrix $P$, be a one-replication estimate of (\ref{eq:Vreg}).
    By Lemma \ref{lem:var} the variance of this estimator $Var[ \hat V^{(\gamma )}(x)] $
converges to zero with rate $\gamma^2$ as $\gamma\downarrow 0$ for each $x\in \X$. See Appendix Section \ref{sec:disc} for the proof.

   \begin{lemma}\label{lem:var}
   We consider  irreducible, aperiodic   Markov chain  with transition matrix $P$ that satisfies drift condition (\ref{eq:drift})
and assume that one-step cost function $g:\X\rightarrow \R$ satisfies $g^2(x)\leq\V(x)$ for each $x\in \X$.
We let (\ref{eq:Vreg})  be a regenerative
  discounted relative value function for discount factor $\gamma$, and assume that   regeneration state  $x^*$ is such that set \begin{align*}
\Big\{x\in \X:\varepsilon\V(x)\leq \V(x^*)\Big\}
\end{align*} is a finite set, where function $\V$ and constant $\varepsilon$ are from (\ref{eq:drift}). We consider an arbitrary $x\in \X$ and let (\ref{eq:VdiscVarEst})
 be an one-replication estimate of   (\ref{eq:Vreg}).

If $\gamma<1$, then there exist   constants $R< \infty$, $B<\infty$, and $r\in (0,1)$ independent of $\gamma$ such that the variance of estimate (\ref{eq:VdiscVarEst}) is bounded as
 \begin{align*}
 Var[ \hat V^{(\gamma )}(x)] \leq \gamma^2\left( R\V(x)\frac{1}{1-\gamma^2r} + (d^T\V)B \frac{1}{1-\gamma^2}  \right), \text{ for each }x\in \X,
\end{align*}
where   $d$ is a stationary distribution of transition matrix $P$.
\end{lemma}

For a fixed  $\gamma\in (0,1)$ and state $x\in \X$  any unbiased
  estimator of $V^{(\gamma)}(x)$ is a biased estimator of
  $h^{(x^*)}(x)$. It turns out that the discount counterparts of the
  estimators (\ref{eq:es1}) and (\ref{eq:es2}) for
  $V^{(\gamma)}(x)$ have smaller variances than the two estimators for $h^{(x^*)}(x)$. This variance
  reduction can be explained intuitively as follows.  Introducing the
discount factor $\gamma$ can be interpreted as a modification of the
original transition dynamics; under the modified dynamics, any action
produces a transition into a regeneration state with probability at
least $1 - \gamma$, thus shortening the length of   regenerative
cycles. See Appendix Section \ref{sec:disc} for details.


We define a \textit{discounted} advantage function for policy $\pi_\phi$ as:
\begin{align}\label{eq:Adisc}
A^{(\gamma)}_{\phi}(x, a):& = \underset{\substack{ y\sim P(\cdot|x, a)}  }{\E} \left[g(x) - \eta_\phi + V_{\phi}^{(\gamma)}(y) -V_{\phi}^{(\gamma)} (x)\right].
\end{align}



We use the function approximation   $f_\psi$ of $V_{\phi}^{(\gamma)}$ to   estimate the advantage function (\ref{eq:Adisc})  as (\ref{eq:Aes1}).

We now present the discounted version of the AMP estimator (\ref{eq:es2}).
We let $\zeta$ be an approximation of the discounted value function $V_{\phi}^{(\gamma)}$ such that   $ d_\phi^T \zeta <\infty$ and $\zeta(x^*)=0$.  We define the sequence $(M_{\phi}^{(n)}:n\geq 0)$:
\begin{align}\label{eq:mart}
M_{\phi}^{(n)}(x):=\sum\limits_{t=k}^{n-1} \gamma ^{t-k+1}\left[ \zeta   (x^{(t+1)}) - \sum\limits_{y\in \X} P_{\phi}\left(y|x^{(t)}\right)  \zeta(y)  \right],
\end{align}
where   $x = x^{(k)}$ and $x^{( t)}$ is a state of the Markov chain   after $ t$ steps.

We define a one-replication of the AMP estimator for the discounted value function:
\begin{align}\label{eq:es4}
\hat V^{ AMP( \zeta), (\gamma)}_{\phi}(x^{(k)} ) :&=  \sum\limits_{t=k}^{\sigma_k-1} \gamma^{t-k}  \left(g(x^{(t)}) - \widehat{ r_\phi(x^*)} \right)  - M^{(\sigma_k )}_\phi(x^{(k)})\\
&=    \zeta(x^{(k)}) +    \sum\limits_{t=k}^{\sigma_k - 1}\gamma^{t-k} \left(g(x^{(t)} ) - \widehat{r_\phi(x^*)}+  \gamma \sum\limits_{y\in \X} P_{\phi}\left(y|x^{(t)}\right)  \zeta(y)   -   \zeta(x^{(t)} )  \right)  \nonumber\\
&\quad\quad- \gamma^{\sigma_k-k} \zeta(x^*) \nonumber\\
&=    \zeta(x^{(k)}) +    \sum\limits_{t=k}^{\sigma_k - 1}\gamma^{t-k} \left(g(x^{(t)} ) -\widehat{r_\phi(x^*)}+  \gamma \sum\limits_{y\in \X} P_{\phi}\left(y|x^{(t)}\right)  \zeta(y)   -   \zeta(x^{(t)} )  \right),  \nonumber
\end{align}
where $ \widehat{r_\phi(x^*)}$ is an estimation of $r(x^*),$ and $\sigma_k = \min\left\{t>k~|~x^{(t)}=x^*\right\}$ is the first time  the regeneration state $x^*$ is visited after time $k$.

The AMP estimator (\ref{eq:es4}) does not introduce any bias subtracting $M_\phi$ from $\hat V^{(\gamma)}_\phi$ since    $\E M_{\phi}^{(n)} =0$ for any $n>0$ by \cite{Henderson2002}.
Function $V_{\phi}^{(\gamma)}$ is a solution of the following equation (see Lemma  \ref{lem:2sol}):
\begin{align}\label{eq:Poiss_reg}
g(x) - r_\phi(x^*) + \gamma\sum\limits_{y\in \X}P_\phi (y|x) h(y) - h(x) =0 \quad \text{ for each }x\in \X.
\end{align}
Therefore, similar to (\ref{eq:es2}), estimator (\ref{eq:es4}) has zero variance if approximation is exact  $\widehat{r_\phi(x^*)} = r_\phi(x^*)$ and $\zeta = V^{(\gamma)}_\phi$,  see Poisson equation (\ref{eq:Poiss_reg}).

Further variance reduction is possible via \textit{$T$-step truncation} \cite[Section 6]{Sutton2018}. We consider an   estimate of  the value function  (\ref{eq:Vreg})  at a state $x\in \X$ as   the sum of the discounted costs before time $T$, where $T<\sigma(x^*)$,   and the discounted costs after time $T$:
\begin{align}\label{eq:Vst}
\hat V^{(\gamma)}(x) =  \sum\limits_{t=0}^{T-1} \gamma^t \left(g(x^{(t)})-\widehat{r(x^*)}\right) + \gamma^T\sum\limits_{t=0}^{\sigma(x^*)-1} \gamma^t \left(g(x^{(T+t)})-\widehat{r(x^*)}\right),
\end{align}
where   $ x^{(0)}=x$,  $x^{(t)}$ is a state of the Markov chain  after   $t$  steps and $\sum\limits_{t=0}^{\sigma(x^*)-1} \gamma^t g(x^{(T+t)})$ is a standard one-replication estimation of the value function at state $x^{(T)}$.
Instead of estimating the value at state $x^{(T)}$ by a random roll-out (second term in (\ref{eq:Vst})), we can use the value of deterministic approximation function $\zeta$ at state $x^{(T)}.$ The $T$-step truncation reduces the variance of the standard estimator but introduces bias unless the approximation is exact $\zeta(x^{(T)}) =  V^{(\gamma)}(x^{(T)})$.

A \textit{$T$-truncated} version of the AMP estimator is
\begin{align}\label{eq:esT}
\hat V^{AMP( \zeta), (\gamma, T)}_k :&= \sum\limits_{t=k}^{T\wedge \sigma_k-1} \gamma^{t-k} \left(g(x^{(t)})-\widehat{r_\phi(x^*)}\right) + \gamma^{T\wedge \sigma_k -k }\zeta(x^{(T\wedge \sigma_k) }) - M_\phi^{(T\wedge \sigma_k )}(x^{(k)}) \\
                                     &=\sum\limits_{t=k}^{T\wedge \sigma_k -1} \gamma^{t-k} \left(g(x^{(t)})-\widehat{r_\phi(x^*)}\right) - \sum\limits_{t=k}^{T\wedge \sigma_k -1} \gamma^{t-k+1}\left(\zeta(x^{(t+1) }) - \sum\limits_{y\in \X} P_\phi\left(y|x^{(t)} \right)\zeta(y) \right) \nonumber \\
  & \quad { } + \gamma^{T\wedge \sigma_k -k}\zeta(x^{(T\wedge \sigma_k)}) \nonumber \\
&=   \zeta(x^{(k)}) + \sum\limits_{t=k}^{T\wedge \sigma_k -1}  \gamma ^{t-k} \left(g(x^{(t)}) -\widehat{r_\phi(x^*)}+  \gamma \sum\limits_{y\in \X} P_{\phi}\left(y|x^{(t)}\right)  \zeta(y)   -   \zeta(x^{(t)})  \right),\nonumber
\end{align}
where $T\wedge \sigma_k = \min(T, \sigma_k)$.
We note that if the value function approximation and present discounted value approximation  are exact, estimator (\ref{eq:esT}) is unbiased for $V^{(\gamma)}(x^{(k)})$ and has zero variance.  We generalize the $T$-truncated estimator by taking  the
  number of summands $T$  to follow the geometrical distribution with parameter $\lambda<1$ as in the TD($\lambda$) method \cite[Section 12]{Sutton2018}, \cite[Section 3]{Schulman2016}:
\begin{align}\label{eq:es5}
\hat V^{AMP( \zeta), (\gamma, \lambda)}_k  :&=  \E_{T\sim Geom(1-\lambda)} \hat V^{AMP( \zeta), (\gamma, T)}_k \\
&=  (1-\lambda)\Big( \hat V^{AMP( \zeta), (\gamma, 1)}_k +\lambda  \hat V^{AMP( \zeta), (\gamma, 2)}_k  + \lambda^2 \hat V^{AMP( \zeta), (\gamma, 3)}_k+\cdots\nonumber\\
&\quad\quad +\lambda^{\sigma_k}V^{AMP( \zeta), (\gamma, \sigma_k)}_k  + \lambda^{\sigma_k+1}V^{AMP( \zeta), (\gamma, \sigma_k)}_k  +\cdots\Big)\nonumber\\
&= \zeta(x^{(k)}) + \sum\limits_{t=k}^{\sigma_k-1} (\gamma\lambda)^{t-k} \left(g(x^{(t)}) -\widehat{r_\phi(x^*)}+  \gamma \sum\limits_{y\in \X} P_{\phi}\left(y|x^{(t)}\right)  \zeta(y)   -   \zeta(x^{(t)})  \right).\nonumber
\end{align}

The regenerative cycles can be very long.  In practice we want to control/predict the  time and memory amount  allocated for the algorithm execution. Therefore, the simulated episodes should have  finite lengths. We use the following estimation for the first $N$ timesteps if an episode with finite length $N+L$ is generated:
\begin{align}\label{eq:esf}
\hat V^{AMP( \zeta), (\gamma, \lambda, N+L)}_k  := \zeta(x^{(k)}) + \sum\limits_{t=k}^{(N+L) \wedge \sigma_k -1} (\gamma\lambda)^{t-k} \left(g(x^{(t)})  - \widehat{r_\phi(x^*)}+  \gamma \sum\limits_{y\in \X} P_{\phi}\left(y|x^{(t)}\right)  \zeta(y)   -   \zeta(x^{(t)})  \right),
\end{align}
 where $k=0,\dotsc, N-1$, and integer $L$ is large enough. We note that if an episode has a finite length $N+L$, regeneration $\sigma_k$  may have not been observed in the generated episode (i.e. $\sigma_k>N+L$).  In this case, we summarize  (\ref{eq:esf}) up to the end of the episode. 

We provide the PPO algorithm where each of $q=1, \dotsc, Q$ parallel actors simulates an episode with length $N+L$: $\left\{x^{(0, q)}, a^{(0,q)}, x^{(1, q)}, a^{(1, q)}, \cdots, x^{(N+L-1, q)}, a^{(N+L-1, q)} \right\}.$   Since we need the approximation $\zeta$ in (\ref{eq:esf}), we use   $f_{\psi_{i-1}}$ that approximates a regenerative discounted value function corresponding to previous policy $\pi_{\theta_{i-1}}$. See Algorithm \ref{alg2}.

\begin{algorithm}[h!]
\SetAlgoLined
\KwResult{policy $\pi_{\theta_I}$ }
 Initialize  policy $\pi_{\theta_0}$ and value function $f_{\psi_{-1}}\equiv 0$ approximators \;
 \For{ policy iteration $i=0, 1,  \dotsc, I-1$}{
  \For{ actor  $q= 1, 2, \dotsc, Q$}{
  Run policy $\pi_{\theta_{i}}$  for $N+L$ timesteps: collect an episode
   $\left\{x^{({0, q})}, a^{({0, q})}, x^{({1 , q})}, a^{({1, q})}, \dotsc, x^{({ N+L-1, q})}, a^{({N+L-1, q})} \right\} $\;
  }
  Estimate the average cost   $\hat\eta_{\theta_i}$ by (\ref{es_av}), the present discounted value $\widehat{r_{\theta_i}(x^*)}$ by (\ref{eq:es_r}) below\;
        Compute  $\hat V^{AMP( f_{\psi_{i-1}}), (\gamma, \lambda)}_{k, q}  $ estimates by (\ref{eq:esf})  for each $q = 1, \dotsc, Q$, $k=0, \dotsc, N-1$\;
        Update $\psi_{i}: = \psi$, where $\psi\in \Psi$ minimizes $\sum\limits_{q=1}^Q \sum\limits_{k=0}^{N-1 } \left ( f_{\psi}\left(x^{({k, q})}\right) -\hat V^{AMP(  f_{\psi_{i-1}}), (\gamma, \lambda)}_{k, q} \right)^2$ following (\ref{eq:Vappr}) \;
  Estimate the advantage functions $\hat A_{\theta_{i}}\left(x^{({k, q})}, a^{({k, q})}\right)$ using (\ref{eq:Aes1}) for each $q = 1, \dotsc, Q$, $k=0, \dotsc,N-1$:\
  \begin{align*}
  D^{(0:N-1)_{q=1}^Q} = \Big\{&\left(x^{({0, q})}, a^{({0, q})}, \hat A^{(\gamma)}_{\theta_{i}} (x^{({0, q})}, a^{({0, q})})\right),  \cdots, \\
&\left(x^{({N-1, q})}, a^{({N-1, q})}, \hat A^{(\gamma)}_{\theta_{i}} (x^{({N-1, q})}, a^{({N-1, q})})\right) \Big\}_{q=1}^Q
   \end{align*}\\
 Minimize the surrogate objective function w.r.t. $\theta\in \Theta$:\
 \begin{align*}
   \hat L^{(\gamma)}\left(\theta, \theta_i,  D^{(0:N-1)_{q=1}^Q} \right) =\sum\limits_{q=1}^Q \sum\limits_{k=0}^{ N-1} &\max\Big[ \frac{\pi_{\theta}\left(a^{({k, q})}| x^{({k, q})}\right)}{\pi_{\theta_{i}}\left(a^{({k, q})}|x^{({k, q})}\right)}  \hat A^{(\gamma)}_{\theta_{i}}\left(x^{({k, q})}, a^{({k, q})}\right) ,\\ &\text{clip}\left(\frac{\pi_{\theta}\left(a^{({k,q})}| x^{({k, q})}\right)}{\pi_{\theta_{i}}\left(a^{({k, q})}|x^{({k, q})}\right)},     1-\epsilon, 1+\epsilon  \right) \hat A^{(\gamma)}_{\theta_{i}}\left(x^{({k, q})}, a^{({k,q})}\right)  \Big];
   \end{align*}\\
 Update $\theta_{i+1}: = \theta$.
 }
 \caption{Proximal policy optimization with discounting}\label{alg2}
\end{algorithm}

In Algorithm \ref{alg2}, we assume that state $x^*$ has been visited $N_q$ times in the $q$th generated episode, when $q=1,\dotsc, Q$ parallel actors are available. Therefore, we estimate the present discounted value $r(x^*)$ as:
\begin{align}\label{eq:es_r}
\widehat{r(x^*)} = (1-\gamma)\frac{1}{\sum_{q=1}^Q N_q} \sum\limits_{q=1}^Q\sum\limits_{n=1}^{N_q} \sum\limits_{k=\sigma^q(n)}^{\sigma^q(n)+L} \gamma^{k-\sigma^q(n)} g(x^{({k,q})}),
\end{align}
where $\sigma^q(n)$ is the $n$th time when state $x^*$ is visited in the $q$th episode, and integer $L$ is a large enough.
If many parallel actors are available, we recommend starting the episodes from state $x^*$ to ensure that  state $x^*$ appears in the generated episodes  a sufficient number of times.

 \begin{remark}\label{rem:regVSinf}
We use the following \textit{discounted value function} as an  approximation of the solution $h$ of the Poisson equation:
\begin{align}\label{eq:disc_inf}
J^{(\gamma)}(x):=\E\left[ \sum\limits_{k=0}^{\infty} \gamma^k\left(g(x^{(k)})- d^Tg \right)~|~x^{(0)}=x \right] \text{ for each }x\in \X.
\end{align}
 We note that  the discounted value function  (\ref{eq:disc_inf}) and the regenerative discounted value function (\ref{eq:Vreg}) are solutions of the same Poisson equation (Lemma \ref{eq:newPoisson}). Therefore, the bias of advantage function estimator (\ref{eq:Adisc}) does not change when $h_\phi$ in (\ref{eq:A}) is replaced   either by $J^{(\gamma)}$ or by $V^{(\gamma)}$.
  The variance of the  regenerative discounted value function estimator (\ref{eq:esf}) can be \textit{potentially} smaller than the variance of the analogous $J^{AMP(\zeta), (\gamma, \lambda, N+L)}$ estimator:
  \begin{align}\label{eq:Jesf}
\hat J^{AMP( \zeta), (\gamma, \lambda, N+L)}_k  := \zeta(x^{(k)}) + \sum\limits_{t=k}^{N+L-1} (\gamma\lambda)^{t-k} \left(g(x^{({t})})  -\hat\eta_{\phi}+  \gamma \sum\limits_{y\in \X} P_{\phi}\left(y|x^{({t})}\right)  \zeta(y)   -   \zeta(x^{({t})})  \right).
\end{align}
Since the upper bound of summation in  (\ref{eq:esf}) is $\min(\sigma_k, N+L)-1$, it includes fewer summands than the summation in (\ref{eq:Jesf}) if the regeneration  frequently occurs.
See Appendix Section \ref{sec:regVSinf}, which describes a numerical experiment for the criss-cross network implying that the choice of  $J^{AMP(\zeta), (\gamma, \lambda, N+L)}$ and  $V^{AMP(\zeta), (\gamma, \lambda,N+ L)}$  estimators can affect the PPO algorithm convergence rate  to the optimal policy.
\end{remark}


The connection between our proposed AMP estimator and the GAE estimator \cite{Schulman2016} suggests another motivation for introducing the GAE estimator.
An accurate estimation of the value function $V^{(\gamma)}_\phi$  by the standard estimator  may require  a large number of state-action pairs samples from the current policy $\pi_\phi$ \cite[Section 13]{Sutton2018}, \cite{Schulman2016, Ilyas2020}. In this chapter we apply the  AMP method to propose the discounted AMP estimator in (\ref{eq:esf}) that has  a smaller variance than the estimators   (\ref{eq:Vst}). 
  The AMP method, however, requires knowledge  of transition probabilities $P(y|x, a)$ in order to exactly compute the expected values for each $(x,~a,~y)~\in~\X~\times~\A~\times~\X$:
  \begin{align*}\underset{\substack{ a\sim \pi_{\phi}(\cdot|x)  \\ y\sim P(\cdot|x, a)}}{\E}  \zeta(y) =  \sum\limits_{y\in \X} \sum\limits_{a\in \A}P(y|x , a)\pi_\phi(a|x )\zeta(y). \end{align*}

One can relax the  requirement by replacing each expected value
 $\underset{\substack{ a\sim \pi_{\phi}\left(\cdot|x^{(t)}\right)  \\ y\sim P(\cdot|x, a)}}{\E}  \zeta(y)$  by its \textit{one-replication estimate}  $\zeta\left(x^{(t+1)}\right)$, where $\left(x^{(t)}, x^{(t+1)}\right)$ are two sequential states from an episode \begin{align*}
 \left(x^{(0)}, \dotsc, x^{(t)}, x^{(t+1)}, \dotsc\right)
 \end{align*} generated under policy $\pi_\phi$.



In Section \ref{sec:ge} we propose  the   AMP estimator of the  regenerative discounted value function (\ref{eq:es5}):
\begin{align}\label{eq:Ves2}
\hat V^{AMP( \zeta), (\gamma, \lambda)}_k =   \zeta(x^{(k)}) + \sum\limits_{t=k}^{\sigma_k-1} (\gamma\lambda)^{t-k} \left(g(x^{(t)})  - \widehat{ r(x^*)}+  \gamma \sum\limits_{y\in\X} P_{\phi}(y|x^{(t)})  \zeta(y)   -   \zeta(x^{(t)})  \right).
\end{align}

Replacing expectations $\sum\limits_{y\in \X} P_{\phi}\left(y|x^{(t)}\right)  \zeta(y)$ by the one-replication estimates $\zeta\left(x^{(t+1)}\right)$  we obtain the following estimator for the value function:
\begin{align}\label{eq:GAE}
\hat V^{GAE( \zeta), (\gamma, \lambda)}_k :=  \zeta(x^{(k) }) + \sum\limits_{t=k}^{\sigma_k-1} (\gamma\lambda)^{t-k} \left(g(x^{(k+t)})  - \widehat{ r(x^*)}+  \gamma \zeta(x^{(t+1)})   -   \zeta(x^{(t)})  \right),
\end{align}
which is a part of  the general advantage estimation (GAE) method \cite{Schulman2016}. We note that the advantage function   is estimated as $
\hat A^{GAE( \zeta), (\gamma, \lambda)}_k : = \hat V^{GAE( \zeta), (\gamma, \lambda)}_k - \zeta\left(x^{(k)}\right)$ in  \cite{Schulman2016}.

If $\lambda=1$, the estimator (\ref{eq:es5}) is transformed into  the standard estimator and does not depend on $\zeta$:
\begin{align*}
\hat V_\phi^{(\zeta), (\gamma)} (x^{(k)}):&= \zeta(x^{(k)}) \nonumber
 +  \sum\limits_{t=k}^{\sigma_k-1} \gamma^{t-k} \left(g(x^{(t) }) -\widehat{ r(x^*)}  +  \gamma \zeta(x^{(t+1)}) -\zeta(x^{(t)})  \right) \\
 &=\sum\limits_{t=k}^{\sigma_k-1}\gamma^{t-k} \left(g(x^{(t)})-\widehat{ r(x^*)}\right) .
\end{align*}
 The martingale-based control variate has no effect on the
  GAE estimator when $\lambda=1$, but it can produce significant variance reductions in the AMP estimator. See Section \ref{sec:cc} for numerical experiments.


\section{Experimental results for multiclass queueing networks}\label{sec:experiments}
In this section we evaluate the performance of the proposed proximal policy optimization Algorithms \ref{alg1},  \ref{alg1amp}, and \ref{alg2} for the multiclass queueing networks control optimization task discussed in Section \ref{sec:MQN}.

We use two separate fully connected feed-forward neural networks to represent  policies $\pi_\theta,$ $\theta\in \Theta$ and value functions $f_\psi$, $\psi\in \Psi$ with the architecture details given in Appendix Section \ref{sec:nn}.  We  refer to the neural network used to represent a policy as \textit{the  policy NN} and to  the neural network used to approximate a value function as \textit{the value function NN}. We  run the algorithm for $I=200$ policy iterations for each experiment. The algorithm uses  $Q=50$ actors to simulate data in parallel for each iteration.    See Appendix Section \ref{sec:par} for the details.

\subsection{Criss-cross network}\label{sec:cc}

 We first study the PPO algorithm and compare its base version Algorithm \ref{alg1} and its modification Algorithm \ref{alg1amp} that incorporates the  AMP method.  We check the robustness of the algorithms for the criss-cross system  with  various load (traffic) intensity regimes, including    I.L. (imbalanced light), B.L. (balanced light), I.M. (imbalanced medium), B.M. (balanced medium), I.H. (imbalanced heavy), and B.H. (balanced heavy) regimes.   Table \ref{t:lp} lists the corresponding arrival and service rates. The criss-cross network in any of these traffic regimes is stable under any work-conserving policy  \cite{Dai1996}. Since we want an initial policy to be stable, we forbid each server in the network to idle unless all its associated buffers are empty.

Table \ref{tab:cc}  summarizes the control policies proposed in the literature. Column 1 reports the load regimes, column 2 reports the
optimal performance obtained by dynamic programming (DP), column 3 reports the performance of a target-pursuing
policy (TP) \cite{Paschalidis2004} , column 4 reports the performance of a threshold policy \cite{Harrison1990}, columns 5 and 6 report the performance of fluid (FP) and robust fluid (RFP) policies respectively \cite{Bertsimas2015}, and column 7 reports the performance and the half width of the $95\%$ confidence intervals (CIs) of the PPO policy  $\pi_{\theta_I}$  resulting from the last iteration of Algorithm  \ref{alg1amp}.

We initialize the policy NN parameters $\theta_0$  using standard Xavier initialization \cite{Glorot2010}.  The resulting policy $\pi_{\theta_0}$ is close to the policy that chooses actions uniformly at random.  We take the empty system state $x^* = (0,0,0)$ as a regeneration state and simulate $N=5,000$ independent regenerative cycles  per actor in each iteration of the algorithm. Although the number of generated cycles is fixed for all traffic regimes, the length of the regenerative cycles varies  and highly depends on the load.

To show the learning curves in Figure \ref{fig:cc_opt},  we  save policy parameters $\{\theta_i\}_{i=0, 10, \dotsc, 200}$ every 10th iteration over the course of learning. After the algorithm terminates we independently
 simulate   policies $\{    \pi_{\theta_i}: i = 0, 10, \dotsc, 200 \}$ (in parallel) starting from the regeneration state $x = (0,\dotsc,0)$ until a fixed number of regenerative events occurs. For light, medium, and heavy  traffic regimes, we run the simulations  for $5\times10^7$, $5\times 10^6$, and  $10^6$ regenerative cycles respectively.
 We compute the $95\%-$confidence intervals   using the strongly consistent estimator of asymptotic variance. See  \cite[Section VI.2d]{Asmussen2003}.

  \begin{figure}[H]
    \subfloat[Imbalanced low (IL) traffic \label{subfig-6:IL}]{%
       \includegraphics[ height=0.19\textwidth, width=0.33\textwidth]{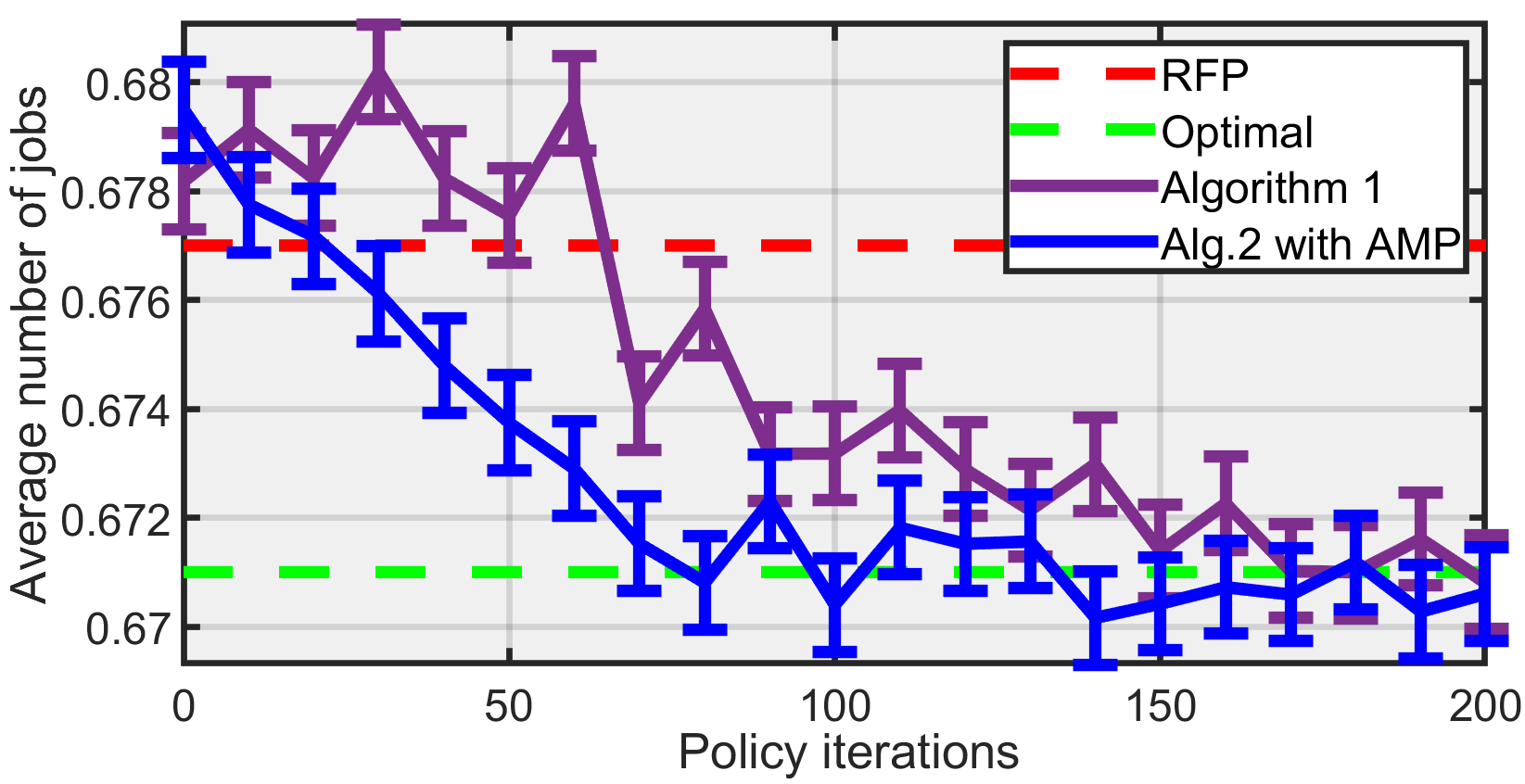}
     }
          \subfloat[Imbalanced medium (IM) traffic \label{subfig-4:IM}]{%
       \includegraphics[ height=0.19\textwidth, width=0.33\textwidth]{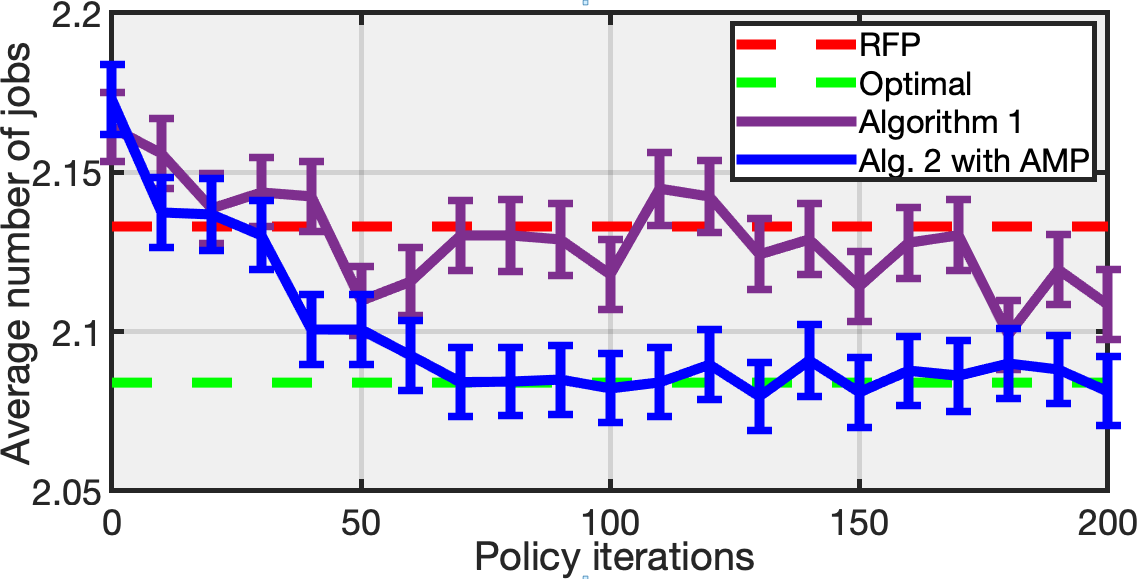}
     }
     \subfloat[Imbalanced heavy (IH) traffic \label{subfig-2:IH}]{%
       \includegraphics[ height=0.19\textwidth, width=0.33\textwidth]{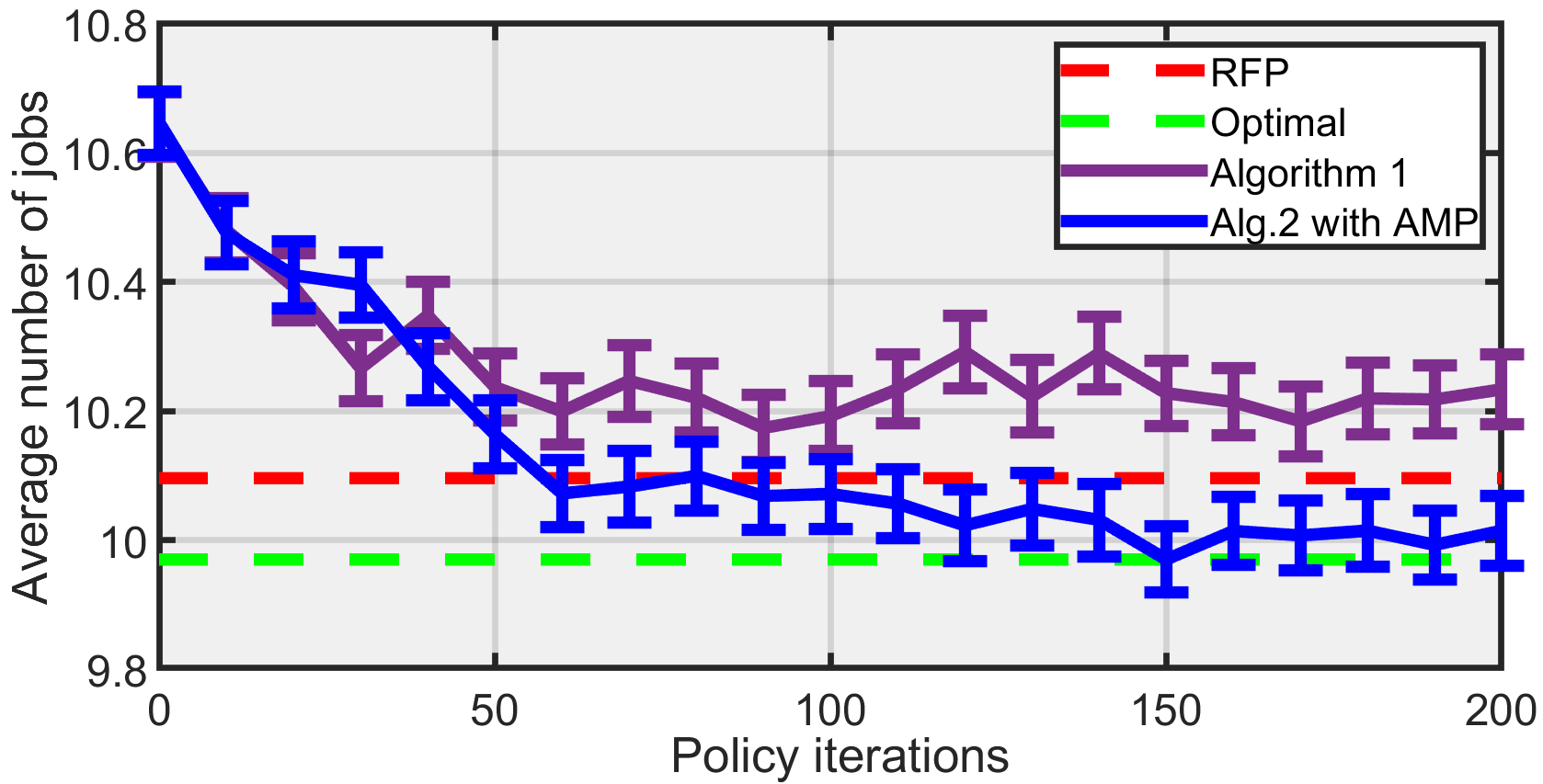}
     }\\
           \subfloat[Balanced low (BL) traffic\label{subfig-5:BL}]{%
       \includegraphics[  height=0.19\textwidth, width=0.33\textwidth]{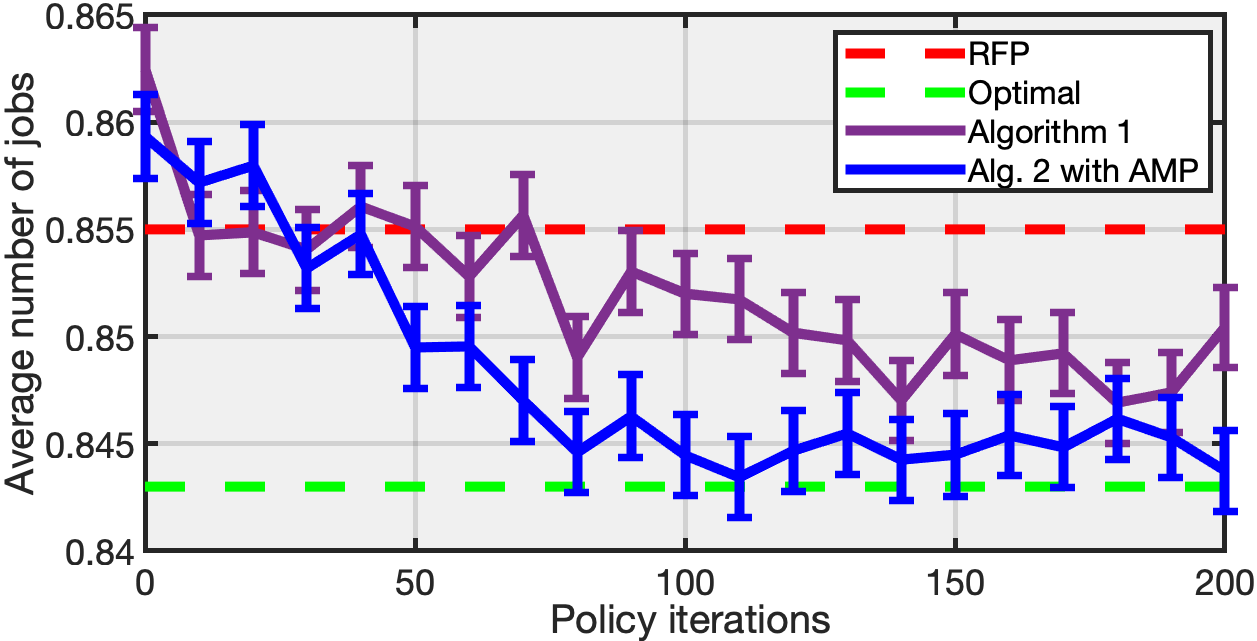}
     }
          \subfloat[Balanced medium (BM) traffic\label{subfig-3:BM}]{%
       \includegraphics[  height=0.19\textwidth, width=0.33\textwidth]{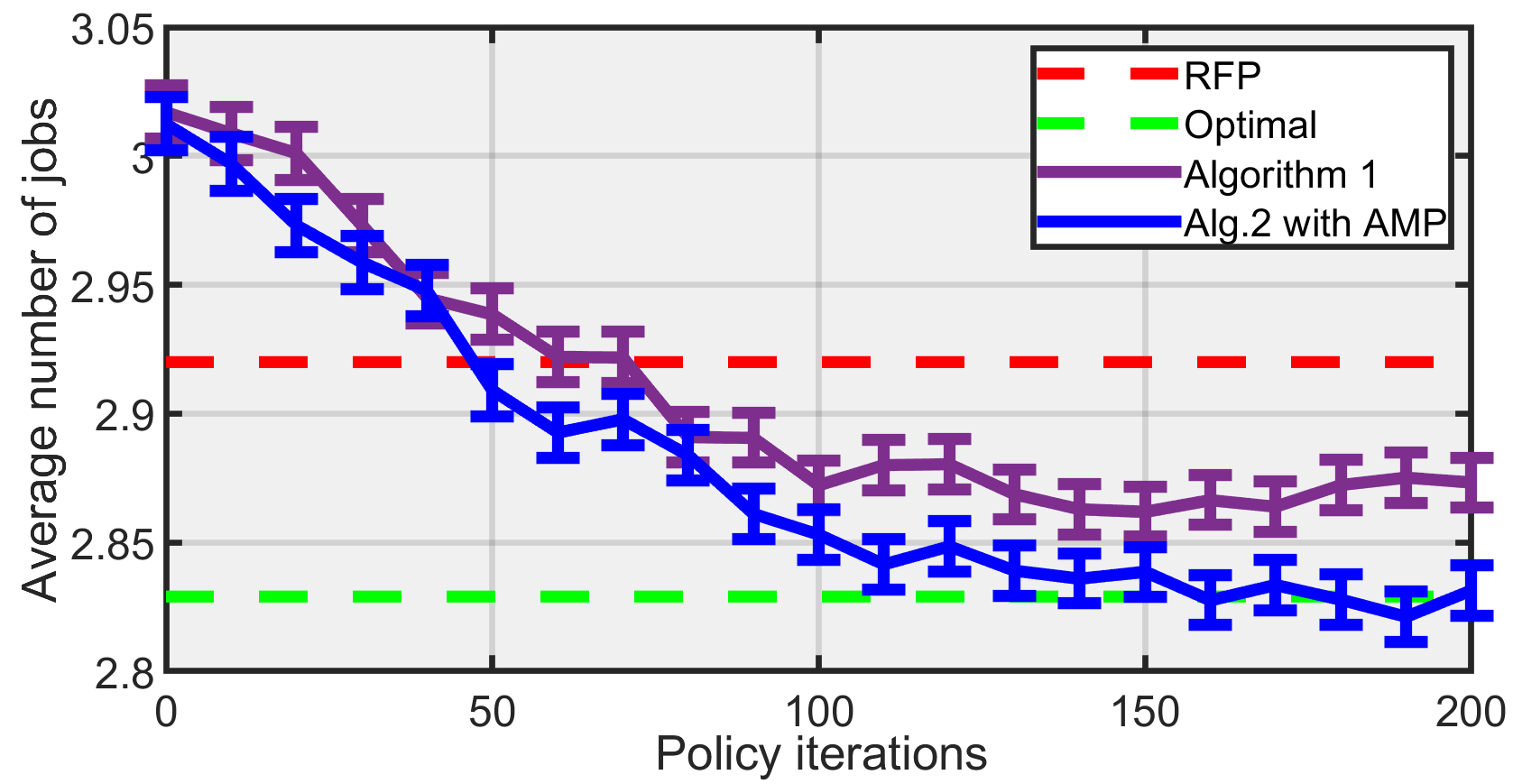}
     }
     \subfloat[Balanced heavy (BH) traffic\label{subfig-1:BH}]{%
       \includegraphics[ height=0.19\textwidth, width=0.33\textwidth]{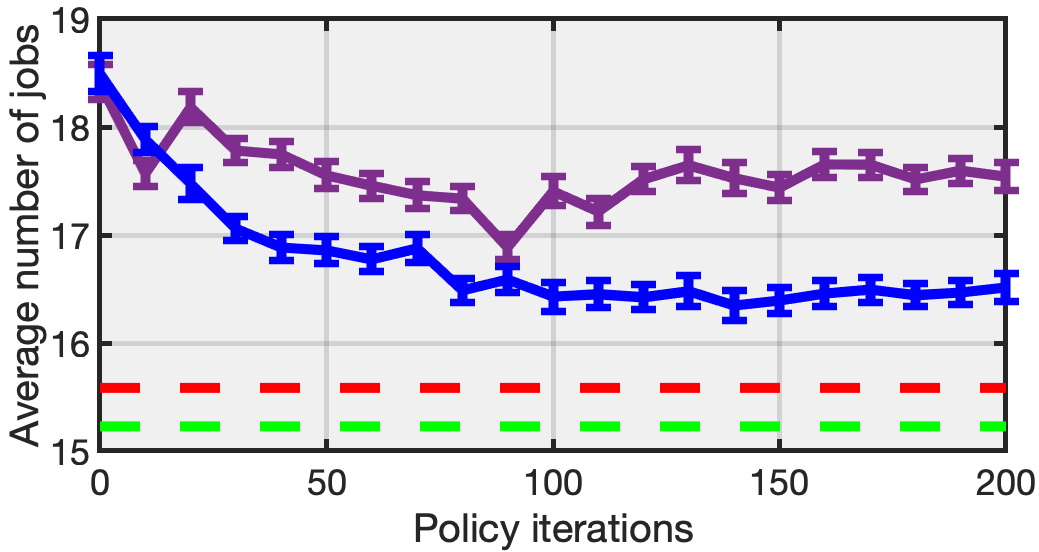}
     }

     \caption[Comparison of learning curves from Algorithm \ref{alg1} and Algorithm \ref{alg1amp}  on the criss-cross network with different traffic regimes.]{
     Comparison of learning curves from Algorithm \ref{alg1} and Algorithm \ref{alg1amp}  on the criss-cross network with different traffic regimes.
    The solid purple and blue lines show the performance of  the PPO policies obtained at the end of every 10th iterations of Algorithm \ref{alg1} and Algorithm \ref{alg1amp}, respectively; the dashed red lines show the performance of the robust fluid policy (RFP), and the dashed green lines show the  performance of the optimal policy.  }
     \label{fig:cc_opt}
   \end{figure}

\begin{table}[H]
\centering
\begin{tabular}{|c|c|c|c|c|c|c|c|}
  \hline
  Load regime  & $\lambda_1$ & $\lambda_3 $ & $\mu_1$ & $\mu_2$ & $\mu_3$ & $\rho_1$ & $\rho_2$\\\hline
  I.L. & 0.3 & 0.3& 2 & 1.5 &2& 0.3&0.2\\\hline
  B.L. & 0.3 & 0.3 & 2 & 1 &2 &0.3&0.3\\\hline
  I.M. & 0.6 & 0.6  & 2 & 1.5 &2& 0.6&0.4\\\hline
  B.M. & 0.6 & 0.6  & 2 & 1 &2& 0.6&0.6\\\hline
  I.H. & 0.9 & 0.9 & 2 & 1.5 &2& 0.9&0.6\\\hline
  B.H. & 0.9 & 0.9 & 2 & 1&2&0.9&0.9 \\

  \hline

\end{tabular}
 \caption{Load parameters for the criss-cross network of Figure \ref{fig:cc}. }\label{t:lp}
\end{table}

\begin{table}[H]
\centering
\begin{tabular}{|c|c|c|c|c|c|c|}
  \hline
  Load  regime  & DP (optimal) & TP & threshold & FP & RFP & PPO (Algorithm \ref{alg1amp})\\\hline
  I.L. & 0.671 & 0.678 & 0.679 & 0.678 &0.677&   $0.671\pm 0.001$ \\\hline
  B.L. & 0.843 & 0.856 & 0.857 & 0.857 &0.855&  $0.844\pm 0.004$\\\hline
  I.M. & 2.084 & 2.117 & 2.129 & 2.162 &2.133& $2.084\pm0.011$ \\\hline
  B.M. & 2.829 & 2.895 & 2.895 & 2.965 &2.920 &  $2.833\pm 0.010$\\\hline
  I.H. & 9.970 & 10.13 & 10.15 & 10.398 &10.096 & $10.014\pm0.055$\\\hline
  B.H. & 15.228 & 15.5 & 15.5 & 18.430 &15.585&   $16.513\pm 0.140$ \\

  \hline

\end{tabular}

 \caption[Average number of jobs per unit time in the criss-cross network under different policies.]{Average number of jobs per unit time in the criss-cross network under different policies.
  Column 1 reports the variances in the load regimes. }\label{tab:cc}%
\end{table}

   We observe that   Algorithm \ref{alg1amp} is not robust enough and converges to a  suboptimal policy when the criss-cross network operates in a balanced heavy load regime.
 We run Algorithm \ref{alg2} that uses discount factor $\gamma=0.998$, TD parameter $\lambda = 0.99$, and the AMP method for the value function estimation   at step 7.
  For each iteration we use $Q = 50$ parallel processes to generate trajectories, each with length $N=50,000$.
   We observe that Algorithm \ref{alg2} uses approximately 10 times fewer samples per iteration than  Algorithm \ref{alg1amp}. Algorithm \ref{alg2} outputs policy $\pi_{\theta_{200}}$ whose long-run average
   performance is $15.353\pm0.138$ jobs, which is lower than the RFP
     performance in \cite{Bertsimas2015}.
     We repeat the experiment with  Algorithm \ref{alg2} with the discounting, but we disable the AMP method in the value function estimation   at step 7.
      Figure \ref{fig:cc23} shows  that both variance reduction techniques (i.e. discounting, AMP method) have been necessary in Algorithm \ref{alg2} to achieve near-optimal performance.


\begin{figure}[H]
\centering%
\includegraphics[width=.7\linewidth]{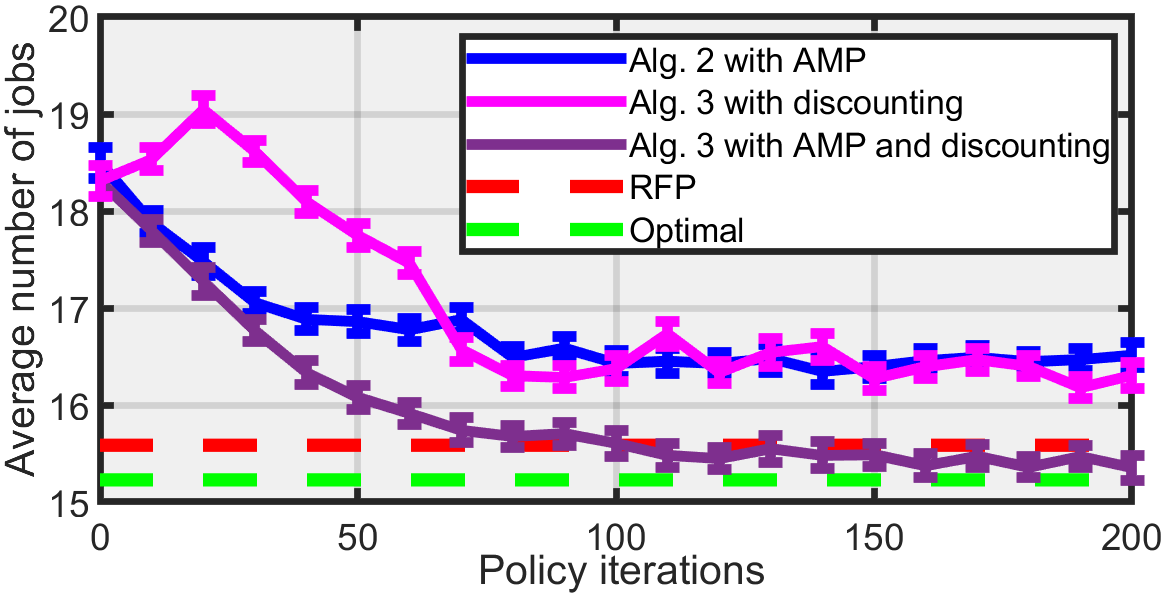}
\caption[Comparison of learning curves from Algorithm \ref{alg1amp} and Algorithm \ref{alg2}  on the criss-cross network with the balanced heavy regime.]{ Comparison of learning curves from Algorithm \ref{alg1amp} and Algorithm \ref{alg2}  on the criss-cross network with the balanced heavy regime. The solid blue and purple lines  show the performance of  the PPO policies obtained at the end of every 10th iterations of Algorithm \ref{alg1amp} and Algorithm \ref{alg2}, respectively. The solid pink   line shows the performance of  the PPO policies obtained at the end of every 10th iterations of Algorithm \ref{alg2} without the  AMP method. The dashed red line shows the performance of the robust fluid policy (RFP), and the dashed green line shows the  performance of the optimal policy.}
\label{fig:cc23}%
\end{figure}

Recall that in the uniformization procedure, the transition matrix
$\tilde P$ in (\ref{eq:unif}) allows ``fictitious'' transitions when
$\tilde P(x|x)>0$ for some state $x\in \X$.  The PPO algorithm
approximately solves the discrete-time MDP (\ref{eq:obj4}), which
allows a decision at every state transition, including fictitious ones.
The algorithm produces randomized stationary Markovian
policies.  In evaluating the performance of any such policies, we
actually simulate a DTMC $\left\{x^{(k)}:k=0,1, 2, \ldots\right\}$ operating under
the policy, estimating the corresponding long-run average cost as in
(\ref{eq:obj4}).  There are two versions of DTMCs depending on how
often the randomized policy is sampled to determine the next
action. In version 1, the policy is re-sampled only when a \emph{real}
state transition occurs. Thus, whenever a fictitious state transition
occurs, no new action is determined and server priorities do not change in this version.
In version 2, the policy is re-sampled at  \emph{every}
transition. Unfortunately, there is no guarantee that these two
versions of DTMCs yield the same long-run average cost. See
\cite[Example 2.2]{Beutler1987} for a counterexample.

When the randomized stationary Markovian policy is optimal for the
discrete-time MDP (\ref{eq:obj4}), under an additional mild condition,
the two DTMC versions yield the same long-run average cost
\cite[Theorem 3.6]{Beutler1987}. Whenever simulation is used to estimate
the performance of a randomized policy in this chapter, we use
version 1 of the DTMC. The reason for this choice is that the long-run
average for this version of DTMC is the same as the continuous-time
long-run average cost in (\ref{eq:obj3}), and the latter performance
has been used as  benchmarks in literature.
Although the final randomized stationary Markovian policy from our PPO
algorithm is not expected to be optimal, our numerical experiments
demonstrate that the performance results of the DTMC versions 1 and 2 are statistically
identical.  Table \ref{tab:cc2} reports  the performance of the DTMC versions 1 and 2  in the criss-cross network.   In Table \label{tab:cc2} column 2, the performance of Version 1 is identical to Table \ref{tab:cc} column 7.
 \begin{table}[H]
\centering
\begin{tabular}{|c|c|c|c|c|c|c|}
  \hline
  Load regime   & Version 1 performance with CIs &   Version 2 performance with CIs
  \\\hline
  I.L. &$0.671\pm 0.001$&       $0.671\pm 0.001$ \\\hline
  B.L. &  $0.844\pm 0.004$&    $0.844\pm 0.004$\\\hline
  I.M. & $2.084\pm0.011$ &       $2.085\pm0.011$ \\\hline
  B.M. &  $2.833\pm 0.010$ &      $2.832\pm 0.010$\\\hline
  I.H. & $10.014\pm0.055$ &   $9.998\pm0.054$\\\hline
  B.H. &   $16.513\pm 0.140$ &       $16.480\pm 0.137$ \\

  \hline

\end{tabular}

 \caption{Average number of jobs per unit time in the discrete-time MDP model of the criss-cross network under PPO policies.}\label{tab:cc2}%
\end{table}


\subsection{Extended  six-class queueing network}\label{sec:ext}

In this subsection, we consider the family of extended six-class networks from \cite{Bertsimas2015} and apply Algorithm \ref{alg2} to find good control policies.

Figure \ref{fig:fig1} shows the structures of the extended six-class networks.
We run experiments for 6 different extended six-class queueing networks with the following traffic parameters: $\lambda_1 = \lambda_3 = 9/140$, the service times are exponentially distributed with service rates determined  by the modulus after dividing the class index  by 6 (i.e. classes associated with server 1 are served with rates $\mu_1 = 1/8$, $\mu_2 = 1/2$, $\mu_3 = 1/4$ and classes associated with server 2 are processed with service rates $\mu_4 = 1/6$, $\mu_5 = 1/7$, $\mu_6 = 1$). The service rates
for the odd servers $S_1, ..., S_{\lfloor L/2\rfloor+1}$ are the same as the service rates for server 1, while the service rates for the even
servers $S_2, ..., S_{\lfloor L/2 \rfloor}$ are the same as the service rates for server 2. The load is the same for each station and is equal to $\rho = 0.9$.

\begin{figure}[tbh]
  \begin{center}
 \begin{tikzpicture}[server/.style={rectangle, inner sep=0.0mm, minimum width=.8cm,
     draw=green,fill=green!10,thick},
  buffer/.style={rectangle, rounded corners=3pt,
    inner sep=0.0mm, minimum width=.9cm, minimum height=.6cm,
    draw=orange,fill=blue!10,thick}]

  \node[server,minimum width=.4in, minimum height=2in] at (6,4) (S1)
  {$S_1$};

  \node[server,minimum width=.4in, minimum height=2in] (S2)
  [right=.75in of S1.east] {$S_2$};

\node[server,minimum width=.4in, minimum height=2in] (S3)
  [right=2.5in of S1.east] {$S_L$};

  \node[buffer] (B1) [left=.2in of $(S1.west)!.5!(S1.north west)$ ] {$B_1$};
  \node[buffer] (B2) [left=.2in of S1.west ] {$B_2$};
  \node[buffer] (B3) [left=.2in of $(S1.west)!.5!(S1.south west)$ ] {$B_3$};
  \node[buffer] (B4) [left=.2in of $(S2.west)!.5!(S2.north west)$ ] {$B_4$};
  \node[buffer] (B5) [left=.2in of S2.west ] {$B_5$};
  \node[buffer] (B6) [left=.2in of $(S2.west)!.5!(S2.south west)$ ] {$B_6$};

  \node[buffer] (B7) [left=.2in of $(S3.west)!.5!(S3.north west)$ ] { $B_{3(L-1)+1}$ };
  \node[buffer] (B8) [left=.2in of S3.west ] {$B_{3(L-1)+2}$};
  \node[buffer] (B9) [left=.2in of $(S3.west)!.5!(S3.south west)$ ] {$B_{3(L-1)+3}$};

\coordinate (source1) at ($ (B1.west) + (-.2in, +.2in)$);
\coordinate (source3) at ($ (B3.west) + (-.2in, -.2in)$);
  \coordinate (departure1) at ($ (B8-|S3.east) + (+.7in, +.0in)$);
 \coordinate (departure3) at ($ (B9-|S3.east) + (+.7in, 0in)$);
  \coordinate (S3minus) at ($(B7-|S3.east)+(0.7in,.7in)$);
  \coordinate (B2plus) at ($(B2.east)+(-1in,0in)$);
 \coordinate (L11) at ($(B7-|S2.east)+(+.1in,0in)$);
 \coordinate (L12) at ($(B7.west)+(-.1in,0in)$);
 \coordinate (L21) at ($(B8-|S2.east)+(+.1in,0in)$);
 \coordinate (L22) at ($(B8.west)+(-.1in,0in)$);
 \coordinate (L31) at ($(B9-|S2.east)+(+.1in,0in)$);
 \coordinate (L32) at ($(B9.west)+(-.1in,0in)$);

  \path[->, thick]
 (B2.east) edge (B2-|S1.west)
 (B3.east) edge (B3-|S1.west)
(B1.east) edge (B1-|S1.west)
(B4.east) edge (B4-|S2.west)
 (B5.east) edge (B5-|S2.west)
(B6.east) edge (B6-|S2.west)
(B7.east) edge (B7-|S3.west)
 (B8.east) edge (B8-|S3.west)
(B9.east) edge (B9-|S3.west)
 (B2plus) edge (B2.west)
(B4-|S1.east) edge (B4.west)
(B5-|S1.east) edge (B5.west)
(B6-|S1.east) edge (B6.west);

\draw[thick, dashed]  (L11) -- (L12);
\draw[thick,->] (L12) |- (B7.west);
\draw[thick] (L11-|S2.east) |- (L11);
\draw[thick, dashed]  (L21) -- (L22);
\draw[thick,->] (L22) |- (B8.west);
\draw[thick] (L21-|S2.east) |- (L21);
\draw[thick, dashed]  (L31) -- (L32);
\draw[thick,->] (L32) |- (B9.west);
\draw[thick] (L31-|S2.east) |- (L31);
\draw [thick,->] (source1) |- (B1.west);
\draw [thick,->] (B8-|S3.east) |-(departure1);
\draw [thick,->] (source3) |- (B3.west);
\draw [thick,->] (B9-|S3.east) |- (departure3);
  \draw [thick] (S3minus) -| (B2plus);
\draw [thick] (B7-|S3.east) -| (S3minus);
  \draw [thick] (B8-|S3.east)  |-(departure1);

\node [above,align=left] at (source1.north) {class 1\\ arrivals};
\node [below,align=left] at (departure1.south) {class $3(L-1)+2$\\ departures};
\node [below,align=left] at (source3.south) {class 3\\ arrivals};
\node [below,align=left] at (departure3.south) {class $3(L-1)+3$\\ departures};

\end{tikzpicture}
  \end{center}
\caption{Extended six-class queueing network.}\label{fig:fig1}
\end{figure}
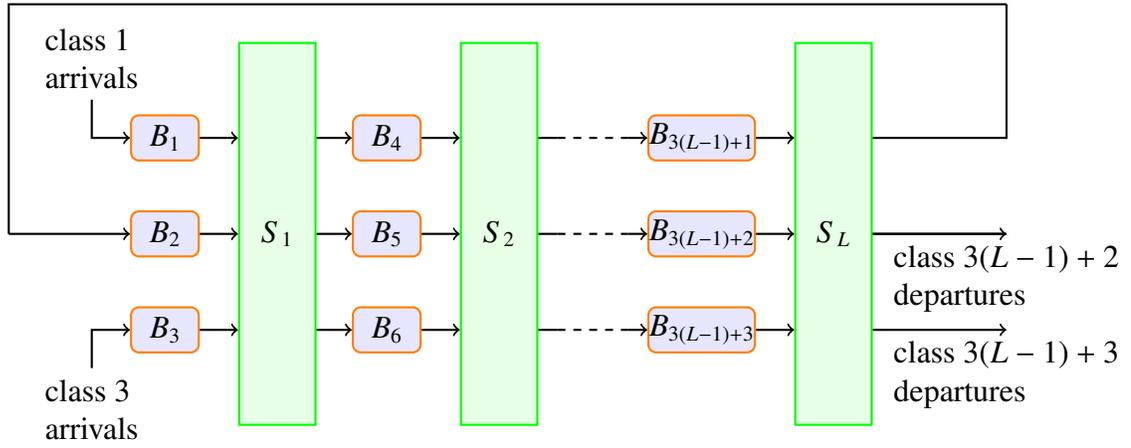

\begin{table}[tbh]
\centering%
\begin{tabular}{|c|c|c|c|c|c|}
  \hline
  No. of classes $3$L  & LBFS & FCFS & FP & RFP & PPO (Algorithm \ref{alg2}) with CIs\\\hline
  6 & 15.749 & 40.173 & 15.422 & 15.286 & $14.130\pm 0.208$\\\hline
  9 & 25.257 & 71.518 & 26.140 & 24.917& $23.269\pm0.251$ \\\hline
  12  & 34.660 & 114.860 & 38.085 & 36.857& $32.171\pm0.556$ \\\hline
  15  & 45.110  & 157.556  & 45.962 & 43.628& $39.300\pm0.612$  \\\hline
  18  & 55.724 & 203.418 & 56.857  & 52.980  & $51.472\pm 0.973$  \\\hline
  21  & 65.980 & 251.657 & 64.713 & 59.051 & $55.124\pm 1.807$\\
  \hline
\end{tabular}
\caption[Numerical results for the extended six-class queueing network in Figure \ref{fig:fig1}.]{Numerical results for the extended six-class queueing network in Figure \ref{fig:fig1}.}
\label{tab:tab6extRes}
\end{table}

We vary the size of the network between 6 and 21 classes to test the robustness of the PPO policies. In all experiments we generate $Q=50$ episodes with $N=50,000$ timesteps.  Table \ref{tab:rt} in Appendix Section \ref{sec:par} reports the running time of the algorithm depending on the size of the queueing network. We set the discount factor and TD parameter  to $\gamma=0.998$ and $\lambda = 0.99$, respectively.
 Table \ref{tab:tab6extRes} shows the performance of the PPO policy and compares it with other heuristic methods for the extended six-class queueing networks.  In the table FP and RFP refer to fluid and robust fluid policies \cite{Bertsimas2015}.    Table \ref{tab:tab6extRes}  reports the performance of the best RFP corresponding to the best choice of policy  parameters for each class.
LBFS refers to the last-buffer first-serve policy,
where the priority at a server is given to jobs with highest index.  FCFS refers to the first-come first-serve policy, where the priority at a server is given to jobs with the longest waiting time for service.

 For each  extended six-class network we consider a corresponding discrete-time MDP.
We fix a stable, randomized policy for the discrete-time MDP,  since the use of Xavier initialization
could yield an unstable NN policy  for extended six-class networks. We refer to this  stable, randomized policy as an expert policy.
We simulate a long episode of the MDP operating under the expert policy.  At each timestep we save the state at the time
and the corresponding probability distribution over the actions.
We use this simulated data set to train the  initial NN policy $\pi_{\theta_0}$.
 In our numerical experiments, we use the \textit{proportionally  randomized (PR)}  policy as the expert policy.
 If the network operates under the PR policy, when an arrival or service completion event occurs  at station $k$, a nonempty buffer $j$ receives a
priority over other classes at the station with
probability
\begin{align}\label{eq:PRPprob}
\frac{x_j}{\sum\limits_{i\in \B(k)}  x_i},
\end{align}
 where  $\B(k)$ is a set of buffers associated with server $k$ and $x = (x_1,\dotsc, x_J)$ is the vector of jobcounts at the time of the event   after accounting for job arrivals and departures.
  The priority  stays fixed   until the next arrival or service completion event occurs.
  The PR policy is  maximally stable for open MQNs, meaning that if the system is unstable under the PR policy, no other policy can stabilize it. See Appendix Section \ref{sec:PR} for the details.  An initial NN policy  for  PPO algorithm plays an important role in its learning process.  PPO algorithm may suffer from an
exploration issue when the initial NN policy  is sufficiently far from the optimal one \cite{Wang2019}.


In  each plot in  Figure \ref{fig:ext_ac}, we save policy NN
parameters $\{\theta_i\}_{i=0, 10, \dotsc, 200}$ every 10th policy
iteration.   For each saved policy NN, we conduct a separate long
simulation of the queueing network operating under the policy for accurate performance evaluation by providing a $95\%$
confidence interval of the long-run average cost.   For any of the six queueing networks, when the
load is high, the regeneration is rare.  Thus, we adopt the \textit{batch means} method to estimate the confidence interval
\cite[Section 6]{Henderson1997}.  For each policy from the set
$\{ \pi_{\theta_i}: i = 0, 10, \dotsc, 200 \}$, we simulate an episode
starting from an empty state $x = (0,\dotsc,0)$ until $5\times10^6$
arrival events occur.  Then we estimate average performance of the
policy based on this episode. To compute the confidence interval from
the episode, we split the episode into 50 sub-episodes (batches), see
also \cite{Nelson1989}. Pretending that the obtained 50 mean estimates
are i.i.d., we compute the $95\%-$confidence intervals as shown
in Figure~\ref{fig:ext_ac}.
\begin{figure}[tbh]
    \subfloat[6-classes network \label{subfig-6:IL}]{%
       \includegraphics[height=0.19\textwidth, width=0.33\textwidth]{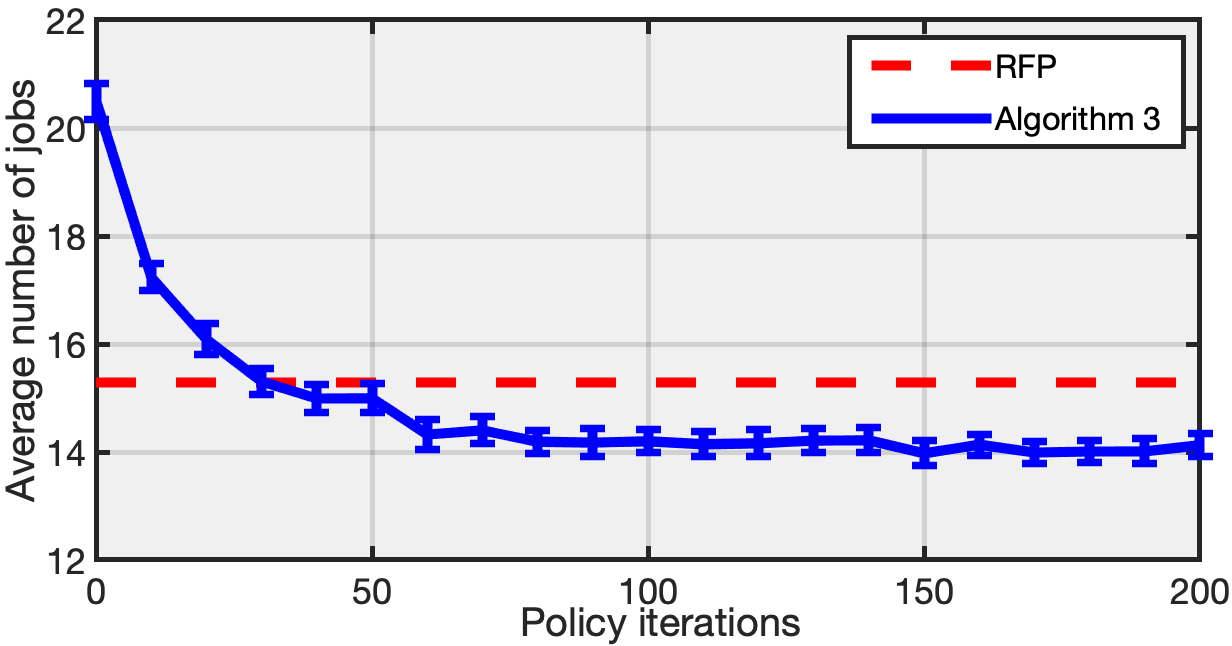}
     }
     \subfloat[9-classes network \label{subfig-4:IM}]{%
       \includegraphics[height=0.19\textwidth, width=0.33\textwidth]{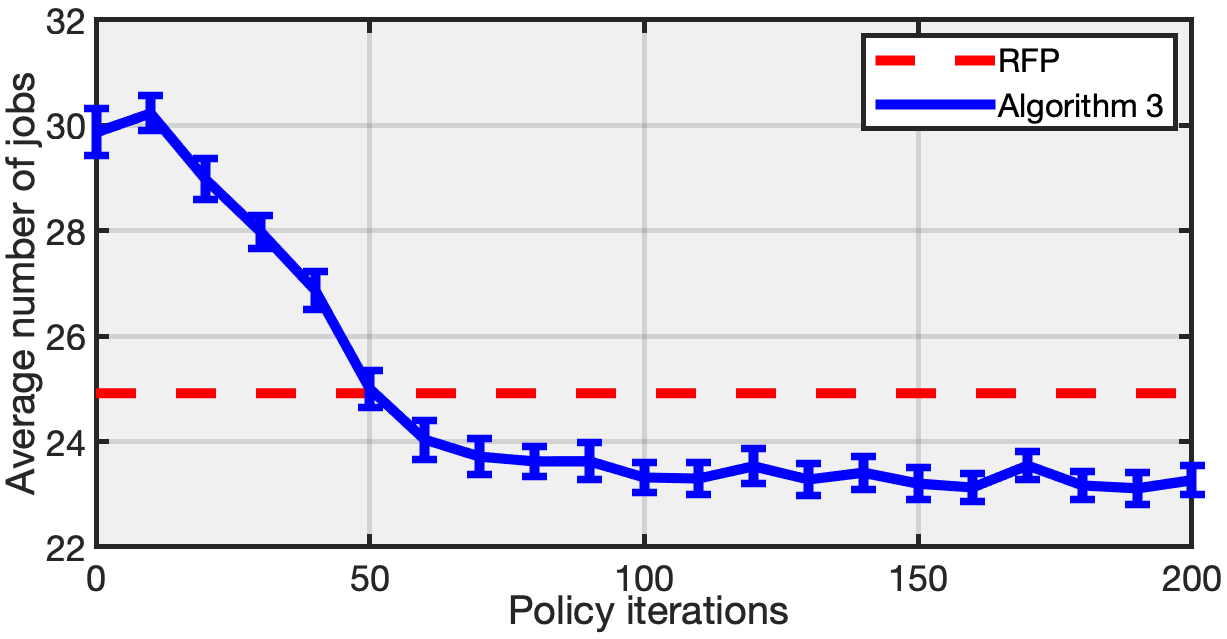}
     }
     \subfloat[12-classes network \label{subfig-2:IH}]{%
       \includegraphics[height=0.19\textwidth, width=0.33\textwidth]{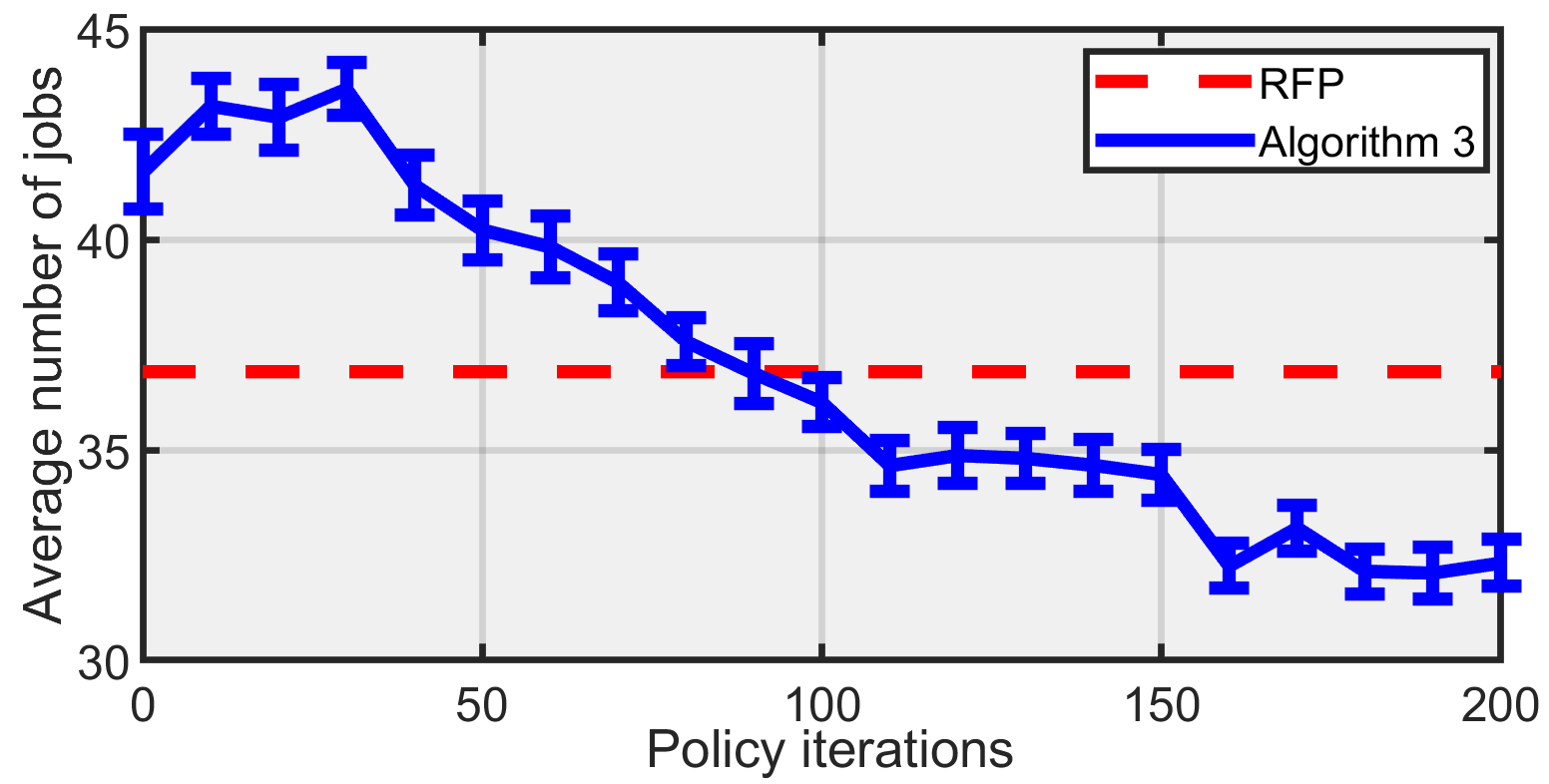}
     }\\
 \subfloat[15-classes network\label{subfig-5:BL}]{%
       \includegraphics[ height=0.19\textwidth, width=0.33\textwidth]{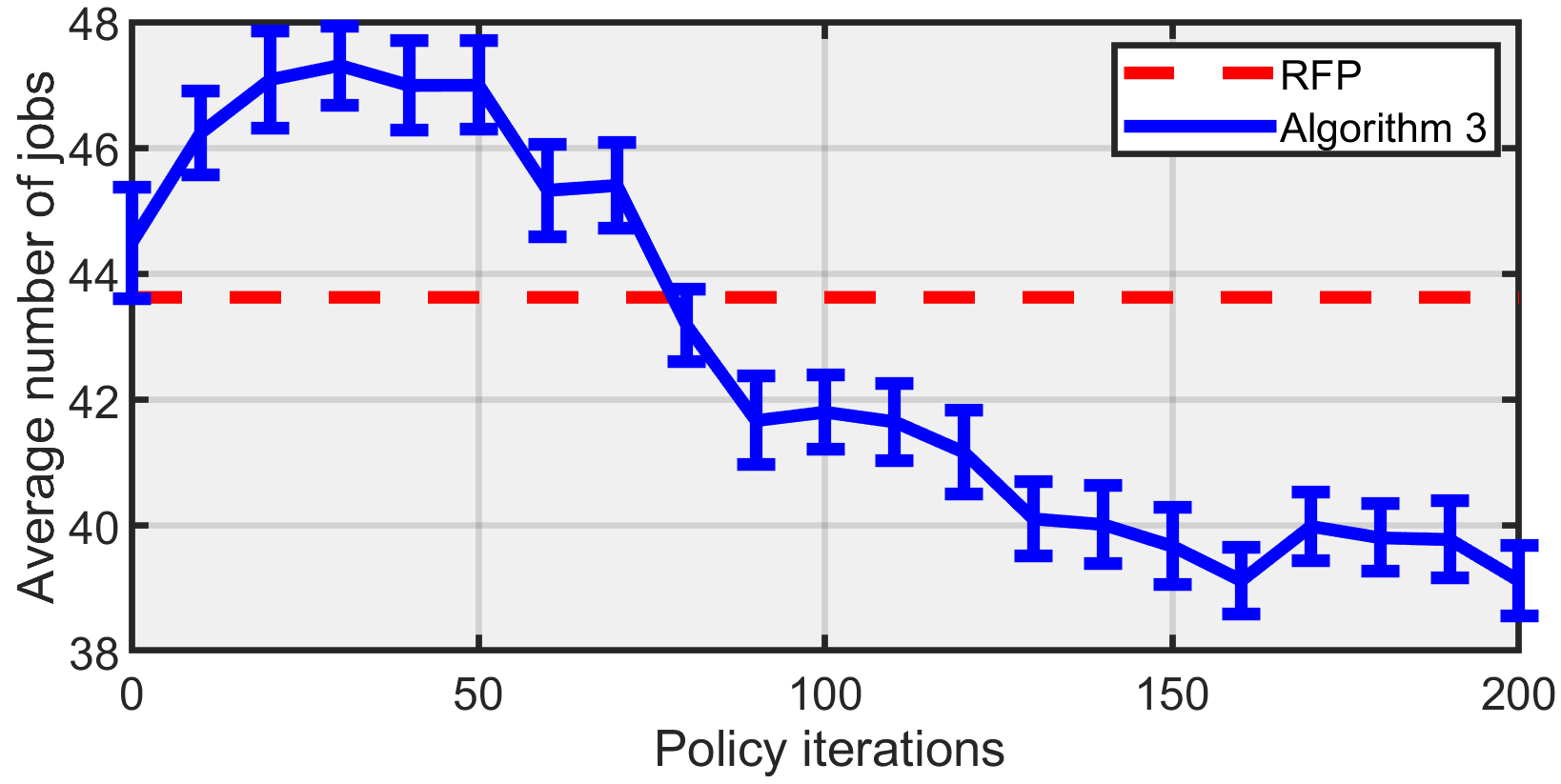}
     }
  \subfloat[18-classes network\label{subfig-3:BM}]{%
       \includegraphics[ height=0.19\textwidth, width=0.33\textwidth]{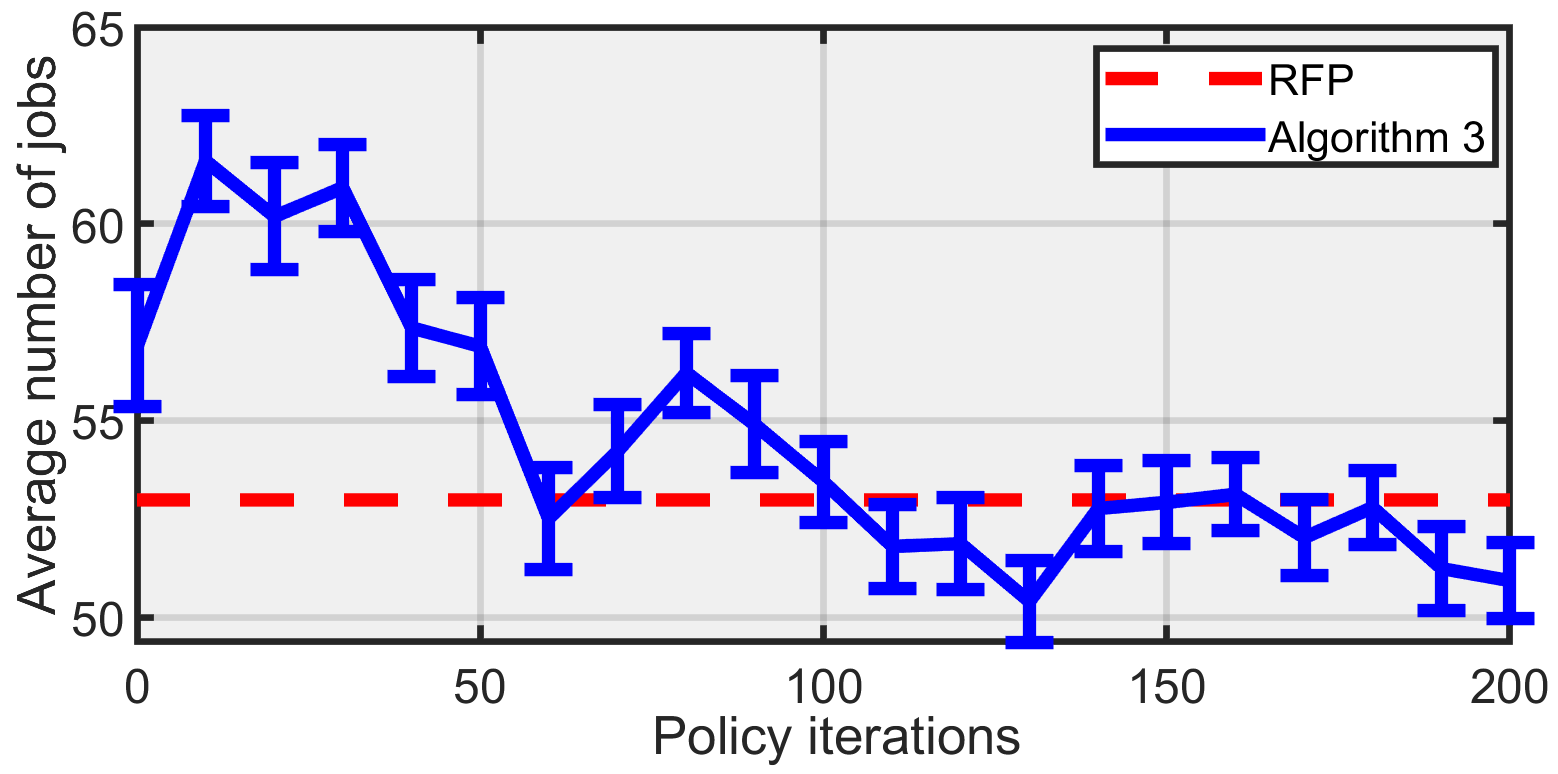}
     }
     \subfloat[21-classes network\label{subfig-1:BH}]{%
       \includegraphics[ height=0.19\textwidth, width=0.33\textwidth]{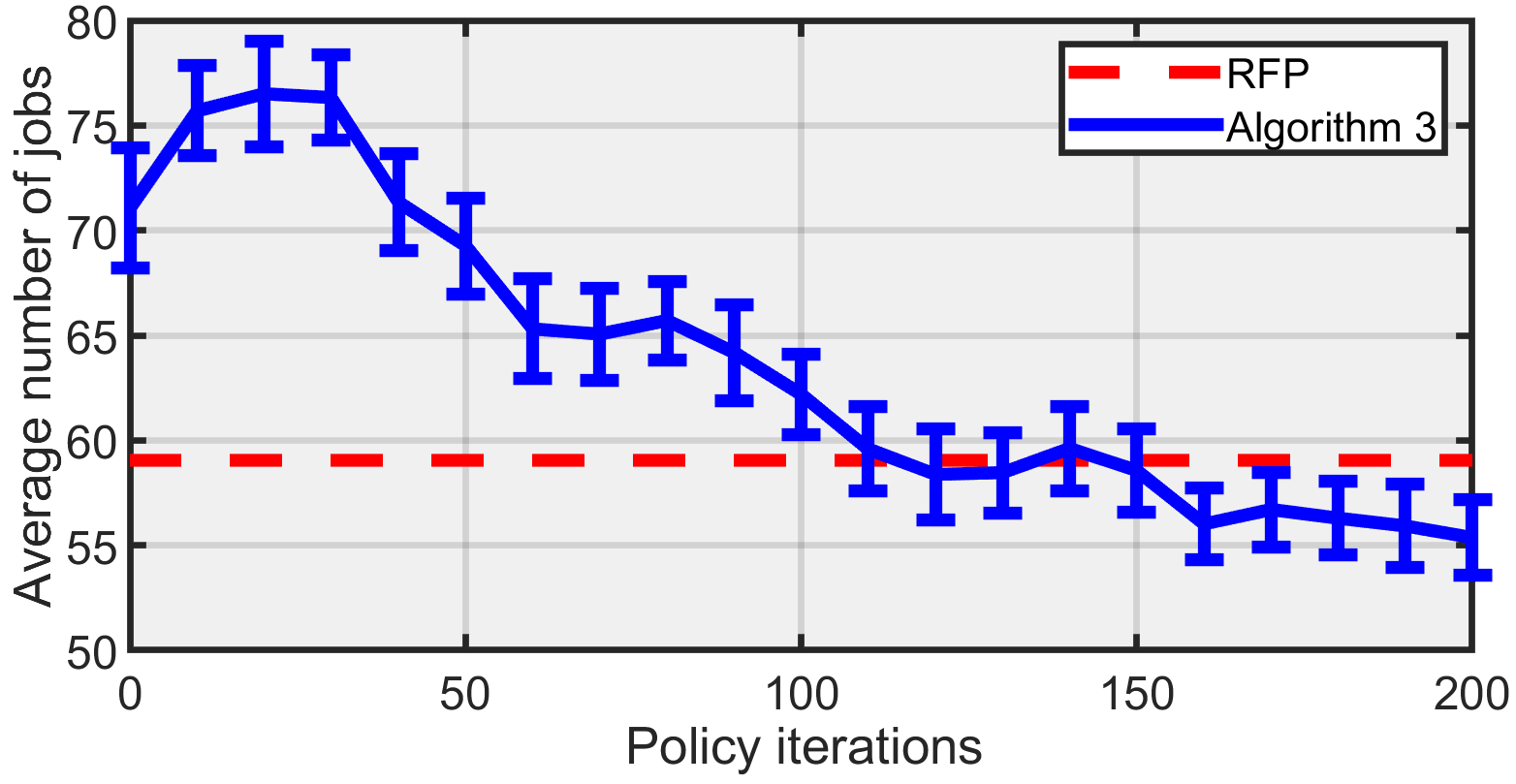}
     }
\caption[Performance of Algorithm~\ref{alg2} on six queueing networks.]{Performance of Algorithm~\ref{alg2} on six queueing networks. The solid   blue lines show the performance of  the PPO policies obtained at the end of every 10th iterations of Algorithm \ref{alg2}; the dashed red lines show the performance of the robust fluid policy (RFP).   }
     \label{fig:ext_ac}
   \end{figure}

In Section \ref{sec:ge} we discussed the relationship between the
GAE   (\ref{eq:GAE}) and  AMP  (\ref{eq:esf})  estimators.
 Figure \ref{fig:AMPvsGAE}  illustrates the benefits of using the AMP method in Algorithm \ref{alg2} on the 6-classes network.
The learning curve for the GAE estimator is obtained by replacing the value function estimation in line 7 of Algorithm \ref{alg2} with the  GAE estimator (\ref{eq:GAE}).

\begin{figure}[H]
\centering%
\includegraphics[width=.7\linewidth]{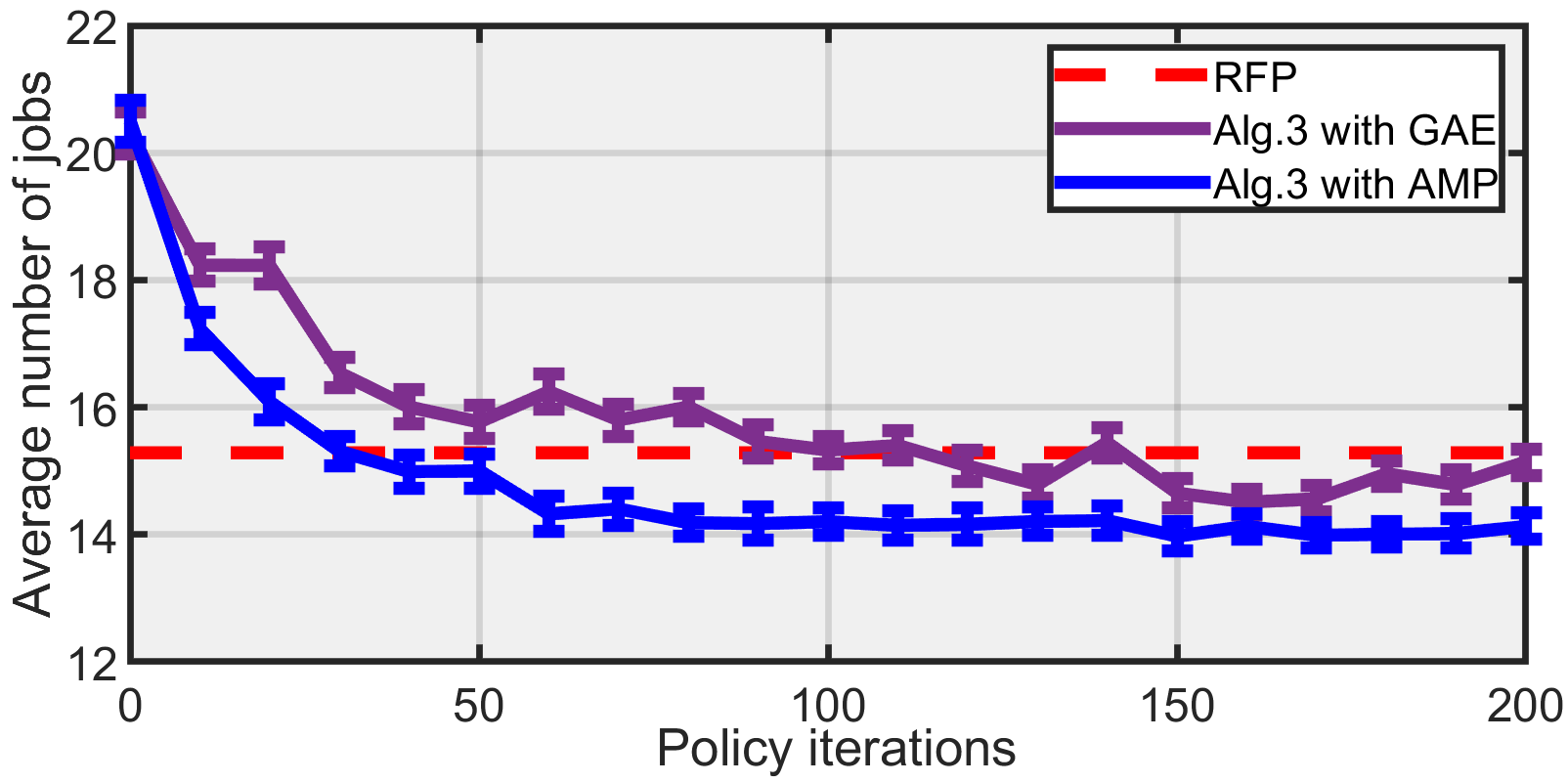}
\caption[Learning curves from Algorithm \ref{alg2} for the 6-class network.]{Learning curves from Algorithm \ref{alg2} for the 6-class network.
The solid blue and purple lines show the performance of  the PPO policies obtained from Algorithm \ref{alg2} in which
the value function estimates are computed by the AMP method   and by the GAE method, respectively; the dashed red line shows the performance of the robust fluid policy (RFP).
}
\label{fig:AMPvsGAE}%
\end{figure}

\subsection{Parallel servers network}\label{sec:nmodel}
In this section we  demonstrate that PPO Algorithm~\ref{alg2} is also effective for a stochastic processing network that is outside the model class of multiclass queueing networks.
Figure \ref{fig:Nmodel}  shows a  processing network system  with two independent Poisson input arrival flows, two servers, exponential
service times, and linear holding costs. Known as the  \textit{N-model network}, it first appeared in \cite{Harrison1998}.

\begin{figure}[H]
  \centering
\begin{tikzpicture}[inner sep=1.5mm] 
  \node[server] at (6,4) (S1)  {S1};
  \node[server] (S2) [right=1in of S1] {S2};
  \node[vbuffer] (B1) [above=.4in of S1.north] {B1};
  \node[vbuffer] (B2) [above=.4in of S2.north] {B2};
  \coordinate (source1) at ($ (B1.north) + (0in, .3in)$) {};
  \coordinate (source2) at ($ (B2.north) + (0in, .3in)$) {};
  \coordinate (departure1) at ($ (S1.south) + (0in, -.3in)$) {};
  \coordinate (departure2) at ($ (S2.south) + (0in, -.3in)$) {};

 \draw[thick,->]    (S1) -- (departure1);
 \draw[thick,->]    (S2) -- (departure2);

  \path[->, thick]
  (source1) edge node [left] {$\lambda_1=1.3\rho$} (B1)
  (source2) edge node [right] {$\lambda_2=0.4\rho$} (B2)
   (B1) edge node [left] {$m_1=1$} (S1.north)
   (B1.south east) edge node [right,near start] {\quad $m_2=2$} (S2)
   (B2) edge node [right] {$m_3=1$} (S2.north);

   \node [inner sep=2.0mm,below,align=left] at (departure1.south east) {server 1 \\ departures};
   \node [inner sep=2.0mm,below,align=left] at (departure2.south west) {server 2\\ departures};
   \node [inner sep=2.0mm,above,align=left] at (source1.north) {class 1\\ arrivals};
   \node [inner sep=2.0mm, above,align=left] at (source2.north) {class 2\\ arrivals};
 \end{tikzpicture}
 \caption{N-model network}
  \label{fig:Nmodel}
\end{figure}
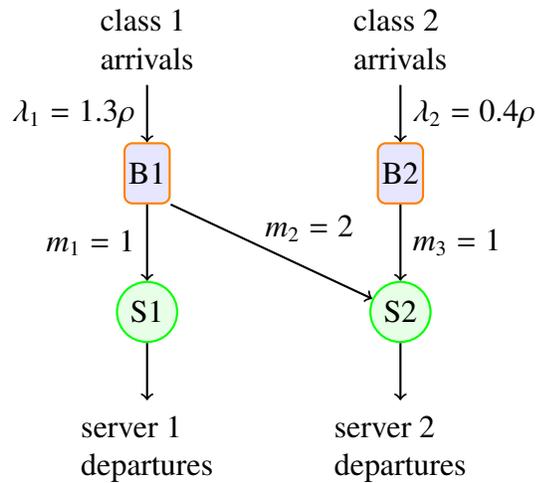

Jobs of class $i$ arrive according to a Poisson process at an average
rate of $\lambda_i$ such that $\lambda_1 = 1.3\rho$ and $\lambda_2 = 0.4\rho$ per unit time, where $\rho=0.95$ is a parameter that specifies the traffic intensity. Each job requires a single service
before it departs, and class 1 can be processed by either server 1 or server 2, whereas
class 2 can be processed only by server 2.   The jobs are processed by three different activities:
\begin{center}
activity 1 = processing of class 1 jobs by server 1,\\
activity 2 = processing of class 1 jobs by server 2,\\
activity 3 = processing of class 2 jobs by server 2.
\end{center}

We assume that  the  service at both servers is preemptive and work-conserving; specifically  we assume that activity 2 occurs only if there is at least one class 1 job in the system.

The service times for activity $i$ are exponentially distributed with mean $m_i$, where $m_1=m_3=1$ and $m_2=2$.
 The holding costs are continuously incurred at a rate of $h_j$ for each class $j$ job that remains within the system, with the specific
numerical values $h_1 = 3$ and $h_2 = 1.$ We note that all model parameter values   correspond to those in \cite{Harrison1998}.

We define $x = (x_1, x_2)\in \X$ as a system state, where $x_j$ is number of class $j$ jobs in the system. We use uniformization to convert the continuous-time control problem to a discrete-time control problem.
Under control $a = 1$ (class 1 has preemption high priority for  server 2) the transition probabilities are given by
\begin{align*}
&P\left((x_1+1, x_2)|(x_1, x_2)\right)= \frac{\lambda_1}{\lambda_1+\lambda_2+\mu_1+\mu_2+\mu_3},\\
&P\left((x_1, x_2+1)|(x_1, x_2)\right) = \frac{\lambda_2}{\lambda_1+\lambda_2+\mu_1+\mu_2+\mu_3},\\
&P\left((x_1-1, x_2)|(x_1, x_2),a= 1\right) = \frac{\mu_1\I_{\{x_1>0\}} + \mu_2\I_{\{x_1>1\}}}{\lambda_1+\lambda_2+\mu_1+\mu_2+\mu_3},\\
&P\left((x_1, x_2-1)|(x_1, x_2), a=1\right) = \frac{\mu_3\I_{\{x_2>0, x_1\leq 1\}}}{\lambda_1+\lambda_2+\mu_1+\mu_2+\mu_3},\\
&P\left((x_1, x_2 )|(x_1, x_2) \right) = 1 - P\left((x_1+1, x_2)|(x_1, x_2)\right)  - \\
 &\quad \quad -P\left((x_1, x_2+1)|(x_1, x_2)\right) -P\left((x_1-1, x_2)|(x_1, x_2) \right) - P\left((x_1, x_2 )|(x_1, x_2), 1\right),
\end{align*}
where $\mu_i  = 1/m_i$, $i=1, 2, 3.$

 Under control $a=2$ (class 2 has high priority), the only changes of transition probabilities are
 \begin{align*}
&P\left((x_1-1, x_2)|(x_1, x_2), a=2\right) = \frac{\mu_1\I_{\{x_1>0\}} + \mu_2\I_{\{x_1>1, x_2=0\}}}{\lambda_1+\lambda_2+\mu_1+\mu_2+\mu_3},\\
&P\left((x_1, x_2-1)|(x_1, x_2), a=2\right) = \frac{\mu_3\I_{\{x_2>0\}}}{\lambda_1+\lambda_2+\mu_1+\mu_2+\mu_3}.\\
\end{align*}

We define the cost-to-go function   as $g(x) := h_1x_1 +h_2x_2 = 3x_1+x_2.$  The objective is to find policy $\pi_\theta, \theta\in \Theta$ that minimizes the long-run average holding costs
\begin{align*}
\lim\limits_{N\rightarrow \infty} \frac{1}{N}  \E\left[\sum_{k=0}^{N-1}g (x^{(k)})\right],
\end{align*}
where $x^{(k)}$ is the system state after $k$ timesteps.

We use   Algorithm \ref{alg2} to find a near-optimal policy. Along with a learning curve from Algorithm \ref{alg2}, Figure \ref{fig:NmodelCurves}  shows the performance the best threshold policy with $T=11$ and the optimal policy. The threshold policy was proposed in \cite{Bell2001}. Server 2 operating under the  threshold policy  gives priority to class 1 jobs, if the number of class 1 jobs in the system is larger than a fixed threshold $T$.

 \begin{figure}[H]
\centering%
\includegraphics[width=.7\linewidth]{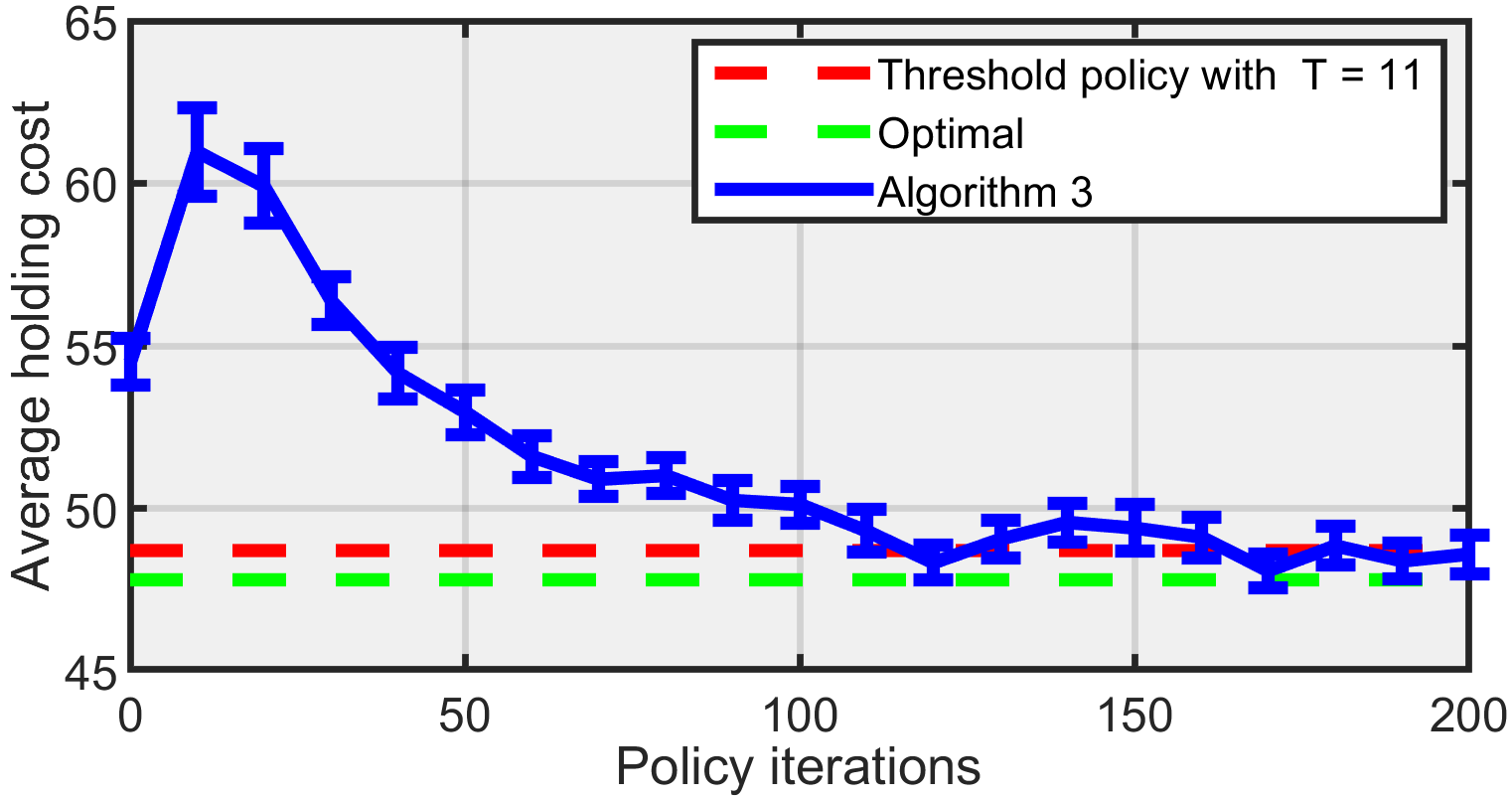}
\caption[Learning curves from Algorithm \ref{alg2} for the N-model network.]{Learning curves from Algorithm \ref{alg2} for the N-model network.  The blue solid line shows the performance of  PPO policies obtained from Algorithm \ref{alg2};       the dashed red line shows the performance of the threshold policy with $T=11$; the dashed greed line shows the performance of the optimal policy.}
\label{fig:NmodelCurves}%
\end{figure}

In Figure \ref{fig:NmodelPolicies} we show the control of randomized PPO policies obtained after 1, 50, 100, 150, and 200 algorithm iterations.  For each policy  we depict the probability distribution over two possible actions for states $x\in \X$ such that $0\leq x_j\leq 50,$ $j=1, 2$, and compare  the PPO, optimal, and  threshold policies.

\begin{figure}[H]
\begin{center}
    \subfloat[after 1 iteration \label{subfig-1:1}]{%
       \includegraphics[height=0.22\textwidth ]{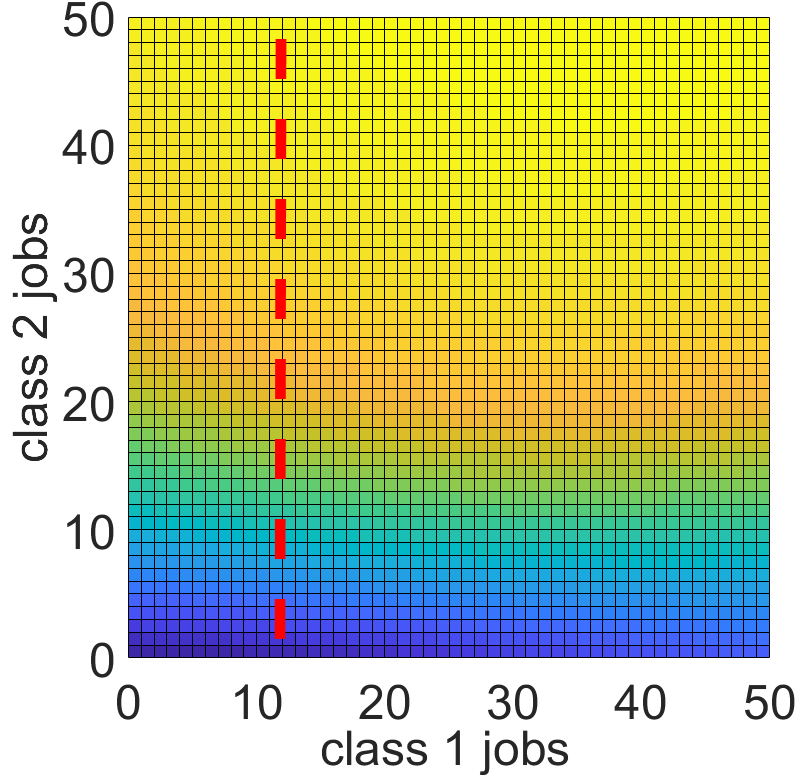}
     }\hspace{8pt}%
     \subfloat[after 50 iterations \label{subfig-2:50}]{%
       \includegraphics[height=0.22\textwidth ]{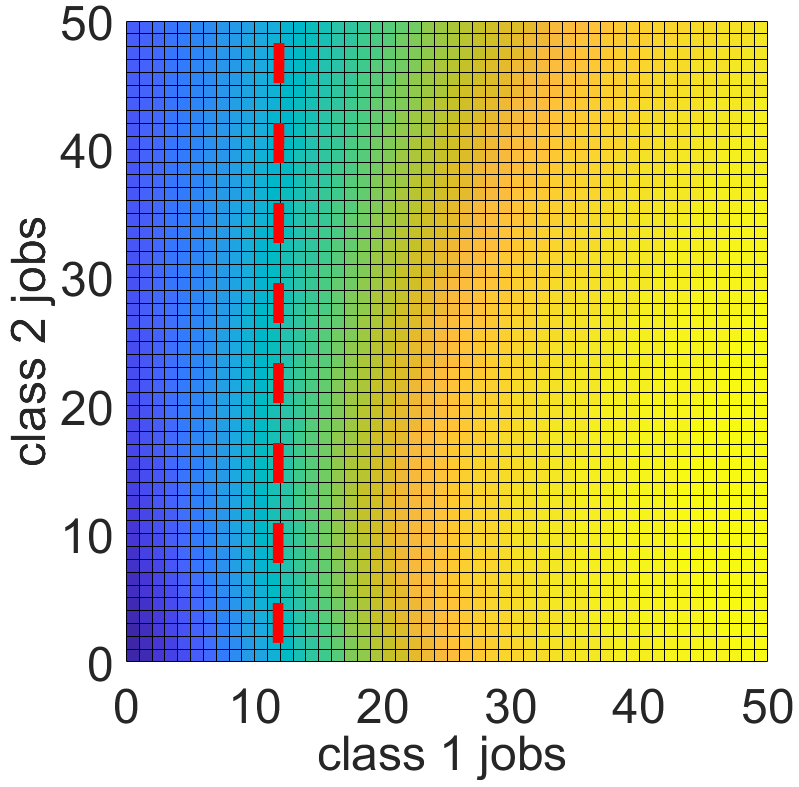}
     }\hspace{8pt}%
     \subfloat[after 100 iterations \label{subfig-3:100}]{%
       \includegraphics[height=0.22\textwidth ]{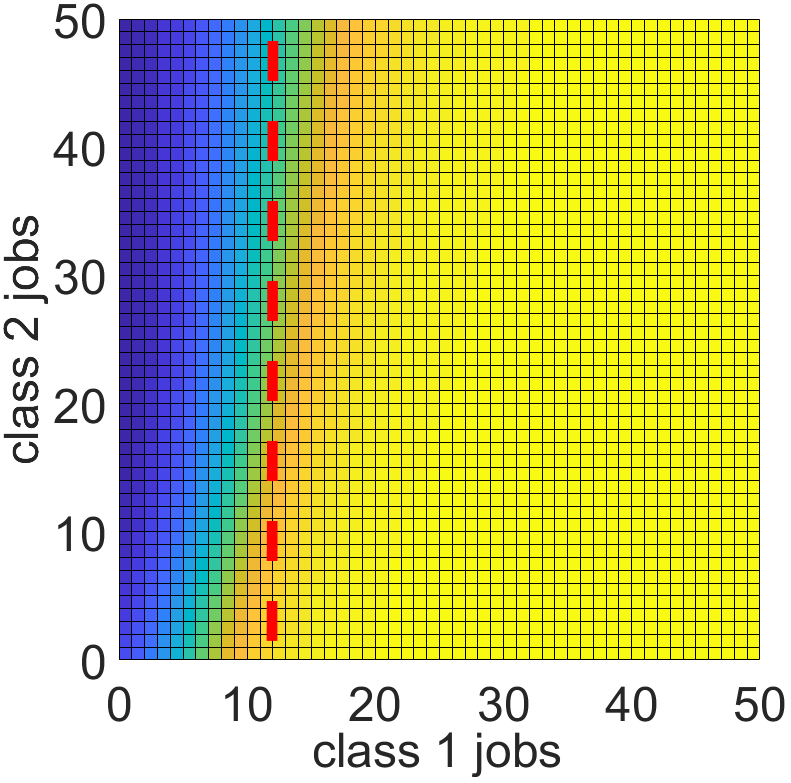}
     }\hspace{8pt}%
 \subfloat[after 150 iterations \label{subfig-4:150}]{%
       \includegraphics[height=0.22\textwidth ]{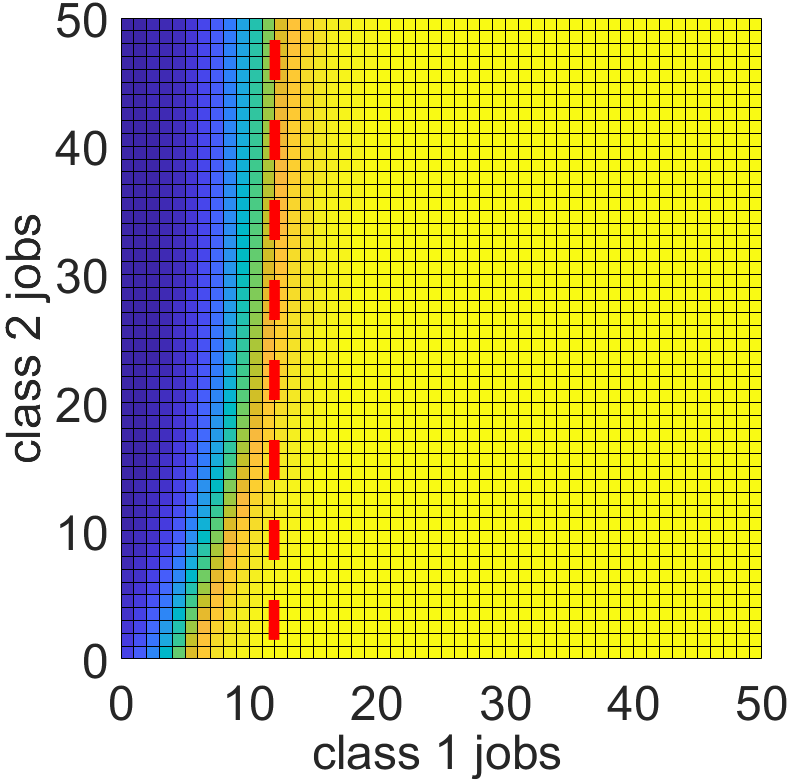}
     }\\
  \subfloat[after 200 iterations\label{subfig-5:200}]{%
       \includegraphics[ height=0.25\textwidth ]{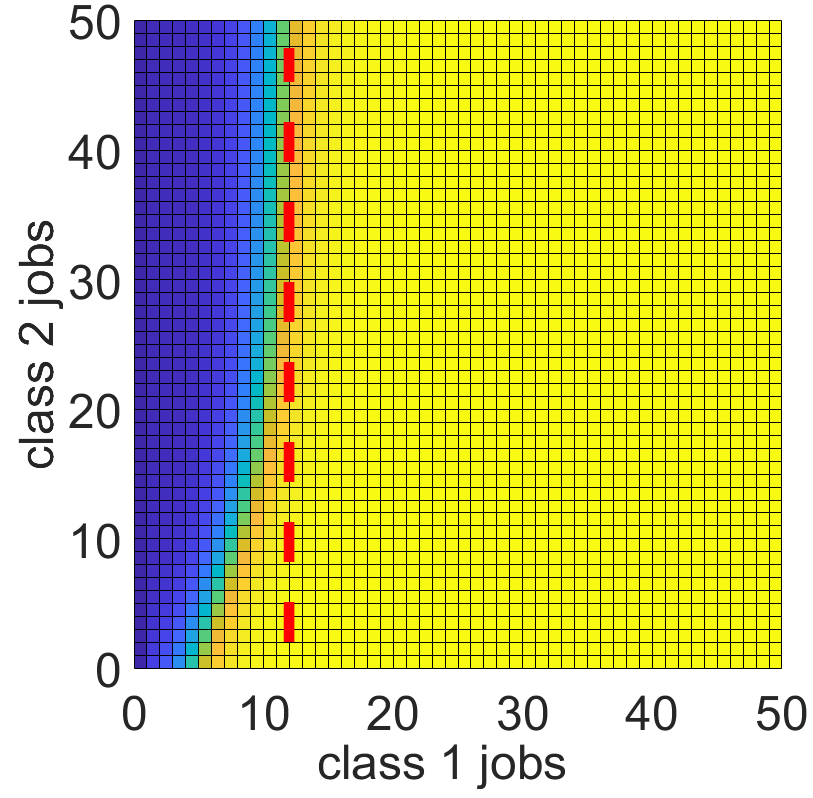}
     }\hspace{8pt}%
     \subfloat[Threshold policy \label{subfig-6:thr}]{%
       \includegraphics[height=0.25\textwidth ]{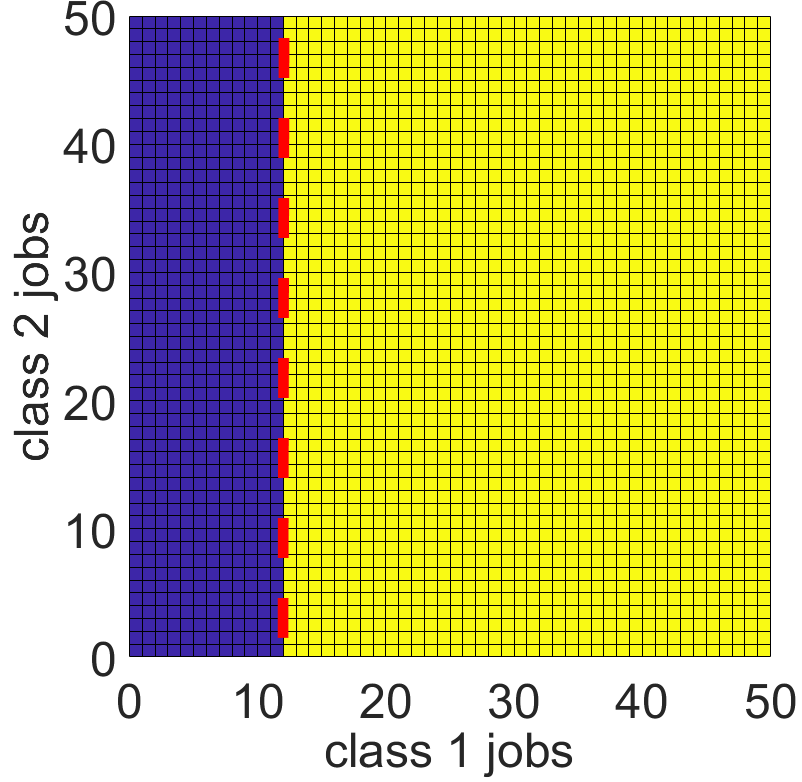}
     }\hspace{8pt}%
\subfloat[Optimal policy \label{subfig-7:opt}]{%
       \includegraphics[height=0.25\textwidth ]{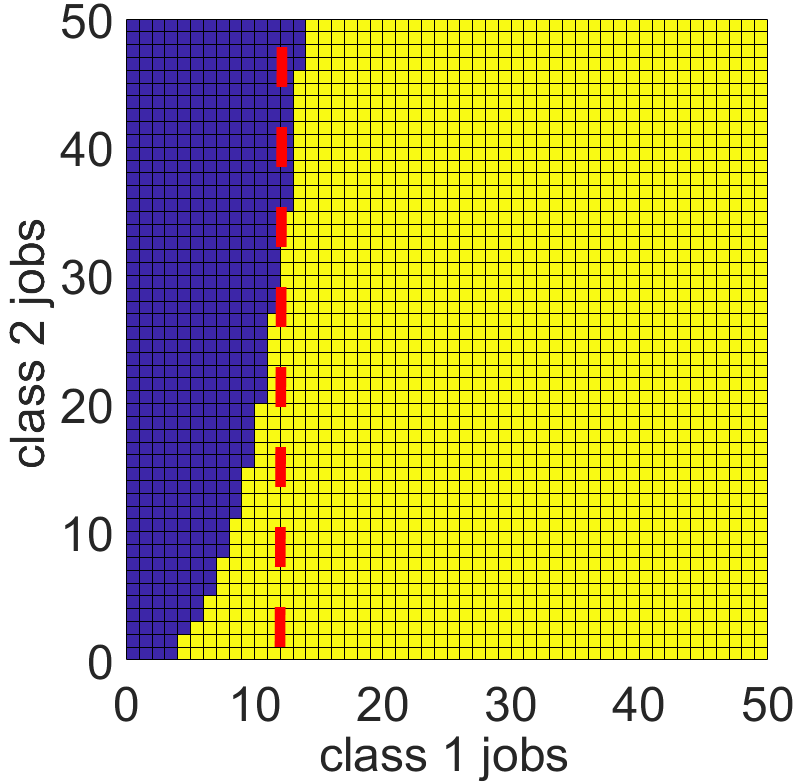}
     }\hspace{8pt}%
     \subfloat {%
       \includegraphics[height=0.25\textwidth ]{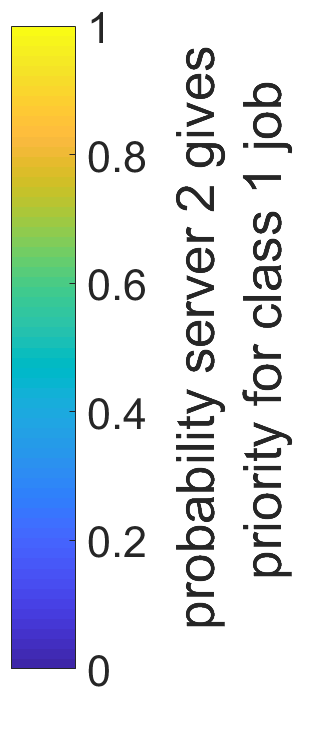}
     }
\end{center}

\caption[Evolution of the PPO policies over the  learning process and comparing them with threshold and optimal policies.]{Evolution of the PPO policies over the  learning process and comparing them with threshold and optimal policies.  The probability that server 2 gives priority to class 1 jobs is shown by a color gradient for system states that have less than 50 jobs in each buffer:
server 2 gives priority to class 1 jobs (yellow), server 2 gives priority to class 2 jobs (blue), and the dashed red lines represent the threshold policy with $T=11.$
}
     \label{fig:NmodelPolicies}
   \end{figure}

\section{Conclusion to Chapter \ref{ch:1}}\label{sec:conc_to_ch1}

This chapter proposed a method for optimizing the long-run average performance in queueing network control problems. It provided a theoretical justification for extending the PPO algorithm in \cite{Schulman2017} for Markov decision problems with infinite state space,  unbounded costs, and the long-run average cost objective.  Our idea  of applying Lyapunov function approach  has a potential to be adapted for other advanced policy gradient methods. We believe that theoretical analysis and performance comparison   of policy gradient \cite{Marbach2001, Peters2008}, trust-region \cite{Schulman2015}, proximal \cite{Schulman2017}, soft actor-critic \cite{Haarnoja2018}, deep Q-learning \cite{Mnih2015, Hessel2018}, and other policy optimization algorithms  can be of great benefit for the research community and deserve a separate  study.

The success of    the PPO algorithm implementation  largely depends on the accuracy of Monte
  Carlo estimates of the relative value function in each policy
  iteration. Numerical experiments using sets of fixed hyperparameters  showed that introducing an appropriate discount factor has the largest effect on the variance reduction of a value function estimator, even though the discount factor introduces biases,  and that introducing a control variate by using the approximating martingale-process (AMP)  further reduces the variance. Moreover, when system loads vary between light and moderate, regeneration  estimators also reduce  variances. The experiments also implied that   AMP estimation, rather than   GAE estimation, is preferable when the transition probabilities are known. Most of hyperparameters, including the discount factor and length of episodes,  were  chosen by experimental tuning rather than a theoretically justified scheme. It is desirable to investigate how to select the  discount factor,  length of episodes,  and other hyperparameters  based on queueing network load, size, topology.

Our numerical  experiments demonstrated that Algorithm~\ref{alg1amp} applied for the criss-cross network control optimization can obtain policies with long-run average performance within $1\%$ from the optimal.  For large-size networks PPO Algorithm~\ref{alg2} produced effective control policies that either outperform or perform as well as the current alternatives. For an extended six-class queueing network, PPO policies outperform the robust fluid policies on average by $10\%$.  The algorithm can be applied for a processing network control problem if the  processing network admits a uniformization representation under any feasible control policy. A wide class of such processing networks  described in  \cite{DaiHarrison2020}.  As an example, we provide the numerical experiment for the N-model network. Although this chapter considered only queueing networks with preemptive service, the proposed algorithm can also be applied for queueing networks with non-preemptive service policies. Two modifications  to the MDP formulation in Section \ref{sec:MQN} are required  if  a queueing network operates under non-preemptive service policies. First, a service status of each server should be included into the system state representation along with the  jobcount vector. Second, at each decision time a set of feasible actions should be restricted based on each server status.

Several research questions should be pursued in future  studies. First, future research could examine the necessity  of $\V$-uniform ergodicity  assumption  in Theorem \ref{thm:main}, and whether this assumption can be replaced by $\V$-ergodicity.   Second, one of the key components  of our PPO algorithm for MQNs is the proposed PR expert policy. Further research is needed to design an algorithm that does not require the knowledge of a stable expert policy. Third, our numerical experiments show that the variance reduction techniques are important for a good
performance of the algorithm. Future investigations are desirable to develop more sample efficient simulation methods, potentially, incorporating problem structure knowledge.

Complexity of the queueing control optimization problems highly depends
not only on the network topology, but on the traffic
intensity. As the network traffic intensity
  increases the long-term effects of the actions become less predictable that presents a challenge  for any reinforcement
  learning algorithm.  We believe that multiclass queueing networks should serve as  useful benchmarks for testing reinforcement learning methods.

\chapter{Scalable Deep Reinforcement Learning for Ride-Hailing}\label{ch:ride_hailing}

Following Chapter \ref{ch:1}, we continue to explore how   deep reinforcement learning (RL) can be used  in various processing network control problems. In this chapter we adapt  proximal policy optimization  algorithm \cite{Schulman2017}     for  order dispatching and relocation optimization in a  ride-hailing transportation network.

A ride-hailing service  is the next generation of taxi service that uses online platforms and mobile apps to connect passengers with drivers. Lyft, Uber, and Didi Chuxing together serve more than 45 million passengers per day \cite{Schlobach2018}. One of the important goals for these companies is to provide a reliable, trustworthy means of transportation, able to fulfill most, if not every, passenger's request \cite{Ferenstein2015}.

 A centralized planner of the ride-hailing service arranges cars in the system, matching them with ride requests from the passengers. Motivated by an empty-car routing mechanism, we follow \cite{Braverman2019} and assume that the centralized planner may also relocate an empty (without a passenger) car to another location in anticipation of future demand and/or shortage.
 Thus, the centralized planner   assigns tasks to multiple drivers over a certain time horizon controlling future geographical distribution of the cars. The centralized planner seeks to allocate enough  drivers at each region to fulfill expected ride requests.
 The optimization of cars routing in ride-hailing services is one of the most challenging problems among vehicle routing
problems \cite{Toth2014, Psaraftis2016}.

   In \cite{Braverman2019}  the authors proposed a closed queueing network model of a ride-hailing service under the assumption of \textit{time-homogeneous} traffic   parameters (e.g. passengers arrival rates, cars travel times). They
 formulated a fluid-based optimization
problem and  found an asymptotically optimal empty-car
routing policy in the “large market” parameter region. For
a \textit{time-varying} traffic pattern,  each decision time the traffic parameters were averaged over the finite time window. The averaged values were used to formulate the fluid-based
optimization problem  as if the ride-hailing service had a time-homogeneous traffic pattern.  The result was a time-dependent
lookahead control policy that  was not designed to be optimal.

In  \cite[Chapter 4]{Feng2020} the authors formulated  a Markov decision process (MDP) model of the ride-hailing system considered in \cite{Braverman2019}.     
The MDP formulation complies with the RL control optimization framework that does not require a centralized planner to know the traffic parameters.
  The ride-hailing operations optimization problem can be considered as an RL problem, i.e.  a (model-free) MDP problem in which the underlying dynamics is unknown, but optimal actions can be learned from sequences of data (states, actions, and rewards) observed or generated under a given policy.
 We follow \cite{Feng2020} and consider the centralized planner that receives real-time data on existing ride
requests and driver activities and assigns tasks to drivers at each decision epoch. The
decision epochs occur at discrete times, and the time between every two consecutive decision epochs is fixed. At
each decision epoch the centralized planner must solve the following combinatorial problem:   each available car should be either  matched with a passenger’s ride request, relocated to another location with no passengers, or asked to stay at its current location until the next decision epoch. The centralized planner’s action
  space grows exponentially with the number of agents, which
  presents a scalability challenge for any policy optimization
  RL algorithm.

Ride-hailing is one of several real-world application domains where deep RL has already been implemented into production.  In 2017 DiDi company deployed a deep RL algorithm for order dispatching in its production
system \cite{Qin2020a}. DiDi's algorithm was designed  to optimize car-passenger matching ignoring empty-car routing. Reported   A\slash B tests showed significant
improvement (0.5\% - 2\%) against the production baseline in several cities in China.  The RL algorithm designed by DiDi takes a
single-driver perspective, i.e. each driver follows a separate control policy that aims to optimize each driver's income separately.

The idea of using deep RL  for ride-hailing services control optimization attracted much attention and support in the scientific community as well, see a survey paper \cite{Qin2022a}. We mention several papers most relevant to our research. Related to empty-car routing,  \cite{Oda2018} employed a   deep Q-network   algorithm, proposed in \cite{Mnih2015},  to proactively dispatch cars to meet  future  demand while limiting   empty-car routing times. The algorithm was applied to find the optimal dispatch actions for individual cars, but did not take into account agents' interactions to achieve scalability. We note that the use of
deep RL algorithms to learn the optimal actions from the perspective of individual cars has been studied intensively \cite{Oda2018, Wang2018}, yet this approach can be ``intrinsically inadequate to be used in a production dispatching system which demands coordinations among multiple agents'' \cite{Tang2019}.
In~\cite{Tang2019, Xu2018a,  Shi2020} the authors applied  a two-stage optimization procedure consisting of  a deep RL algorithm used for a policy evaluation followed by dispatching policy optimization done by solving a  bipartite graph matching.
In~\cite{Ke2019} the authors proposed to delay the car assignment for some passengers in order to   accumulate more   drivers and waiting passengers in the matching pool, where the passengers and cars would be connected later via optimal bipartite matching.  
A multi-agent
actor-critic RL algorithm was developed to optimize the choice of  the delayed times for the passengers.

Over the last few years, deep RL algorithms with conservative policy updates \cite{Kakade2002, Schulman2015, Schulman2017} have become popular control optimization algorithms because of their good performance. These algorithms  involve a neural network (NN), which parametrizes  control policies. The size of  the NN output layer is equal to the number of actions.  Such NN architecture makes
 the algorithms  computationally expensive because the number of  parameters in the NN grows exponentially with the number of agents.

In this chapter we  suggest a way to use the deep RL algorithms with conservative policy updates  for ride-hailing services control optimization.
 We incorporate   a special decomposition of actions  and   assigns a different role to
the  control policy that \emph{sequentially} matches drivers with tasks. 
The control policy, observing current time, outstanding ride requests, and available cars, suggests
a  trip  from location A to location B. Then, the centralized planner matches the generated trip type with a driver and, preferably, with a passenger requesting a trip from A to B.
This trip-generating control policy is used repeatedly until all available cars have been assigned   some task.
Thus, addressing all available cars, the centralized planner forms an action for the current decision epoch. The  idea of actions decomposition is not new and many authors used it to address the scalability issue arisen out of large action spaces, see, for example, \cite{Spivey2004, Mao2016}. We focus our discussion on the actions decomposition applicability for ride-hailing system controls, in particular, 
scalability of an  NN architecture used to parametrize control policies.
A deep RL algorithm is applied to optimize the control policy  that should generate the most beneficial trip to fulfill at a current system state.  We use a proximal policy optimization (PPO) algorithm \cite{Schulman2017} for the  policy optimization  in our numerical experiments. A preliminary numerical experiment with the proposed  decomposition of actions and adapted PPO algorithm was performed in  \cite{Feng2020}.

We summarize the major contributions of this chapter:
\begin{enumerate}
\item In Section \ref{subsec:action} we propose a special decomposition for the MDP model of ride-hailing transportation network actions by
sequentially assigning tasks to the drivers. We discuss why the new actions structure resolves the scalability problem
and enables the use of deep RL algorithms for control policy optimization.
\item In Section \ref{sec:control_policy} we justify the use of PPO algorithm to solve MDPs with the proposed actions structure. In particular, we derive a novel policy improvement bound for the finite horizon setting.  
\item  In Section \ref{sec:num-study} we test the proposed PPO on the nine-region transportation network. We also mention the experiment with  the PPO algorithm on  the five-region transportation network from \cite{Feng2020}.  The resulting
policies outperform  the  time-dependent lookahead policy proposed in \cite{Braverman2019} and achieve the state-of-the-art performance for both networks. The benefits of empty-car routing   are demonstrated via a comparison test on the nine-region transportation network. Additional  experiment  is conducted disabling the empty-car routing. The matching rate of the best policy  learned  via the PPO algorithm for such configuration  was significantly lower than the performance of the policy learned with the enabled  empty-car routing.        
\end{enumerate}

This chapter is primarily based on the research reported in \cite{Feng2021}.

\section{The transportation network}
\label{sec:ridesharing-network}

 In this section we describe our model of the ride-hailing service and transportation network, following \cite{Braverman2019, Feng2020}. The service consists of a centralized planner, passengers requesting rides, and a fixed number of geographically distributed agents (cars). The transportation network consists of $N$ cars distributed across a service territory divided into $R$ regions.
For   ease of exposition, we assume that each working day (``episode'') of the  ride-hailing service starts at the same time  and  lasts for $H$ minutes.

  We assume that the number of passenger arrivals at region $o$ in the $t$-th minute (i.e., $ t$ minutes elapsed since the start of the working day) is a Poisson random variable with mean
  \begin{equation}
    \lambda_o(t),  \text{ for each } o=1,\dotsc,R, \, t=1,\dotsc,H. \nonumber
\end{equation}
 The collection of all  Poisson random variables is independent. Passengers only arrive \emph{after} a working day starts (i.e., there are no passengers at the 0-th minute).

Upon arrival  at region $o$,  a passenger travels to region $d$ with probability that depends on  time $t$, origin region $o$, and destination region $d$
\begin{equation}
    P_{od}(t), \quad o,d=1,\dotsc,R, \, t=1,\dotsc,H. \nonumber
\end{equation}

After a trip   from region $o$ to $d$ has been initiated, its duration is \emph{deterministic} and equals to
\begin{equation}
    \tau_{od}(t), \quad o,d=1,\dotsc,R, \, t=1,\dotsc,H.
    \label{eq:trav-time}
\end{equation}

  We let
 \begin{equation}\label{eq:tau_d}
 \tau_d :=\max_{t=1,\dotsc,H, \, o=1,\dotsc,R}\tau_{o d}(t),\quad   d=1,\dotsc, R
 \end{equation} be the maximum travel time to  region $d$ from any region of the transportation network at any time.

 While Section 1  in \cite{Braverman2019}  assumed that travel times were random variables having
an exponential distribution, the experiments in \cite[Section 3.2.1]{Braverman2019} were conducted under constant travel times. For ease of exposition, we use deterministic travel times in Section \ref{sec:prob-form} below.

Patience time denotes a new passenger's maximum waiting time for a car. We assume that each passenger  has a deterministic patience time   and we fix it as equal to $L$ minutes.  We assume that the centralized planner knows the patience time.

 We say a car is  \textit{available} if it is at or less than $L$ minutes away from its final destination, where $L$ is  the patience time.
In real time, the centralized planner receives   ride requests, observes the location and activity of each car in the system, and considers three types of tasks for the available cars: (1) car-passenger matching, (2) empty-car routing, and  (3)  do-nothing  (a special type of   empty-car routing).
  We assume that each passenger requires an immediate response to his or her request.   If the centralized planner assigns a matching between a passenger and an available car, we assume  the passenger has to accept the matching and to wait  up to $L$ minutes for the assigned car to pick him or her up. A passenger who is not matched with a car in the first decision epoch leaves the system before the next decision epoch. Hence, a passenger waits up to one decision epoch to be matched and, if matched, up to $L$ minutes  to be picked up.

Unlike \cite{Braverman2019},  the constraint that   only  cars   idling at the passenger's location can be matched with the passenger are relaxed. We assume that the centralized planner can match cars with subsequent ride requests before current trips are completed.
We assume that the   patience time   satisfies
 \begin{equation}
    L<\min_{t=1,\dotsc,H, \, o=1,\dotsc,R}
    \tau_{o d}(t), 
    \quad \text{for each } d = 1,\dotsc,R
    .
    \label{eq:a-trav-time}
\end{equation}
The assumption implies that the travel time of any trip   is larger than the patience time.    Therefore, no more than one subsequent trip can be assigned to a driver.

If a car  reaches its destination and has not been matched with a new passenger, it becomes empty.
The centralized planner may let the empty car stay at the destination or relocate to another region.  For the former, we note that the centralized planner's decision belongs to  the do-nothing task and does not cost any travel time. The centralized planner will be able to assign the car a new task at the next decision epoch. For the latter, the centralized planner chooses a region  for the relocation and the travel time remains the same as in  equation~\eqref{eq:trav-time}.
 Unlike \cite{Braverman2019}, the centralized planner can assign two empty-car routing tasks in succession.

\section{Optimal control problem formulation}
\label{sec:prob-form}

 Our goal is to find a control policy for the centralized planner that maximizes the total reward collected during one working day by the entire ride-hailing service. Following \cite{Feng2020}, we formulate the problem as a   finite-horizon, discrete-time,  undiscounted MDP. We set the time interval between two successive epochs to one minute.
 Under this setting, the time in minute, $t = 1,\dotsc,H$, also represents the decision epochs. As a result, a passenger waits at most one minute for a decision.  

\subsection{State space}
\label{subsec:state}

 The state space $\X^\Sigma$ of the MDP includes states $x^{(t)} = \left[ x^{(t)}_e, x^{(t)}_c, x^{(t)}_p\right]$, such that each state consists of three components:   current epoch  $x^{(t)}_e := t$,  cars status  $x^{(t)}_c$, and   passengers status  $x^{(t)}_p$. 

The cars status component represents the  number of   cars of every   type   in the system:
\begin{equation}
x^{(t)}_c  := \left(x^{(t)}_c(d,\eta )~\Big| ~ \substack{d = 1,\dotsc, R, \\ \eta = 0,1,\dotsc,\tau_d, \tau_d+1,\dotsc,\tau_d+L}\right), \nonumber
\end{equation}
 where $x^{(t)}_c(d,\eta )$ is the number of cars in the system whose \emph{final} destination region is $d$,  and the \emph{total} remaining travel time  (``distance'') to the destination is equal to $\eta$, and $\tau_d$ is the maximum travel time to region $d$ defined by (\ref{eq:tau_d}).

The passengers status component is equal to
\begin{equation}
x_p^{(t)} := \left( x_p^{(t)}(o, d)~\Big|~o,d=1,\dotsc, R \right),\nonumber
\end{equation}
where $x_p^{(t)}(o, d)$ characterizes the number of passengers in the system requesting rides from region $o$ to region $d$.

\subsection{Sequential decision making process}
\label{subsec:action}

At each epoch $t$, the centralized planner  observes  the system state $x^{(t)}$, and makes a decision $a^{(t)}$   that should  address  all    $I_t$ available cars, where
\begin{equation}
   I_t:= \sum_{o = 1}^R \sum_{\eta = 0}^L x^{(t)}_c (o, \eta).\nonumber
\end{equation}

 We let $\A^\Sigma$ denote the action space of the MDP.
  We propose to  decompose every decision $a^{(t)}\in \A^\Sigma$ into a sequence of \textit{atomic actions}, each addressing a single available car. Therefore, we consider action $a^{(t)}$ as:
\begin{align*}
  a^{(t)} := \left(  a^{(t,1)}, \dotsc, a^{(t, I_t)}\right),
\end{align*}
where $a^{(t, i)}$ is an atomic action that encodes a trip  by  one of the available cars. We let $\A$  denote the atomic action space. We  note that $\A = \{(o,d)\}_{o,d=1}^R$.

We call the  sequential generation of atomic actions a \textit{sequential decision making process} (SDM process). We let $x^{(t,i)}$   denote  a state of the SDM process after   $i-1$ steps, for each decision epoch $t=1,\dotsc, H$.  Figure \ref{fig:SDM} illustrates the SDM process  at  decision epoch $t$.

We let $\X$ be the state space of the SDM process.
Each  state $x^{(t,i)}$ of the SDM process has four components $x^{(t,i)} :=\left [x^{(t,i)}_e, x^{(t, i)}_c, x^{(t,i)}_p, x^{(t, i)}_\ell\right]$, where, as in the original MDP, the first three components $x^{(t,i)}_e$, $x^{(t,i)}_c$, $x^{(t, i)}_p$  represent current epoch,  cars status, and passengers status, respectively, and a new component $x^{(t,i)}_\ell$ tracks the cars exiting the available cars pool until the next decision epoch.
The SDM process is initialized with state $x^{(t, 1)}$ such that
$x^{(t,1)}_e = x^{(t)}_e,$ $x^{(t, 1)}_c = x^{(t)}_c,$ $x^{(t,1)}_p= x^{(t)}_p$,
 and   $x^{(t,1)}_\ell$ is  a zero vector,  for each decision epoch $t=1,\dotsc,H$.

 \begin{figure}
     \centering
     \includegraphics[width=0.8\textwidth]{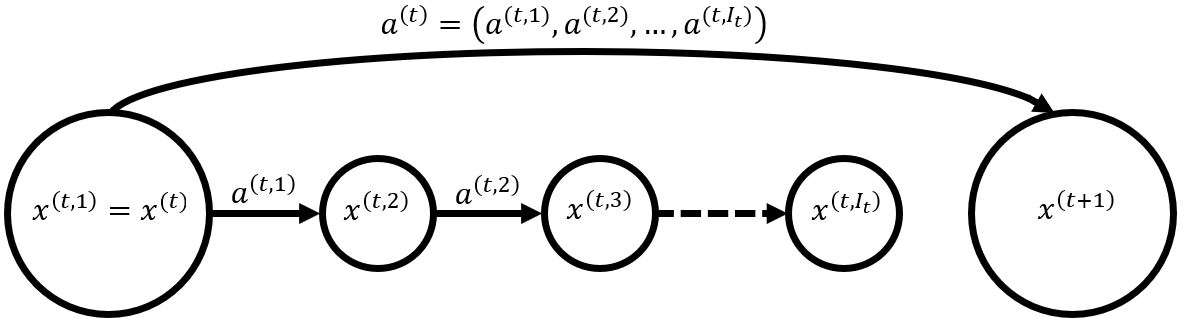}
     \caption[The SDM process  at   decision epoch $t$]{The SDM process  at   decision epoch $t$. Control policy $\pi$ sequentially generates atomic actions $a^{(t, 1)}$, $a^{(t, 2)}$, ..., $a^{(t, I_t)}$ which form an action $a^{(t)}\in \A^\Sigma$ for the original MDP.}
     \label{fig:SDM}
 \end{figure}

Each atomic action represents a \textit{feasible} trip type $a^{(t,i)}  = (o^{(t,i)}, d^{(t,i)})$,
where $o^{(t,i)},$ $d^{(t,i)}$ are the origin   and  destination regions of the trip, respectively.
Action $a^{(t,i)}$  is \emph{feasible} if there exists an available car that is $L$ minutes (or less) away from the origin region $o^{(t,i)}$, (i.e. $\sum_{\eta = 0}^{L} x_c^{(t,i)}(o^{(t,i)}, \eta) > 0$).
 Although atomic action $a^{(t,i)}$  only encodes the   origin and destination of a trip, we set a few rules that   specify which car will conduct the trip and if the car will carry a passenger.
 Among the set of available cars, we select  the car closest to origin $o^{(t,i)}$ to take the trip.
Then,  we  prioritize car-passenger matching over empty-car routing, (i.e. if there exists a passenger requesting a ride from the trip origin to the trip destination, we assign the car to the requesting passenger; if there are several passengers requesting such a ride, we assign the car to a passenger at random).
If there is no passenger requesting a ride from $o^{(t,i)}$ to $d^{(t,i)}$, we interpret atomic action $a^{(t,i)}$ as either an empty-car routing task or a do-nothing task depending on the value of $d^{(t,i)}$. Namely,  if the trip relocates the car to a different region ($d^{(t,i)} \neq o^{(t,i)}$), then the car fulfills an empty-car routing task. Otherwise, we interpret the atomic action as a "do nothing" task, and the car becomes a \textit{do-nothing} car.

Once an available car,  possibly  a do-nothing car,  has been assigned  a task at the $i$th step of the SDM process, the centralized planner should exclude it from the available cars pool. If the car has been assigned  a passenger ride request  or an empty-car routing task,   the cars status component of the SDM process state is updated such that the car becomes associated with its new final destination.
The car is automatically excluded from the available cars pool by assumption (\ref{eq:a-trav-time}). The do-nothing tasks require   special transitions that the original MDP does not have, so we use $x^{(t,i)}_\ell$ to track the do-nothing cars
\begin{equation}
x^{(t,i)}_\ell := \Big(x^{(t,i)}_\ell(d,\eta) ~\Big|~ d = 1,\dotsc, R, ~\eta = 0,1,\dotsc,L \Big), \nonumber
\end{equation}
 where $x^{(t,i)}_\ell(d,\eta)$ is the number of  do-nothing cars which drive to or idle at region $d$, $\eta$ minutes away from their destinations.
  The do-nothing component excludes do-nothing cars from the  available cars pool until the next decision epoch.

 The atomic actions are generated sequentially  under   control policy
 $\pi:\X\rightarrow \A,$
 which is a mapping from the state space into a set of the trip types. The  control policy $\pi$, given a current state of the SDM process $x^{(t,i)}$, sequentially generates feasible atomic actions $ a^{(t,i)} = \pi(x^{(t,i)})$. The   SDM process terminates when all available cars become unavailable cars, producing action $a^{(t)} = (a^{(t,1)}, \dotsc, a^{(t,I_t)})$.

 At each decision epoch the centralized planner observes system state $x^{(t)}$ and  exercises control policy $\pi$ sequentially  in  the SDM process to obtain action $a^{(t)}$. Then the transition of the system to the next state $x^{(t+1)}$ occurs according to the dynamics of the original MDP.

\subsection{Reward functions and objective}

A car-passenger matching generates an immediate reward
\begin{align*}
    g^{(t)}_f(o, d, \eta), \quad  o, d = 1,\dotsc,R;   ~ \eta = 0, 1, \dotsc, L;  ~    t = 1,\dotsc, H,
\end{align*}
where  $o$ and $d$ are the passenger's origin and destination regions, respectively, $\eta$ is the distance (in minutes) between the  matched car  and the passenger's location, and $t$ is the time of the decision.  The superscript $f$ denotes a full-car trip.

Every  empty-car routing atomic action  generates  a cost that  depends on   origin region $o$, destination region $d$,   and   decision time $t$
\begin{equation*}
     g^{(t)}_e(o,d), \quad o,~d = 1,\dotsc,R; ~ t = 1,\dotsc, H,
\end{equation*}
where the superscript denotes an empty-car trip.

We assume the do-nothing actions do not generate any rewards.
Therefore, a one-step reward function generated on  the $(i-1)$-th step of SDM process at epoch $t$ is equal to
\begin{align*}
    g(x^{(t,i)}, a^{(t,i)}) =  \begin{cases}
    g^{(t)}_f(o,d,\eta), \text{ if action }a^{(t,i)} \text{ implies a car-passenger matching,}\\
    -g_e^{(t)}(o,d), \text{ if action }a^{(t,i)} \text{ implies an empty-car routing},\\
0,\text{ if action }a^{(t,i)} \text{ implies a do-nothing action}.
    \end{cases}
\end{align*}

We want to find   control policy $\pi$  that maximizes the expected total  rewards over the finite time horizon
\begin{equation}
    \E_\pi \left[\sum_{t = 1}^H \sum_{i = 1}^{I_t}
    g(x^{(t,i)}, a^{(t, i)})
    \right].\nonumber
\end{equation}

\subsection{Control policy optimization}\label{sec:control_policy}

Here, a \textit{randomized} control policy refers to a map
\begin{equation*}
\pi:\X \rightarrow [0,1]^{R^2},
\end{equation*}  that outputs a probability distribution over all trip types  given   state $x\in \X$. We use $\pi(a | x)$ to denote a probability of choosing atomic action $a$ at state $x$ if the system operates under policy $\pi$. We assume that $\pi(a | x) =0 $ if action $a$ is infeasible at system state $x$. Then,  at each epoch $t$, the SDM process  selects atomic actions sampled according to  distribution $\pi(\cdot|x^{(t,i)})$ under      randomized control policy $\pi$, for each step $i=1,\dotsc, I_t.$

We define a   value function $V_\pi: \X\rightarrow \R$ of policy $\pi$
\begin{equation}
    V_\pi(x^{(t,i)}): =  \E_\pi   \left[\sum\limits_{k=i}^{I_t} g( x^{(t,k)}, a^{(t,k)}) + \sum\limits_{j=t+1}^H\sum\limits_{k=1}^{I_j} g( x^{(j,k)}, a^{(j,k)}) \right],\nonumber
\end{equation}
 for each $t=1, \dotsc, H$,  $i=1,\dotsc, I_t$, and $x\in \X.$
 For notation convenience we set $V_\pi(x^{(H+1, 1)})=0$ for any policy $\pi$.

Next, we define  advantage function $A_\pi:\X\times \A\rightarrow \R$ of policy $\pi$
\begin{align*}
    A_\pi(x^{(t,i)}, a^{(t,i)}):= \begin{cases}
    g(x^{(t,i)}, a^{(t,i)})+  V_\pi(x^{(t, i+1)}) - V_\pi(x^{(t,i)}), \quad \text{ if }i\neq I_t\\
     g(x^{(t,i)}, a^{(t,i)})+\sum\limits_{y\in \X} \mathcal{P}(x^{(t,i)}, a^{(t,i)}, y) V_\pi(y) - V_\pi(x^{(t, i)}), \text{ if }i= I_t,\nonumber
    \end{cases}
\end{align*}
 for each $t=1,\dotsc,H$; $i=1,\dotsc, I_t$; $x^{(t, i)}\in \X$, and $a^{(t,i)}\in \A = \{(o, d)\}_{o,d=1}^R$. We note that the transitions within the SDM process are deterministic.
 We use $\mathcal{P}$ to denote the probabilities of transitions that come  from random passenger arrival processes.

 We  let  $\{\pi_\theta,~ \theta\in \Theta\}$ be a set of parametrized control policies,
where $\Theta$ is an open subset of  $\R^d$,   $d\geq 1$.
Hereafter, we abuse the notation and use $V_\theta$ and $A_\theta$   to denote the value function and the advantage function  of policy $\pi_\theta$, $\theta\in \Theta$, respectively.

 In Lemma \ref{lem:per_diff} we obtain  the  performance difference identity for an MDP operating under the actions generated by the SDM process.    Performance difference identity was first obtained for MDPs with infinite-horizon discounted cost objectives in \cite{Kakade2002}.

 \begin{lemma}\label{lem:per_diff}
  We consider two policies $\pi_\theta$ and $\pi_\phi$, where $\theta, \phi\in \Theta$.
  Their value functions satisfy
 \begin{align*}
 V_\theta( x^{(1, 1)}) - V_\phi(x^{(1, 1)})  = \E_{\pi_\theta} \left[\sum\limits_{t=1}^{H} \sum\limits_{i=1}^{I_t} A_\phi ( x^{(t,i)}, a^{(t,i)})  \right].
  \end{align*}

\end{lemma}
 The proof of Lemma \ref{lem:per_diff} can be found in Appendix \ref{appendix:ride_hailing_proofs}.

 We  define an advantage function for the original MDP $A^\Sigma_\pi: \X^\Sigma\rightarrow \A^\Sigma$ of policy $\pi$ as
 \begin{align*}
     A^{\Sigma}_\pi(x^{(t)}, a^{(t)}): = \sum\limits_{i=1}^{I_t} A_\pi(x^{(t,i)}, a^{(t,i)}),
 \end{align*}
  where $a^{(t)} = (a^{(t,1)},\dotsc, a^{(t,I_t)})\in \A^\Sigma$.  We note that
\begin{align*}
A_\pi^\Sigma (x^{(t)}, a^{(t)}) = \sum\limits_{i=1}^{I_t}\left [g(x^{(t,i)}, a^{(t,i)})\right] + \sum\limits_{y\in \X} \mathcal{P}(x^{(t,I_t)}, a^{(t,I_t)}, y)V_\pi(y) - V_\pi(x^{(t,1)}).
\end{align*}
 We also let $\pi^\Sigma(a^{(t)}|x^{(t,1)})$  denote  the  probability of selecting action $a^{(t)} = (a^{(t,1)}, \dotsc, a^{(t, I_t)})$ through the SDM process   initialized at state $x^{(t,1)}$ under policy $\pi$.

We define an occupation measure   of policy $\pi_\theta^\Sigma$ at epoch $t$ as a distribution over states  of  $\X^\Sigma$:
\begin{align*}
   \mu_\theta(t, x) := \Prob(x^{(t)} = x),\quad \text{ for each }t=1,\dotsc, H,~x\in \X^\Sigma,
\end{align*}
 where $x^{(t) }$ is a state of the  MDP at epoch $t$   under policy $\pi^\Sigma_\theta$.
We define another occupation measure for the states of the SDM process under policy $\pi_\theta$, $\theta\in \Theta$. We denote  the  probability that starting at state $x^{(t,1)}=x$ at epoch $t$ under policy $\pi_\theta$ the SDM process is at state $y$  after  $i-1$ steps as
\begin{align*}
   \xi_\theta(t,i, x, y )  := \Prob(x^{(t,i)} = y~|~x^{(t,1)} = x),
\end{align*}
  for each  $t=1,\dotsc,H,$ $i=1,\dotsc,I_t,$ $y\in \X,$ $x\in \X^\Sigma$.

Next, we obtain a policy improvement bound on the difference of finite horizon objectives of two control policies.

\begin{theorem}\label{thm:pib_finite}

 We consider two policies $\pi_\theta$ and $\pi_\phi$, where $\theta, \phi\in \Theta$. Then the difference of the objectives of these policies satisfy the following policy improvement bound
\begin{align}\label{eq:two_terms}
      &V_\theta( x^{(1, 1)}) - V_\phi(x^{(1, 1)})  \geq \E_{x^{(t, i)}\sim \pi_\phi} \left[ \sum\limits_{t=1}^{H}\sum\limits_{i=1}^{I_t}\frac{\pi_\theta(a^{(t,i)}|x^{(t,i)})}{\pi_\phi(a^{(t,i)}|x^{(t,i)})} A_\phi(x^{(t,i)}, a^{(t,i)})  \right]  \\
     &\quad\quad- \max\limits_{x\in \X,a\in \A} |A_\phi (x,a)|\sum\limits_{t=1}^{H }\sum\limits_{x\in  \X^\Sigma} \mu_\phi(t, x )  \sum\limits_{i=1}^{I_t}\sum\limits_{y\in \X}|\xi_\phi(t, i, x, y) - \xi_\theta(t, i, x, y)|  \nonumber\\
     &\quad\quad-\max\limits_{x\in\X^\Sigma,~a\in  \A^\Sigma}| A_\phi^\Sigma (x, a)|\sum\limits_{t=1}^{H}\sum\limits_{x\in \X^\Sigma} |\mu_\phi(t, x) - \mu_\theta(t,x) |\nonumber .
 \end{align}
\end{theorem}

 The proof of Theorem \ref{thm:pib_finite} can be found in Appendix \ref{appendix:ride_hailing_proofs}.

We assume that randomized control policy $\pi_\phi$ is the centralized planner's current policy. We want to improve it and get policy $\pi_\theta$  that outperforms the current policy (i.e. $V_\theta(x^{(1, 1)}) - V_\phi( x^{(1,1)})>0$).  We can guarantee the improvement if we find policy $\pi_\theta$ such that the right-hand side (RHS) of   (\ref{eq:two_terms}) is positive. We address the maximization of the RHS of (\ref{eq:two_terms}) following the approach previously used in \cite{Kakade2002, Schulman2015, DaiGluzman2021}: we  bound $\sum\limits_{x\in \X} |\mu_\phi(t, x) - \mu_\theta(t,x) |$ and $\sum\limits_{y\in \X}\left|\xi_\phi(t,i,x ,y) - \xi_\theta(t,i,x ,y)\right|$   by controlling
 the maximum change between policies $\pi_\theta$ and $\pi_\phi$, and focus on maximization of the first term of the RHS of (\ref{eq:two_terms}).

In \cite{Schulman2017} the authors proposed  to maximize a \emph{clipping} surrogate objective function:
\begin{align}\label{eq:PO_ride_hailing}
L(\theta, \phi):= \E_{{\pi_\phi}} \Big[ \sum\limits_{t=1}^{H} \sum\limits_{i=1}^{I_t}\min \Big( &r_{\theta, \phi}(x^{(t,i)}, a^{(t,i)}) A_{\phi} (x^{(t,i)}, a^{(t,i)}) , \\
&\text{clip} (r_{\theta, \phi}(x^{(t,i)}, a^{(t,i)}),  1-\epsilon, 1+\epsilon)  A_{\phi} (x^{(t,i)}, a^{(t,i)})  \Big) \Big],\nonumber
\end{align}
where $r_{\theta, \phi}(x, a) := \frac{\pi_\theta(x, a)}{\pi_\phi( x, a)}$, $\epsilon\in(0, 1)$ is a hyperparameter, and     clipping function is equal to
  \begin{equation*}
  \text{clip}(y,  1-\epsilon, 1+\epsilon):=
  \begin{cases}
   \min(y, 1+\epsilon)   \text{ if }  y\geq 1,\\
   \max(y, 1-\epsilon), \text{ otherwise.}
   \end{cases}
  \end{equation*}

We note that the clipping term $\text{clip} (r_{\theta, \phi}( x, a),  1-\epsilon, 1+\epsilon)  A_{\phi}  ( x, a) $ of the objective function  (\ref{eq:PO_ride_hailing}) prevents large changes to the policy and keeps $r_{\theta, \phi}(  x, a)$ close to 1, therefore promoting a conservative update.

We use Monte Carlo simulation to obtain an estimate of the objective function (\ref{eq:PO_ride_hailing}).
We generate $K$ episodes, $K\geq 1$, each  $H$ epochs long. For now, we assume that the advantage function estimates $\hat A$     required to evaluate  (\ref{eq:PO_ride_hailing}) are available. At each step of the SDM process we record a separate datapoint with the following fields (state, action, and advantage function estimate for the state-action pair) to get  a dataset:
\begin{align}\label{eq:data}
D_\phi^{(K)}:=& \left\{ \left(   \Big(  x^{(t,1,k)}, a^{(t,1,k)}, \hat A ( x^{(t,1,k)}, a^{(t,1,k)}) \Big ),\dotsc,    \Big(  x^{(t, I_{t,k} ,k)}, a^{(t, I_{t,k} ,k)}, \hat A ( x^{(t, I_{t,k},k)}, a^{(t,I_{t,k},k)})  \Big)  \right)_{t=1}^{H}
\right\}_{k=1}^{K},
\end{align}
where $x^{(t, i, k)}$ and $a^{(t, i, k)}$ are the state  and   action at the SDM process step $i$,   epoch $t$,  episode $k$, respectively.

Given dataset (\ref{eq:data}) we estimate the objective function as:
 \begin{align}\label{eq:POest}
\hat L(\theta, \phi, D^{(K)}_\phi):=\frac{1}{K}\sum\limits_{k=1}^K \Big[ \sum\limits_{t=1}^{H} \sum\limits_{i=1}^{I_{t,k}} \min \Big(& r_{\theta, \phi}(x^{(t,i,k)}, a^{(t,i,k)}) \hat A_{\phi} (  x^{(t,i,k)}, a^{(t,i,k)}) ,   \\
&\text{clip} \left(r_{\theta, \phi}(  x^{(t,i,k)}, a^{(t,i,k)}),  1-\epsilon, 1+\epsilon\right)  \hat A_{\phi}  (  x^{(t,i,k)}, a^{(t,i,k)})  \Big)\Big].\nonumber
\end{align}
Next, we discuss   estimating  the advantage function of policy $\pi_\phi$. First, we estimate the value function $V_\phi$. We compute a Monte Carlo estimate of the value function that corresponds to each step in the generated episodes (\ref{eq:data}), such as:
\begin{equation}\label{eq:Vest}
    \hat V_{t, i, k} := \sum\limits_{j=i}^{I_{t,  k}} \left[g(x^{(t,j,k)}, a^{(t,j,k)})\right] +\sum\limits_{\ell = t+1}^{H}  \sum\limits_{j=1}^{I_{\ell,k}} g(x^{(\ell,j, k)}, a^{(\ell,j, k)}),
\end{equation}
which is a one-replication estimate of the value function $V(x^{(t, i, k)})$ at state $x^{(t, i, k)}$ that is visited at epoch $t$, episode $k$, after  $i-1$ steps of the SDM process. We note that the approximating martingale-process (AMP) method  from Section \ref{sec:AMP}  is not incorporated into value function estimation (\ref{eq:Vest}). There are two reasons why we do not apply AMP for the considered ride-hailing system  model. First, the AMP method requires knowledge of the transition probabilities, which are assumed to be unknown to the centralized planner. Second, even if we assume that the transition probabilities are known, the AMP method requires an accurate estimation of the expected value of the value function at the subsequent state  each simulation step. This estimation of the expected value is a computationally intense task due to complexity of the ride-hailing system dynamics. 

We   use function approximator $V_\psi:\X\rightarrow \R$ to get a low-dimensional representation of value function $V_\phi$. We consider a set of function approximators $\{V_\psi, ~\psi\in \Psi\}$, and based on one-replication estimates we find the optimal $V_\psi $ that minimizes the mean-square norm:
 \begin{equation}\label{eq:mse}
     \sum\limits_{k=1}^{K}\sum\limits_{t=1}^{H}\sum\limits_{i=1}^{I_{t,k}} \left\| V_\psi (x^{(t, i, k)}) - \hat V_{t, i, k} \right\|^2.
 \end{equation}
Next, we obtain   the advantage function estimates   as
\begin{align}\label{eq:Aest}
    \hat A ( x^{(t, i,  k)}, a^{(t, i, k)}):=\begin{cases}
    g(x^{(t, i, k)}, a^{(t, i, k)}) +V_\psi( x^{(t, i+1, k)}) - V_\psi(x^{(t, i, k)})\quad \text{ if } i\neq I_{t,  k},\\
    g (x^{(t, i, k)}, a^{(t, i, k)}) +V_\psi (x^{(t+1, 1, k)}) - V_\psi(x^{(t, i, k)})\quad \text{ otherwise},
\end{cases}
\end{align}
for each $t=1, \dotsc,H$; $k=1,\dotsc, K;$ and $i=1,\dotsc, I_{t,k}$.

 Our proposed PPO algorithm consists of the following steps.

\begin{algorithm}[H]
\SetAlgoLined
\LinesNumbered
\SetKwBlock{Begin}{Begin}{}
\SetAlgoLined
\SetKwProg{Loop}{LOOP}{}{}
\KwResult{policy $\pi_{\theta_J}$}
 Initialize  policy function  $\pi_{\theta_0}$ and value function approximator $V_{\psi_{-1}}$\;
 \For{ policy iteration $j = 1, 2, \dotsc, J$}{

Run policy $\pi_{\theta_{j-1}}$ for $K$ episodes and collect dataset (\ref{eq:data}).

   Construct  Monte-Carlo estimates of the value function $V_{\theta_{j-1}}$  following (\ref{eq:Vest}).

Update function approximator $V_\psi$ minimizing (\ref{eq:mse}).

Estimate advantage functions $\hat A (x^{(t, i,  k)}, a^{(t, i, k)})$ by (\ref{eq:Aest}).

 Maximize surrogate objective function (\ref{eq:POest}) w.r.t. $\theta$.
 Update $\theta_{j} \leftarrow  \theta$
      }
 \caption{The PPO algorithm for the ride-hailing system}
 \label{alg:ppo}
\end{algorithm}

 \section{Experimental results}\label{sec:num-study}

In this section we report numerical experiments for two   transportation networks considered in \cite{Braverman2019}.
We evaluate the performance of the proposed PPO algorithm for a transportation network consisting of $R = 9$ regions, $N = 2000$ cars, and $H = 240$ minutes,   designed based on real data from Didi Chuxing. Previously, in \cite{Feng2020},  the numerical experiment with  the proposed algorithm was conducted for a transportation network consisting of $R = 5$ regions, $N = 1000$ cars, and $H = 360$ minutes, designed ``artificially''.  For completeness, we report this experiment as well. The traffic parameters, i.e. passengers arrival rates $\lambda_o(t)$, travel times $\tau_{od}(t)$, and destination probabilities $P_{od}(t)$, of the nine-region and five-region transportation networks can be found in Appendix~EC.3.1 and  Appendix~EC.3.2 of \cite{Braverman2019}, respectively.

In both experiments, following \cite{Braverman2019}, at the start of each working day, the centralized planner distributes the cars in proportion to each region's expected demand.
 We set patience time at $L = 5$. We establish the reward functions (i.e., car-passenger matching rewards are equal to $g_f^{(t)}(o,d,\eta) = 1$, and empty-car routing costs are equal to $g^{(t)}_e(o,d) = 0,$ for each $o, d = 1,\dotsc,R$, $\eta = 0,1,\dotsc,L $, and $t = 1,\dotsc, H$) such that the total reward accumulated at the end of the working day corresponds to the number of completed ride requests.
In this way, the total reward accumulated by the end of a working day correspond to the  number  of ride requests fulfilled. This can be reinterpreted as the fraction of ride requests fulfilled,  given  a sample path of the passenger arrivals. The number of completed ride requests fulfilled is the common objective considered by the dynamic matching problems, see, for example \cite{Ozkan2020}.

We  run the proposed PPO algorithm for the nine-region transportation network for $J =150$ policy iterations.
We use two separate and fully connected feed-forward neural networks (NNs) to represent randomized control policies $\pi_\theta$,  $\theta\in \Theta$ and value functions $V_\psi$, $\psi\in \Psi$, see the details in Appendix \ref{appendix:ride_hailing_nn}.
The algorithm simulates $K = 250$ episodes (working days) in each iteration.  See Appendix \ref{appendix:ride_hailing} for  more details about   hyperparameter values.
Figure \ref{fig:nine_regions} shows that our PPO algorithm   achieves  $85.6\%$  fulfilled ride requests  after $J = 150$ policy iterations. The performance of the randomized control policy was evaluated after every iteration by taking the average of the fractions of fulfilled ride requests on each of $K=250$ episodes.
We use the best result for the nine-region experiment in \cite{Braverman2019} as the benchmark.
The ``time-dependent lookahead'' policy from \cite{Braverman2019} could achieve 83.8\%  fulfilled ride requests. We also test a \textit{closest-driver} policy for the nine-region network to have another benchmark. Each time a new passenger arrives to the system, the closest-driver policy assigns this passenger to the closest available driver, if any.   In our test the closest-driver policy  fulfilled  $65.1\%$  ride requests, that is much worth performance if we compare it with the results of  the PPO and time-dependent lookahead policies.

\begin{figure}[H]
\centering%
\includegraphics[width=.9\linewidth]{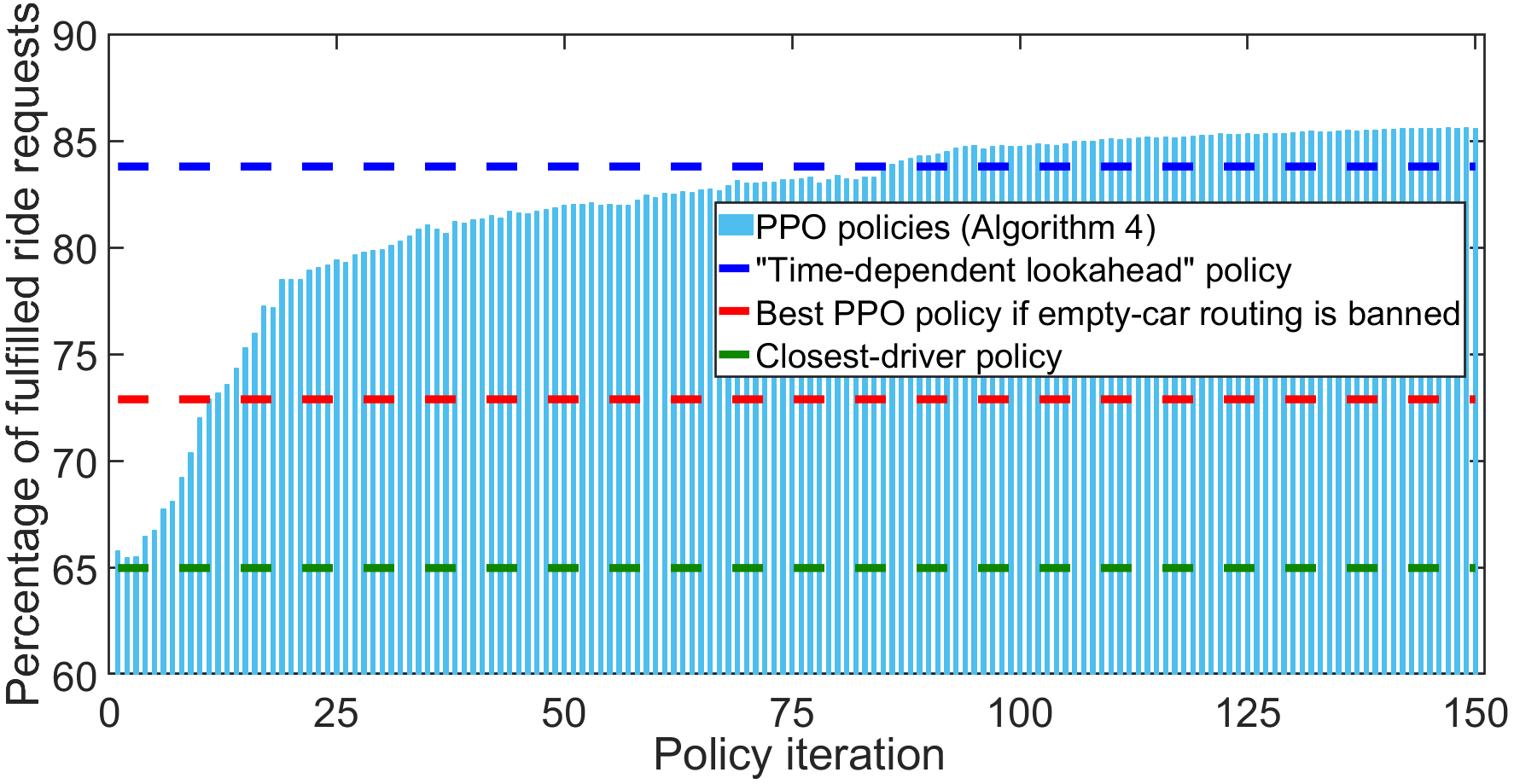}
\caption[Learning curve from Algorithm \ref{alg:ppo} for the nine-region  transportation network from \cite{Braverman2019}]{Learning curve from Algorithm \ref{alg:ppo} for the nine-region  transportation network from \cite{Braverman2019}. The    columns  show the  performance of the randomized control policies obtained at every iteration of Algorithm \ref{alg:ppo}. The dashed blue line shows the
best performance of the ``time-dependent lookahead'' policy from \cite{Braverman2019}.   The dashed red line shows the performance of the  best PPO policy learned assuming the empty-car routing is disabled. The dashed green line shows the performance of the closest-driver policy. }
  \label{fig:nine_regions}
\end{figure}

 We have performed again  our experiment with the PPO algorithm on the nine-region transportation network,
but we have disabled the empty-car routing. The number of iterations and the number of  episodes have remained unchanged. The PPO algorithm has achieved 72.9\% fulfilled ride requests. This result demonstrates the importance of the empty-car routing mechanism for ride-hailing services reliability.

 Similarly,  the PPO algorithm was tested on the five-region network running it for $J =75$ policy iterations (allowing the empty-car routing) in \cite{Feng2020}. Figure \ref{fig:five_regions} shows that the  algorithm   achieves  $87\%$  fulfilled ride requests while the performance of the ``time-dependent lookahead'' policy reported in \cite{Braverman2019} was $84\%.$
In fact,  Algorithm \ref{alg:ppo} needs only $8$ policy iterations to boost the performance to $80\%$ from the initial $59\%$ attained by a policy NN with random weights.

\begin{figure}[H]
\centering%
\includegraphics[width=.9\linewidth]{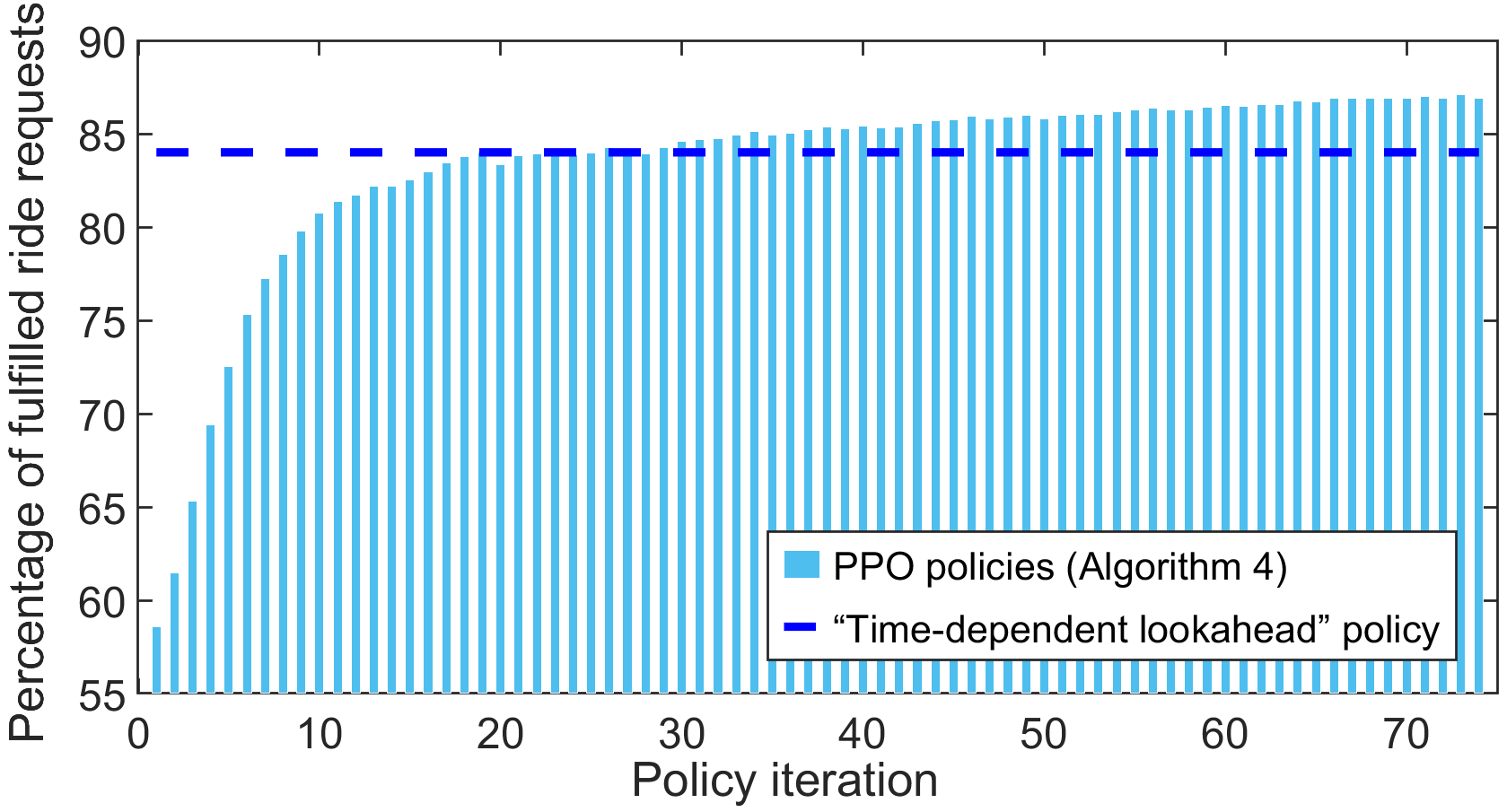}
\caption[Learning curve from Algorithm \ref{alg:ppo} for the five-region  transportation network reported in \cite{Feng2020}]{Learning curves from Algorithm \ref{alg:ppo} for the five-region  transportation network reported in \cite{Feng2020}. The    columns  show the  performance of the randomized control policies obtained at every iteration of Algorithm \ref{alg:ppo}. The dashed blue line shows the
best performance of the ``time-dependent lookahead'' policy from \cite{Braverman2019}.}
  \label{fig:five_regions}
\end{figure}

\section{Conclusion to Chapter \ref{ch:ride_hailing}}\label{sec:conclusion_ride_hailing}
This chapter proposes a method to optimize the  finite horizon total reward objective in a ride-hailing service system.
The large  action space
prohibits the direct use of policy optimization RL methods. Although the standard PPO algorithm \cite{Schulman2017} suggests designing a policy NN such that the the number of units in its output layer is equal to the number of actions, the algorithm becomes computationally expensive because the number of parameters grows exponentially with the number of agents. The large action space also makes challenging for  the policy NN to capture
  similarities among the actions.

Instead of searching for the optimal actions directly, we use the PPO algorithm to learn the most beneficial type of trip
to fulfill at a given   state. Repeated executions of the trip-generating policy allow a centralized planner to prioritize
trip types and sequentially assign tasks to all available cars.

  Numerical experiments  demonstrate that the PPO algorithm applied to the considered MDP formulation outperforms  the
policy proposed in \cite{Braverman2019} by 2-3\%. The importance of the empty-car routing mechanism was evidenced through the numerical tests.

Although, incorporation of the atomic action might resolve the scalability issue caused by a large number of cars, this modification does not address potential scalability issues with respect to the number of regions.  Moreover, in this chapter we omit any discuss on how practitioners should  divide ride-hailing operational territory into the regions. Such division is expected to be nontrivial  for most  real-world cases.

\chapter{Policy Improvement Bounds for Markov Decision Processes}\label{ch:2}
Policy improvement bounds on the difference of the discounted and average returns play a crucial role in the theoretical justification of the trust-region policy optimization (TRPO), proximal policy optimization (PPO), and related algorithms. Theorem \ref{thm:main} in Section \ref{sec:TRPOforAC} is an example of a policy improvement bound.

In this chapter  we focus our investigation on policy improvement bounds due to discrepancy between the implication of the existing bounds in the literature and common practical use of reinforcement learning algorithms. The policy improvement  bound   in  \cite{Schulman2015, Achiam2017}  leads to a degenerate bound when the discount factor approaches one, making the applicability of TRPO and related algorithms questionable when the discount factor is close to one. We refine the results in \cite{Schulman2015, Achiam2017} and propose a novel bound that is  continuous  in the discount factor. Furthermore, we generalize the obtained novel bounds on Markov decision process (MDP) problems with countable state spaces and on semi-Markov decision process (SMDP) problems.

 In \cite{Kakade2002} the authors developed a conservative
   policy iteration algorithm for MDPs
   that can avoid catastrophic large policy updates; each iteration
   generates a new policy as a mixture of the old policy and a greedy
   policy. They proved that the updated policy is guaranteed to
   improve when the greedy policy is properly chosen and the updated
   policy is sufficiently close to the old one.  In
   \cite{Schulman2015} the authors generalized the proof of
   \cite{Kakade2002} to a policy improvement  bound for two
   \emph{arbitrary} randomized policies.  This policy improvement
   bound allows one to find an updated policy that guarantees to
   improve by solving an unconstrained optimization problem.
   \cite{Schulman2015} also proposed a practical algorithm, called
   trust region policy optimization (TRPO), that approximates the
   theoretically-justified update scheme by solving a constrained
   optimization problem in each iteration.  In recent years, several
 modifications of TRPO have been proposed \cite{Schulman2016,
   Schulman2017, Achiam2017, Abdolmaleki2018}. These studies continued
 to exploit the policy improvement bound to theoretically motivate
 their algorithms.

  The policy improvement bounds in \cite{Schulman2015,
     Achiam2017} are   lower bounds on the difference of the expected
   \emph{discounted reward returns} under two policies.  Unfortunately, the
 use of these policy improvement bounds becomes questionable and
 inconclusive when the discount factor is close to one. These policy
 improvement bounds degenerate as discount factor converges to one.
 That is, the lower bounds on the difference of discounted reward returns
 converge to negative infinity as the discount factor goes to one,
 although the difference of discounted returns converges to the
 difference of (finite) average rewards.  Nevertheless, numerical experiments demonstrate that the TRPO
 algorithm and its variations perform best when the discount factor
 $\gamma$ is close to one, a region that the existing
 bounds do not justify;  e.g. \cite{Schulman2015, Schulman2016,
   Schulman2017} used $\gamma=0.99$, and \cite{Schulman2016,
   Achiam2017} used $\gamma=0.995$ in their experiments.

 Recent studies \cite{Zhang2021} and \cite{DaiGluzman2021} (see  Section \ref{sec:TRPOforAC})
 proposed policy improvement bounds for   average returns, justifying  that
 a family of TRPO algorithms can be used for continuing problems with
 long-run average reward or cost objectives. It remains unclear how the
 large values of the discount factor can be justified and why the
 policy improvement bounds in \cite{Schulman2015, Achiam2017} for the
 discounted returns do not converge to one of the bounds provided in
 \cite{DaiGluzman2021, Zhang2021}.

In this chapter we conduct a comprehensive study on policy improvement bounds in different settings, i.e. for average and discounted cost objectives, finite and countable state spaces, MDP and SMDP models. Following the narrative of Chapter \ref{ch:1} we focus on cost minimization MDP and SMDP problems.  As a result, unlike some other papers \cite{Schulman2015,
     Achiam2017, Zhang2021} that proposed lower bounds on  the difference of the expected reward returns,  we  derive \textit{upper bounds} on the difference of the expected cost returns.

      We summarize the major contributions of this chapter:

\begin{enumerate}
  \item  In Section \ref{sec:PIB_f} we provide a unified derivation of policy
   improvement bounds for both discounted and average cost MDPs on finite state spaces. Our
   bounds depend on the discount factor \emph{continuously}.  When the
   discount factor converges to $1$, the corresponding
   bound for discounted returns converges to a policy improvement
   bound for average costs.
   Our results justify the use of a large discount
   factor in TRPO algorithm and its variations.

\item Policy improvement bounds proposed for finite state spaces in Section \ref{sec:PIB_f} are not valid for MDPs on infinite state spaces.  In Section \ref{sec:PIB_V}  we  obtain  policy improvement bounds on the difference of performances of two policies of an MDP on the countable state space,
assuming $\V$-uniform ergodicity of the transition matrix of one of the policies. We introduce a  $\V$-weighted   ergodicity coefficient and relate it to the  $\V$-uniform ergodicity assumption.

\item In Section \ref{sec:PIB_SMDP} we provide performance identity and policy improvement bound for an SMDP. These results imply that TRPO, PPO, and related algorithms can be applied for SMDPs.

\end{enumerate}

 Section \ref{sec:PIB_f} is based on \cite{Dai2021}. The results in Sections \ref{sec:PIB_V} and \ref{sec:PIB_SMDP} have not been published elsewhere before.

\section{Policy improvement bounds in finite state spaces}\label{sec:PIB_f}

In this section  we provide a single policy improvement bound  for
both discounted and average cost objectives for an MDP on a finite state space.
This   result is achieved by   two innovative
   observations. First, we embed the discounted future state
   distribution under a fixed policy as the stationary distribution of
   a  modified Markov chain.  Second, we introduce an
   \emph{ergodicity coefficient} from Markov chain perturbation theory
   to bound the one-norm of the difference of discounted future state
   distributions, and prove that this bound is optimal in a certain
   sense.

\subsection{Preliminaries}

We consider an MDP defined by the tuple $(\X, \A, P, g, \mu)$, where $\X$ is a  finite state space; $\A$ is a finite action space; $P(y|x,a)$ is the probability of transitioning into state $y$ upon taking action $a$ in state $x$;
$g:\X\times\A\rightarrow \R$ is the cost function; $\mu$ is the probability distribution of the initial state $x^{(0)}$.



 We let $\pi$ denote a stationary randomized policy $\pi:\X\rightarrow \Delta(\A)$, where $\Delta(\A)$ is the probability simplex over $\A$.  Under policy $\pi$, the corresponding Markov chain has a transition matrix $P_\pi$ given by
$ P_\pi(x,y):=\sum\limits_{a\in\A }\pi(a|x)P(y|x,a),~x,y\in \X.
$
 We assume that MDPs we consider are unichain, meaning that for any stationary policy $\pi$ the corresponding  Markov chain with transition matrix $P_\pi$ contains only one recurrent class \cite{Puterman2005}.
 We use $d_\pi$ to denote a unique stationary distribution of  a Markov chain with transition matrix $P_\pi$.

 We discuss two formulations of an MDP problem: with infinite horizon discounted cost objective and long-run average cost objective.


We let  $\gamma\in [0,1)$ be a discount factor.  We define the value function for a given policy $\pi$ as
  \begin{equation*}
V^{(\gamma)}_\pi(x) :=\E_\pi    \left[ \sum\limits_{t=0}^\infty \gamma^tg(x^{(t)}, a^{(t)}) ~\Big|~  x^{(0)}=x \right],
 \end{equation*}

\noindent
 where $x^{(t)}$, $a^{(t)}$ are random variables for the state and action at time $t$ upon executing the policy $\pi$ from the initial state $x$. For policy $\pi$ we define the state-action value function as
 \begin{equation*}
Q^{(\gamma)}_\pi(x, a) :=g(x, a) + \gamma \underset{y\sim P_\pi(\cdot|x,a)}{\E} \left[V_\pi^{(\gamma)}(y) \right],
\end{equation*}
 and  the advantage function as
 \begin{equation*}
A^{(\gamma)}_\pi(x, a):  =Q^{(\gamma)}_\pi(x, a) -  V^{(\gamma)}_\pi(x).
\end{equation*}



We define  the discounted future state distribution of policy $\pi$ as
\begin{align*}
d^{(\gamma)}_\pi(x):=(1-\gamma)\sum\limits_{t=0}^\infty \gamma^t\Prob\Big[x^{(t)}=x~|~x^{(0)}\sim \mu; x^{(1)}, x^{(2)},...\sim \pi\Big].
\end{align*}

 We measure the performance of policy $\pi$ by its expected discounted   return from the initial state distribution $\mu$:
  \begin{align*}
 \eta^{(\gamma)}_\pi (\mu):=(1-\gamma)\E_{x\sim \mu} \left[V^{(\gamma)}_\pi(x)\right]= \underset{\substack{x\sim d^{(\gamma)}_\pi\\  a\sim \pi(\cdot|x)}  }{\E}[ g(x, a) ].
 \end{align*}

 In the following lemma we give an alternative definition of  the discounted future state distribution as a stationary distribution of a modified transition matrix. See Appendix \ref{appendix:bounds} for the proof of the following Lemma \ref{lem:stat}.

  \begin{lemma}\label{lem:stat}
 For a stationary policy $\pi$, we define a discounted transition matrix for policy $\pi$ as
 \begin{align}\label{eq:mod_trans}
 P_\pi^{(\gamma)}:=\gamma P_\pi + (1-\gamma) e\mu^T,
 \end{align}
 where $e := (1,1,\dotsc,1)^T$ is a vector of ones, $e\mu^T$ is the matrix which rows are equal to $\mu^T$.

 Then    the discounted future state distribution of policy $\pi$, $d^{(\gamma)}_\pi$,  is the stationary distribution of transition matrix $P_\pi^{(\gamma)}$.

 \end{lemma}





 The long-run average cost of policy $\pi$ is defined as
\begin{align*}
 \eta_\pi :&= \lim\limits_{N\rightarrow \infty}\frac{1}{N} \E_\pi \left[ \sum\limits_{t=0}^{N-1} g (x^{(t)},a^{(t)} ) ~|~ x^{(0)}\sim \mu\right] \nonumber\\
&=   \underset{\substack{x\sim d_\pi\\  a\sim \pi(\cdot|x)}  }{\E}   \left[ g(x,a) \right].
 \end{align*}
The long-run average cost  $\eta_\pi$ does not depend on the initial state initial state distribution $\mu$, since
the Markov chain with transition matrix $P_\pi$ is assumed to be a unichain, see  \cite[Section 8.2]{Puterman2005}.

For an MDP with a long-run average cost objective we define  the relative value function
 \begin{align*}
 V_\pi(x) :=\lim\limits_{N\rightarrow \infty}\E_\pi   \left[ \sum\limits_{t=0}^{N-1} \left(g(x^{(t)}, a^{(t)}) - \eta_\pi\right)~|~x^{(0)}=x \right],
 \end{align*}
the relative state-action value function $
 Q_\pi(x, a) :=g(x, a) - \eta_\pi + \E_{y\sim P_\pi(\cdot|x,a)}\left[V_\pi(y)\right],
$ and  the  relative advantage function $
 A_\pi(x, a): =    Q_\pi(x, a) -V_\pi(x)$. The following relations hold for value, state-action value, and advantage functions.

\begin{lemma}\label{lem:adv}

We let $\pi$ be a stationary policy, $\gamma$ be the discount factor, and $\mu$ be the initial state distribution. Then the following limits hold for each $x\in \X$, $a\in \A$:
 \begin{align}
\eta_\pi &= \lim\limits_{\gamma\rightarrow 1} \eta_\pi^{(\gamma)}  \label{eq:conv1}\\
 V_\pi(x) &=\lim\limits_{\gamma \rightarrow 1} \left( V_\pi^{(\gamma)}(x)  - (1-\gamma)^{-1} \eta_\pi\right) ,\label{eq:conv2}\\
 Q_\pi(x,a)&=\lim\limits_{\gamma \rightarrow 1} \left( Q_\pi^{(\gamma)}(x, a)  - (1-\gamma)^{-1}\eta_\pi\right), \nonumber\\
A_\pi(x,a)&=\lim\limits_{\gamma \rightarrow 1}A_\pi^{(\gamma)}(x,a).\nonumber
\end{align}
\end{lemma}
The proofs of identities (\ref{eq:conv1}), (\ref{eq:conv2}) can be found in \cite[Section 8]{Puterman2005}. The rest results  of Lemma \ref{lem:adv} follow directly.  

\subsection{Novel policy improvement bounds}


The policy improvement bound in \cite{Schulman2015, Achiam2017} for  the discounted  returns serves to theoretically justify the TRPO algorithm and its variations.
The following lemma is a   reproduction of Corollary 1 in \cite{Achiam2017}. We state its upper bound version because it  is  more appropriate for MDPs with cost minimization objectives.
\begin{lemma}[Corollary 1 in \cite{Achiam2017}]\label{lem:perf_bound}
For any  two policies $\pi_1$ and $\pi_2$ the following bound holds:
\begin{equation}\label{eq:perf_bound}
\eta^{(\gamma)}_{\pi_2}(\mu) - \eta^{(\gamma)}_{\pi_1}(\mu)\leq \underset{    \substack{ x\sim d^{(\gamma)}_{  \pi_1}, a \sim \pi_2(\cdot|x) }     }{\E} \left[A^{(\gamma)}_{\pi_1}(x,a)   \right]      +\frac{ 2\gamma \epsilon^{(\gamma)}_{\pi_2}}{1-\gamma}   \underset{  x\sim d^{(\gamma)}_{ \pi_1} }{\E} \left[\text{TV}\Big(\pi_2(\cdot|x)~||~\pi_1(\cdot|x)\Big) \right],
\end{equation}

\noindent
where ${\text{TV}}\Big(\pi_2(\cdot|x)||\pi_1(\cdot|x)\Big):=\frac{1}{2}\sum\limits_{a\in \A} \abs{\pi_2(a|x) - \pi_1(a|x)}$, and $\epsilon^{(\gamma)}_{\pi_2}:=\max\limits_{x\in \X}\Big|\underset{a\sim \pi_2(\cdot|x)}{\E}[A^{(\gamma)}_{\pi_1}(x,a)]\Big|$.

\end{lemma}

We provide a summary of the proof of \cite[Corollary 1]{Achiam2017} in Appendix \ref{appendix:bounds}.

The left-hand side of (\ref{eq:perf_bound}) converges to the difference of long-run average costs as $\gamma\rightarrow 1$. Unfortunately, the right-hand side of (\ref{eq:perf_bound})  converges to the positive infinity   because of $(1-\gamma)^{-1}$ factor in the second term.
Our goal is to get a new policy improvement bound for discounted returns that does not degenerate.


The group inverse $D$ of a matrix $A$ is the unique matrix such that
$
ADA = A,~ DAD=D,  \text{ and } DA=AD.
$ From \cite{Meyer1975}, we know that if stochastic matrix $P$ is aperiodic and  irreducible  then the group inverse matrix of $I-P$ is well-defined and equals to
$
 D = \sum\limits_{t=0}^\infty (P^t - ed^T),
$ where $d$ is the stationary distribution of $P$. \cite{Meyer1975} also established a connection between the fundamental matrix of transition matrix $P$ and the group inverse $D$ of a matrix $P$:
\begin{align}\label{eq:DZ}
 D = Z - ed^T,
\end{align}
where $Z :=\sum\limits_{t=0}^\infty\left(P - ed^T\right)^t$ is called the fundamental matrix of $P$, see \cite{Kemeny1976a}.

We let $D^{(\gamma)}_\pi$ be the group inverse of matrix $I-P^{(\gamma)}_\pi$, where $P^{(\gamma)}_\pi$ is defined by (\ref{eq:mod_trans}).
Following \cite{Seneta1991}, we define a one-norm ergodicity coefficient for a matrix $A$ as
\begin{align}\label{eq:erg_def}
\tau_1[A]:= \underset{ \substack{ \|x\|_1=1\\ x^Te = 0}}{\max}\|A^Tx\|_1.
\end{align}
The one-norm ergodicity coefficient has important property that
\begin{equation}\label{eq:prop}
\tau_1[A] = \tau_1[A+ec^T], ~\text{for any vector } c.
\end{equation}
By Lemma \ref{lem:matr_diff} below, $\tau_1\left[D^{(\gamma)}_\pi\right] = \tau_1\left[(I-\gamma P_\pi)^{-1}\right]$, for $\gamma<1$.

\begin{lemma}\label{lem:matr_diff}
We let $\pi$ be an arbitrary policy. Then
\begin{align*}
D^{(\gamma)}_{ \pi} = (I-\gamma P_\pi)^{-1}+e (d^{(\gamma)}_\pi)^T( I - \left(I - \gamma P_\pi \right)^{-1}) - ed_\pi^T.
\end{align*}
\end{lemma}
The proof of Lemma \ref{lem:matr_diff} can be found in Appendix \ref{appendix:bounds}.

We are ready to state the main result of Section \ref{sec:PIB_f}.

  \begin{theorem}\label{thm:main_bounds}
The following bound on the difference of discounted returns of two policies ${\pi_1}$ and $\pi_2$  holds:
\begin{align}\label{eq:bound_opt}
\eta^{(\gamma)}_{\pi_2}(\mu) -   \eta^{(\gamma)}_{\pi_1}(\mu)\leq \underset{    \substack{ x\sim d^{(\gamma)}_{  \pi_1}\\ a \sim \pi_2(\cdot|x) }     }{\E} \left[A^{(\gamma)}_{\pi_1}(x,a)   \right]   + 2\gamma \epsilon^{(\gamma)}_{\pi_2} \tau_1\left[D^{(\gamma)}_{\pi_2}\right] \underset{  x\sim d^{(\gamma)}_{ \pi_1} }{\E} \left[\text{TV}\Big(\pi_2(\cdot|x)~||~\pi_1(\cdot|x)\Big) \right].
\end{align}
\end{theorem}

  \begin{proof}[\textbf{Proof of Theorem \ref{thm:main_bounds}}]
We   closely follow the first steps in the proof of Lemma 2 in \cite{Achiam2017} and start with
\begin{align*}
\eta^{(\gamma)}_{\pi_2}(\mu) -   \eta^{(\gamma)}_{\pi_1}(\mu)\leq \underset{    \substack{ x\sim d^{(\gamma)}_{ \pi_1}\\a \sim \pi_2(\cdot|x) }     }{\E} \left[A^{(\gamma)}_{\pi_1}(x,a)   \right]     +   \max\limits_{x\in \X}\Big|\underset{a\sim \pi_2(\cdot|x)}{\E}[A^{(\gamma)}_{\pi_1}(x,a)]\Big|~ \left\|d^{(\gamma)}_{\pi_1} - d^{(\gamma)}_{\pi_2}\right\|_1.
\end{align*}
Next,  unlike \cite{Achiam2017}, we obtain an upper bound on $ \|d^{(\gamma)}_{\pi_2}-d^{(\gamma)}_{ \pi_1}\|_1$ that does not degenerate as $\gamma\rightarrow 1$. We use the following  perturbation identity:
\begin{align}\label{eq:old_new}
(d^{(\gamma)}_{\pi_2})^T-(d^{(\gamma)}_{ \pi_1})^T &= (D^{(\gamma)}_{\pi_1})^T (P_{\pi_1}^{(\gamma)} - P_{\pi_2}^{(\gamma)})d^{(\gamma)}_{\pi_2}\\
&=\gamma (D^{(\gamma)}_{\pi_1})^T (P_{  \pi_1} - P_{ \pi_2})d^{(\gamma)}_{\pi_2}.\nonumber
\end{align}
Identity (\ref{eq:old_new}) follows from the perturbation identity for stationary distributions, see equation (4.1) in \cite{Meyer1980}, and the fact that $d^{(\gamma)}_{\pi_2}$ and $d^{(\gamma)}_{ \pi_1}$ are the stationary distributions of the discounted transition matrices $P_{\pi_2}^{(\gamma)}$ and $P_{\pi_1}^{(\gamma)}$, respectively.
We make use of the ergodicity coefficient (\ref{eq:erg_def}) to get a new perturbation bound:
\begin{align}\label{eq:inv_tau}
\|d^{(\gamma)}_{\pi_2}-d^{(\gamma)}_{ \pi_1} \|_1 &=\gamma \left\| \left(D^{(\gamma)}_{\pi_2}\right)^T(P_{  \pi_1} - P_{ \pi_2})^Td^{(\gamma)}_{\pi_1} \right\|_1 \nonumber \\
&\leq \gamma \tau_1 \left[D^{(\gamma)}_{\pi_2}  \right]\left\|(P_{  \pi_1} - P_{ \pi_2})^Td^{(\gamma)}_{\pi_1} \right\|_1\\
& \leq 2\gamma\tau_1\left[D^{(\gamma)}_{\pi_2}\right] \underset{  x\sim d^{(\gamma)}_{ \pi_1} }{\E} \left[\text{TV}\Big(\pi_2(\cdot|x)~||~\pi_1(\cdot|x)\Big) \right],\nonumber
\end{align}
where first equality follows from (\ref{eq:old_new}), first inequality follows from ergodicity coefficient (\ref{eq:erg_def}) definition and the fact that $\left( (P_{  \pi_1} - P_{\pi_2})^Td^{(\gamma)}_{\pi_1}\right)^Te = 0$, the second inequality follows from a definition of the total variation distance, see the proof of \cite[Lemma 3]{Achiam2017}.
 \end{proof}

The novel policy improvement bound (\ref{eq:bound_opt}) converges to a meaningful bound on the difference of average costs as $\gamma$ goes to 1. Corollary \ref{col:main_av} follows from Theorem \ref{thm:main_bounds}, Lemma \ref{lem:adv} and the fact that $\tau_1\left[D^{(\gamma)}_{\pi_2}\right] \rightarrow \tau_1\left[D_{\pi_2}\right]$ as $\gamma\rightarrow 1$.

  \begin{corollary}\label{col:main_av}
The following bound on the difference of long-run average costs of two policies $\pi_1$ and $\pi_2$  holds:
\begin{align}\label{eq:bound_av}
 \eta_{\pi_2} -  \eta_{\pi_1}  \leq  \underset{    \substack{ x\sim d_{  \pi_1}\\ a \sim \pi_2(\cdot|x) }     }{\E} \left[A_{\pi_1}(x,a)   \right]  + 2 \epsilon_{\pi_2} \tau_1\left[D_{\pi_2}\right] \underset{  x\sim d_{ \pi_1} }{\E} \left[\text{TV}\Big(\pi_2(\cdot|x)~||~\pi_1(\cdot|x)\Big) \right],
\end{align}
where  $D_{\pi_2}$ is the group inverse of matrix $I -  P_{\pi_2}$, $\epsilon_{\pi_2}:=\max\limits_{x\in \X}\Big|\underset{a\sim \pi_2(\cdot|x)}{\E}[A_{\pi_1}(x,a)]\Big|$.
\end{corollary}

\subsection{Interpretation of $\tau_1\left[ D^{(\gamma)}_{\pi}\right]$ }\label{sec:interp_tau_1_gamma}

 We provide several  bounds on $\tau_1[D^{(\gamma)}_{\pi}] $ to reveal its dependency on the discount factor $\gamma$ and policy $\pi$.
First, we compute $\tau_1[D^{(\gamma)}_{\pi}] $ using  geometric convergence rates of transition matrices to their stationary distributions.
For any policy $\pi$, there exists constants $k_\pi\in (0,1)$ and $C_\pi>0$ such that $
\|P_\pi^t(x, \cdot) - d_\pi(\cdot)\|_1\leq C_\pi k_\pi^t,
$  for each $x\in \X$, where $d_\pi$ is the stationary distribution of transition matrix $P_\pi$.
By  Appendix Section 5 in \cite{Cinlar2013}, $k_\pi$ can be taken to be $|\lambda_2|$, the largest number among the absolute values of the
eigenvalues of $P_\pi$ excluding the eigenvalue 1.

Matrix $P^{(\gamma)}_\pi$ defined by (\ref{eq:mod_trans}) is called the Google matrix, see \cite{Langville2003}. If  the spectrum of  transition matrix $P_\pi$ is $\{1, \lambda_2, \lambda_3, ..., \lambda_n\}$, then   the spectrum of    matrix $P^{(\gamma)}_\pi$ is $\{1, \gamma\lambda_2, \gamma\lambda_3, ..., \gamma\lambda_n\}$.
Google matrix $P^{(\gamma)}_\pi$ exhibits  a faster convergence rate than $P_\pi$: there exists constant $C^{(\gamma)}_\pi>0$,  such that   $
\|(P^{(\gamma)}_\pi)^t(x, \cdot) - d^{(\gamma)}_{\pi}(\cdot)\|_1\leq C_\pi^{(\gamma)} |\gamma\lambda_2|^t,\text{ for each }x\in \X.
$ Inequalities $ \tau_1\left[ D^{(\gamma)}_{\pi} \right]\leq \left\|D^{(\gamma)}_{\pi}\right\|_\infty \leq \sum\limits_{t=0}^\infty \left\|(P^{(\gamma)}_\pi)^t - e(d^{(\gamma)}_{\pi})^T\right\|_\infty $ lead to the following bound.
  \begin{lemma}\label{lem:lambda}
  We let $D^{(\gamma)}_{\pi}  $  be the group inverse matrix of $I-P^{(\gamma)}_{\pi}$. Then for any discount factor $\gamma\in (0,1]$ there exists  $C_\pi^{(\gamma)}>0$ such that
 \begin{align*}
  \tau_1\left[ d^{(\gamma)}_{\pi}\right ]\leq C_\pi^{(\gamma)} \frac{1}{1- \gamma|\lambda_2|},
 \end{align*}
where $\lambda_2$ is an eigenvalue of $P_\pi$ with the second largest absolute value.
\end{lemma}

It may be  difficult to express how constant $C_\pi^{(\gamma)}$  depends  on the discount factor for a general transition matrix $P_\pi.$ In Lemma~\ref{lem:group} below we derive another upper bound on $\tau_1\left[D^{(\gamma)}_{ \pi}\right] $ that does not include additional constants dependent on $\gamma$. The proof of Lemma \ref{lem:group} is in Appendix \ref{appendix:bounds}.

  For a given policy $\pi$, we assume the transition matrix $P_\pi$ is aperiodic and irreducible. By Proposition 1.7 in \cite{Levin2017}, there exists an integer $\ell$ such that $P_\pi^q(x,y)>0$ for all $x,y\in \X$, and $q\geq \ell.$  Then, there exists a sufficiently small constant $\delta_\pi^{(\mu)}>0$, such that
   \begin{align}\label{eq:minor}
 P_\pi^\ell(x,y)\geq \delta^{(\mu)}_\pi \mu(y), \quad \text{ for each }x,y\in \X,
  \end{align}
where $\mu$ denotes the distribution of the initial state.

  \begin{lemma}\label{lem:group}
  We let $D^{(\gamma)}_{\pi}  $  be the group inverse matrix of $I-P^{(\gamma)}_{\pi}$.

 We let $\delta^{(\mu)}_\pi$ be a constant  that satisfies (\ref{eq:minor}) for $P_\pi$ and some integer $\ell$. Then
  \begin{align*}
  \tau_1[ D^{(\gamma)}_{\pi} ]\leq  \frac{2\ell}{1 - \gamma +  \gamma^\ell \delta_\pi^{(\mu)}},
  \end{align*}
  where $\delta^{(\mu)}_\pi$ and $\ell$ are  independent of $\gamma$.
\end{lemma}

\subsection{Condition numbers in policy improvement bounds}

 Lemma \ref{lem:opt_bound}  demonstrates that inequality (\ref{eq:inv_tau}) is   the best (smallest) norm-wise bound on the difference of stationary distributions given "averaged" transition matrix perturbation.
Lemma \ref{lem:opt_bound} directly follows from the results in  \cite{Kirkland2008}.

\begin{lemma}\label{lem:opt_bound}
We consider an irreducible and aperiodic transition matrix   $ P$ with the stationary distribution   $  d$. We say that $\tau[  P]$ is a condition number of matrix $  P$ if inequality
\begin{align}\label{eq:cond}
\|d - \tilde d\|_1 \leq \tau[  P]~\| ( \tilde P  -  P)^T \tilde d \|_1,
\end{align}
holds for any irreducible and aperiodic transition matrix $\tilde P$ with the stationary distribution $\tilde d$.
We let $  D$ be a group inverse matrix of $I-  P$.

 Then  $\tau_1[  D]$ is the smallest condition number, i.e.  inequality $ \tau_1[ D]\leq \tau[  P]$
holds for any condition number $\tau[ P]$ satisfying (\ref{eq:cond}).

\end{lemma}

Lemma \ref{lem:opt_bound} shows that inequality (\ref{eq:inv_tau}) in the proof of Theorem~\ref{thm:main}  is a key to the improvement of the policy improvement bounds in \cite{Schulman2015, Achiam2017}. Moreover, it follows from Lemma \ref{lem:opt_bound}  that Corollary \ref{col:main_av} provides a better policy improvement bound for the average cost criterion than \cite{DaiGluzman2021, Zhang2021}. We note that in Section \ref{sec:TRPOforAC} we essentially used $\|Z\|_{\V}$, where $Z$ is the fundamental matrix, as a condition number. In finite state spaces this  condition number corresponds to $\|Z\|_\infty$, i.e. $\V\equiv1$. In contrast, \cite{Zhang2021} used Kemeny's constant  (\ref{eq:kemeny_def}) as a condition number. Lemma \ref{lem:cond_nums} compares these condition numbers with $\tau_1[D]$. The proof of Lemma \ref{lem:cond_nums} can be found in Appendix \ref{appendix:bounds}.

\begin{lemma}\label{lem:cond_nums}

We consider an  irreducible and aperiodic transition matrix $P$. We let $M:\X\times\X\rightarrow \R_+$ be the mean first hitting time matrix, where $M(x, y)$ is the expected number of steps is taken to reach state $y$ from state $x$ for the Markov chain with transition matrix $P$, for each $x, y\in \X$. We note $M(x, x) = 0$, for each $x\in \X$. We let $d$ be the stationary distribution of $P$.

We define   Kemeny's constant of transition matrix $P$ as
\begin{equation}\label{eq:kemeny_def}
\kappa:= \sum\limits_{y\in \X} d(y)M(x, y),
\end{equation}
where $\kappa$ is a constant independent of $x\in \X$, see \cite{Kemeny1976a}.
Matrices $D$ and $Z$ are the group inverse of $I-P$ and the fundamental matrix of $P$, respectively.

Then
\begin{enumerate}
\item[(a)] $ \tau_1[D]  = \tau_1[Z] = \tau_1[ M  I_d]$;
\item[(b)] $\tau_1[D] \leq \|Z\|_\infty$;
\item[(c)] $\tau_1[D] = \kappa  - \min\limits_{x, y\in \X} \sum\limits_{z\in \X} d(z) \min[M(x, z), M(y, z)]\leq\kappa$,
\end{enumerate}
where $I_d$ is the diagonal matrix with diagonal elements $d(x)$, $x\in \X$.
\end{lemma}
There is no general superior relation between $\|Z\|_\infty$ and $\kappa$. In other words,  either $\|Z\|_\infty$ or $\kappa$  may provide a superior bound for different examples of transition matrices, see \cite{Hunter2006}.

\section{Policy improvement bounds for  countable state spaces}\label{sec:PIB_V}

In this section we derive proper policy improvement bounds for MDPs on countable state spaces. As we observed in Section \ref{sec:TRPOforAC}, a policy improvement bound is a necessary building block if we want to design a deep RL algorithm for MDP problems with countable state spaces. In this section we provide more detailed exploration of this topic. We discuss degeneracy of the bounds derived in Section \ref{sec:PIB_f} for countable  state spaces,  propose novel policy improvement bounds via a $\V$-weighted one-norm ergodicity coefficient, and provide bounds on this ergodicity coefficient  via the Lyapunov function $\V$.

\subsection{Preliminaries}

Policy improvement bounds obtained in Section \ref{sec:PIB_f} are not valid in countable state spaces for most MDPs with long-run average cost objectives. These  bounds   have been refined by incorporating sharp perturbation bounds on the difference between stationary distributions.  These perturbation bounds depend on condition numbers that  typically become infinite for DTMC on countable state spaces. Specifically, Kemeny's constant $\kappa$ is infinite for any   DTMC on countable state spaces \cite{Angel2019, Liu2020}, whereas $\|Z\|_\infty$ and $\tau_1[D]$ are guaranteed to be finite only for uniformly ergodic Markov chains \cite[Section 13]{Meyn2009}, \cite{Liu2012}.

 In this section we consider an  MDP problem with a countable state space $\X$, finite action space $\A$, one-step cost function $g(x, a)$,  transition function $P(\cdot|x, a)$, and long-run average cost objective:
\begin{align}\label{eq:eta}
\eta_\pi:= \lim\limits_{N\rightarrow \infty} \frac 1 N \E_\pi\left[ \sum\limits_{t=0}^{N-1} g(x^{(t)}, a^{(t)})~|~x^{(0)}\sim \mu \right],
\end{align}
where $\mu$ is an initial state distribution. We note that one-step cost function $g$ might be unbounded. Existence of a Lyapunov function $\V$ such that $|g|\leq \V$ is a sufficient condition  for the long-run average cost (\ref{eq:eta}) to be finite, see  Lemma \ref{lem:drift} in Section \ref{sec:MC}.

We recall that policy $\pi$ and its  corresponding transition matrix $P_\pi$ satisfy the drift condition if there exists a Lyapunov function $\V:\X\rightarrow [1, \infty)$,  constants  $\varepsilon\in (0,1)$ and $b\geq0$, and a finite subset $C\subset \X$ such that
\begin{align}\label{eq:drift_rep}
\sum\limits_{y\in \X}P_\pi(y|x)\V(y)\leq  \varepsilon \V(x) + b \I_{C}(x), \quad \text{for each }x\in \X,
\end{align}
where $\I_{C}(x)=1$ if $x\in C$ and $\I_{C}(x)=0$ otherwise.

\subsection{Novel policy improvement bound}

 Following \cite[Chapter 2]{Kartashov1996}, we define a $\V$-weighted one-norm ergodicity coefficient for a $\X\times \X$ matrix $A$ and function $\V:\X\rightarrow [1, \infty)$   as
\begin{align}\label{eq:erg_def_V}
\tau_{1,\V}^{ }[A]:&= \underset{ \substack{ \|x\|_{1, \V}=1\\ x^Te = 0}}{\sup}\|A^Tx\|_{1,\V}^{ } \nonumber \\
 &=\sup\limits_{x, y\in \X} \frac{1}{\V(x)+\V(y)}\sum\limits_{z\in \X} |A(x, z) - A(y,z)  |\V(z),
\end{align}
where $\|\nu\|_{1, \V} := \sum\limits_{x\in \X} |\nu(x)|\V(x)$  for any $\nu:\X\rightarrow \R$, $A(x, y)$ is the $(x, y)$th element of matrix $A$ for $x, y\in \X$.

Definition  (\ref{eq:erg_def_V}) directly implies that the following property continues to hold for the $\V$-weighted one-norm ergodicity coefficient:
\begin{equation}\label{eq:prop_V}
\tau_{1, \V}[A] = \tau_{1, \V}[A+ec^T], ~\text{for any vector } c,
\end{equation}
for any $\V:\X\rightarrow [1, \infty).$

We state our   policy improvement bound for an MDP on countable state space with long-run average cost objective.

 \begin{theorem}\label{thm:main_V}

We consider two policies $\pi_1$ and $\pi_2$. We assume that transition matrix $P_{\pi_2}$ of policy $\pi_2$ is such that the drift condition (\ref{eq:drift_rep}) holds for a Lyapunov function $\V\geq 1$ and the cost function satisfies $|g|\leq \V$. We also assume that  transition matrix $P_{\pi_1}$ of  policy $\pi_1$ is positive recurrent and has stationary distribution $d_{\pi_1}$.

Then  ergodicity coefficient $  \tau_{1,\V}^{ }\left[D_{\pi_2}\right]$ is finite and the following bound on the difference of average returns of two policies $\pi_1$ and $ \pi_2$  holds:
\begin{align}\label{eq:bound_opt_V}
\eta_{\pi_2} -  \eta_{\pi_1} \leq \underset{    \substack{ x\sim d_{\pi_1}\\ a \sim  \pi_1(\cdot|x) }     }{\E} \left[\frac{\pi_{2}(a|x)}{\pi_1(a|x)}A_{\pi_1}(x,a)   \right]  + 2\epsilon_{\pi_2, \V} \tau_{1,\V}^{ }\left[D_{\pi_2}\right] \underset{\substack{ x\sim d_{ \pi_1}\\  a \sim  \pi_1(\cdot|x) \\ y\sim P(\cdot|x, a) }  }{\E}\left[ \left|\frac{\pi_2(a|x) }{\pi_1(a|x)}-1 \right|   \V (y)\right].
\end{align}
where $\epsilon_{\pi_2, \V} :=\left\|\underset{a\sim  \pi_2(\cdot|x)}{\E}[A_{\pi_1}(x,a)]\right\|_{\infty, \V}.$
\end{theorem}
The proof of Theorem \ref{thm:main_V} is provided in Appendix \ref{appendix:bounds_V}.

\begin{remark}
In practice, when a deep reinforcement learning algorithm is designed, policy $\pi_1$ is interpreted  as  a current policy and policy $\pi_2$ as a next next. In other words, policy $\pi_2$  is interpreted as an unknown policy that is yet to be found, see,   for example,  Section \ref{sec:ppo}. Therefore, we need to make additional assumptions  to ensure that next policy $\pi_2$  satisfies the drift condition with some known Lyapunov function. Some authors, for example \cite{HernandezLerma1997},  make a strong assumption that  all policy of an MDP under consideration satisfy the drift condition  for same Lyapunov function $\V$, i.e. all policies are $\V$-uniform ergodic. In Chapter \ref{ch:1} we developed another approach.  We assumed that current policy $\pi_1$ satisfies the drift condition and the change from policy $\pi_1$ to policy $\pi_2$ is sufficiently small to preserve $\V$-uniform ergodicity, see Theorem \ref{thm:main}. Several papers proposed other   sufficient conditions to  ensure that a transition matrix continues to be $\V$-uniform ergodic after update or perturbation, see \cite{Glynn1996, Roberts1998, Ferre2013, Herve2014, Negrea2021}.
\end{remark}

\subsection{Interpretation of $\tau_{1, \V}\left[ D \right]$ }

 We consider a Markov chain on countable state space such that its corresponding transition matrix $P$ satisfies the drift condition (\ref{eq:drift_rep}). In this section we  relate ergodicity coefficient $\tau_{1, \V}\left[ D\right]$ to the drift condition, where $D$ is the group inverse of $I-P$.   We provide several bounds on $\tau_{1, \V}\left[ D\right]$ in terms of Lyapunov function $\V$, constants $\varepsilon$ and $b$, finite set $C$ used in the drift condition.

We start with a simple case when finite set $C$ consists of a single state $C = \{x^*\}$. Similar bounds were proposed in \cite[Corollary 3.1.]{Liu2012}. We provide the proof of Lemma \ref{lem:D_V} in Appendix \ref{appendix:bounds_V}.

\begin{lemma}\label{lem:D_V}

We assume that transition matrix $P $ satisfies the following  drift condition:
\begin{align}\label{eq:drift_single}
\sum\limits_{y\in \X}P (y|x)\V(y)\leq  \varepsilon \V(x) + b \I_{x = x^*}(x), \quad \text{for each }x\in \X,
\end{align}
where $\V:\X\rightarrow [1, \infty)$, $\varepsilon\in (0, 1)$, $b\geq0$.

Then
\begin{align*}
\tau_{1, \V}[D ]\leq \frac{1}{1-\varepsilon}\left( 1+ \min\limits_{x\in \X}[\V(x)] ~\|d \|_{1, \V} \right),
\end{align*}
where $d$ is the stationary distribution of $P$. Norm $\|d \|_{1, \V} $ can be further bounded as
  \begin{align}\label{eq:d_bound}
 \|d \|_{1, \V}&\leq  \frac{b}{1-\varepsilon} d (x^*).
  \end{align}

\end{lemma}

We provide an example of applying Lemma \ref{lem:D_V}.
 \begin{example}\label{exp:mm1}
We consider the Bernoulli random walk on the integer lattice $\Z_+$ with transition probabilities $P(x, x+1) = \lambda$ for $x\in \Z_+$, $P(x, x-1) = \mu$ for $x\geq 1$, and $P(0,0) = \mu$. We assume that $\lambda+\mu=1$ and $\rho:=\frac{\lambda}{\mu}<1$.

This Markov chain satisfies the following drift condition with $\V(x) = \rho^{-x/2}$, $x\in \Z_+$, and $C=\{0\}$:
\begin{align}\label{eq:Lyapunov_f}
\sum\limits_{y\in \Z_+}P(x, y) \V(y)\leq \frac{2\sqrt{\rho}}{1+\rho} \V(x) + \frac{1-\sqrt{\rho}}{1+\rho} I_{x = 0}(x), \text{ for each }x\in \Z_+.
\end{align}

We provide the proof that the Bernoulli random walk transition probabilities satisfy (\ref{eq:Lyapunov_f}) in Appendix \ref{appendix:bounds_V}.    Lyapunov function $\V(x) = \rho^{-x/2}$ was proposed in \cite{Lund1996}. 

It is known that the stationary distribution of the   Bernoulli random walk is $d(x) = (1-\rho)\rho^x$ for $x\in \Z_+.$ Hence, we can explicitly find
\begin{align*}
\|d\|_{1, \V} &= 1+\sqrt{\rho}.
\end{align*}

From Lemma \ref{lem:D_V}, we get that
\begin{align*}
\tau_{1, \V}[D]\leq \frac{1+\rho}{(1-\sqrt{\rho})^2} \left( 2+ \sqrt{\rho}\right) = O\left(\frac{1}{(1-\rho)^2}\right),
\end{align*}
where $D$ is the group inverse matrix of $I-P$.
We should mention that same order bound on $\tau_{1, \V}[D]$ for the Bernoulli random walk was previously obtained in \cite{Mouhoubi2010}.
\end{example}

Next, we consider a more general case where finite state $C$ might include more than one state. The proof of  Lemma \ref{lem:D_V2} can be found in Appendix \ref{appendix:bounds_V}.

\begin{lemma}\label{lem:D_V2}

We assume that transition matrix $P $ satisfies the drift condition (\ref{eq:drift_rep})  for Lyapunov function $\V$, constants $\varepsilon, b$, and finite set $C$ such that
\begin{align}\label{eq:as1}
b\leq\sum\limits_{y\in \X}P(x^*,y)\V(y),
\end{align}
where $x^*\in \X$ is such that $\sum\limits_{x\in C}P(x^*, x)>0$.

Then
\begin{align*}
\tau_{1, \V}[D ]\leq \frac{1}{1-\varepsilon}+\min\limits_{x\in \X}[\V(x)]\frac{b}{(1-\varepsilon)^2}\sum\limits_{x\in C}d(x).
\end{align*}

\end{lemma}

The discounting can   be considered as an update or perturbation of the initial Markov chain. Lemma \ref{lem:D_V2} allows to bound the ergodicity coefficient of $D^{(\gamma)}$, where $D^{(\gamma)}$ is the group inverse matrix of $I-P^{(\gamma)}$, if the drift condition holds for transition matrix $P$. We recall definition of a discounted transition matrix:
\begin{align*}
P^{(\gamma)}:=\gamma P +(1-\gamma)e\mu^T,
\end{align*}
where $\mu$ is a  state distribution, $\gamma\in(0, 1]$ is a discount  factor.

We start with Lemma \ref{lem:D_disc} that shows that the discounted transition matrix $P^{(\gamma)}$  is $\V$-uniformly ergodic if initial transition matrix $P$  satisfies the drift condition for Lyapunov function $\V$. The proof of Lemma  \ref{lem:D_disc}  can be found in Appendix \ref{appendix:bounds_V}.

\begin{lemma}\label{lem:D_disc}

We assume that transition matrix $P$ satisfies the drift condition (\ref{eq:drift_rep}).

Then the modified transition matrix $P^{(\gamma)} :=\gamma P  +(1-\gamma) e\mu^T$ satisfies the following drift condition
\begin{align}\label{eq:drift_gamma}
\sum\limits_{y\in \X}P^{(\gamma)} (x, y) \V(y)\leq \frac{1}{2}(\varepsilon\gamma +1)\V(x) + \max[\gamma b, (1-\gamma)\mu^T\V]\I_{C\cup \Omega}(x),
\end{align}
where $\Omega = \left\{ x\in \X:\V(x)<\frac{2(1-\gamma)\mu^T\V}{1-\gamma\varepsilon} \right\}$, $\gamma\in (0, 1]$.

\end{lemma}

Corollary \ref{cor:D_V3} directly follows from Lemma \ref{lem:D_V2} and Lemma \ref{lem:D_disc}.

\begin{corollary}\label{cor:D_V3}

We assume that transition matrix $P$ satisfies the drift condition (\ref{eq:drift_rep}). We let $P^{(\gamma)}$ be a discounted transition matrix $P^{(\gamma)} :=\gamma P  +(1-\gamma) e\mu^T$, where $\gamma\in (0, 1]$.

If there exists $x^*\in \X$ such that $\sum\limits_{y\in C\cup\Omega }P^{(\gamma)}(x^*, y)>0$ and
\begin{align*}
\max[\gamma b, (1-\gamma)\mu^T\V]\leq\sum\limits_{y\in \X}P^{(\gamma)}(x^*,y)\V(y),
\end{align*}
where  $\Omega = \left\{ x\in \X:\V(x)<\frac{2(1-\gamma)\mu^T\V}{1-\gamma\varepsilon} \right\}$,
then
\begin{align*}
\tau_{1, \V}[D^{(\gamma)} ]\leq \frac{2}{1-\varepsilon\gamma}\left( 1+ \min\limits_{x\in \X}[\V(x)] ~\|d^{(\gamma)} \|_{1, \V} \right),
\end{align*}
where $D^{(\gamma)}$ is the group inverse matrix of $I-P^{(\gamma)}$.
The norm of the discounted stationary distribution $\|d^{(\gamma)} \|_{1, \V} $ can be bounded as
\begin{align*}
\|d^{(\gamma)} \|_{1, \V} \leq \frac{2}{1-\varepsilon\gamma} \max[\gamma b, (1-\gamma)\mu^T\V]\sum\limits_{x\in C\cup\Omega}d^{(\gamma)}(x).
\end{align*}

\end{corollary}

\section{Policy improvement bound for  semi-Markov decision processes}\label{sec:PIB_SMDP}

In this section we analyze infinite-horizon semi-Markov decision processes (SMDPs) with
average cost criterion. We derive a policy improvement bound for SMDPs on finite state spaces.

\subsection{Preliminaries}\label{sec:prel_SMDP}

We consider an SMDP problem model which formulation closely follows the one in \cite[Section 11]{Puterman2005}.    SMDPs  generalize MDPs  by modeling the sytem evolution in continuous time and allowing the time between state transition to follow an arbitrary probability distribution. Nevertheless, the SMDP    should be distinguished from the \textit{natural process}.  The
natural process models the state evolution of the system continually throughout time. The SMDP aims to accurately represent the evolution of the system  at decision epochs only.  Decision epochs occur at random points of time determined by the model description  requiring the decision maker to choose actions. We let $\X$ be a finite state-space. At decision epoch $t$, the system occupies state $x^{(t)}\in \X$ and the decision maker chooses an action $a$ from the decision set $\A. $ We denote $F(t|x, a)$ as the probability that the next decision epoch occurs within $t$ time units of the current decision epoch, given that the decision maker chooses action $a$ at state $x$ at the current decision epoch. We assume that there exist $\epsilon>0$ and $\delta>0$ such that
\begin{align}\label{eq:as_smdp}
F(\delta|x,a)\leq 1-\epsilon,
\end{align}
for each $x\in \X$ and $a\in \A$.

We use $ P(y | x, a )$ to  denote the probability that the SMDP is at state $y$ when the next decision epoch occurs, given that the decision maker chooses action $a$ at state $x$ at the current decision epoch. In other words, $P(y|x, a)$ is a transition kernel of the
\textit{embedded Markov decision process}, that describes the state transition evolution only, see \cite[equation (11.4.5)]{Puterman2005}.

We let $t_k$ be the $k$th decision epoch that happens after time $t_0=0$. For convention, we assume that the SMDP starts with the first decision epoch $t_0$ at time $0$.
At time $t_k$, the system occupies state $x^{(t_k)}$  and the decision maker chooses action $a^{(t_k)}$. As a consequence of this action choice, the system remains in state $x^{(t_k)}$ for $\tau_k\sim F(\cdot|x^{(t_k)},a^{(t_k)})$ units of time, where $F(\cdot|x^{(t_k)},a^{(t_k)})$ is the cumulative distribution function of $\tau_k$. At time $t_{k+1}=t_k+\tau_k$ the system state changes
to $x^{(t_{k+1})}\sim P(\cdot|x^{(t_k)},a^{(t_k)})$, and the next decision epoch occurs. We define $m(x, a)$ as the expected time until the next decision epoch, given that action $a$ is chosen in state $x$ at the current decision epoch:
\begin{align*}
m(x, a):=\E\left[\tau_1~|~x^{(0)}=x, a^{(0)}=a\right] = \int\limits_{0}^\infty (1- F(s|x, a)) ds.
\end{align*}
Below, we also refer to $m:\X\times \A\rightarrow\R_+$ as the  \textit{expected time function}.
We denote
 \begin{align*}
m_\pi(x): = \sum\limits_{a\in \A}m(x, a)\pi(a|x)
\end{align*}
 as the expected time until the next decision epoch at state $x\in \X$  according to policy $\pi$.

We assume that when action $a$ is chosen in state $x$, instantaneous cost $k(x, a)$  is incurred. Moreover, infinitesimal cost  is incurred at rate $c(y, x, a)$ as long as the natural process occupies state $y$, and action $a$ was chosen in state $x$ at the preceding decision epoch.  We define $g(x, a)$ as an expected cost that is accumulated between decision epoch when the system is at state $x$ and action $a$ is chosen, and the following decision epoch:
\begin{align*}
g(x, a):=k(x, a) + \E\left[\int\limits_{0}^{\tau_1} c\left(x^{(t)}, x, a\right )dt ~| ~x^{(0)}=x, a^{(0)}=a \right].
\end{align*}
 We denote $g_\pi(x): = \sum\limits_{a\in \A}g(x, a)\pi(a|x)$ as the expected cost at state $x\in \X$  according to policy $\pi$.

We define a randomized stationary Markovian service policy as a map $\pi:\X\rightarrow \Delta(\A)$, where $\Delta A$ is the probability distribution over action space $\A$.
Under policy $\pi$, the corresponding embedded Markov chain has transition matrix $P_\pi$ defined as $P_\pi(x, y):=\sum\limits_{a\in \A}\pi(a|x)P(y|x,a).$
We assume that, for every stationary policy, the embedded Markov chain has a
unichain transition probability matrix.

We define a long-run average cost $\eta_\pi$ of policy $\pi$ for an SMDP:
\begin{align*}
\eta^\pi:=\lim\limits_{N\rightarrow \infty}\frac{     \underset{    \substack{  x^{(t_{k+1})}\sim P(\cdot| x^{(t_k)}, a^{(t_k)} ) }     }{\E}    \left[  g \big(x^{(t_k)} , a^{(t_k)} \big)~|~x^{(0)}\sim \mu\right]}{  \underset{    \substack{ \tau_k \sim F(\cdot| x^{(t_k)} , a^{(t_k)}) }     }{\E} \left[ \sum\limits_{k=0}^N \tau_k ~|~x^{(0)}\sim \mu\right]},
\end{align*}
where $\eta^\pi$ does not depend on the initial state distribution $\mu$ since $P_\pi$ is unichain, see \cite[Proposition 11.4.1]{Puterman2005}.

We define a Poisson equation of the SMDP with expected cost function $g$, expected time function $m$, and transition kernel $P$ for a stationary policy $\pi$:
\begin{align}\label{eq:Possion_SMDP}
h(x) = g_\pi(x) - \eta m_\pi(x) + \sum\limits_{y\in \X } P_\pi(y|x)h(y), \text{ for each }x\in \X.
\end{align}

 By \cite[Theorem 11.4.3]{Puterman2005}, the long-average cost $\eta_\pi$ of policy $\pi$  and function
\begin{align}\label{eq:solution_SMDP}
h_\pi(x) := \sum\limits_{k=0}^\infty \E_\pi \left[ g_\pi(x^{(t_k)}) - \eta_\pi m_\pi(x^{(t_k)})~|~x^{(0)}=x\right]
\end{align}
satisfy Poisson equation (\ref{eq:Possion_SMDP}). Function $h_\pi:\X\rightarrow \R$ that satisfies equation (\ref{eq:Possion_SMDP}) is called a solution to the Poisson equation for policy $\pi$.

We define advantage function $A_\pi:\X\times\A\rightarrow \R$ of policy $\pi$ as
\begin{align*}
A_\pi(x,a): = g(x, a) - \eta_\pi m(x, a)+ \sum\limits_{y\in \X}P(y|x,a) h_\pi(y) - h_\pi(x),
\end{align*}
for each $x\in \X$, $a\in \A.$

\subsection{Novel policy improvement bound}

In this section we derive a policy improvement bound for SMDPs.
 The following lemma establishes performance difference identity for SMDPs. An analogous performance difference identity for MDPs was proposed in \cite{Kakade2002}. The proof of Lemma \ref{lem:perf_iden_SMDP} can be found in Appendix \ref{appendix:bounds_SMDP}.

\begin{lemma}\label{lem:perf_iden_SMDP}
We consider the SMDP model described in Section \ref{sec:prel_SMDP}. The following policy performance identity holds  for any two policies $\pi_1$ and $\pi_2$ of the SMDP:
\begin{align}\label{eq:perf_iden_SMDP}
\eta_{\pi_2} -\eta_{\pi_1} =    \frac{1}{\overline m_{\pi_2}} \E_{x\sim  d_{\pi_2}, a\sim \pi_2(\cdot|x)}\left[ A_{\pi_1}(x, a)\right],
\end{align}
where $d_{\pi_2}$ is the stationary distribution of transition matrix $P_{\pi_2}$, and $\overline m_{\pi_2}:=\E_{x\sim d_{\pi_2}}[m_{\pi_2}(x)]$ is a mean time between decision epochs under policy $\pi_2$.
\end{lemma}

Performance difference identity (\ref{eq:perf_iden_SMDP}) allows us to establish a policy improvement bound for SMDPs.  The proof of Theorem \ref{thm:perf_bound_SMDP} can be found in Appendix \ref{appendix:bounds_SMDP}.

\begin{theorem}\label{thm:perf_bound_SMDP}

We consider the SMDP model described in Section \ref{sec:prel_SMDP}. The following bound on the difference of long-run average costs of two policies
$\pi_1$ and $\pi_2$ holds:
\begin{align}\label{eq:bound_av_SMDP}
 \eta_{\pi_2} -  \eta_{\pi_1}  \leq  \frac{1}{\overline m_{\pi_2}}  \underset{    \substack{ x\sim d_{  \pi_1}\\ a \sim \pi_1(\cdot|x) }     }{\E} \left[\frac{\pi_2(a|x)}{\pi_1(a|x)}A_{\pi_1}(x,a)   \right]  + 2 \epsilon_{\pi_2} \tau_1\left[D_{\pi_2}\right] \underset{  x\sim d_{ \pi_1} }{\E} \left[\text{TV}\Big(\pi_2(\cdot|x)~||~\pi_1(\cdot|x)\Big) \right],
\end{align}
where  $D_{\pi_2}$ is the group inverse of matrix $I -  P_{\pi_2}$, $\epsilon_{\pi_2}:= \frac{1}{\overline m_{\pi_2}}\max\limits_{x\in \X}\Big|\underset{a\sim \pi_2(\cdot|x)}{\E}[A_{\pi_1}(x,a)]\Big|$.

\end{theorem}

\section{Conclusion to Chapter \ref{ch:2} }\label{sec:PIB_conclusion}

In this chapter we introduce several novel policy improvement bounds for different setting. We get a unified policy improvement bound for discounted and average cost criterions. This new bound refines  previous existing policy improvement bounds for MDPs with the discounted objectives and suggests  a meaningful bound for the average cost objective. This result clears up the existing doubts about validity of the use of the on-policy deep RL algorithms with a large discount factor and generalizes the use of this class of RL algorithms on MDPs with average cost objectives.

The refined policy improvement bound for the infinite-horizon discounted setting is  \textit{optimal} in some sense and depends on the     one-norm ergodicity coefficient. We propose several bounds on this ergodicity coefficient to uncover its dependency on the discount factor.  Nevertheless, to the best of our knowledge, it is an open problem  whether the discounting leads to smaller one-norm ergodicity coefficient, i.e. $\tau_1[D_\pi]\geq \tau_1[D_\pi^{(\gamma)}]$ for $\gamma<1$. It is known that the discounting does decrease  Kemeny's constant, see \cite[Theorem 4.10]{Catral2010},  hence  identity (c) from Lemma \ref{lem:cond_nums} might be a potential path for proving the equivalent  result for the one-norm ergodicity coefficient.

We propose a  policy improvement bound for MDPs with countable state spaces. We bound the performance difference between two policies  assuming the $\V$-uniform ergodicity of the transition matrix of one of the policies.  Additional conditions on the "closeness" between two policies also makes the bound practical for reinforcement learning. Specifically, APG algorithms can be justified for solving MDP problems on countable state spaces. We obtain several bounds on the $\V$-weighted one-norm ergodicity coefficient that allow to estimate it based on the drift condition satisfied by the corresponding Markov chain. We believe these results are of independent interest for the Markov chains perturbation theory. 

Another potential application of the  ergodicity coefficient estimation is an adaptive adjustment of the allowed magnitude of the policy changes in each iteration of TRPO, PPO, and similar algorithms. Policy improvement bounds (\ref{eq:bound_opt}), (\ref{eq:bound_opt_V}), (\ref{eq:bound_av_SMDP}) show that the larger the ergodicity coefficient is, the more challenging minimization  of the bounds  becomes, e.g. in (\ref{eq:argmin}). 
 A large ergodicity coefficient  indicates that the corresponding Markov chain is sensitive to the updates and perturbations, and significant changes to its transition probabilities might lead to the  performance degradation. While the original TRPO and PPO algorithms suggest fixing a trust region parameter, $\delta$, and a clipping parameter, $\epsilon$, respectively, through the course of learning, we believe that adjustment of these parameters proportionally to the ergodicity coefficient each iteration might improve the robustness and sample complexity. Further research is needed to design and implement  TRPO and/or PPO algorithms with the adaptive step sizes depending on the ergodicity coefficient estimates.

Novel results are obtained for SMDPs. We derive the performance difference  identity and policy performance bound for SMDP policies. As a result, deep RL algorithms, such as PPO, TRPO and their variations, can be generalized and directly use to solve SMDP problems. Numerical experiments with PPO algorithm on a class of queueing networks with general arrival/service distributions are in our investigation plans.

\appendix

\chapter{Chapter 1 of Appendix}\label{appendix:qnc}

\section{Proofs of the theorems in Section \ref{sec:countable}}\label{sec:proofs}

 \begin{proof}[\textbf{Proof of Lemma \ref{lem:Zeq}}]
  We define vector $h:= Z\left(g -(d^Tg) e\right)$. Matrix $Z$ has a finite $\V-$norm, therefore the inverse matrix of $Z$ is unique and equal to $I-P+\Pi$. Then by definition vector $h$ satisfies
\begin{align}\label{eq:extPois}
(I-P+\Pi) h = g -(d^Tg) e.
\end{align}

Multiplying both sides of (\ref{eq:extPois}) by $d$ we get $d^T h = 0$ (and $\Pi h = 0$). Hence, vector $h$ is a solution of the Poisson equation  (\ref{eq:Poisson}) such that $d^T h = 0$. It follows from Lemma \ref{lem:poisson_sol} that $h=h^{(f)}.$
 \end{proof}

\begin{proof}[\textbf{Proof of Lemma \ref{lem:st}}]
We denote \begin{align}\label{eq:U}U_{\theta, \phi}: = (P_{\theta} - P_{\phi}) Z_{\phi}\end{align} and define matrix $H_{\theta, \phi}$ as
\begin{align}\label{eq:H}
H_{\theta,\phi} := \sum\limits_{k=0}^\infty U^k_{\theta, \phi}.
\end{align}
The convergence in the $\V$-weighted norm in definition (\ref{eq:H}) follows  from assumption $\|U_{\theta, \phi}\|^{ }_\V<1$.

The goal of this proof is to show that the Markov chain  has a unique stationary distribution $d_\theta$ such that
\begin{align}\label{eq:mu12}
d_\theta^T = d_\phi^T H_{\theta,\phi}^{ }.
\end{align}

We let $\nu^T :=  d_\phi^T H_{\theta,\phi}$.  We use $e = (1, 1, ...,1,...)^T$ to denote  the unit vector.  First, we  verify that $\nu^T e = \sum\limits_{x\in \X} \nu(x)=1.$ We note that  $Z_{\phi} e = e.$ Then
 \begin{align*}
 \nu ^T e = d_\phi^T H^{ }_{\theta,\phi} e =  d_\phi^T \sum\limits_{k=0}^\infty \left( (P_{\theta} - P_{\phi})Z_{\phi}\right)^k e  = d_\phi^T I e = 1.
 \end{align*}

  Second, we verify that $\nu^T P_{\theta} = \nu^T.$ We prove it by first assuming
  that \begin{align}\label{eq:Pfinite}
  P_{\theta}- P_{\phi} + U_{\theta, \phi}  P_{\phi} = U_{\theta, \phi}\end{align}  holds. Indeed,
 \begin{align*}
\nu^T P_{\theta} &=d_\phi^T \sum\limits_{k=0}^\infty U_{\theta, \phi}^k P_{\theta} \\
&=   d_\phi^T \sum\limits_{k=0}^\infty U_{\theta, \phi}^k P_{\theta}- d_\phi^T \sum\limits_{k=0}^\infty U_{\theta, \phi}^k P_{\phi} + d_\phi^T \sum\limits_{k=0}^\infty U_{\theta, \phi}^{k} P_{\phi}\\
& =  d_\phi^T +   d_\phi^T \sum\limits_{k=0}^\infty U_{\theta, \phi}^k P_{\theta}- d_\phi^T \sum\limits_{k=0}^\infty U_{\theta, \phi}^k P_{\phi} + d_\phi^T \sum\limits_{k=0}^\infty U_{\theta, \phi}^{k+1} P_{\phi} \\
& =  d_\phi ^T+ d_\phi^T \sum\limits_{k=0}^\infty U_{\theta, \phi}^k (P_{\theta}- P_{\phi} + U_{\theta, \phi}  P_{\phi})\\
& =d_\phi^T + d_\phi^T \sum\limits_{k=0}^\infty U_{\theta, \phi}^{k+1}\\
& =d_\phi^T \sum\limits_{k=0}^\infty U_{\theta, \phi}^{k}\\
& = \nu^T.
 \end{align*}

 It remains to prove (\ref{eq:Pfinite}).   Indeed,
      \begin{equation*}
\begin{aligned}[c]
   P_{\theta}- P_{\phi} + U_{\theta, \phi}  P_{\phi}& = P_{\theta}- P_{\phi} +  (P_{\theta} - P_{\phi}) Z_{\phi}   P_{\phi}\\
   & =  (P_{\theta}- P_{\phi} )(I + Z_{\phi}   P_{\phi} )\\
   & = (P_{\theta}- P_{\phi} )(I -\Pi_{\phi} + Z_{\phi}   P_{\phi} )   \\
   & =   (P_{\theta}- P_{\phi} )(I - Z_{\phi}  \Pi_{\phi} + Z_{\phi}   P_{\phi} )    \\
   & =  (P_{\theta}- P_{\phi} ) Z_{\phi} \\
   & = U_{\theta, \phi},
\end{aligned}
\end{equation*}
where the second equality follows from
  \begin{align*}
    \Big((P_{\theta} - P_{\phi}) Z_{\phi} \Big)  P_{\phi} =     (P_{\theta} - P_{\phi}) \Big(Z_{\phi}  P_{\phi} \Big),
  \end{align*}
  which holds   by \cite[Corollary 1.9]{Kemeny1976}, the third equality holds due to $(P_{\theta}- P_{\phi}) \Pi_{\phi}=0$, the fourth equality holds because $ Z_{\phi}  \Pi_{\phi} = \Pi_{\phi}$,  and the fifth equality follows from $ I = Z_{\phi}(I + \Pi_\phi - P_\phi) = Z_\phi + Z_\phi \Pi_\phi - Z_\phi P_{\phi}$.

The uniqueness of the stationary distribution follows from the fact that the Markov chain with transition matrix $P_{\theta}$ is assumed to be irreducible.

\end{proof}

 The following Lemma \ref{lem:norms} will be used in the proofs of Theorem \ref{thm:main} and Lemma  \ref{lem:policies} below. We believe the claim of  Lemma \ref{lem:norms} should be a well-known mathematical fact, but we have not found its proof in any textbook. For completeness, we present it here.

 \begin{lemma}\label{lem:norms}

Let $\M_{\X, \X}$ be a set of all matrices on the countable space $\X\times \X$.
 The operator norm $\|\cdot\|_\V$ on  $\M_{\X\times \X}$ is equivalent to the operator norm  induced from  vector norm $\|\cdot\|_{1, \V}$ and to the operator norm  induced from vector norm $\|\cdot\|_{\infty, \V}$ in the following sense:
\begin{align*}
\|T\|_\V = \sup\limits_{\nu:\|\nu\|_{1, \V}=1} \|\nu T \|_{1, \V} = \sup\limits_{h:\|h\|_{\infty,\V} = 1} \| Th \|_{\infty, \V}\quad\text{ for any }T\in \M_{\X\times \X},
\end{align*}
where $\|\nu\|_{1, \V}=\sum\limits_{x\in \X} |\nu(x)|\V(x)$, $\|\nu\|_{\infty, \V} = \sup\limits_{x\in \X} \frac{|\nu(x)|}{\V(x)}$, $\|T\|_\V=\sup\limits_{x\in \X} \frac{1}{\V(x)} \sum\limits_{y\in \X} |T(x, y)|\V(y)$.

 Furthermore, for any vectors $\nu_1, \nu_2$ on $\X$ and matrices $T_1,T_2\in \M_{\X\times \X}$ the following inequalities hold:
\begin{align}\label{eq:norms1}
\|\nu_1^T T \nu_2\|_\V\leq \|\nu_1\|_{1, \V}\|T\|_\V\| \nu_2 \|_{\infty, \V}
\end{align}
and
\begin{align}\label{eq:norms2}
\|  T_1 T_2\|_\V\leq \|T_1\|_{\V}\|T_2\|_\V.
\end{align}

\begin{proof}

First, we show that $\sup\limits_{h:\|h\|_{1, \V} = 1} \| Th \|_{1, \V} = \|T\|_\V$. On the one hand,
\begin{align*}
 \sup\limits_{\nu:\|\nu\|_{1, \V}=1} \|\nu T \|_{1, \V}  &= \sup\limits_{\nu: \X\rightarrow \R} \frac{1}{\|\nu\|_{1, \V}} \|\nu T \|_{1, \V}  \\
 &\geq\sup\limits_{ \nu\in \{e_x\}} \frac{1}{\|\nu\|_{1, \V}} \|\nu T \|_{1, \V}   \\
 &=  \sup\limits_{x\in \X} \frac{1}{\V(x)} \sum\limits_{y\in X} \V(y) |T(x, y)| = \|T\|_\V,
\end{align*}
where in the second step we choose a set $\{e_x\}$ of unit vectors $ e_x = (0,\dotsc,0,1,0,\dotsc)$, where the $x$ coordinate is $1$ and the other coordinates are $0$s.

On the other hand,
\begin{align*}
\|T\|_\V &= \sup\limits_{x\in X} \sum\limits_{y\in X}   \frac{1}{\V(x)} |T(x, y)|  \V(y) \sup\limits_{\nu:\|\nu\|_{1, \V}=1}  \sum\limits_{x\in X} |\nu(x)| \V(x)\\
&\geq \sup\limits_{\nu:\|\nu\|_{1, \V}=1}   \sum\limits_{y\in X}   \sum\limits_{x\in X}  \frac{1}{\V(x)}~|T(x, y)| ~ \V(y) ~ |\nu(x)|~ \V(x)\\
& =\sup\limits_{\nu:\|\nu\|_{1, \V}=1}   \sum\limits_{y\in X}   \sum\limits_{x\in X} |T(x, y) \nu(x)| \V(y) \\
&=\sup\limits_{\nu:\|\nu\|_{1, \V}=1} \|\nu T \|_{1, \V}.
\end{align*}

Similarly, we can show the equivalency of  $\sup\limits_{h:\|h\|_{\infty, \V} = 1} \| Th \|_{\infty, \V}$ and $\|T\|_\V$ norms.

Inequalities (\ref{eq:norms1}) and (\ref{eq:norms2}) follow from the properties of a linear operator norm.
\end{proof}
\end{lemma}

\begin{proof}[\textbf{Proof of Theorem \ref{thm:main}}]
 We denote $U_{\theta, \phi}:=(P_\theta - P_\phi)Z_\phi$. Under assumption $\|U_{\theta, \phi}\|_\V = D_{\theta, \phi} <1$  operator $H_{\theta, \phi}:=\sum\limits_{k=0}^{\infty}U^k_{\theta, \phi} $ is well-defined and
 \begin{align}\label{eq:normH}
 \|H_{\theta, \phi}\|_\V\leq \frac{1}{1 - D_{\theta, \phi}}.
 \end{align}

We represent the stationary distribution of the Markov chain with transition matrix $P_\theta$  as  $d_\theta^T = d_\phi^T H_{\theta, \phi}$, see   (\ref{eq:mu12}).  We get
\begin{align}\label{eq:mu_theta}
 \|d_\theta\|^{ }_{1, \V} = d_\theta^T\V = d_\phi H_{\theta, \phi}\V\leq\|H_{\theta, \phi}\|_\V^{ } (d_\phi^T \V)<\infty,
\end{align}
 since $H_{\theta,\phi}\V\leq \|H_{\theta, \phi}\|_\V^{ }\V$   by definition of the $\V$-norm.

The long-run average costs difference is equal to
\begin{equation*}
\begin{aligned}[c]
d_\theta^Tg - d_\phi^Tg  &= d_\theta^Tg+ d_\theta^T\left( (P_\theta-I)h_{\phi} \right)- d_\phi^Tg \\
&=d_\theta^T(g + P_\theta h_{\phi}  - h_{\phi} ) -d_\phi^Tg\\
&=d_\phi^T(g  - \eta_{\phi}  e + P_\theta h_{\phi}  -h_{\phi}) + (d_\theta^T - d_\phi^T)(g+P_\theta h_{\phi}  - h_{\phi} )\\
& = d_\phi^T(g  -\eta_{\phi}  e + P_\theta h_{\phi}  -h_{\phi}) + (d_\theta^T - d_\phi^T)(g - \eta_{\phi}  e +P_\theta h_{\phi}  - h_{\phi} ).
\end{aligned}
 \end{equation*}

Now we are ready to bound the last term:
\begin{equation*}
\begin{aligned}[c]
 \left|(d_\theta^T - d_\phi^T)   (g - \eta_{\phi}e +P_\theta h_{\phi}  - h_{\phi} )\right|&\leq \|d_\theta - d_\phi  \|_{1, \V}~ \|g - (d_\phi^T g)e +P_\theta h_{\phi}  - h_{\phi} \|_{\infty, \V} \\
 &=\|d_\theta - d_\phi  \|_{1, \V}~\|( P_\theta - P_{\phi}) h_{\phi}  \|_{\infty, \V}\\
 &= \|d_\theta - d_\phi  \|_{1, \V}~\| (P_\theta - P_{\phi})Z_{\phi} (g - (d_\phi^T g)e)  \|_{\infty, \V} \\
 &\leq  \|d_\theta - d_\phi  \|_{1, \V}~\| (P_\theta - P_{\phi}) Z_{\phi}\|_\V~ \| g - \eta_{\phi} e  \|_{\infty, \V}\\
 &=  D_{\theta, \phi}  \|d_\theta - d_\phi  \|_{1, \V}~\| g -\eta_{\phi}  e  \|_{\infty, \V}   \\
 &  =  D_{\theta, \phi} \|d_\theta^T U_{\theta, \phi} \|_{1, \V}~\| g - \eta_{\phi} e  \|_{\infty, \V},\\
 &\leq D_{\theta, \phi}  \|d_\theta\|_{1,\V}^{  } \|U_{\theta, \phi} \|_{ \V}~\| g - \eta_{\phi} e  \|_{\infty, \V},\\
  &\leq  D_{\theta, \phi}^2 \|d_\theta\|_{1,\V}^{ } \| g - \eta_{\phi} e  \|_{\infty, \V },\\
  &\leq D^2_{\theta, \phi} \|H_{\theta, \phi}\|_\V^{ }  \| g - \eta_{\phi} e  \|_{\infty, \V}  (d_\phi^T \V),\\
 &\leq  \frac{D^2_{\theta, \phi}}{1-D_{\theta, \phi}} \| g - \eta_{\phi} e  \|_{\infty, \V} (d_\phi^T \V),
\end{aligned}
 \end{equation*}
 where the first, second and third inequalities follow from Lemma  \ref{lem:norms}, the second equality follows from Lemma \ref{lem:Zeq},  the last equality holds due to
$d_\theta^T - d_\phi^T  =  d_\theta^T U_{\theta, \phi}$
from (\ref{eq:mu12}),   the fourth inequality follows from  (\ref{eq:D}), the fifth inequality follows from (\ref{eq:mu_theta}),  and the last inequality holds due to  (\ref{eq:normH}).

\end{proof}

\begin{proof}[\textbf{Proof of Lemma \ref{lem:policies}}]
    \begin{equation*}
\begin{aligned}[c]
     \|  (P_{\theta} - P_{\phi}) Z_{\phi}\|_\V&\leq  \|  P_{\theta} - P_{\phi} \|_\V\|Z_{\phi}\|_\V\\
     & = \|Z_{\phi}\|_\V\sup\limits_{x\in \X}\frac{1}{\V(x)} \sum\limits_{y\in \X} |P_{\theta} - P_{\phi}|_{x, y} \V(y)\\
     & = \|Z_{\phi}\|_\V\sup\limits_{x\in \X}\frac{1}{\V(x)} \sum\limits_{y\in \X} | \sum\limits_{a\in \A} P(y|x, a)\pi_\theta(a|x) - \sum\limits_{a\in \A} P(y|x, a)\pi_{\phi}(a|x)| \V(y)\\
          & \leq \|Z_{\phi}\|_\V\sup\limits_{x\in \X}\frac{1}{\V(x)} \sum\limits_{y\in \X} \sum\limits_{a\in \A} P(y|x, a) | \pi_\theta(a|x) -  \pi_{\phi}(a|x)| \V(y)\\
          & =\|Z_{\phi}\|_\V\sup\limits_{x\in \X}    \sum\limits_{a\in \A} | \pi_\theta(a|x) -  \pi_{\phi}(a|x)|    \frac{\sum\limits_{y\in \X} P(y|x, a) \V(y) }{\V(x)} \\
          &= \|Z_{\phi}\|_\V\sup\limits_{x\in \X}    \sum\limits_{a\in \A} \Big| \frac{\pi_\theta(a|x)}{ \pi_{\phi}(a|x)} - 1\Big|   G(x, a)
\end{aligned}
 \end{equation*}
\end{proof}

 \section{Proofs of the theorems in Section \ref{sec:M1} }\label{sec:disc}

 We consider the Poisson equation for a Markov chain with the transition kernel $P$, stationary distribution $d$, and cost function $g:\X\rightarrow \R$:
\begin{align*}
g(x) - d^Tg +\sum\limits_{y\in \X} P(y|x) h(y) - h(x) = 0, \text{ for each }x\in \X,
\end{align*}
which admits a solution
\begin{align*}
h^{(x^*)}(x):=\E\left[ \sum\limits_{k=0}^{\sigma(x^*)-1} \left(g (x^{(k)} )-d^Tg\right)|x^{(0)}=x \right] \text{ for each }x\in \X,
\end{align*}
where $\sigma(x^*) = \min\left\{k>0~|~x^{(k)}=x^*\right\}$ is the first time when state $x^*$ is visited.

Since regenerative cycles can be long in large-size systems, we propose to change the original dynamics and increase the probability of transition to the regenerative state $x^*$ from each state $x\in \X.$

We let $P(y|x)$ be an original transition probability from state $x$ to state $y$, for each $x,y\in \X$. We consider a new Markov reward process with cost function $g$ and a modified transition kernel $\tilde P^{(\gamma)}$:
\begin{align}\label{eq:newP}
\begin{cases}
\tilde P^{(\gamma)}(y|x) := \gamma P(y|x)\quad \text{ for }y\neq x^*,\\
\tilde P^{(\gamma) }(x^*|x) := \gamma P(x^*|x)+(1-\gamma),
\end{cases}
\end{align}
for each $x\in \X.$

We  modified the transition kernel so that the probability of transition to the regenerative state $x^*$ is at least $1-\gamma$ from any state.

The Poisson equation for the modified problem is equal to:
\begin{align}\label{eq:newPoisson}
g(x) - \tilde d^Tg +\sum\limits_{y\in \X}  \tilde P^{(\gamma)}(y|x) \tilde h(y) - \tilde h(x) = 0, \text{ for each }x\in \X,
\end{align}
 where $\tilde d$ is the stationary distribution of the Markov chain $\tilde P^{(\gamma)}$.

 Equation (\ref{eq:newPoisson}) admits a solution
\begin{align}\label{eq:newSol}
\tilde h^{(x^*)}(x):=\E\left[ \sum\limits_{k=0}^{\tilde \sigma(x^*)-1} \left (g (x^{(k)} )- \tilde d^Tg \right )~|~x^{(0)}=x \right] \text{ for each }x\in \X,
\end{align}
where   $x^{(k)}$ is the state of the Markov chain with transition matrix $\tilde P^{(\gamma)}$ after $k$ timesteps,  and $\tilde \sigma(x^*) = \min\left\{k>0~|~x^{(k)} = x^*\right\}$.
According to the new dynamics the regeneration occurs more frequently and we can estimate solution (\ref{eq:newSol})   by using  fewer replications of the regenerative simulation.

\begin{lemma}\label{lem:2sol}
Consider the Poisson equation for the Markov chain with the transition kernel $\tilde P^{(\gamma)}$ defined by (\ref{eq:newP}), stationary distribution $\tilde d$, and cost function $g:\X\rightarrow \R$:
\begin{align}\label{eq:newPoisson1}
g(x) - \tilde d^Tg +\sum\limits_{y\in \X}  \tilde P^{(\gamma)}(y|x) \tilde h(y) - \tilde h(x) = 0, \text{ for each }x\in \X.
\end{align}
Equation (\ref{eq:newPoisson1}) admits solutions:
\begin{align*}
J^{(\gamma)}(x):=\E\left[ \sum\limits_{k=0}^{\infty} \gamma^k\left(g(x^{(k)})- d^Tg \right)~|~x^{(0)}=x \right] \text{ for each }x\in \X,
\end{align*}
and
\begin{align*}
V^{(\gamma )}(x):=\E\left[ \sum\limits_{k=0}^{\sigma(x^*)-1} \gamma^k\left(g (x^{(k)} )-r(x^*)\right )~|~x^{(0)}=x \right] \text{ for each }x\in \X,
\end{align*}
where   $x^{(k)}$ is the state of the Markov chain with transition matrix $P$   after $k$ timesteps.
\begin{proof}
We substitute the definition of $ \tilde P^{(\gamma)}$ (\ref{eq:newP}) and rewrite equation (\ref{eq:newPoisson1}) as
 \begin{align}\label{eq:newPoisson2}
g(x) - (\tilde d^Tg - (1-\gamma) \tilde h(x^*)) + \gamma \sum\limits_{y\in \X} P(y|x) \tilde h(y)  - \tilde h(x) = 0, \text{ for each }x\in \X.
\end{align}

Equation (\ref{eq:newPoisson2}) admits infinitely many solutions, but we specify a unique solution fixing $\tilde h(x^*).$
Next, we consider two options.

First, we let $ \tilde h(x^*) = \frac{1}{1-\gamma}(d^Tg -\tilde d^Tg )$. Then the Poisson equation (\ref{eq:newPoisson2}) becomes
 \begin{align*}
g(x) - d^Tg + \gamma \sum\limits_{y\in \X} P(y|x) \tilde h(y)  - \tilde h(x) = 0, \text{ for each }x\in \X,
\end{align*}
and admits   solution
\begin{align*}
J^{(\gamma)}(x):=\E\left[ \sum\limits_{k=0}^{\infty} \gamma^k\left(g (x^{(k)} )- d^Tg \right)~|~x^{(0)}=x \right] \text{ for each }x\in \X.
\end{align*}

Second, we let $\tilde h(x^*) = 0$ in equation (\ref{eq:newPoisson2}). We note that $\tilde d^Tg = r(x^*)$, where $r(x^*) = (1-\gamma)\E\left[\sum\limits_{k=0}^\infty \gamma^k g\left(x^{(k)}\right)~|~x^{(0)}=x^*\right]$ is a present discounted value at $x^*.$

We get  the Poisson equation (\ref{eq:Poiss_reg})
\begin{align*}
g(x) - r(x^*) +\gamma \sum\limits_{y\in \X} P(y|x) \tilde  h(y) -\tilde  h(x) = 0, \text{ for each }x\in \X,
\end{align*}
which admits   solution  (\ref{eq:Vreg})
\begin{align*}
V^{(\gamma )}(x):=\E\left[ \sum\limits_{k=0}^{\sigma(x^*)-1} \gamma^k\left(g (x^{(k)} )-r(x^*) \right)|x^{(0)}=x \right] \text{ for each }x\in \X.
\end{align*}

Indeed,
\begin{align*}
V^{(\gamma )}(x) &= \E\left[ \sum\limits_{k=0}^{\infty} \gamma^k\left(g (x^{(k)} )-r(x^*)\right )|x^{(0)}=x \right]\\
&=\E\left[ \sum\limits_{k=0}^{\sigma(x^*)-1} \gamma^k\left(g (x^{(k)} )-r(x^*) \right)|x^{(0)}=x \right]+ \E\left[ \gamma^{\sigma(x^*)}\E\left[\sum\limits_{k=0}^{\infty} \gamma^k\left(g(x^{(k)} )-r(x^*)\right)|x^{(0)}=x^* \right]\right]\\
&=\E\left[ \sum\limits_{k=0}^{\sigma(x^*)-1} \gamma^k\left(g (x^{(k)} )-r(x^*) \right)|x^{(0)}=x \right].
\end{align*}

\end{proof}
\end{lemma}

\begin{proof}[\textbf{Proof of Lemma \ref{lem:disc}}]
By Lemma \ref{lem:2sol} function $ V^{(\gamma )}$ is a solution of Poisson equation (\ref{eq:newPoisson}). We consider the discounted value function  $J^{(  \gamma)}  = \E\left[ \sum\limits_{k=0}^{\infty} \gamma^k\left(g\left(x^{(k)}\right)- d^Tg \right)~|~x^{(0)}=x \right] $ that is another solution.

Then, for an arbitrary $x\in \X$,
\begin{align}\label{eq:VJ}
\left|V^{(\gamma)}(x) - h^{(x^*)}(x)\right| \leq \left|J^{(\gamma)}(x) - h^{(f)}(x)\right| + \left|V^{(\gamma)}(x) - h^{(x^*)}(x)  - \left(J^{(\gamma)}(x) - h^{(f)}(x)\right)\right|,
\end{align}
where $h^{(f)}$ is the fundamental solution of the Poisson equation (\ref{eq:Poisson}).

First, we  bound $\left|J^{(\gamma)}(x) - h^{(f)}(x)\right|$:
\begin{align*}
|J^{(\gamma)}(x) - h^{(f)}(x) | &\leq \sum\limits_{t=0}^\infty |\gamma^{t} - 1| \Big| \sum\limits_{y\in \X} P^t(y|x) (g(y) - d^Tg)   \Big| \\
&\leq R \V(x) \sum\limits_{t=0}^\infty (1  - \gamma^{t}) r^t\\
&= R\V(x)r\frac{1-\gamma}{(1-r)(1-\gamma r) },
\end{align*}
where the second inequality follows from (\ref{eq:geo}).

 Since  $V^{(\gamma)}$ and $J^{(\gamma)}$ are both solutions of the Poisson equation (\ref{eq:newPoisson}), therefore,
\begin{align*}
J^{(\gamma)}(x) - V^{(\gamma)}(x) =J^{(\gamma)}(x^*) - V^{(\gamma)}(x^*) = \frac{1}{1-\gamma}(d^Tg - r(x^*)) \text{ for each }x\in \X.
\end{align*}

Similarly, $ h^{(f)}(x) - h^{(x^*)}(x) = h^{(f)}(x^*)$ for each $x\in \X.$

Second, we bound the last term in inequality (\ref{eq:VJ}):
\begin{align*}
\Big|V^{(\gamma)}(x) - h^{(x^*)}(x) & - \left(J^{(\gamma)}(x) - h^{(f)}(x)\right)\Big|\\
&=\left|h^{(f)}(x^*) - \frac{1}{1-\gamma}(r(x^*)-d^Tg ) \right|\\
&=\left| \sum\limits_{t=0}^{\infty}\sum\limits_{y\in \X} P^t(y|x^*)(g(y)-d^Tg) - \sum\limits_{t=0}^{\infty}\sum\limits_{y\in \X}  \gamma^tP^t(y|x^*)(g(y)-d^Tg)  \right|\\
&=\left|  \sum\limits_{t=0}^{\infty}\sum\limits_{y\in \X}  (1-\gamma^t) P^t(y|x^*)(g(y)-d^Tg) \right|\\
&\leq    \sum\limits_{t=0}^{\infty} |1-\gamma^t| \left|\sum\limits_{y\in \X} P^t(y|x^*)(g(y)-d^Tg) \right|  \\
&\leq R\V(x^*)r\frac{1-\gamma}{(1-r)(1-\gamma r) },
\end{align*}
where the last inequality holds due to (\ref{eq:geo}).
\end{proof}

\begin{proof}[\textbf{Proof of Lemma \ref{lem:var}}]

We denote $\overline g(x^{(k)}): = \left(g (x^{(k)} )- r (x^*) \right )$.
 Following   \cite[Section 17.4.3]{Meyn2009}, we can show that
\begin{align*}
Var[ \hat V^{(\gamma )}(x)] &= \E\left[  \left(\sum\limits_{k=0}^{\sigma(x^*)-1} \gamma^k \overline  g (x^{(k)} ) \right)^2~\Big|~x^{(0)}=x \right] -\left( V^{(\gamma)}(x)\right)^2 \\
&=\E\left[ -\sum\limits_{k=0}^{\sigma(x^*)-1} \gamma^{2k}\overline g^2(x^{(k)}) + 2\sum\limits_{k=0}^{\sigma(x^*)-1} \sum\limits_{j=k}^{\sigma(x^*)-1}\gamma^{k+j}\overline g(x^{(k)})\overline g(x^{(j)}) ~\Big|~x^{(0)}= x \right] -\left( V^{(\gamma)}(x)\right)^2 \\
&=\E\left[ \sum\limits_{k=0}^{\sigma(x^*)-1}  \E\left[ 2  \sum\limits_{j=k}^{\sigma(x^*)-1}\gamma^{k+j}\overline g(x^{(k)})\overline g(x^{(j)})  - \gamma^{2k}\overline g^2(x^{(k)}) ~|~\F_k\right]~\Big|~x^{(0)}= x \right] -\left( V^{(\gamma)}(x)\right)^2 \\
&=\E\left[ \sum\limits_{k=0}^{\sigma(x^*)-1}  2\gamma^{2k}\overline g(x^{(k)})\E\left[   \sum\limits_{j=k}^{\sigma(x^*)-1}\gamma^{j-k}\overline g(x^{(j)})   ~|~\F_k\right]- \gamma^{2k}\overline g^2(x^{(k)})~\Big|~x^{(0)}= x \right] -\left( V^{(\gamma)}(x)\right)^2 \\
&= \E\left[ \sum\limits_{k=0}^{\sigma(x^*)-1}  2\gamma^{2k}\overline  g(x^{(k)} ) V^{(\gamma)}(x^{(k)})-\gamma^{2k} \overline g^2(x^{(k)}) ~\Big|~x^{(0)}= x \right] -\left( V^{(\gamma)}(x)\right)^2,
\end{align*}
where $\F_k$ is a $\sigma-$algebra generated by $x^{(0)}, x^{(1)}, ..., x^{(k)}$.

Hereafter, we denote $\sum\limits_{y\in \X}P(y|x) V^{(\gamma)}(y)$ and $\sum\limits_{y\in \X}P(y|x) \left(V^{(\gamma)}(y)\right)^2$  as $PV^{(\gamma)}(x)$ and $P\left(V^{(\gamma)}\right)^2(x)$, respectively,  to improve readability. We use the Poisson equation (\ref{eq:Poiss_reg}) and replace $ \overline g(x^{(k)} )$ by $V^{(\gamma)}(x^{(k)}) -\gamma P V^{(\gamma)}(x^{(k)})$.
\begin{align*}
  &\E\left[ \sum\limits_{k=0}^{\sigma(x^*)-1}  2\gamma^{2k}\overline  g(x^{(k)} ) V^{(\gamma)}(x^{(k)})-\gamma^{2k} \overline g^2(x^{(k)})  ~\Big|~x^{(0)}= x \right] -\left( V^{(\gamma)}(x)\right)^2\\
&=\E\left[ \sum\limits_{k=0}^{\sigma(x^*)-1}  \gamma^{2k}\left(2\left(V^{(\gamma)}(x^{(k)}) -\gamma P V^{(\gamma)}(x^{(k)})\right)V^{(\gamma)}(x_k)-  \left(V^{(\gamma)}(x^{(k)}) - \gamma PV^{(\gamma)}(x^{(k)})\right)^2\right)~\Big|~x^{(0)}= x \right] \\
&\quad -\left( V^{(\gamma)}(x)\right)^2\\
&=\E\left[ \sum\limits_{k=0}^{\sigma(x^*)-1}  \gamma^{2k}\left(\left(V^{(\gamma)}(x^{(k)})\right)^2  -  \left(\gamma PV^{(\gamma)}(x^{(k)})\right)^2 \right)~\Big|~x^{(0)}= x \right] -\left( V^{(\gamma)}(x)\right)^2.
\end{align*}
Next, we subtract the expectation of  a martingale
\begin{align*}
\E\left[  -\left( V^{(\gamma)}(x)\right)^2 +\gamma^{2\sigma(x^*)}\left(V^{(\gamma)}(x^{(\sigma(x^*))})\right)^2 +\sum\limits_{k=0}^{\sigma(x^*)-1}  \gamma^{2k}\left(\left(V^{(\gamma)}(x^{(k)})\right)^2  -  \gamma^2P\left(V^{(\gamma)}\right)^2(x^{(k)}) \right)~\Big|~x^{(0)}= x \right]
\end{align*}
 that is equal to zero by \cite[Proposition 1]{Henderson1997}. Since $V^{(\gamma)}(x^{(\sigma(x^*))})=0$, we get
\begin{align*}
Var[ \hat V^{(\gamma )}(x)] &=\E\left[ \sum\limits_{k=0}^{\sigma(x^*)-1}  \gamma^{2k}\left(\gamma^2 P\left(V^{(\gamma)}\right)^2(x^{(k)}) -    \left( \gamma PV^{(\gamma)}(x^{(k)})\right)^2\right)~\Big|~x^{(0)}= x \right]\\
  &=\gamma^2\E\left[ \sum\limits_{k=0}^{\sigma(x^*)-1}  \gamma^{2k}\left( Var\left[V^{(\gamma)}(x^{(k+1)})~|~x^{(k)}\right]\right)~\Big|~x^{(0)}= x \right],
\end{align*}
where $Var\left[V^{(\gamma)}(x^{(k+1)})~|~x^{(k)}\right] = \sum\limits_{y\in \X}P(y|x^{(k)}) \left(V^{(\gamma)}(y)\right)^2 - \left(\sum\limits_{y\in \X}P(y|x^{(k)}) V^{(\gamma)}(y)\right)^2.$

Next, we want to show that there exists   constant $B_1>0$ such that for any  $\gamma\in [0,1]$
 \begin{align*}\left(V^{(\gamma)}(x)\right)^2\leq B_1 \V(x)\quad \text{for each }x\in \X.\end{align*}
First, we recall that function $V^{(\gamma)}$ is a solution of Poisson equation (\ref{eq:newPoisson}) for the system with modified dynamics (\ref{eq:newP}) and cost function $g$. Given that transition matrix $P$ satisfies the drift condition (\ref{eq:drift}), for the modified dynamics we have the following drift inequality
\begin{align*}
\sum\limits_{y\in \X}\tilde P^{(\gamma)}(y|x)\V(y) \leq \varepsilon \V(x) +(\V(x^*)+b)\I_{C\cup \{x:b \V(x)\leq \V(x^*)\}}(x)\quad \text{for each }x\in \X,
\end{align*}
for any $\gamma\in [0,1].$
Indeed,
\begin{align*}
\sum\limits_{y\in \X}\tilde P^{(\gamma)}(y|x)\V(y) &=\gamma \sum\limits_{y\in \X} P(y|x)\V(y) +(1-\gamma)\V(x^*)\\
&\leq \gamma \varepsilon \V(x) +(1-\gamma)\V(x^*) +b\I_C\\
&\leq \varepsilon \V(x) +(\V(x^*)+b)\I_{C\cup \{x\in \X:\varepsilon \V(x)\leq \V(x^*)\}}(x),
\end{align*}
where the first inequality follows from the drift condition (\ref{eq:drift}), and the second inequality follows from the fact that $ \gamma \varepsilon \V(x) +(1-\gamma)\V(x^*) \leq \varepsilon \V(x)$ if $\varepsilon \V(x)\geq \V(x^*)$, and $ \gamma \varepsilon\V(x) +(1-\gamma)\V(x^*) \leq   \V(x^*)$ otherwise.

Second, we use Jensen's inequality and get that function $\sqrt{\V}$ is also a Lyapunov function for the modified system:
\begin{align*}
\sum\limits_{y\in \X}\tilde P^{(\gamma)}(y|x) \sqrt{\V(y)} \leq \sqrt{\varepsilon \V(x)} +\sqrt{\V(x^*)+b}~\I_{C\cup \{x\in \X:\varepsilon \V(x)\leq \V(x^*)\}}(x)\quad \text{for each }x\in \X.
\end{align*}
This drift inequality and the assumption that $|g(x)|\leq \sqrt{\V(x)}$ for each $x\in \X$  allow us to apply \cite[Theorem 17.7.1]{Meyn2009}, see also \cite[equation (17.39)]{Meyn2009},   and conclude that, for some $c_0>0$ independent of $\gamma$, Poisson equation (\ref{eq:newPoisson}) admits the fundamental solution $J^{(\gamma)}:\X\rightarrow \R$ such that
\begin{align*}
|J^{(\gamma)}(x)|\leq c_0(\sqrt{\V(x)}+1), \text{ for each } x\in \X.
\end{align*}
 Function $V^{(\gamma)}$ is another solution of Poisson equation (\ref{eq:newPoisson}), such that  $J^{(\gamma)}(x) =V^{(\gamma)}(x)+J^{(\gamma)}(x^*)$, for each $x\in \X$, because $V^{(\gamma)}(x^*)=0$. Since $\V\geq 1$, there exists   constant $B_1>0$ such that
\begin{align*}
|V^{(\gamma)}(x)|\leq |J^{(\gamma)}(x)| +|J^{(\gamma)}(x^*)|\leq c_0(\sqrt{\V(x)}+1) +c_0(\sqrt{\V(x^*)}+1)\leq \sqrt{B_1 \V(x)},
\end{align*}
for each $x\in \X$.
We have proved that, for some   constant $B_1>0$,
  \begin{align*}\left(V^{(\gamma)}(x)\right)^2\leq B_1 \V(x)\quad \text{for each }x\in \X,\end{align*} for any $\gamma\in [0,1].$

We let $G^{(\gamma)}(x) := Var\left[V^{(\gamma)}(x^{(1)})~|~x^{(0)}=x\right].$  Then there exists a positive constant $B$ such that
  \begin{align}\label{eq:var_ineq1}
  G^{(\gamma)}(x) \leq  \sum\limits_{y\in \X}P(y|x) \left(V^{(\gamma)}(y)\right)^2\leq B_1  \sum\limits_{y\in \X}P(y|x)  \V(y)\leq B\V(x),
  \end{align}
where the last inequality follows from the drift condition (\ref{eq:drift}).
By  \cite[Theorem 15.0.1]{Meyn2009} we have that there exist constants $R>0$ and $r\in (0,1)$ such that
\begin{align}\label{eq:var_ineq2}
\left|   \sum\limits_{y\in \X}P^k(y|x)    G^{(\gamma)}(y) -  \beta^{(\gamma)}\right|\leq R\V(x) r^k,
\end{align}
where  $\beta^{(\gamma)} := \sum\limits_{x\in \X} d(x) G^{(\gamma)}(x)$ is a discounted asymptotic variance.

Then
\begin{align*}
Var[ \hat V^{(\gamma )}(x)] &=\gamma^2\E\left[ \sum\limits_{k=0}^{\sigma(x^*)-1}  \gamma^{2k}\left( Var\left[V^{(\gamma)}(x^{(k+1)})~|~x^{(k)}\right]\right)~\Big|~x^{(0)}= x \right]\\
&\leq \gamma^2 \sum\limits_{k=0}^\infty \gamma^{2k}\E[G^{(\gamma)}(x_k)~|~x^{(0)}=x ]\\
&\leq  \gamma^2 \sum\limits_{k=0}^\infty  \gamma^{2k}\left( R\V(x) r^k + \beta^{(\gamma)} \right)\\
&=\gamma^2\left( R\V(x)\frac{1}{1-\gamma^2r} + \beta^{(\gamma)} \frac{1}{1-\gamma^2}  \right)\\
&\leq \gamma^2\left( R\V(x)\frac{1}{1-\gamma^2r} + (d^T\V) B\frac{1}{1-\gamma^2}  \right),
\end{align*}
where the second inequality follows from (\ref{eq:var_ineq2}) and the last inequality follows from  (\ref{eq:var_ineq1}).

\end{proof}

\section{Maximal stability of the proportionally randomized policy}\label{sec:PR}

We assert that a discrete-time MDP obtained by
uniformization of the multiclass queueing network semi-Markov decision process model is stable
under the proportionally randomized (PR) policy if the load conditions
(\ref{eq:load}) are satisfied. We illustrate the proof for the criss-cross
queueing network. We let $x\in \Z_+^3$ be a state for the discrete-time
MDP. The proportionally randomized policy $\pi$ is given by
\begin{align*}
  & \pi(x) =
    \begin{cases}
      \Big(\frac{x_1}{x_1+x_3}, 1, \frac{x_3}{x_1+x_3}\Big) & \text{ if } x_2\ge 1 \text{ and } x_1+x_3\ge 1, \\
      \Big(\frac{x_1}{x_1+x_3}, 0, \frac{x_3}{x_1+x_3}\Big) & \text{ if } x_2= 0 \text{ and } x_1+x_3\ge 1, \\
    \big(0, 1, 0\big) & \text{ if } x_2\ge 1 \text{ and } x_1+x_3=0 , \\
     \big(0, 0, 0\big) & \text{ if } x_2=0 \text{ and } x_1+x_3=0.
      \end{cases}
\end{align*}
Recall the transition probabilities $\tilde P$ defined by (\ref{eq:unif}).  The
discrete-time MDP operating under policy $\pi$ is a DTMC. Now we can  specify its
transition matrix.  For $x\in \Z_+^3$ with $x_1\ge 1$, $x_3\ge 1$, and $x_2\ge 1$,
\begin{align*}
  P(y|x)& = \pi_1(x)\tilde P\big(y|x,(1,2)\big)+\pi_3(x)\tilde P\big(y|x,(3,2)\big) \quad \text{ for each } y\in \Z_+^3.
\end{align*}
For $x\in \Z_+^3$ on the boundary with $x_1\ge 1$, $x_3\ge 1$, and $x_2= 0$,
\begin{align*}
  P(y|x)& = \pi_1(x)\tilde P\big(y|x,(1,0)\big)+\pi_3(x)\tilde P\big(y|x,(3,0)\big) \quad \text{ for each } y\in \Z_+^3.
\end{align*}
For $x\in \Z_+^3$ on the boundary with $x_1= 0$, $x_3\ge 1$, and $x_2\ge 1$,
\begin{align*}
  P(y|x)& = \tilde P\big(y|x,(3,2)\big) \quad \text{ for each } y\in \Z_+^3.
\end{align*}
Similarly, we write the transition probabilities for other boundary cases. One can verify that
\begin{align}
  & P\big((x_1+1,x_2, x_3)| x\big)  = \frac{\lambda_1}{B},
    \quad    P\big((x_1,x_2, x_3+1)| x\big)  = \frac{\lambda_3}{B}, \label{eq:pr1}\\
  &  P\big((x_1-1,x_2+1, x_3)| x\big) =\frac{\mu_1}{B}\frac{x_1}{x_1+x_3}  \quad \text{ if } x_2\ge 1, \\
  & P\big((x_1,x_2, x_3-1)| x\big) =\frac{\mu_3}{B}\frac{x_3}{x_1+x_3} \quad \text{ if } x_3\ge 1, \\
  &  P\big((x_1,x_2-1, x_3)| x\big) =\frac{\mu_2}{B} \quad \text{ if } x_2\ge 1, \\
  & P(x|x) = 1- \sum_{y\neq x} P(y|x),\label{eq:pr5}
\end{align}
where $B=\lambda_1+\lambda_3+\mu_1+\mu_2+\mu_3$. This transition
matrix $P$ is irreducible. Now we consider the continuous-time
criss-cross network operating under the head-of-line
proportional-processor-sharing (HLPPS) policy  defined in
\cite{Bramson1996}. Under the HLPPS policy, the  jobcount  process
$\{Z(t), t\ge 0\}$ is a CTMC. Under the load condition (\ref{eq:load_cc}),
\cite{Bramson1996} proves that the CTMC is positive recurrent. One can
verify that the transition probabilities in
(\ref{eq:pr1})-(\ref{eq:pr5}) are identical to the ones for a
uniformized DTMC of this CTMC. Therefore, the DTMC corresponding to the
transition probabilities (\ref{eq:pr1})-(\ref{eq:pr5}) is positive
recurrent,  and proves  the stability of the discrete-time MDP operating under the
proportionally randomized policy.

\section{Additional experimental results}\label{sec:regVSinf}

In Remark \ref{rem:regVSinf} we discussed two possible biased
estimators of the solution to the Poisson equation. In this section we
compare the performance of the PPO algorithm with these two
estimators.  We consider two versions of line 7 in Algorithm \ref{alg2}: Version 1 uses the regenerative discounted value function
(VF) estimator (\ref{eq:esf}), and Version 2 uses the discounted value
function estimator (\ref{eq:Jesf}). We apply two versions of the PPO
algorithm for the criss-cross network operating under the  balanced medium (B.M.) load
regime.  The queueing network parameter setting is identical to
  the one detailed in Section \ref{sec:cc}, except that the quadratic
  cost function $g(x) = x_1^2+x_2^2+x_3^2$ replaces the linear
  cost function that is used to minimize the long-run average cost, where
$x_i$ is a number of jobs in buffer $i$, $i=1, 2, 3.$

We use  Xavier initialization to initialize
the policy NN parameters $\theta_0$. We take the empty system state $x^* = (0,0,0)$ as a regeneration state. Each episode in each iteration starts at the regenerative state and runs for $6,000$ timesteps. We compute the one-replication estimates of a value function (either regenerative discounted VF or discounted VF)   for the first $N=5,000$ steps at each episode. In this experiment we simulated $Q=20$ episodes in parallel.  The values of the remaining  hyperparameters (not mentioned yet) are  the same as  in Table \ref{tab:par2}.

In Figure \ref{fig:regVSinf} we compare the learning curves of PPO algorithm \ref{alg2}  empirically to demonstrate the benefits of using the regenerative discounted VF estimator over the discounted VF estimator when the system regeneration occurs frequently.

\begin{figure}[H]
\centering%
\includegraphics[width=.7\linewidth]{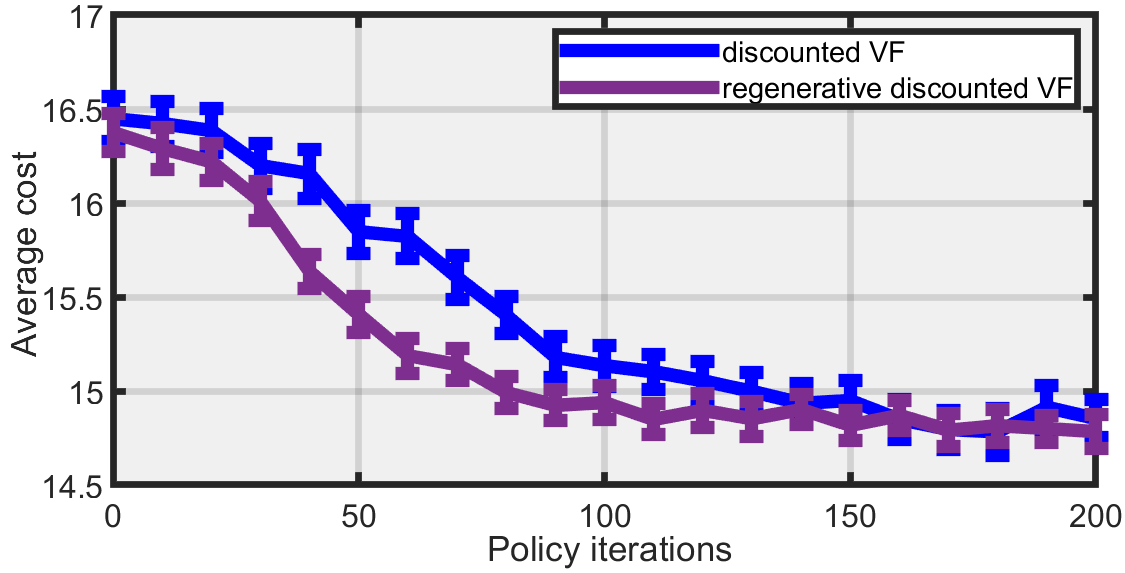}
\caption[Learning curves from Algorithm \ref{alg2} for the criss-cross network with the B.M. load and quadratic cost function.]{
Learning curves from Algorithm \ref{alg2} for the criss-cross network with the B.M. load and quadratic cost function.
The solid purple and blue lines show the performance of  the PPO policies obtained from Algorithm \ref{alg2}  in which
the solutions to the Poisson equations are estimated  by the discounted VF estimator and by the regenerative discounted VF estimator, respectively.
}

 \label{fig:regVSinf}
\end{figure}

\section{Neural network structure}\label{sec:nn}

In the experiments we parameterized the RL policy with a neural
network. Here, we use $\theta$ to denote the vector of weights and
  biases of the neural network. For a fixed parameter $\theta$, the
  neural network outputs deterministically distribution
  $\pi_\theta(\cdot|x)$ over the action space for each state
  $x\in \X$. Therefore, the resulting policy $\pi_\theta$ is a randomized policy as explained in Section~\ref{sec:MQN}.

To represent the policy we use a fully connected  feed-forward neural network with one input layer,   three hidden layers with tanh activation functions, and one output layer. The input layer has $J$ units, one for each job class, the first hidden layer has $10\times J$ units, the third hidden layer has $10\times L$, where $L$ is  number of stations in the queueing system. The number of units in the  second hidden layer is a geometric mean of units in the first and third hidden layers (i.e. $10\times \sqrt{LJ}$).

We use $z^{(k)}_j$ to denote the variable in the $j$th unit of hidden
layer $k$, $k=1, 2, 3$. Thus, our feed-forward neural network has the following representations:
\begin{align*}
  & z^{(1)}_j =  h\Big(\sum_{i=1}^J A^{(1)}_{ji}x_i + b^{(1)}_j\Big), \quad j=1, \ldots, 10J, \\
  & z^{(2)}_j =  h\Big(\sum_{i=1}^{10J} A^{(2)}_{ji}z^{(1)}_i + b^{(2)}_j\Big), \quad j=1, \ldots, 10\sqrt{LJ},\\
  & z^{(3)}_j =  h\Big(\sum_{i=1}^{10\sqrt{LJ}} A^{(3)}_{ji}z^{(2)}_i + b^{(3)}_j\Big), \quad j=1, \ldots, 10\sqrt{L},
\end{align*}
where $h:\R\to\R$ is the activation function given by $h(y)=\tanh(y)$ for each $y\in \R$.

 We denote $p=(p_j)$ as the output vector, which is given by
\begin{align*}
  p_j = \sum_{i=1}^{10\sqrt{L}} A^{(4)}_{ji}z^{(3)}_i + b^{(4)}_j, \quad j\in \J,
\end{align*}
and normalize output vector $p$ via a \textit{softmax function}  into $L$ probability
distributions as:
\begin{align}\label{eq:softmax}
  \pi_{\theta}(j|x) = \frac{\exp(p_j)}{\sum_{i\in \mathcal{L}(\ell)} \exp(p_i)} \quad \text{ for each } j\in \mathcal{B}(\ell), \, \ell\in \LL,
\end{align}
where the sets $\J$, $\LL$, and $\mathcal{B}(\ell)$ are defined in Section~\ref{sec:MQN}.

The neural network parameter $\theta$ is the vector of weights $A$'s and biases $b$'s
\begin{align*}
\theta=  \Big(A^{(1)},\, b^{(1)}, A^{(2)},\, b^{(2)}, A^{(3)},\, b^{(3)}, A^{(4)},\, b^{(4)}\Big),
\end{align*}
which has dimension
\begin{align*}
  K=\big(J\times 10J + 10J\big) +   \big(10J \times 10\sqrt{LJ} + 10\sqrt{LJ}\big) +    \big(10\sqrt{LJ}\times 10L + 10L\big) +
    \big(10L\times J + J\big).
\end{align*}
For example, when $J=21$ and $L=7$, this dimension is approximately equal to
$30203$.

To represent the value function we use a neural network whose
architecture is almost identical to the policy neural network except
that the third hidden layer has $10$ units. The
number of units in the second hidden layer is $10\times \sqrt{J}$. The
output layer contains one unit with a linear activation function, which means
that
\begin{align*}
  V(x) = \sum_{i=1}^{10} A^{(4)}_{i} z^{(3)}_i + b^{(4)}.
\end{align*}

 For the N-model processing network  in Section  \ref{sec:nmodel} the  structure of policy and value NNs is the same as for the MQNs described above, except for the meaning of  set $\B(\ell)$ in (\ref{eq:softmax}). We consider   station $\ell\in \LL$ of a processing network. The set   $\B(\ell)$ includes a job class  if and only if  there is an activity  such that station $\ell$ can process jobs from this   class.

\section{Implementation details of numerical experiments in Section \ref{sec:experiments}}\label{sec:par}

 We use Tensorflow v1.13.1  \cite{Abadi2016} to build a training routine of the neural networks and Ray package v0.6.6 \cite{Moritz2018}
    to maintain parallel simulation of the actors. We run all experiments on a   2.7 GHz  96-core processor with 1510 GB of RAM.

    We optimize  the value and policy functions to minimize the corresponding loss functions (\ref{eq:Vappr}), (\ref{eq:popt}) by the Adaptive Moment Estimation (Adam) method \cite{Kingma2017}.
    The Adam method is an algorithm for mini-batch gradient-based optimization.
We assume that $N$ datapoints have been generated  $D^{(1:N)}= \Big\{  (x^{(j)}, a^{(j)}, \hat A_j)\Big\}_{j=1}^N$ to update the policy NN  or $D^{(1:N)} = \left\{  ( x^{(j)}, \hat V_j)\right\}_{j=1}^N$ to update the value NN.
 The Adam algorithm runs for $E$ epochs. The number of epochs is the number of complete passes through the entire dataset. In the beginning of a new epoch $e$ the entire dataset
is randomly reshuffled and divided into   batches with size $m$. Then each batch (indexed by $n$) is passed to the learning algorithm and the parameters
of the neural networks $\theta = (\theta^1, \dotsc, \theta^i,\dotsc, \theta^K)$  are updated at the end of every such step  according to
     \begin{align*}
     \theta_{n+1}^i = \theta_n^i - \varsigma \frac{1}{\sqrt{\hat H^i_n} +\varrho} \hat G^i_n,
     \end{align*}
     where $\varsigma$ is a learning rate, $\hat G_n$ and $\hat H_n$ are moving average estimates of the first and second moments of the gradient, respectively, and $\varrho<<1$ is a constant.

     We use the batches to compute the gradient of a loss function $\hat L(\theta, D)$, such as for (\ref{eq:popt}):
    \begin{align}\label{grad}
    g_n = \frac{1}{m} \nabla_\theta   L\left(\theta,  D_e^{ (nm:nm+m)}\right),
    \end{align}
where $D_e^{ (nm:nm+m)}$ denotes the data segment in the $n$th batch of size $m$ at epoch $e$.

 The moving averages $G_{0}$ and $H_0$ are initialized as vectors of zeros at the first epoch.   Then    the Adam method updates the moving average estimates   and includes bias corrections to   account for their initialization at the origin:
   \begin{align*}
   \begin{cases}
    G_{n} = \beta_1 G_{n-1} +(1-\beta_1) g_{n-1},\\
    \hat G_{n} = \frac{G_{n}}{1 - \beta_1^n},
    \end{cases}
    \end{align*}
where  $\beta_1>0$,  with $\beta_1^n$  denoting $\beta_1$ to the power $n$.

    Similarly, we compute the second moments by
    \begin{align*}
    \begin{cases}
    H_{n} = \beta_2 H_{n-1} +(1-\beta_2)  g_{n-1}^2,\\
    \hat H_{n} = \frac{H_{n}}{1 - \beta_2^n},
    \end{cases}
    \end{align*}
    where $g_n^2$ means the elementwise square, $\beta_2>0$   with $\beta_2^n$  denoting $\beta_2$ to the power $n$.

 Each subsequent epoch continues   the count over $n$  and keeps  updating the moving average estimates $G_n, H_n$  starting from their final values of the latest  epoch.

Table \ref{tab:par1} and Table \ref{tab:par2} list the PPO hyperparameters we choose for the experiments in Section \ref{sec:experiments}. Table \ref{tab:rt} reports the estimates of the running time of Algorithm \ref{alg2}.

\begin{table}[H]
\centering%
\begin{tabular}{l|@{\quad}l}
  \hline
  Parameter  & Value\\\hline
   Clipping parameter  $(\epsilon)$ & $0.2\times \max[ \alpha,0.01]$ \\
 No. of regenerative cycles per actor (N) & 5,000 \\
  No. of  actors $(Q)$  & 50 \\
  Adam parameters for policy NN& $\beta_1 = 0.9$, $\beta_2 = 0.999$,    \\
 & $\varrho = 10^{-8},$  $\varsigma=5\cdot 10^{-4}\times \max[ \alpha,0.05]  $   \\
  Adam parameters for value NN& $\beta_1 = 0.9$, $\beta_2 = 0.999$,  \\
&  $\varrho = 10^{-8},$  $\varsigma=2.5\cdot 10^{-4}  $ \\
  No. of epochs (E) & 3\\
  Minibatch size in Adam method (m) & 2048\\
\end{tabular}
\caption[PPO hyperparameters used in Algorithms \ref{alg1} and \ref{alg1amp} for the experiments in Section \ref{sec:cc}.]{PPO hyperparameters used in Algorithms \ref{alg1} and \ref{alg1amp} for the experiments in Section \ref{sec:cc}.  Parameter $\alpha$ decreases linearly  from $1$ to $0$ over the course of learning: $\alpha =(I - i)/I$ on the $i$th policy iteration, $i=0, 1,\dotsc,I-1. $  }\label{tab:par1}
\end{table}

\begin{table}[H]
\centering%
\begin{tabular}{l|@{\quad}l}
  \hline
  Parameter  & Value\\\hline
   Clipping parameter  $(\epsilon)$ & $0.2\times \max[ \alpha,0.01]$ \\
  Horizon $(N)$ & 50,000 \\
  No. of  actors $(Q)$  & 50 \\
  Adam parameters for policy NN& $\beta_1 = 0.9$, $\beta_2 = 0.999$,    \\
  &   $\varrho = 10^{-8},$  $\varsigma=5\cdot 10^{-4}\times \max[ \alpha,0.05]  $   \\
  Adam parameters for value NN& $\beta_1 = 0.9$, $\beta_2 = 0.999$,  \\
 &   $\varrho = 10^{-8},$  $\varsigma=2.5\cdot 10^{-4}  $ \\
  Discount factor $(\beta)$  & 0.998  \\
  GAE parameter $(\lambda)$  & 0.99 \\
  No. of epochs (E)& 3\\
  Minibatch size in Adam method (m)  & 2048\\
\end{tabular}
\caption[PPO hyperparameters used in  Algorithm \ref{alg2} for the experiments in Sections \ref{sec:ext}  and  \ref{sec:nmodel}.]{PPO hyperparameters used in  Algorithm \ref{alg2} for the experiments in Sections \ref{sec:ext}  and  \ref{sec:nmodel}. Parameter $\alpha$ decreases linearly  from $1$ to $0$ over the course of learning: $\alpha =(I - i)/I$ on the $i$th policy iteration, $i=0, 1,\dotsc,I-1. $}\label{tab:par2}
\end{table}

   \begin{table}[H]
\centering%
\begin{tabular}{|c|c|}
  \hline
  Num. of classes $3L$  & Time (minutes)  \\\hline
  6 & 0.50 \\\hline
  9 & 0.73 \\\hline
  12  & 1.01\\\hline
  15  & 2.12 \\\hline
  18  & 4.31\\\hline
  21  & 7.61 \\
  \hline
\end{tabular}
\caption[Running time of one policy iteration of Algorithm \ref{alg2}  for the extended six-class network in Figure \ref{fig:fig1}.]{Running time of \textit{one policy iteration} of Algorithm \ref{alg2}  for the extended six-class network in Figure \ref{fig:fig1}.}\label{tab:rt}
\end{table}

In the Algorithm \ref{alg2} we use  finite length episodes to estimate the expectation of the loss function in line 10. For each episode we need  to specify an initial state.
 We propose sampling the initial states from the set of states generated during previous policy iterations.
  We consider the $i$th policy iteration of the algorithm. We need to choose initial states to simulate policy $\pi_i$.  Since policy $\pi_{i-1}$ has been simulated in the $(i-1)$th iteration of the algorithm,   we can    sample $Q$ states uniformly at random from the episodes generated under policy $\pi_{i-1}$ and  save them in memory. Then we use them as initial states for   $Q$ episodes under $\pi_i$ policy. For policy $\pi_0$ all $Q$ episodes start from state $x = (0,\dotsc,0).$

\chapter{Chapter 2 of Appendix}

 \section{Proofs of the theorems in Section \ref{sec:control_policy} }\label{appendix:ride_hailing_proofs}

\begin{proof}[\textbf{Proof of Lemma \ref{lem:per_diff}}]

First, we note that due to the martingale property
 \begin{align*}
 \E_{\pi_\theta} \Big[&V_\phi(  x^{(1,1)}) -V _\phi(  x^{(H+1, 1)})\\
 &\quad\quad\quad+ \sum\limits_{t=1}^{H }   \Big( \sum\limits_{y\in \X}\mathcal{P}( x^{(t,I_t)},a^{(t,I_t)}, y) V_\phi(y) - V _\phi(  x^{(t,I_t)}) +
   \sum\limits_{i=1}^{I_t-1} \Big(    V_\phi(x^{(t, i+1)}) - V _\phi(  x^{(t,i)})\Big) \Big)      \Big]=0,
 \end{align*}
 see, for example, \cite{Henderson2002}.

 Then
\begin{align*}
&V_\theta(  x^{(1, 1)}) - V_\phi (x^{(1, 1)})  =  V_\theta( x^{(1,1 )}) - V_\phi( x^{(1, 1)})+ \E_{\pi_\theta} \Big[V_\phi(  x^{(1,1)}) -V _\phi(  x^{(H+1, 1)}) \\
&\quad\quad\quad\quad+ \sum\limits_{t=1}^{H }   \Big( \sum\limits_{y\in \X}\mathcal{P}( x^{(t,I_t)},a^{(t,I_t)}, y) V_\phi(y) - V _\phi(x^{(t,I_t)}) +
   \sum\limits_{i=1}^{I_t-1} \Big(    V_\phi(x^{(t, i+1)}) - V _\phi(  x^{(t,i)})\Big) \Big)      \Big]\\
     &= V_\theta(  x^{(1, 1)})  + \E_{\pi_\theta} \left[ \sum\limits_{t=1}^{H }   \left(\sum\limits_{y\in \X} \mathcal{P}( x^{(t,I_t)},a^{(t,I_t)}, y) V_\phi(y) - V _\phi(  x^{(t,I_t)}) +
   \sum\limits_{i=1}^{I_t-1} \Big(    V_\phi(x^{(t, i+1)}) - V _\phi(  x^{(t,i)})\Big) \right)      \right]\\
     &= \E_{\pi_\theta} \left[
 \sum\limits_{t=1}^{H}
   \sum\limits_{i=1}^{I_t}  g(s^{(t,i)},a^{(t,i)})    \right] \\
&\quad\quad\quad\quad+\E_{\pi_\theta} \left[ \sum\limits_{t=1}^{H }   \left(\sum\limits_{y\in \X} \mathcal{P}( x^{(t,I_t)},a^{(t,I_t)}, y) V_\phi(y) - V _\phi(  x^{(t,I_t)}) +
   \sum\limits_{i=1}^{I_t-1} \Big(    V_\phi(x^{(t, i+1)}) - V _\phi(  x^{(t,i)})\Big) \right)      \right]\\
   &=\E_{\pi_\theta} \Big[
 \sum\limits_{t=1}^{H }   \Big(g(x^{(t,I_t)},a^{(t,I_t)}) + \sum\limits_{y\in \X}\mathcal{P}( s^{(t,I_t)},a^{(t,I_t)}, y) V_\phi(y) - V _\phi( x^{(t,I_t)}) \\
&\quad\quad\quad\quad+  \sum\limits_{i=1}^{I_t-1} \Big(g(x^{(t,i)},a^{(t,i)}) +   V_\phi(x^{(t, i+1)}) - V _\phi( x^{(t,i)})\Big) \Big)      \Big]\\
     &= \E_{\pi_\theta} \left[\sum\limits_{t=1}^{H} \sum\limits_{i=1}^{I_t} A_\phi (x^{(t,i)}, a^{(t,i)}) \right].
\end{align*}

\end{proof}

\begin{proof}[\textbf{Proof of Theorem \ref{thm:pib_finite}}]

We recall  that we defined an occupation measure   of policy $\pi_\theta^\Sigma$ at epoch $t$ as a distribution over states  of  $\X^\Sigma$ as
\begin{align*}
   \mu_\theta(t, x) := \Prob(x^{(t)} = x),\quad \text{ for each }t=1,\dotsc, H,~x\in \X^\Sigma,
\end{align*}
 where $x^{(t) }$ is a state of the  MDP at epoch $t$   under policy $\pi^\Sigma_\theta$.
We also defined  the  probability that starting at state $x^{(t,1)}=x$ at epoch $t$  the SDM process is at state $y$  after  $i-1$ steps under policy $\pi_\theta$ as
\begin{align*}
   \xi_\theta(t,i, x, y )  := \Prob(x^{(t,i)} = y~|~x^{(t,1)} = x),
\end{align*}
  for each  $t=1,\dotsc,H,$ $i=1,\dotsc,I_t,$ $y\in \X,$ $x\in \X^\Sigma$.

First, starting from the result of  Lemma \ref{lem:per_diff} we get the policy improvement bound for the original MDP:
\begin{align}\label{eq:pib_finite1}
    V_\theta( x^{(1,1)}) - V_\phi( x^{(1, 1)})   & \geq \sum\limits_{t=1}^{H }\sum\limits_{x\in  \X^\Sigma} \mu_\phi(t, x)  \sum\limits_{a\in \A^\Sigma} \pi^\Sigma_\theta(a| x) A^\Sigma_\phi (x, a) \\
     &\quad\quad-\max\limits_{x\in \X^\Sigma,~a\in  \A^\Sigma}\left| A_\phi^\Sigma (x, a)\right|\sum\limits_{x\in \X^\Sigma} \left|\sum\limits_{t=1}^{H}\mu_\phi(t, x) - \mu_\theta(t,x) \right|\nonumber .
\end{align}

Indeed,
\begin{align*}
    V_\theta( x^{(1,1)}) - V_\phi( x^{(1, 1)})   & =  \E_{\pi_\theta} \left[\sum\limits_{t=1}^{H} \sum\limits_{i=1}^{I_t} A_\phi (x^{(t,i)}, a^{(t,i)}) \right] \nonumber  \\
    & =  \E_{\pi_\theta} \left[\sum\limits_{t=1}^{H}   A_\phi^\Sigma (x^{(t) }, a^{(t)}) \right]\nonumber   \\
     &= \sum\limits_{t=1}^{H }\sum\limits_{x\in  \X^\Sigma } \mu_\theta(t, x)  \sum\limits_{a\in \A^\Sigma} \pi^\Sigma_\theta(a| x) A^\Sigma_\phi (x, a)  \nonumber\\
     & \geq \sum\limits_{t=1}^{H }\sum\limits_{x\in  \X^\Sigma} \mu_\phi(t, x)  \sum\limits_{a\in \A^\Sigma} \pi^\Sigma_\theta(a| x) A^\Sigma_\phi (x, a) \\
     &\quad\quad-\max\limits_{x\in \X^\Sigma,~a\in  \A^\Sigma}\left| A_\phi^\Sigma (x, a)\right|\sum\limits_{x\in \X^\Sigma} \left|\sum\limits_{t=1}^{H}\mu_\phi(t, x) - \mu_\theta(t,x) \right|\nonumber .
\end{align*}

Next, we find the policy improvement bound for the policies that define atomic actions. A sampled composed action $a^{(t)}\sim \pi^\Sigma$ defines the path of the SDM process.  The goal is to replace values related to the original MDP (i.e. $\pi^\Sigma$, $A^\Sigma$) by the values related to the SDM process (i.e. $\pi$, $A$). First, we reformulate $  \sum\limits_{a\in \A^\Sigma} \pi^\Sigma_\theta(a| x) A^\Sigma_\phi (x, a)$. Although, this expression does not directly depend on $t$, we abuse the notation and use superscript $t$ to distinguish   states and actions $(x^{(t, i)},  a^{(t, i)} )$ of the SDM process from states and actions  of the original MDP.
\begin{align}\label{eq:pib_finite2}
\sum\limits_{a^{(t)}\in \A^\Sigma} \pi^\Sigma_\theta(a| x) A^\Sigma_\phi (x, a) &=  \underset{a \sim \pi_\theta^\Sigma(\cdot|x )}{\E}\left[A_\phi^\Sigma(x , a )\right] \nonumber\\
&= \underset{(x^{(t, i)}, a^{(t, i)})\sim \pi_\theta^\Sigma(\cdot|x )}{\E}\left[\sum\limits_{i=1}^{I_t}A_\phi(x^{(t, i)}, a^{(t, i)})\right] \nonumber\\
&=\sum\limits_{i=1}^{I_t}\underset{(x^{(t, i)}, a^{(t, i)})\sim \pi_\theta^\Sigma(\cdot|x^{(t )})}{\E}\left[ A_\phi(x^{(t, i)}, a^{(t, i)})\right] \nonumber\\
&=\sum\limits_{i=1}^{I_t}\underset{\substack{x^{(t, i)}\sim \xi_\theta(t, i, x , y) \\ a^{(t, i)}\sim \pi_\theta(\cdot|x^{(t, i)})}}{\E}\left[ A_\phi(x^{(t, i)}, a^{(t, i)})\right] \nonumber\\
&=\sum\limits_{i=1}^{I_t} \sum\limits_{y\in \X} \xi_\theta(t, i, x , y) \sum\limits_{a^{(t, i)}\in \A}\pi_\theta(a^{(t, i)}|y)  A_\phi(y, a^{(t, i)}),
\end{align}
where in the second equality we started to use superscript $t$ to specify states and actions of the SDM process: action $a$ is decomposed as $a = (a^{(t, 1)},\dotsc,a^{(t, I_t)}  )$,  state $x^{(t, i)}$ denotes a state of the SDM process   after $i-1$ steps, the SDM process starts at state $x$.

 We combine (\ref{eq:pib_finite1}), (\ref{eq:pib_finite2}) and obtain a policy improvement bound for an MDP with a SDM process:
\begin{align*}
   V_\theta( x^{(1,1)}) - V_\phi( x^{(1, 1)})   & \geq  \sum\limits_{t=1}^{H }\sum\limits_{x\in  \X^\Sigma} \mu_\phi(t, x)  \sum\limits_{a\in \A^\Sigma} \pi^\Sigma_\theta(a| x) A^\Sigma_\phi (x, a)  \\
     &\quad\quad-\max\limits_{x\in \X^\Sigma,~a\in  \A^\Sigma}\left| A_\phi^\Sigma (x, a)\right|\sum\limits_{x\in \X^\Sigma} \left|\sum\limits_{t=1}^{H}\mu_\phi(t, x) - \mu_\theta(t,x) \right|\\
&\geq  \sum\limits_{t=1}^{H }\sum\limits_{x\in  \X^\Sigma} \mu_\phi(t, x)  \sum\limits_{i=1}^{I_t} \sum\limits_{y\in \X}\xi_\theta(t, i, x, y) \sum\limits_{a^{(t,i)}\in \A} \pi_\theta(a^{(t,i)}| y) A_\phi (y, a^{(t,i)})  \\
     &\quad\quad-\max\limits_{s\in \X^\Sigma,~a\in  \A^\Sigma}\left| A_\phi^\Sigma (x, a)\right|\sum\limits_{t=1}^{H}\sum\limits_{x \in \X^\Sigma} |\mu_\phi(t, x) - \mu_\theta(t,x) |   \\
     & \geq \sum\limits_{t=1}^{H }\sum\limits_{x\in  \X^\Sigma} \mu_\phi(t, x)  \sum\limits_{i=1}^{I_t} \sum\limits_{y\in \X}\xi_\phi(t, i, x, y) \sum\limits_{a^{(t,i)}\in \A} \pi_\theta(a^{(t,i)}| y) A_\phi (y, a^{(t,i)})  \\
     &\quad\quad- \max\limits_{x\in \X,a\in \A} |A_\phi (x,a)|\sum\limits_{t=1}^{H }\sum\limits_{x \in  \X^\Sigma} \mu_\phi(t, x)  \sum\limits_{i=1}^{I_t}\sum\limits_{y\in \X}|\xi_\phi(t, i, x, y) - \xi_\theta(t, i, x, y)|  \\
     &\quad\quad-\max\limits_{x\in \X^\Sigma,~a\in  \A^\Sigma}| A_\phi^\Sigma (x, a)|\sum\limits_{t=1}^{H}\sum\limits_{x\in \X^\Sigma} |\mu_\phi(t, x) - \mu_\theta(t,x) |   \\
     & =\E_{x^{(t, i)}\sim \pi_\phi} \left[ \sum\limits_{t=1}^{H}\sum\limits_{i=1}^{I_t}\frac{\pi_\theta(a^{(t,i)}|x^{(t,i)})}{\pi_\phi(a^{(t,i)}|x^{(t,i)})} A_\phi(x^{(t,i)}, a^{(t,i)})  \right]  \\
     &\quad\quad- \max\limits_{x\in \X,a\in \A} |A_\phi (x,a)|\sum\limits_{t=1}^{H }\sum\limits_{x\in  \X^\Sigma} \mu_\phi(t, x )  \sum\limits_{i=1}^{I_t}\sum\limits_{y\in \X}|\xi_\phi(t, i, x, y) - \xi_\theta(t, i, x, y)|  \\
     &\quad\quad-\max\limits_{x\in\X^\Sigma,~a\in  \A^\Sigma}| A_\phi^\Sigma (x, a)|\sum\limits_{t=1}^{H}\sum\limits_{x\in \X^\Sigma} |\mu_\phi(t, x) - \mu_\theta(t,x) | .
 \end{align*}

\end{proof}

\section{Neural network structure }\label{appendix:ride_hailing_nn}

In this section we focus on  the architecture of the policy neural network (NN) used for atomic actions sampling in the SDM process.
 The value NN  has    identical architecture  except
the output layer. The
output layer of the value NN contains one unit with a linear activation function.

 We use $\theta$ to denote the vector of weights and
  biases of the neural network.
 For a fixed parameter $\theta$, the
  neural network outputs deterministically distribution
  $\pi_\theta(\cdot|x)$ over the atomic action space (trip types) for each state
  $x\in \X$. 
We consider a ride-hailing transportation network with $R$ regions,  patience time $ L$, and length of a working day $H$.

We start with a description of the input layer.
We recall that each  state $x^{(t,i)}$ of the SDM process has four components $x^{(t,i)} =\left [x^{(t,i)}_e, x^{(t, i)}_c, x^{(t,i)}_p, x^{(t, i)}_\ell\right]$, where  the first three components $x^{(t,i)}_e$, $x^{(t,i)}_c$, $x^{(t, i)}_p$  represent current epoch,  cars status, and passengers status, respectively, and  component $x^{(t,i)}_\ell$ tracks the cars exiting the available cars pool until the next decision epoch.  Based on a system state, the policy NN generates a sampling probability distribution over atomic actions. Next, we discuss how we encode each state component as an input to the NN.

Component $x^{(t,i)}_e$ is a categorical variable that takes   integer values in range $1, \dotsc, H$. We use entity embedding, see \cite{Guo2016},  to encode this component as a low-dimensional vector. First, we apply one-hot embedding to represent a value of $x^{(t, i)}_e$ as a vector in $\R^H$. Namely, we map    $x^{(t,i)}_e = t$ into   vector $ x^{(t)}_{\text{one-hot}} = (0, \dotsc,0,  1, 0, \dotsc, 0)^T$, where the $t$-th element of vector   $  x^{(t)}_{\text{one-hot}}$ is equal to $1$ and the rest elements are equal to $0$.

We define an \textit{embedding matrix} $E$ as an $H\times B$ matrix that is a transformation from a set of one-hot vectors into a continuous vector space with dimensionality $B$. Matrix $E$ maps each one-hot embedded vector $  x^{(t)}_{\text{one-hot}}$ into  a vector $y^{(t)}$ of size $B$:
\begin{align*}
y_e^{(t)} = E  x^{(t)}_{\text{one-hot}},
\end{align*}
where $y_e^{(t)}$ is a part of the input vector of the NN. 

Elements of matrix $E$ are training parameters of the  NN and are included in $\theta$. In other words, the  embedding matrix $E$ is not given but learned during the NN parameters optimization. 

Next, the cars status component $x^{(t, i)}_c$ is represented by a vector that is divided into $R$ parts, one for  each region. The $r$-th part of  $x^{(t, i)}_c$ contains $\tau^{\text{max}}_r+L$ elements, where $\tau^{\text{max}}_r$ is the maximum time (in minutes) that is required for a driver to reach region $r$ from any location of the transportation network. The $k$-th element of the $r$-th part of  $x^{(t, i)}_c$ counts the number of cars that have final destination in region $r$ and that are  $k$ minutes away from it.  
 
The passengers status vector $x^{(t, i)}_p$ has $R^2$ elements. Each element of $x^{(t, i)}_p$ corresponds to one of the trip types $(o, d)$ and counts the number of passengers that want to get a ride from region $o$ to region $d$, where $o, d=1, \dotsc, R$. Component $x^{(t,i)}_\ell$, that tracks the cars exiting the available cars pool, is represented by an  $R(L+1)$-dimensional vector. Element $(x^{(t,i)}_\ell)_{r, k}$ of this vector counts the number of cars which final destination or current location is region $r$ and which are $k$ minutes away from their destination, where $r=1,\dotsc,R $, $k=0, 1,\dotsc, L$.

After standard normalization, vector $x^{(t, i)}_{\text{input}} =\left [y_e^{(t)}, x^{(t, i)}_c, x^{(t,i)}_p, x^{(t, i)}_\ell\right]$ is   used as an input to the NN. The input layer has $K:= B+\sum\limits_{r=1}^R(\tau^{\text{max}}_r+L)+R^2+R(L+1)$ units, one for each element of the input vector $x^{(t, i)}_{\text{input}} $.

The input layer has $K$ units, the first hidden layer has $K$ units, the third hidden layer has $5$ units, where $5$ is a fixed  number. The number of units in the  second hidden layer is a geometric mean of units in the first and third hidden layers (i.e. integer part of $\sqrt{5K}$). The output layer of the policy NN has $R^2$ units,  one for each trip type. The softmax function is used as the activation function in the output layer of the policy neural network. The policy NN is a feed-forward policy NN and its layers are joined accordantly, see Appendix   \ref{sec:nn} for details.

 For the nine-region transportation network, the input layer is of size $527$.
It consists of the cars status component $x^{(t, i)}_c$ ($388$ entries), the passengers status component $x^{(t,i)}_p$ ($81$ entries), the ``do nothing'' cars component $x^{(t, i)}_\ell$ ($54$ entries). The time-of-day component $x^{(t,i)}_e $ is a  categorical  variable taking one of $H=240$ values and  it is additionally processed into an embedding layer $y_e^{(t)}$ of size $B = 4$. The first, second, and third hidden layers are of size $527$, $51$, $5$, respectively. The output layer of the policy NN has $81$ units.

 \section{Implementation details of numerical experiments in Section \ref{sec:num-study} }\label{appendix:ride_hailing}

We use Tensorflow v1.13.1  \cite{Abadi2016} to build a training routine of the neural networks and Ray package v0.6.6 \cite{Moritz2018}
    to maintain parallel simulation of the actors. We run all experiments on a   2.7 GHz  96-core processor with 1510 GB of RAM.

    We optimize  the value and policy functions to minimize the corresponding loss functions (\ref{eq:Vappr}), (\ref{eq:popt}) by the Adaptive Moment Estimation (Adam) method \cite{Kingma2017}, see the details in Appendix \ref{sec:par}.

 Table \ref{tab:dNN_hyperparams} summarizes the hyper-parameters of Algorithm \ref{alg:ppo} used in the nine-region experiment in Section \ref{sec:num-study}.

\begin{table}[H]
\centering%
\begin{tabular}{l|@{\quad}l}
  \hline
  Parameter  & Value\\\hline
   Number of policy iterations  $(J)$ & $150$ \\
  Number of episodes per policy iteration $(K)$ &250\\
   No. of  actors $(Q)$  & 50 \\
  Adam parameters for policy NN& $\beta_1 = 0.9$, $\beta_2 = 0.999$,    \\
(see parameter descriptions in  Appendix \ref{sec:par})  &   $\varrho = 10^{-8},$  $\varsigma=5\cdot 10^{-5}\times \max[ \alpha,0.05]  $   \\
  Adam parameters for value NN& $\beta_1 = 0.9$, $\beta_2 = 0.999$,  \\
 &   $\varrho = 10^{-8},$  $\varsigma=2.5\cdot 10^{-4}  $ \\
 Clipping parameter  $(\epsilon)$ & $0.2\times \max[ \alpha,0.01]$ \\
  No. of epochs  for policy NN update& 3\\
(passes over training data for policy NN update) & \\
  No. of epochs  for value NN update& 10\\
(passes over training data for value NN update) & \\
  Minibatch size in Adam method & 4096
\end{tabular}
 \caption[Algorithm \ref{alg:ppo} hyperparameters used for experiments in Section \ref{sec:num-study}.]{Algorithm \ref{alg:ppo} hyperparameters used for experiments in Section \ref{sec:num-study}. Parameter $\alpha$ decreases linearly  from $1$ to $0$ over the course of learning: $\alpha =(J - j)/J$ on the $j$th policy iteration, $j=0, 1, \dotsc,J. $  }
    \label{tab:dNN_hyperparams}
\end{table}

Table \ref{tab:mdp-params9} summarizes the values for the parameters used in the experiment with the  nine-region transportation network. The values for the traffic   parameters $ \lambda,$ $P$, and $\tau$ for  the nine-region transportation network  are the same as in  \cite[Appendix EC.3.1]{Braverman2019}.

\begin{table}[H]
\centering%
\begin{tabular}{l|@{\quad}l}
  \hline
  Parameter  & Value\\\hline
   Number of regions  $(R)$ & $9$ \\
  Number of cars $(N)$ &2,000\\
   Length of a working day $(H)$  & $240$ (minutes)\\
  Passenger patience time  ($L$)& $5$ (minutes),    \\
 Immediate rewards for a car-passenger matching ($g_f^{(t)}$) &  $g^{(t)}_f(o,d,\eta)\equiv 1$,  \\
Immediate rewards for an empty-car routing ($g_e^{(t)}$) & $g_e^{(t)}(o,d)\equiv 0$.
\end{tabular}
\caption{The nine-region transportation network configuration.}\label{tab:mdp-params9}
\end{table}

In Table \ref{tab:mdp-params5} we summarize  the  five-region transportation network configuration details used in the experiment in \cite{Feng2020}. The values for the traffic parameters $ \lambda,$ $P$, and $\tau$ for  the five-region transportation network  can be found in \cite[Appendix C.2]{Feng2020} or \cite[Appendix EC.3.2.]{Braverman2019}.

\begin{table}[H]
\centering%
\begin{tabular}{l|@{\quad}l}
  \hline
  Parameter  & Value\\\hline
   Number of regions  $(R)$ & $5$ \\
  Number of cars $(N)$ &1,000\\
   Length of a working day $(H)$  & $360$ (minutes)\\
  Passenger patience time  ($L$)& $5$ (minutes),    \\
 Immediate rewards for a car-passenger matching ($g_f^{(t)}$) &  $g^{(t)}_f(o,d,\eta)\equiv 1$,  \\
Immediate rewards for an empty-car routing ($g_e^{(t)}$) & $g_e^{(t)}(o,d)\equiv 0$.
\end{tabular}
\caption{The five-region transportation network configuration.}\label{tab:mdp-params5}
\end{table}

\chapter{Chapter 3 of Appendix}

 \section{Proofs of the theorems in Section \ref{sec:PIB_f} }\label{appendix:bounds}

\begin{proof}[\textbf{Proof of Lemma \ref{lem:stat}}]

The discounted future state distribution can be expressed in a vector form as:
  \begin{align}\label{eq:disc_distr2}
(d^{(\gamma)}_{\pi})^T&= (1-\gamma)\mu^T\sum\limits_{t=0}^\infty (\gamma P_\pi)^t \nonumber\\
&=(1-\gamma)\mu^T(I-\gamma P_\pi)^{-1}.
\end{align}

If $\gamma<1$,  matrix $P^{(\gamma)}_\pi$  is irreducible and aperiodic regardless of $P_\pi$, see \cite{Langville2003}.

 Now, we need to show that $(d^{(\gamma)}_{\pi})^TP_\pi^{(\gamma)}=(d^{(\gamma)}_{\pi})^T$. Indeed, using (\ref{eq:disc_distr2}) and (\ref{eq:mod_trans}), we get
 \begin{align*}
(d^{(\gamma)}_{\pi})^TP_\pi^{(\gamma)}&= (1-\gamma)\mu^T\sum\limits_{t=0}^\infty (\gamma P_\pi)^tP_\pi^{(\gamma)}\\
&=(1-\gamma)\mu^T\sum\limits_{t=0}^\infty (\gamma P_\pi)^t\Big(  \gamma P_\pi  + (1 - \gamma) e\mu^T \Big)\\
&=(1-\gamma)\mu^T\sum\limits_{t=0}^\infty (\gamma P_\pi)^{t+1}+(1-\gamma)^2\mu^T \sum\limits_{t=0}^\infty \gamma^t e\mu^T\\
& = (1-\gamma)\mu^T\sum\limits_{t=0}^\infty (\gamma P_\pi)^{t+1}+(1-\gamma)\mu^T\\
&=(1-\gamma)\mu^T \left(\sum\limits_{t=0}^\infty (\gamma P_\pi)^{t+1}+I \right)\\
&=(1-\gamma)\mu^T \sum\limits_{t=0}^\infty (\gamma P_\pi)^{t}\\
&= (d^{(\gamma)}_{\pi})^T.
\end{align*}
\end{proof}

  \begin{proof}[\textbf{Proof of Lemma \ref{lem:perf_bound}}]

The performance difference identity, proposed in \cite{Cao1999, Kakade2002}, allows to express the difference in performance between two policies $\pi_1$ and $\pi_2$ as
\begin{equation}\label{eq:perf_diff3}
 \eta^{(\gamma)}_{\pi_2}(\mu) -  \eta^{(\gamma)}_{\pi_1}(\mu) = \underset{\substack{ x\sim d^{(\gamma)}_{\pi_2}\\  a \sim \pi_2(\cdot|x) }  }{\E} \left[A^{(\gamma)}_{\pi_1}(x,a)   \right].
\end{equation}

Starting from the performance difference identity (\ref{eq:perf_diff3}) we get
\begin{align*}
\eta^{(\gamma)}_{\pi_2}(\mu) -  \eta^{(\gamma)}_{\pi_1}(\mu) &=   \underset{\substack{ x\sim d^{(\gamma)}_{\pi_2}\\  a \sim \pi_2(\cdot|x) }  }{\E} \left[A^{(\gamma)}_{\pi_1}(x,a)   \right]\\
&=  \underset{\substack{ x\sim d^{(\gamma)}_{\pi_1}\\  a \sim \pi_2(\cdot|x) }  }{\E} \left[A^{(\gamma)}_{\pi_1}(x,a)   \right] + \sum\limits_{x\in \X}\left(d^{(\gamma)}_{\pi_2}(x)-d^{(\gamma)}_{\pi_1}(x)\right) \underset{\substack{ a \sim \pi_2(\cdot|x) }  }{\E} \left[A^{(\gamma)}_{\pi_1}(x,a)   \right]\\
&\leq \underset{\substack{ x\sim d^{(\gamma)}_{\pi_1}\\  a \sim \pi_2(\cdot|x) }  }{\E} \left[A^{(\gamma)}_{\pi_1}(x,a)   \right] +\max\limits_{x\in \X}\Big[\underset{a\sim \pi_2(\cdot|x)}{\E}[A^{(\gamma)}_{\pi_1}(x,a)]\Big]  \left\|d^{(\gamma)}_{\pi_2}-d^{(\gamma)}_{\pi_1}\right\|_1.
\end{align*}

Then in \cite[Lemma 1]{Achiam2017}, the following  perturbation identity was derived for discounted stationary distributions:
\begin{align}\label{eq:old_id}
(d^{(\gamma)}_{\pi_2})^T-(d^{(\gamma)}_{\pi_1})^T =\gamma (d_\gamma^{   \pi_1})^T (P_{  \pi_1} - P_{ \pi_2})(I-\gamma P_{\pi_2})^{-1}.
\end{align}

\cite[Lemma 1]{Achiam2017} finalized the proof of \cite[Corollary 1]{Achiam2017} showing that
\begin{align}\label{eq:inv}
\left\|(I-\gamma P_{\pi_2})^{-1}\right\|_\infty\leq \frac{1}{1-\gamma}
\end{align}
and
\begin{align*}
\left\| (P_{  \pi_1} - P_{ \pi_2})^Td^{(\gamma)}_{   \pi_1} \right\|_1\leq 2 \underset{  x\sim d^{(\gamma)}_{\pi_1} }{\E} \left[\text{TV}\Big(\pi_2(\cdot|x)~||~\pi_1(\cdot|x)\Big) \right],
\end{align*}
which combined result in
\begin{align*}
\left\|d^{(\gamma)}_{\pi_2}-d^{(\gamma)}_{\pi_1}\right\|_1 &=\gamma \left\| \left((I-\gamma P_{\pi_2})^{-1} \right)^T(P_{  \pi_1} - P_{  \pi_2})^Td^{(\gamma)}_{   \pi_1} \right\|_1\\
& \leq \frac{2\gamma}{(1-\gamma)} \underset{  x\sim d^{(\gamma)}_{\pi_1} }{\E} \left[\text{TV}\Big(\pi_2(\cdot|x)~||~\pi_1(\cdot|x)\Big) \right].
\end{align*}

 \end{proof}

  \begin{proof}[\textbf{Proof of Lemma \ref{lem:matr_diff}}]
 Using the definition of a discounted transition matrix (\ref{eq:mod_trans}), group inverse $D^{(\gamma)}_{\pi_1}$ can be written as
 \begin{align*}
D^{(\gamma)}_{\pi_1} + e(d^{(\gamma)}_{\pi})^T &=\left(I - P^{(\gamma)}_\pi + e(d^{(\gamma)}_{\pi})^T\right)^{-1}\\
&=\left(I - \gamma P_\pi - (1-\gamma)e\mu^T + e(d^{(\gamma)}_{\pi})^T\right)^{-1}.
\end{align*}

 Then
  \begin{align*}
D^{(\gamma)}_{\pi_1} + e(d^{(\gamma)}_{\pi})^T &=\left(I - \gamma P_\pi - (1-\gamma)e\mu^T + e(d^{(\gamma)}_{\pi})^T\right)^{-1}\\
& = \left(I - \gamma P_\pi \right)^{-1} - \frac{\left(I - \gamma P_\pi \right)^{-1}(- (1-\gamma)e\mu^T + e(d^{(\gamma)}_{\pi})^T)\left(I - \gamma P_\pi \right)^{-1}}{1+(-(1-\gamma)\mu + d^{(\gamma)}_{\pi})^T(I-\gamma P_\pi)^{-1}e}\\
& = \left(I - \gamma P_\pi \right)^{-1} - \frac{\left(I - \gamma P_\pi \right)^{-1}e(d^{(\gamma)}_{\pi})^T\left(- I + \left(I - \gamma P_\pi \right)^{-1}\right)}{1+(-(1-\gamma)\mu + d^{(\gamma)}_{\pi})^T(I-\gamma P_\pi)^{-1}e}\\
& = \left(I - \gamma P_\pi \right)^{-1} - \frac{e(d^{(\gamma)}_{\pi})^T\left(- I + \left(I - \gamma P_\pi \right)^{-1}\right)}{1-\gamma+(-(1-\gamma)\mu + d^{(\gamma)}_{\pi})^Te}\\
&= \left(I - \gamma P_\pi \right)^{-1} + e (d^{(\gamma)}_{\pi})^T\Big( I - \left(I - \gamma P_\pi \right)^{-1}\Big),
\end{align*}
 where the second equality is due to Sherman--Morrison formula \cite[Section 2.1.4]{Golub2013}, the third equality follows from (\ref{eq:disc_distr2}), and the fourth equality is held since
\begin{equation*}
(1-\gamma)(I-\gamma P_\pi)^{-1}e = e.
\end{equation*}
\end{proof}

\begin{proof}[\textbf{Proof of Lemma \ref{lem:group}}]

First, let us prove that for any  $\ell\in \Z_+$ the    discounted transition matrix $P^{(\gamma)}_\pi$   satisfies
   \begin{align}\label{eq:ind}
 (P^{(\gamma)}_\pi)^\ell& =  \left(\gamma P_\pi +(1-\gamma)e\mu^T \right)^\ell \widetilde  \geq \gamma^\ell \left(P_\pi\right)^\ell + (1-\gamma) e \mu^T,
  \end{align}
  where matrix inequality $A \widetilde  \geq B$ means $A(x, y)\geq B(x, y)$ for each $x, y\in \X$.

 We prove inequality (\ref{eq:ind}) by induction. The base case $\ell=1$ is obvious. We assume that (\ref{eq:ind}) holds for $\ell-1$. Then
 \begin{align*}
 \left(\gamma P_\pi +(1-\gamma)e\mu^T \right)^\ell &= \left(\gamma P_\pi +(1-\gamma)e\mu^T \right) \left(\gamma P_\pi +(1-\gamma)e\mu^T \right)^{\ell-1}  \\
 &\widetilde \geq \left(\gamma P_\pi +(1-\gamma)e\mu^T \right) \left(  \gamma^{\ell-1} P_\pi^{\ell-1} + (1-\gamma) e \mu^T  \right)\\
 &=\gamma^\ell P_\pi^{\ell} +  \gamma (1-\gamma) e\mu^T + (1-\gamma)\gamma^{\ell-1}e\mu^T P_\pi^{\ell-1} + (1-\gamma)^2e \mu^T\\
 &= \gamma^\ell P_\pi^{\ell} +  (1-\gamma) e\mu^T + (1-\gamma)\gamma^{\ell-1}e\mu^T P_\pi^{\ell-1} \\
 &\widetilde \geq\gamma^\ell P_\pi^{\ell} +  (1-\gamma) e\mu^T.
 \end{align*}

Hence, we have proved (\ref{eq:ind}).

By (\ref{eq:minor}) and (\ref{eq:ind}) the following inequalities hold for $P^{(\gamma)}_\pi$:
    \begin{align}\label{eq:minor_disc}
 (P^{(\gamma)}_\pi)^\ell(x,y) &\geq \gamma^\ell  P_\pi^{\ell}(x,y) +  (1-\gamma) \mu(y) \\
 & \geq \left(\gamma^\ell \delta^{(\mu)}_\pi  + (1-\gamma)\right)\mu(y), \nonumber
  \end{align}
 for each $x,y\in \X$.

We let $d^{(\gamma)}_{\pi}$ define the stationary distribution of $P^{(\gamma)}_\pi$. By (\ref{eq:minor_disc}) and \cite[Lemma 2]{Rosenthal2007}, we have the following convergence rate bound for transition matrix $ P^{(\gamma)}_\pi$:
\begin{align*}
 \max\limits_{x\in \X} \sum\limits_{y\in \X} \Big|(P^{(\gamma)}_\pi)^t(x, y) - d^{(\gamma)}_{\pi}(y) \Big | \leq 2 (\gamma -\gamma^\ell \delta^{(\mu)}_\pi)^{[ t/\ell]},
\end{align*}
 where $[x]$  is the greatest integer not exceeding $x$.

  We are ready to bound the norm of the group inverse matrix:
  \begin{align*}
\left\|D^{(\gamma)}_{\pi }\right\|_\infty   &= \left\| \sum\limits_{t=0}^\infty \left((P^{(\gamma)}_{  \pi})^t- e(D^{(\gamma)}_{\pi})^T\right)\right\|_\infty\\
 &\leq  \sum\limits_{t=0}^\infty\| (P^{(\gamma)}_{  \pi})^t- e(d^{(\gamma)}_{\pi })^T \|_\infty\\
 &\leq  2\sum\limits_{t=0}^\infty (\gamma -\gamma^\ell \delta^{(\mu)}_\pi)^{[ t/\ell]}\\
 &=2\ell \sum\limits_{t=0}^\infty (\gamma -\gamma^\ell \delta^{(\mu)}_\pi)^{t}\\
 & =  \frac{2\ell}{1 - (\gamma -\gamma^\ell \delta_\pi^{(\mu)})}.
 \end{align*}
\end{proof}

  \begin{proof}[\textbf{Proof of Lemma \ref{lem:cond_nums}}]
(a) The fact that $\tau_1[Z] = \tau_1[D]$ follows from (\ref{eq:DZ}) and (\ref{eq:prop}). From \cite[Theorem 4.4.7]{Kemeny1976a}
\begin{equation*}
Z = ee^T I_Z - M I_d,
\end{equation*}
where $I_Z$ is the diagonal matrix whose elements are diagonal elements of matrix $Z$.

The equality $\tau_1[Z] = \tau_1[M I_d]$ follows from (\ref{eq:prop}).

(b) The statement directly follows from definition (\ref{eq:erg_def}) and (a), see \cite[Lemma 4.1(f)]{Cho2001}.

(c) We note that each row of matrix $MI_d$ sums up to $\kappa$ as we noted in (\ref{eq:kemeny_def}), see \cite[Theorem 4.4.10]{Kemeny1976a}. By \cite[equation (6)]{Seneta1991}
\begin{equation*}
\tau_1[M I_d] = \kappa - \min\limits_{x, y\in \X} \sum\limits_{z\in \X} d(z) \min[M(x, z), M(y, z)].
\end{equation*}
Since $d(x), M(x, y)$ take nonnegative values for any $x, y\in \X$,
\begin{equation*}
\tau_1[D] = \tau_1[M I_d] \leq \kappa.
\end{equation*}
A different proof of the fact that $\tau_1[D] \leq \kappa$ can be found in \cite{Seneta1993}.
 \end{proof}

 \section{Proofs of the theorems in Section \ref{sec:PIB_V} }\label{appendix:bounds_V}

  \begin{proof}[\textbf{Proof of Theorem \ref{thm:main_V}}]

  Starting from the performance difference identity we get
\begin{align*}
\eta_{\pi_2} -  \eta_{\pi_1} &=   \underset{\substack{ x\sim d_{\pi_2}\\  a \sim \pi_2(\cdot|x) }  }{\E} \left[A_{\pi_1}(x,a)   \right]\\
&=  \underset{\substack{ x\sim d_{\pi_1}\\  a \sim  \pi_2(\cdot|x) }  }{\E} \left[A_{\pi_1}(x,a)   \right] + \sum\limits_{x\in \X}\left(d_{\pi_2}(x)-d_{\pi_1}(x)\right) \underset{\substack{ a \sim \pi_2(\cdot|x) }  }{\E} \left[A_{\pi_1}(x,a)   \right]\\
&\leq \underset{\substack{ x\sim d_{\pi_1}\\  a \sim  \pi_2(\cdot|x) }  }{\E} \left[A_{\pi_1}(x,a)   \right] + \sum\limits_{x\in \X}\left|d_{\pi_2}(x)-d_{\pi_1}(x)\right|\V(x) \sup\limits_{x\in \X} \frac{\underset{\substack{ a \sim \pi_2(\cdot|x) }  }{\E} \left[A_{\pi_1}(x,a)   \right]}{\V(x)}\\
&= \underset{\substack{ x\sim d_{\pi_1}\\  a \sim \pi_1(\cdot|x) }  }{\E} \left[\frac{\pi_2(a|x)}{\pi_1(a|x)}A_{\pi_1}(x,a)   \right] +\left\|\underset{a\sim  \pi_2(\cdot|x)}{\E}[A_{\pi_1}(x,a)]\right\|_{\infty, \V}  \|d_{\pi_2}-d_{\pi_1}\|_{1, \V}.
\end{align*}

Next,   we use   perturbation identity (\ref{eq:old_new}) with $\gamma=1$
\begin{align*}
d_{\pi_2}^T-d_{\pi_1}^T = d_{   \pi_1}^T (P_{  \pi_1} - P_{  \pi_2})D_{\pi_2}.
\end{align*}
 to get a new perturbation bound:
\begin{align}\label{eq:inv_V}
\|d_{\pi_2}-d_{\pi_1} \|_{1, \V} &= \left\| \left(D_{\pi_2}\right)^T(P_{  \pi_1} - P_{ \pi_2})^Td_{   \pi_1} \right\|_{1, \V}\nonumber  \\
&\leq \tau_{1, \V}^{ } \left[D_{\pi_2}  \right]\left\|(P_{  \pi_1} - P_{  \pi_2})^Td_{   \pi_1} \right\|_{1, \V}
\end{align}

Ergodicity coefficient $\tau_{1, \V}^{ } \left[D_{\pi_2}  \right]$ is finite because
\begin{align*}
\tau_{1, \V}^{ } \left[D_{\pi_2}  \right]&= \tau_{1, \V}^{ } \left[Z_{\pi_2}  \right] \leq \|Z_{\pi_2}\|_{\V} < \infty,
\end{align*}
where the first equality follows from (\ref{eq:prop_V}). Norm $\|Z_{\pi_2}\|_{\V}$ is finite since $P_{\pi_2}$ satisfies the drift condition for Lyapunov function $\V$.

We simplify $\left\|(P_{  \pi_1} - P_{  \pi_2})^Td_{   \pi_1} \right\|_{1, \V}$ term  in (\ref{eq:inv_V})  to make it more explicit:
\begin{align*}
\left\|(P_{  \pi_1} - P_{ \pi_2})^Td_{   \pi_1} \right\|_{1, \V}&=\sum\limits_{y\in \X}\V(y) \left| \sum\limits_{x\in \X}  (P_{\pi_1}(x, y) - P_{\pi_2}(x,y)) d_{\pi_1}(x)   \right|\\
&\leq\sum\limits_{y\in \X}\V(y)  \sum\limits_{x\in \X}  \left|P_{\pi_1}(x, y) - P_{\pi_2}(x,y) \right| d_{\pi_1}(x)  \\
&=\sum\limits_{y\in \X}\V(y)  \sum\limits_{x\in \X}  \left|\sum\limits_{a\in \A} \pi_1(a|x)P(y| x, a) - \sum\limits_{a\in \A}  \pi_2(a|x)P(y| x, a) \right| d_{\pi_1}(x) \\
&\leq\sum\limits_{y\in \X}\V(y)  \sum\limits_{x\in \X} \sum\limits_{a\in \A}  P(y| x, a)   \left| \pi_1(a|x) -  \pi_2(a|x)  \right| d_{\pi_1}(x) \\
&=\sum\limits_{x\in \X} \sum\limits_{a\in \A}  \left| \pi_1(a|x) - \pi_2(a|x)  \right| d_{\pi_1}(x)     \sum\limits_{y\in \X}\V(y)    P(y| x, a)    \\
&=\E_{x\sim d_{\pi_1}} \left[ \sum\limits_{a\in \A} \left| \pi_1(a|x) -  \pi_2(a|x)  \right| \E_{y\sim P(y|x, a)} \V(y)   \right]\\
&=\E_{x\sim d_{\pi_1}} \left[ \sum\limits_{a\in \A} \pi_1(a|x) \left|\frac{ \pi_2(a|x) }{\pi_1(a|x)}-1 \right| \E_{y\sim P(\cdot|x, a)} \V(y)   \right]\\
&=\underset{\substack{ x\sim d_{\pi_1}\\  a \sim  \pi_1(\cdot|x) \\ y\sim P(\cdot|x, a) }  }{\E}\left[ \left|\frac{ \pi_2(a|x) }{\pi_1(a|x)}-1 \right|   \V(y)\right].
\end{align*}

 \end{proof}

\begin{lemma}\label{lem:kartashov}

We consider a  $\V$-uniformly ergodic Markov chain with transition matrix $P$ on  countable state space $\X$ that satisfies the drift condition (\ref{eq:drift_rep}). We let $D$ be  a group inverse matrix of $I-P$, and let $d$ be the stationary distribution of $P$. 

We consider a nonnegative matrix  $\T:\X\times\X\rightarrow \R_+$ 
\begin{align*}
\T(x, y) = P(x, y) - \nu(x) \omega(y) \text{ for each }x,y\in \X,
\end{align*}   
where   $\nu:\X\rightarrow\R_+$ is a  nonnegative vector and   $\omega:\X\rightarrow [0,1]$ is a probability distribution such that   $\omega^Te=1$, $d^T\nu>0$, $\nu^T\omega>0$.

If there exists constant $\lambda<1$  such that 
$
\|\T\|_\V^{ } \leq \lambda,
$
 then 
\begin{itemize}
\item[(a)] the following identity holds
  \begin{align}\label{eq:a}
  D  = e d ^T \left( \left(d ^T\sum\limits_{n=0}^\infty   \mathcal{T} ^n e\right) I - \sum\limits_{n=0}^\infty   \mathcal{T}^n \right) +\sum\limits_{n=0}^\infty   \mathcal{T}^n\left(I-e d^T\right);
  \end{align}

\item[(b)]
the following bound holds
  \begin{align}\label{eq:b}
  \tau_{1, \V}[D ] \leq \frac{  1+ \|e\|_{\infty, \V} \|d \|_{1, \V}  }{1-\lambda}.
  \end{align}

\end{itemize}
\end{lemma}
\begin{proof}
Identity (\ref{eq:a}) was originally proved in  \cite{Kartashov1986}. The proposed formulation of item (a)  is similar to \cite[Proposition A.2]{Liu2012}.

We show (\ref{eq:b}) using identity (\ref{eq:a}) and  property (\ref{eq:prop_V}):
  \begin{align*}
  \tau_{1, \V}[D] &= \tau_{1, \V}\left[\sum\limits_{n=0}^\infty   \mathcal{T}^n\left(I-e d ^T\right)   \right]\\
                      &\leq \left\|\sum\limits_{n=0}^\infty   \mathcal{T}^n\left(I-e d^T\right) \right\|_\V \\
                      &\leq\sum\limits_{n=0}^\infty\left\|   \mathcal{T}^n\left(I-e d^T\right) \right\|_\V\\
                      &\leq \left\| I-e d^T  \right\|_\V\sum\limits_{n=0}^\infty\left\|   \mathcal{T}\right\|^n\\
                      &\leq \frac{  1+ \|e\|_{\infty, \V} \|d\|_{1, \V} }{1-\lambda}.
  \end{align*}
\end{proof}

 \begin{proof}[\textbf{Proof of Lemma \ref{lem:D_V}}]

  We define matrix $\T$ on state space $\X\times \X$ such that
  \begin{align*}
  \mathcal{T} (x, y) :=
  \begin{cases}
  P (x,y)\quad \text{ if }x\neq x^*,\\
  0, \quad \text{ if }x=x^*.
  \end{cases}
  \end{align*}

From (\ref{eq:drift_single}), 
 \begin{align*}
\sum\limits_{y\in \X}\mathcal{T}(x, y)\V(y) \leq  \begin{cases}
\varepsilon \V(x), \text{ if } x\neq x^*\\
0, \text{otherwise}.
\end{cases}
\end{align*}
Hence, $\|T \|^{ }_\V\leq \varepsilon$. Moreover, $\T$ is nonnegative and $\T  = P - \nu \omega^T$, where $\nu(x) =\I_{x = x^*}$ and $\omega(x) = P(x^*, x)$ for each $x\in \X$. 
  Then, from Lemma \ref{lem:kartashov},
  \begin{align*}
  \tau_{1, \V}[D]  \leq \frac{ 1+ \|e\|_{\infty, \V} \|d\|_{1, \V}  }{1-\varepsilon}.
  \end{align*}

  We multiply  both side of (\ref{eq:drift_single}) by $d $, we get
  \begin{align*}
  d  ^T P  \V \leq \varepsilon  d ^T  \V + b  d(x^*)
  \end{align*}

 Since  $ d ^T P  =  d ^T$, we get
  \begin{align*}
 \|d \|_{1, \V} \leq  \frac{b d(x^*)}{1-\varepsilon}.
 \end{align*}
  \end{proof}

\begin{proof}[\textbf{Proof of the drift inequality in Example \ref{exp:mm1}}]

We want to confirm that the Bernoulli random walk transition probabilities  satisfy the following drift condition with $\V(x) = \rho^{-x/2}$, $x\in \Z_+$, and $C=\{0\}$:
\begin{align*}
\sum\limits_{y\in \Z_+}P(x, y) \V(y)\leq \frac{2\sqrt{\rho}}{1+\rho} \V(x) +\frac{1-\sqrt{\rho}}{1+\rho} I_{x = 0}(x), \text{ for each }x\in \Z_+.
\end{align*}

First, we note that $\mu=1/(1+\rho)$ and $\lambda=\rho/(1+\rho)$. Hence,
\begin{itemize}
\item For $x=0$: 
\begin{align*}
\sum\limits_{y\in \Z_+}P(x, y)\V(y) - \frac{2\sqrt{\rho}}{1+\rho}\V(x)&= \frac{1}{1+\rho}+\frac{\rho}{(1+\rho)\sqrt{\rho}} - \frac{2\sqrt{\rho}}{1+\rho}\\
&=\frac{1-\sqrt{\rho}}{1+\rho}.
\end{align*}

\item For $x\geq 1$:
\begin{align*}
\sum\limits_{y\in \Z_+}P(x, y)\V(y) - \frac{2\sqrt{\rho}}{1+\rho}\V(x)&=\frac{1}{1+\rho}\rho^{-x/2+1/2}+\frac{\rho}{1+\rho}\rho^{-x/2-1/2} - \frac{2\sqrt{\rho}}{1+\rho}\rho^{-x/2}\\
&=\frac{\rho^{-x/2}}{1+\rho}\left(\sqrt{\rho} + \sqrt{\rho} - 2 \sqrt{\rho} \right) = 0.
\end{align*}
\end{itemize}
\end{proof}

  \begin{proof}[\textbf{Proof of Lemma \ref{lem:D_V2}}]

  We define a nonnegative matrix $\mathcal{T} $   as
  \begin{align*}
  \mathcal{T} (x, y) =
  \begin{cases}
  P(x,y) -   P (x^*,y)\quad \text{ if }x\in C,\\
  P(x,y), \quad \text{ otherwise}.
  \end{cases}
  \end{align*}
We note that $\T  = P - \nu \omega^T$, where $\nu(x) =\I_{C}(x)$ and $\omega(y) =   P(x^*, y)$, for each $x, y\in \X$.   For each $x\in \X$
 \begin{align*}
\sum\limits_{y\in \X}\mathcal{T}(x, y)\V(y) = & \sum\limits_{y\in \X} P(x,y) \V(y)- \I_C(x) \sum\limits_{y\in \X} P(x^*, y)\V(y)\\
&\leq  \varepsilon \V(x) +\left(b -  \sum\limits_{y\in \X} P(x^*,y)\V(y)\right)\I_C(x)\\
&\leq\varepsilon \V(x),
\end{align*}
where the first inequality follows from the drift condition (\ref{eq:drift_rep}) and the second inequality follows from (\ref{eq:as1}). Hence, $\|\T\|_\V^{ }\leq\varepsilon $ and,
by Lemma \ref{lem:kartashov},
  \begin{align*}
  \tau_{1, \V}[D ] \leq \frac{\left( 1+ \|e\|_{\infty, \V} \|d \|_{1, \V} \right)}{1-\varepsilon}.
  \end{align*}

  We get a bound on $\|d \|_{1, \V}$ multiplying  both side of (\ref{eq:drift_rep}) by $d$:
  \begin{align*}
    d   ^T P  \V \leq \varepsilon d ^T  \V + b\sum\limits_{x\in C} d (x)
  \end{align*}

As a result, we obtain
  \begin{align*}
 \|d \|_{1, \V}&\leq \frac{b}{1-\varepsilon} \sum\limits_{x\in C}d(x).
  \end{align*}
  \end{proof}

  \begin{proof}[\textbf{Proof of Lemma \ref{lem:D_disc}}]

 \begin{align*}
\sum\limits_{y\in \X}P^{(\gamma)} (x, y) \V(y)&=\sum\limits_{y\in \X}\Big(\gamma P  (x, y)+(1-\gamma)\mu(y)\Big) \V(y)\\
&\leq \gamma \varepsilon\V(x) +\gamma b\I_C(x) +(1-\gamma)\mu^T\V\\
&=\frac{1}{2}(\gamma \varepsilon+1) \V(x) +\gamma b\I_C(x) - \frac{1-\varepsilon\gamma}{2}\V(x) +(1-\gamma)\mu^T\V\\
&\leq \frac{1}{2}(\gamma \varepsilon+1) \V(x) +\gamma b\I_C(x) +(1-\gamma)\mu^T\V\I_{\Omega}(x)\\
&\leq \frac{1}{2}(\gamma \varepsilon+1) \V(x) + \max[\gamma b, (1-\gamma)\mu^T\V] \I_{C\cup\Omega}(x),
\end{align*}
where $\Omega = \left\{ x\in \X:\V(x)<\frac{2(1-\gamma)\mu^T\V}{1-\gamma\varepsilon} \right\}$.

We note that the first inequality follows from the drift condition (\ref{eq:drift_rep}),  the second inequality follows from the fact that
\begin{align*}
- \frac{1-\varepsilon\gamma}{2}\V(x) +(1-\gamma)\mu^T\V\leq
\begin{cases}
0, \quad\text{if } x\not\in \Omega, \text{ i.e. } - \frac{1-\varepsilon\gamma}{2}\V(x) +(1-\gamma)\mu^T \V\leq 0,\\
(1-\gamma)\mu^T\V, \quad \text{otherwise.}
\end{cases}
\end{align*}

  \end{proof}

 \section{Proofs of the theorems in Section \ref{sec:PIB_SMDP} }\label{appendix:bounds_SMDP}

  \begin{proof}[\textbf{Proof of Lemma \ref{lem:perf_iden_SMDP}}]

We define $\tilde d_{\pi_2}:\X\rightarrow [0, 1]$ as the   stationary distribution of the "generator" $\frac{1}{m_{\pi_2}(x)}(P_{\pi_2}(y|x) - \I_{x=y})$:
\begin{align*}
\tilde d_{\pi_2}(x):=\frac{ d_{\pi_2}(x)m_{\pi_2}(x) }{ \sum\limits_{x\in\X}d_{\pi_2}(x)m_{\pi_2}(x) } = \frac{1}{\overline m_{\pi_2}}d_{\pi_2}(x)m_{\pi_2}(x).
\end{align*}

Probability distribution $\tilde d_{\pi_2}$ is called the the   stationary distribution of   "generator" $\frac{1}{m_{\pi_2}(x)}(P_{\pi_2}(y|x) - \I_{x=y})$ since $\tilde d_{\pi_2}^Te = 1$ and
\begin{align*}
\sum\limits_{x\in \X} \tilde d_{\pi_2}(x) \frac{1}{m_{\pi_2}(x)}(P_{\pi_2}(y|x) - \I_{x=y}) = 0, \text{ for each }y\in \X.
\end{align*}

We start with performance identity  \cite[equation (20)]{Cao2003}:
\begin{align}\label{eq:cao}
\eta_{\pi_2} - \eta_{\pi_1} = \E_{x\sim \tilde d_{\pi_2}}\Big[ \frac{g_{\pi_2}(x)}{m_{\pi_2}(x)} &-  \frac{g_{\pi_1}(x)}{m_{\pi_1}(x)}  +\frac{1}{m_{\pi_2}(x)}\sum\limits_{y\in \X}(P_{\pi_2}(y|x) - \I_{x=y} )h_{\pi_1}(y)\nonumber \\
&- \frac{1}{m_{\pi_1}(x)}\sum\limits_{y\in \X}(P_{\pi_1}(y|x) - \I_{x=y} )h_{\pi_1}(y) \Big].
\end{align}

Starting from (\ref{eq:cao}) we get the statement
\begin{align*}
\eta_{\pi_2} - \eta_{\pi_1} & = \sum\limits_{x\in \X} \frac{m_{\pi_2}(x)d_{\pi_2}(x)}{\overline m_{\pi_2}}\Big[ \frac{g_{\pi_2}(x)}{m_{\pi_2}(x)} -  \frac{g_{\pi_1}(x)}{m_{\pi_1}(x)}  +\frac{1}{m_{\pi_2}(x)}\sum\limits_{y\in \X}(P_{\pi_2}(y|x) - \I_{x=y} )h_{\pi_1}(y) \\
& \quad\quad  \quad\quad - \frac{1}{m_{\pi_1}(x)}\sum\limits_{y\in \X}(P_{\pi_1}(y|x) - \I_{x=y} )h_{\pi_1}(y) \Big],\\
&=\frac{1}{\overline m_{\pi_2}} \E_{x\sim  d_{\pi_2}}\Big[ g_{\pi_2}(x) + \sum\limits_{y\in \X}P_{\pi_2}(y|x) h_{\pi_1}(y) - h_{\pi_1}(x)\\
& \quad\quad \quad\quad  - \frac{m_{\pi_2}(x)}{m_{\pi_1}(x)}\Big( g_{\pi_1}(x) + \sum\limits_{y\in \X} P_{\pi_1}(y|x)h_{\pi_1}(y) - h_{\pi_1}(x) \Big) \Big]\\
&=\frac{1}{\overline m_{\pi_2}} \E_{x\sim  d_{\pi_2}}\left[ g_{\pi_2}(x) + \sum\limits_{y\in \X}P_{\pi_2}(y|x) h_{\pi_1}(y) - h_{\pi_1}(x) - m_{\pi_2}(x) \eta_{\pi_1} \right]\\
& = \frac{1}{\overline m_{\pi_2}} \E_{x\sim  d_{\pi_2}}\left[ \sum\limits_{a\in \A}\pi_2(a|x)  \left( g(x, a) + \sum\limits_{y\in \X}P(y|x, a) h_{\pi_1}(y) - h_{\pi_1}(x) - m(x, a) \eta_{\pi_1} \right) \right]\\
&= \frac{1}{\overline m_{\pi_2}} \E_{x\sim  d_{\pi_2}, a\sim \pi_2(\cdot|x)}\left[ A_{\pi_1}(x, a)\right].
\end{align*}

\end{proof}

  \begin{proof}[\textbf{Proof of Theorem \ref{thm:perf_bound_SMDP}}]
 Starting from the performance difference identity from Lemma \ref{lem:perf_iden_SMDP} we get
\begin{align*}
\eta_{\pi_2} -  \eta_{\pi_1} &= \frac{1}{\overline m_{\pi_2} }  \underset{\substack{ x\sim d_{\pi_2}\\  a \sim \pi_2(\cdot|x) }  }{\E} \left[A_{\pi_1}(x,a)   \right]\\
&=  \frac{1}{\overline m_{\pi_2} }  \underset{\substack{ x\sim d_{\pi_1}\\  a \sim  \pi_2(\cdot|x) }  }{\E} \left[A_{\pi_1}(x,a)   \right] + \frac{1}{\overline m_{\pi_2} }  \sum\limits_{x\in \X}\left(d_{\pi_2}(x)-d_{\pi_1}(x)\right) \underset{\substack{ a \sim \pi_2(\cdot|x) }  }{\E} \left[A_{\pi_1}(x,a)   \right]\\
&\leq \frac{1}{\overline m_{\pi_2} } \underset{\substack{ x\sim d_{\pi_1}\\  a \sim  \pi_2(\cdot|x) }  }{\E} \left[A_{\pi_1}(x,a)   \right] + \frac{1}{\overline m_{\pi_2} } \sum\limits_{x\in \X}\left|d_{\pi_2}(x)-d_{\pi_1}(x)\right|  \max\limits_{x\in \X}  \underset{\substack{ a \sim \pi_2(\cdot|x) }  }{\E} \left[A_{\pi_1}(x,a)   \right] \\
&= \frac{1}{\overline m_{\pi_2} }  \underset{\substack{ x\sim d_{\pi_1}\\  a \sim \pi_1(\cdot|x) }  }{\E} \left[\frac{\pi_2(a|x)}{\pi_1(a|x)}A_{\pi_1}(x,a)   \right] +\frac{1}{\overline m_{\pi_2} } \left\|\underset{a\sim  \pi_2(\cdot|x)}{\E}[A_{\pi_1}(x,a)]\right\|_{\infty }  \|d_{\pi_2}-d_{\pi_1}\|_{1 }.
\end{align*}

Next,   we use   perturbation identity (\ref{eq:old_new}) with $\gamma=1$
\begin{align*}
d_{\pi_2}^T-d_{\pi_1}^T = d_{   \pi_1}^T (P_{  \pi_1} - P_{  \pi_2})D_{\pi_2}.
\end{align*}
 to get the following perturbation bound:
\begin{align*}
\|d_{\pi_2}-d_{\pi_1} \|_{1 } & \leq \tau_{1 }^{ } \left[D_{\pi_2}  \right]\left\|(P_{  \pi_1} - P_{  \pi_2})^Td_{   \pi_1} \right\|_{1 }.
\end{align*}

Following the proof of Lemma \ref{lem:perf_bound} we simplify $\left\|(P_{  \pi_1} - P_{  \pi_2})^Td_{   \pi_1} \right\|_{1 }$ term to get
\begin{align*}
\left\| (P_{  \pi_1} - P_{ \pi_2})^Td_{   \pi_1} \right\|_1\leq 2 \underset{  x\sim d_{\pi_1} }{\E} \left[\text{TV}\Big(\pi_2(\cdot|x)~||~\pi_1(\cdot|x)\Big) \right].
\end{align*}

\end{proof}

\bibliography{Thesis, dai20200529}

\def\cprime{$'$} \def\cprime{$'$} \def\cprime{$'$} \def\cprime{$'$}
  \def\cprime{$'$} \def\cprime{$'$} \def\cprime{$'$}
\begin{thebibliography}{100}

\bibitem{Abadi2016}
Mart{\'{i}}n Abadi, Paul Barham, Jianmin Chen, Zhifeng Chen, Andy Davis,
  Jeffrey Dean, Matthieu Devin, Sanjay Ghemawat, Geoffrey Irving, Michael
  Isard, Manjunath Kudlur, Josh Levenberg, Rajat Monga, Sherry Moore, Derek~G.
  Murray, Benoit Steiner, Paul Tucker, Vijay Vasudevan, Pete Warden, Martin
  Wicke, Yuan Yu, and Xiaoqiang Zheng.
\newblock {TensorFlow: a system for large-scale machine learning}.
\newblock In {\em 12th USENIX Symposium on Operating Systems Design and
  Implementation (OSDI 16)}, pages 265--283. USENIX Association, 2016.
\newblock URL:
  \url{https://www.usenix.org/conference/osdi16/technical-sessions/presentation/abadi}.

\bibitem{Abbasi_Yadkori2014}
Yasin Abbasi-Yadkori, Peter Bartlett, and Alan Malek.
\newblock {Linear programming for large-scale Markov decision problems}.
\newblock In {\em Proceeding ICML'14 - Volume 32}, pages 496--504, 2014.
\newblock \href {https://doi.org/10.48550/arXiv.1402.6763}
  {\path{doi:10.48550/arXiv.1402.6763}}.

\bibitem{Abdolmaleki2018}
Abbas Abdolmaleki, Jost~Tobias Springenberg, Yuval Tassa, Remi Munos, Nicolas
  Heess, and Martin Riedmiller.
\newblock {Maximum a posteriori policy optimisation}.
\newblock In {\em Proceedings of ICLR'18}, 2018.
\newblock \href {http://arxiv.org/abs/1806.06920} {\path{arXiv:1806.06920}}.

\bibitem{Achiam2017}
Joshua Achiam, David Held, Aviv Tamar, and Pieter Abbeel.
\newblock {Constrained policy optimization}.
\newblock {\em Proceedings of ICML'17}, 70:22--31, 2017.
\newblock \href {http://arxiv.org/abs/1705.10528} {\path{arXiv:1705.10528}}.

\bibitem{Andradottir1993}
Sigr{\'{u}}n Andrad{\'{o}}ttir, Daniel~P. Heyman, and Teunis~J. Ott.
\newblock {Variance reduction through smoothing and control variates for Markov
  chain simulations}.
\newblock {\em ACM Transactions on Modeling and Computer Simulation (TOMACS)},
  3(3):167--189, 1993.
\newblock \href {https://doi.org/10.1145/174153.174154}
  {\path{doi:10.1145/174153.174154}}.

\bibitem{Angel2019}
Omer Angel and Mark Holmes.
\newblock {Kemeny's constant for infinite DTMCs is infinite}.
\newblock {\em Journal of Applied Probability}, 56(4):1269--1270, 2019.
\newblock \href {https://doi.org/10.1017/jpr.2019.64}
  {\path{doi:10.1017/jpr.2019.64}}.

\bibitem{Asmussen2003}
S{\o}ren Asmussen.
\newblock {\em {Applied Probability and Queues}}.
\newblock Springer, New York, 2003.
\newblock \href {https://doi.org/10.1007/b97236} {\path{doi:10.1007/b97236}}.

\bibitem{Ata2005}
Baris Ata and Sunil Kumar.
\newblock {Heavy traffic analysis of open processing networks with complete
  resource pooling: Asymptotic optimality of discrete review policies}.
\newblock {\em The Annals of Applied Probability}, 15(1A):331--391, 2005.
\newblock \href {https://doi.org/10.1214/105051604000000495}
  {\path{doi:10.1214/105051604000000495}}.

\bibitem{Avram1995}
Florin Avram, Dimitris Bertsimas, and Michael Ricard.
\newblock {Fluid models of sequencing problems in open queueing networks: an
  optimal control approach}.
\newblock In F.~Kelly and R.~J. Williams, editors, {\em Stochastic Networks},
  volume~71, page 237. Springer, New York, 1995.

\bibitem{Bauerle2001}
Nicole B{\"{a}}uerle.
\newblock {Asymptotic optimality of tracking policies in stochastic networks}.
\newblock {\em The Annals of Applied Probability}, 10(4):1065--1083, 2001.
\newblock \href {https://doi.org/10.1214/aoap/1019487606}
  {\path{doi:10.1214/aoap/1019487606}}.

\bibitem{Baxter2001}
Jonathan Baxter and Peter~L Bartlett.
\newblock {Infinite-horizon policy-gradient estimation}.
\newblock {\em Journal of Artificial Intelligence Research}, 15(1):319--350,
  2001.
\newblock \href {https://doi.org/10.1613/jair.806}
  {\path{doi:10.1613/jair.806}}.

\bibitem{Bell2001}
S.~L. Bell and R.~J. Williams.
\newblock {Dynamic scheduling of a system with two parallel servers in heavy
  traffic with resource pooling: asymptotic optimality of a threshold policy}.
\newblock {\em The Annals of Applied Probability}, 11(3):608--649, 2001.
\newblock \href {https://doi.org/10.1214/AOAP/1015345343}
  {\path{doi:10.1214/AOAP/1015345343}}.

\bibitem{Bellemare2013}
Marc~G. Bellemare, Yavar Naddaf, Joel Veness, and Michael Bowling.
\newblock {The arcade learning environment: an evaluation platform for general
  agents}.
\newblock {\em Journal of Artificial Intelligence Research}, 47(1):253--279,
  2013.
\newblock \href {https://doi.org/10.1613/jair.3912}
  {\path{doi:10.1613/jair.3912}}.

\bibitem{Bertsimas2011}
Dimitris Bertsimas, David Gamarnik, and Alexander {Anatoliy Rikun}.
\newblock {Performance analysis of queueing networks via robust optimization}.
\newblock {\em Operations Research}, 59(2):455--466, 2011.
\newblock \href {https://doi.org/10.1287/opre.1100.0879}
  {\path{doi:10.1287/opre.1100.0879}}.

\bibitem{Bertsimas2015}
Dimitris Bertsimas, Ebrahim Nasrabadi, and Ioannis~Ch. Paschalidis.
\newblock {Robust fluid processing networks}.
\newblock {\em IEEE Transactions on Automatic Control}, 60(3):715--728, 2015.
\newblock \href {https://doi.org/10.1109/TAC.2014.2352711}
  {\path{doi:10.1109/TAC.2014.2352711}}.

\bibitem{Bertsimas1994}
Dimitris Bertsimas, Ioannis~Ch. Paschalidis, and John~N Tsitsiklis.
\newblock {Optimization of multiclass queueing networks: polyhedral and
  nonlinear characterizations of achievable performance}.
\newblock {\em The Annals of Applied Probability}, 4(1):43--75, 1994.
\newblock \href {https://doi.org/10.1214/aoap/1177005200}
  {\path{doi:10.1214/aoap/1177005200}}.

\bibitem{Beutler1987}
Frederick~J. Beutler and Keith~W. Ross.
\newblock {Uniformization for semi-Markov decision processes under stationary
  policies}.
\newblock {\em Journal of Applied Probability}, 24(3):644--656, 1987.
\newblock \href {https://doi.org/10.2307/3214096} {\path{doi:10.2307/3214096}}.

\bibitem{Bhatnagar2012}
Shalabh Bhatnagar and K.~Lakshmanan.
\newblock {An online actor-critic algorithm with function approximation for
  constrained Markov decision processes}.
\newblock {\em Journal of Optimization Theory and Applications},
  153(3):688--708, 2012.
\newblock \href {https://doi.org/10.1007/s10957-012-9989-5}
  {\path{doi:10.1007/s10957-012-9989-5}}.

\bibitem{Bramson1996}
Maury Bramson.
\newblock {Convergence to equilibria for fluid models of head-of-the-line
  proportional processor sharing queueing networks}.
\newblock {\em Queueing Systems}, 23(1-4):1--26, 1996.
\newblock \href {https://doi.org/10.1007/bf01206549}
  {\path{doi:10.1007/bf01206549}}.

\bibitem{Bramson1998}
Maury Bramson.
\newblock {State space collapse with application to heavy traffic limits for
  multiclass queueing networks}.
\newblock {\em Queueing Systems}, 30(1-2):89--140, 1998.
\newblock \href {https://doi.org/10.1023/a:1019160803783}
  {\path{doi:10.1023/a:1019160803783}}.

\bibitem{Braverman2019}
Anton Braverman, J.~G. Dai, Xin Liu, and Lei Ying.
\newblock {Empty-car routing in ridesharing systems}.
\newblock {\em Operations Research}, 67(5):1437--1452, 2019.
\newblock \href {https://doi.org/10.1287/OPRE.2018.1822}
  {\path{doi:10.1287/OPRE.2018.1822}}.

\bibitem{Cao1999}
Xi~Ren Cao.
\newblock {Single sample path-based optimization of Markov chains}.
\newblock {\em Journal of Optimization Theory and Applications},
  100(3):527--548, 1999.
\newblock \href {https://doi.org/10.1023/A:1022634422482}
  {\path{doi:10.1023/A:1022634422482}}.

\bibitem{Cao2003}
Xi~Ren Cao.
\newblock {Semi-Markov decision problems and performance sensitivity analysis}.
\newblock {\em IEEE Transactions on Automatic Control}, 48(5):758--769, 2003.
\newblock \href {https://doi.org/10.1109/TAC.2003.811252}
  {\path{doi:10.1109/TAC.2003.811252}}.

\bibitem{Catral2010}
M.~Catral, S.~J. Kirkland, M.~Neumann, and N.~S. Sze.
\newblock {The Kemeny constant for finite homogeneous ergodic Markov chains}.
\newblock {\em Journal of Scientific Computing}, 45(1):151--166, 2010.
\newblock \href {https://doi.org/10.1007/S10915-010-9382-1}
  {\path{doi:10.1007/S10915-010-9382-1}}.

\bibitem{Chen1993}
Hong Chen and David~D. Yao.
\newblock {Dynamic scheduling of a multiclass fluid network}.
\newblock {\em Operations Research}, 41(6):1104--1115, 1993.
\newblock \href {https://doi.org/10.1287/OPRE.41.6.1104}
  {\path{doi:10.1287/OPRE.41.6.1104}}.

\bibitem{Chen1999}
Rong~Rong Chen and Sean Meyn.
\newblock {Value iteration and optimization of multiclass queueing networks}.
\newblock {\em Queueing Systems}, 32:65--97, 1999.
\newblock \href {https://doi.org/10.1023/A:1019182903300}
  {\path{doi:10.1023/A:1019182903300}}.

\bibitem{Chenetal2009}
W.~Chen, D.~Huang, A.~A. Kulkarni, J.~Unnikrishnan, Q.~Zhu, P.~Mehta, S.~Meyn,
  and A.~Wierman.
\newblock Approximate dynamic programming using fluid and diffusion
  approximations with applications to power management.
\newblock In {\em Proceedings of the 48h IEEE Conference on Decision and
  Control (CDC) held jointly with 2009 28th Chinese Control Conference}, pages
  3575--3580, Dec 2009.
\newblock \href {https://doi.org/10.1109/CDC.2009.5399685}
  {\path{doi:10.1109/CDC.2009.5399685}}.

\bibitem{Cho2001}
Grace~E. Cho and Carl~D. Meyer.
\newblock {Comparison of perturbation bounds for the stationary distribution of
  a Markov chain}.
\newblock {\em Linear Algebra and Its Applications}, 335(1-3):137--150, 2001.
\newblock \href {https://doi.org/10.1016/S0024-3795(01)00320-2}
  {\path{doi:10.1016/S0024-3795(01)00320-2}}.

\bibitem{Cinlar2013}
Erhan {\c{C}}inlar.
\newblock {\em {Introduction to Stochastic Processes}}.
\newblock Dover Publications, Mineola, NY, 2013.

\bibitem{Cooper2003}
William~L. Cooper, Shane~G. Henderson, and Mark~E. Lewis.
\newblock {Convergence of simulation-based policy iteration}.
\newblock {\em Probability in the Engineering and Informational Sciences},
  17(2):213--234, 2003.
\newblock \href {https://doi.org/10.1017/S0269964803172051}
  {\path{doi:10.1017/S0269964803172051}}.

\bibitem{Dai2021}
J.~G. Dai and Mark Gluzman.
\newblock {Refined policy improvement bounds for MDPs}.
\newblock In {\em Workshop on Reinforcement Learning Theory, ICML}, 2021.
\newblock URL:
  \url{https://lyang36.github.io/icml2021_rltheory/camera_ready/82.pdf}, \href
  {http://arxiv.org/abs/2107.08068} {\path{arXiv:2107.08068}}.

\bibitem{DaiGluzman2021}
J.~G. Dai and Mark Gluzman.
\newblock {Queueing network controls via deep reinforcement learning}.
\newblock {\em Stochastic Systems}, 12(1):30--67, 2022.
\newblock \href {https://doi.org/10.1287/STSY.2021.0081}
  {\path{doi:10.1287/STSY.2021.0081}}.

\bibitem{DaiHarrison2020}
J.~G. Dai and J.~Michael Harrison.
\newblock {\em {Processing Networks: Fluid Models and Stability}}.
\newblock Cambridge University Press, Cambridge, UK, 2020.
\newblock \href {https://doi.org/10.1017/9781108772662}
  {\path{doi:10.1017/9781108772662}}.

\bibitem{Dai2019a}
J.~G. Dai and Pengyi Shi.
\newblock {Inpatient overflow: an approximate dynamic programming approach}.
\newblock {\em Manufacturing \& Service Operations Management}, 21(4):894--911,
  2019.
\newblock \href {https://doi.org/10.1287/msom.2018.0730}
  {\path{doi:10.1287/msom.2018.0730}}.

\bibitem{Dai1996}
J.~G. Dai and G.~Weiss.
\newblock {Stability and instability of fluid models for reentrant lines}.
\newblock {\em Mathematics of Operations Research}, 21(1):115--134, 1996.
\newblock \href {https://doi.org/10.1287/moor.21.1.115}
  {\path{doi:10.1287/moor.21.1.115}}.

\bibitem{DeFarias2003a}
D.~P. de~Farias and B.~{Van Roy}.
\newblock {The linear programming approach to approximate dynamic programming}.
\newblock {\em Operations Research}, 51(6):850--865, 2003.
\newblock \href {https://doi.org/10.1287/opre.51.6.850.24925}
  {\path{doi:10.1287/opre.51.6.850.24925}}.

\bibitem{Feng2020}
Jiekun Feng.
\newblock {\em {Markov chain, Markov decision process, and deep reinforcement
  learning with applications to hospital management and real-time
  ride-hailing}}.
\newblock PhD thesis, Cornell University, Ithaca, NY, 2020.
\newblock \href {https://doi.org/10.7298/0x9s-6r20}
  {\path{doi:10.7298/0x9s-6r20}}.

\bibitem{Feng2021}
Jiekun Feng, Mark Gluzman, and J.~G. Dai.
\newblock {Scalable deep reinforcement learning for ride-hailing}.
\newblock {\em IEEE Control Systems Letters}, 5(6):2060--2065, 2021.
\newblock \href {http://arxiv.org/abs/2009.14679} {\path{arXiv:2009.14679}},
  \href {https://doi.org/10.1109/LCSYS.2020.3046995}
  {\path{doi:10.1109/LCSYS.2020.3046995}}.

\bibitem{Ferenstein2015}
Gregory Ferenstein.
\newblock {Uber CEO Spells Out His Endgame, In 2 Quotes}, 2015.
\newblock URL:
  \url{https://www.forbes.com/sites/gregoryferenstein/2015/09/16/uber-ceo-spells-out-his-endgame-in-2-quotes/?sh=117fb2607bec}.

\bibitem{Ferre2013}
D{\'{e}}borah Ferr{\'{e}}, Lo{\"{i}}c Herv{\'{e}}, and James Ledoux.
\newblock {Regular perturbation of V-geometrically ergodic Markov chains}.
\newblock {\em Journal of Applied Probability}, 50(1):184--194, 2013.
\newblock \href {https://doi.org/10.1239/JAP/1363784432}
  {\path{doi:10.1239/JAP/1363784432}}.

\bibitem{Glorot2010}
Xavier Glorot and Yoshua Bengio.
\newblock {Understanding the difficulty of training deep feedforward neural
  networks}.
\newblock In {\em Proceedings of the Thirteenth International Conference on
  Artificial Intelligence and Statistics}, pages 249--256, 2010.

\bibitem{Glynn1996}
Peter~W. Glynn and Sean~P. Meyn.
\newblock {A Liapounov bound for solutions of the Poisson equation}.
\newblock {\em Annals of Probability}, 24(2):916--931, 1996.
\newblock \href {https://doi.org/10.1214/aop/1039639370}
  {\path{doi:10.1214/aop/1039639370}}.

\bibitem{Golub2013}
Gene~H. Golub and Charles~F. {Van Loan}.
\newblock {\em {Matrix Computations}}.
\newblock Johns Hopkins University Press, Baltimore, Maryland, 4th edition,
  2013.

\bibitem{Guo2016}
Cheng Guo and Felix Berkhahn.
\newblock {Entity embeddings of categorical variables}.
\newblock 2016.
\newblock \href {http://arxiv.org/abs/1604.06737} {\path{arXiv:1604.06737}}.

\bibitem{Haarnoja2018}
Tuomas Haarnoja, Aurick Zhou, Pieter Abbeel, and Sergey Levine.
\newblock {Soft actor-critic: off-policy maximum entropy deep reinforcement
  learning with a stochastic actor}.
\newblock {\em Proceedings of Machine Learning Research}, 80:1861--1870, 2018.
\newblock \href {http://arxiv.org/abs/1801.01290} {\path{arXiv:1801.01290}}.

\bibitem{Harr1988}
J.~Michael Harrison.
\newblock Brownian models of queueing networks with heterogeneous customer
  populations.
\newblock In W.~Fleming and P.~L. Lions, editors, {\em Stochastic Differential
  Systems, Stochastic Control Theory and Applications}, volume~10 of {\em The
  IMA Volumes in Mathematics and Its Applications}, pages 147--186. Springer,
  New York, NY, 1988.
\newblock \href {https://doi.org/10.1007/978-1-4613-8762-6_11}
  {\path{doi:10.1007/978-1-4613-8762-6_11}}.

\bibitem{Harr1996}
J.~Michael Harrison.
\newblock The {bigstep} approach to flow management in stochastic processing
  networks.
\newblock In S.~Zachary F.~P.~Kelly and I.~Ziedins, editors, {\em Stochastic
  Networks: Theory and Applications}, volume~4 of {\em Lecture Note Series},
  pages 57--90. Oxford University Press, 1996.

\bibitem{Harrison1998}
J~Michael Harrison.
\newblock {Heavy traffic analysis of a system with parallel servers: asymptotic
  optimality of discrete-review policies}.
\newblock {\em The Annals of Applied Probability}, 8(3):822--848, 1998.
\newblock \href {https://doi.org/10.1214/aoap/1028903452}
  {\path{doi:10.1214/aoap/1028903452}}.

\bibitem{Harr2000}
J.~Michael Harrison.
\newblock Brownian models of open processing networks: canonical representation
  of workload.
\newblock {\em Ann. Appl. Probab.}, 10(1):75--103, 2000.
\newblock corrections: {\bf 13}, 390--393 (2003) and {\bf 16}, 1703-1732
  (2006).
\newblock \href {https://doi.org/10.1214/aoap/1019737665}
  {\path{doi:10.1214/aoap/1019737665}}.

\bibitem{Harr2002}
J.~Michael Harrison.
\newblock Stochastic networks and activity analysis.
\newblock In Yu.~M. Suhov, editor, {\em Analytic Methods in Applied
  Probability: In memory of {Fridrikh Karpelevich}}, volume 207 of {\em
  American Mathematical Society Translations: Series 2}, pages 53--76,
  Providence, RI, 2002. American Mathematical Society.
\newblock \href {https://doi.org/10.1090/trans2/207/04}
  {\path{doi:10.1090/trans2/207/04}}.

\bibitem{Harrison1993}
J.~Michael Harrison and Vi{\^{e}}n Nguyen.
\newblock {Brownian models of multiclass queueing networks: Current status and
  open problems}.
\newblock {\em Queueing Systems}, 13(1-3):5--40, mar 1993.
\newblock \href {https://doi.org/10.1007/BF01158927}
  {\path{doi:10.1007/BF01158927}}.

\bibitem{Harrison1990}
J.~Michael Harrison and Lawrence~M. Wein.
\newblock {Scheduling networks of queues: heavy traffic analysis of a
  two-station closed network}.
\newblock {\em Operations Research}, 38(6):1052--1064, 1990.
\newblock \href {https://doi.org/10.1007/978-1-4684-0302-2}
  {\path{doi:10.1007/978-1-4684-0302-2}}.

\bibitem{Henderson2002}
Shane~G. Henderson and Peter~W. Glynn.
\newblock {Approximating martingales for variance reduction in Markov process
  simulation}.
\newblock {\em Mathematics of Operations Research}, 27(2):253--271, 2002.
\newblock \href {https://doi.org/10.1287/moor.27.2.253.329}
  {\path{doi:10.1287/moor.27.2.253.329}}.

\bibitem{Henderson1997}
Shane~G. Henderson and Sean~P. Meyn.
\newblock {Efficient simulation of multiclass queueing networks}.
\newblock In {\em Proceedings of the 29th conference on Winter simulation - WSC
  '97}, pages 216--223, New York, New York, USA, 1997. ACM Press.
\newblock \href {https://doi.org/10.1145/268437.268482}
  {\path{doi:10.1145/268437.268482}}.

\bibitem{Henderson2003}
Shane~G. Henderson, Sean~P. Meyn, and Vladislav~B. Tadi{\'{c}}.
\newblock {Performance evaluation and policy selection in multiclass networks}.
\newblock {\em Discrete Event Dynamic Systems: Theory and Applications},
  13(1-2):149--189, 2003.
\newblock \href {https://doi.org/10.1023/A:1022197004856}
  {\path{doi:10.1023/A:1022197004856}}.

\bibitem{HernandezLerma1997}
On{\'{e}}simo Hern{\'{a}}ndez-Lerma and Jean~B. Lasserre.
\newblock {Policy iteration for average cost Markov control processes on Borel
  spaces}.
\newblock {\em Acta Applicandae Mathematica}, 47(2):125--154, 1997.
\newblock \href {https://doi.org/10.1023/A:1005781013253}
  {\path{doi:10.1023/A:1005781013253}}.

\bibitem{Herve2014}
Lo{\"{i}}c Herv{\'{e}} and James Ledoux.
\newblock {Approximating Markov chains and V-geometric ergodicity via weak
  perturbation theory}.
\newblock {\em Stochastic Processes and their Applications}, 124(1):613--638,
  2014.
\newblock \href {https://doi.org/10.1016/J.SPA.2013.09.003}
  {\path{doi:10.1016/J.SPA.2013.09.003}}.

\bibitem{Hessel2018}
Matteo Hessel, Joseph Modayil, Hado van Hasselt, Tom Schaul, Georg Ostrovski,
  Will Dabney, Dan Horgan, Bilal Piot, Mohammad Azar, and David Silver.
\newblock {Rainbow: combining improvements in deep reinforcement learning}.
\newblock {\em 32nd AAAI Conference on Artificial Intelligence}, pages
  3215--3222, 2018.
\newblock \href {http://arxiv.org/abs/1710.02298} {\path{arXiv:1710.02298}}.

\bibitem{Hunter2006}
Jeffrey~J. Hunter.
\newblock {Mixing times with applications to perturbed Markov chains}.
\newblock {\em Linear Algebra and Its Applications}, 417(1):108--123, 2006.
\newblock \href {https://doi.org/10.1016/j.laa.2006.02.008}
  {\path{doi:10.1016/j.laa.2006.02.008}}.

\bibitem{Ilyas2020}
Andrew Ilyas, Logan Engstrom, Shibani Santurkar, Dimitris Tsipras, Firdaus
  Janoos, Larry Rudolph, and Aleksander Madry.
\newblock {A closer look at deep policy gradients}.
\newblock In {\em ICLR}, 2020.
\newblock \href {http://arxiv.org/abs/1811.02553} {\path{arXiv:1811.02553}}.

\bibitem{Jaakkola1994}
Tommi Jaakkola, Satinder {P. Singh}, and Michael~I. Jordan.
\newblock {Reinforcement learning algorithm for partially observable Markov
  decision problems}.
\newblock In {\em Proceedings of the 7th International Conference on Neural
  Information Processing Systems}, pages 345--352, 1994.

\bibitem{Jiang2017}
Shuxia Jiang, Yuanyuan Liu, and Yingchun Tang.
\newblock {A unified perturbation analysis framework for countable Markov
  chains}.
\newblock {\em Linear Algebra and Its Applications}, 529:413--440, 2017.
\newblock \href {https://doi.org/10.1016/j.laa.2017.05.002}
  {\path{doi:10.1016/j.laa.2017.05.002}}.

\bibitem{Kakade2001}
Sham Kakade.
\newblock {Optimizing average reward using discounted rewards}.
\newblock In {\em COLT '01/EuroCOLT '01}, pages 605--615, 2001.
\newblock \href {https://doi.org/10.1007/3-540-44581-1_40}
  {\path{doi:10.1007/3-540-44581-1_40}}.

\bibitem{Kakade2002}
Sham Kakade and John Langford.
\newblock {Approximately optimal approximate reinforcement learning}.
\newblock In {\em Proceedings of ICML'02}, pages 267--274, 2002.

\bibitem{Kartashov1986}
N.~V. Kartashov.
\newblock {Strongly stable Markov chains}.
\newblock {\em Journal of Soviet Mathematics}, 34(2):1493--1498, 1986.
\newblock \href {https://doi.org/10.1007/BF01089787}
  {\path{doi:10.1007/BF01089787}}.

\bibitem{Kartashov1996}
N.~V. Kartashov.
\newblock {\em {Strong Stable Markov Chains}}.
\newblock De Gruyter, Berlin, 1996.
\newblock \href {https://doi.org/10.1515/9783110917765}
  {\path{doi:10.1515/9783110917765}}.

\bibitem{Ke2019}
Jintao Ke, Feng Xiao, Hai Yang, Jieping Ye, and Senior Member.
\newblock {Optimizing online matching for ride-sourcing services with
  multi-agent deep reinforcement learning}.
\newblock 2019.
\newblock \href {http://arxiv.org/abs/1902.06228} {\path{arXiv:1902.06228}}.

\bibitem{Kemeny1976a}
John~G. Kemeny and J.~Laurie Snell.
\newblock {\em {Finite Markov Chains}}.
\newblock Springer-Verlag, New York, 1st edition, 1976.

\bibitem{Kemeny1976}
John~G. Kemeny, J.~Laurie. Snell, and Anthony~W. Knapp.
\newblock {\em {Denumerable Markov Chains}}.
\newblock Springer New York, 1976.

\bibitem{Kingma2017}
Diederik~P. Kingma and Jimmy Ba.
\newblock {Adam: a method for stochastic optimization}.
\newblock In {\em ICLR}, 2015.
\newblock \href {http://arxiv.org/abs/1412.6980} {\path{arXiv:1412.6980}}.

\bibitem{Kirkland2008}
Stephen~J. Kirkland, Michael Neumann, and Nung~Sing Sze.
\newblock {On optimal condition numbers for Markov chains}.
\newblock {\em Numerische Mathematik}, 110(4):521--537, 2008.
\newblock \href {https://doi.org/10.1007/s00211-008-0172-8}
  {\path{doi:10.1007/s00211-008-0172-8}}.

\bibitem{Konda2003}
Vijay~R. Konda and John~N. Tsitsiklis.
\newblock {On actor-critic algorithms}.
\newblock {\em SIAM Journal on Control and Optimization}, 42(4):1143--1166,
  2003.
\newblock \href {https://doi.org/10.1137/S0363012901385691}
  {\path{doi:10.1137/S0363012901385691}}.

\bibitem{Kumar1993}
P.~R. Kumar.
\newblock {Re-entrant lines}.
\newblock {\em Queueing Systems}, 13(1-3):87--110, 1993.
\newblock \href {https://doi.org/10.1007/BF01158930}
  {\path{doi:10.1007/BF01158930}}.

\bibitem{Kumar1994}
S.~Kumar and P.~R. Kumar.
\newblock {Performance bounds for queueing networks and scheduling policies}.
\newblock {\em IEEE Transactions on Automatic Control}, 39(8):1600--1611, 1994.
\newblock \href {https://doi.org/10.1109/9.310033}
  {\path{doi:10.1109/9.310033}}.

\bibitem{Kumar1996}
S.~Kumar and P.~R. Kumar.
\newblock {Fluctuation smoothing policies are stable for stochastic re-entrant
  lines}.
\newblock {\em Discrete Event Dynamic Systems}, 6(4):361--370, 1996.
\newblock \href {https://doi.org/10.1007/BF01797136}
  {\path{doi:10.1007/BF01797136}}.

\bibitem{Langville2003}
Amy Langville and Carl Meyer.
\newblock {Deeper inside PageRank}.
\newblock {\em Internet Mathematics}, 1(3):335--380, jan 2004.
\newblock \href {https://doi.org/10.1080/15427951.2004.10129091}
  {\path{doi:10.1080/15427951.2004.10129091}}.

\bibitem{Lehnert2018}
Lucas Lehnert, Romain Laroche, and Harm van Seijen.
\newblock {On value function representation of long horizon problems}.
\newblock {\em Thirty-Second AAAI Conference on Artificial Intelligence},
  32(1), 2018.

\bibitem{Levin2017}
David~A. Levin and Yuval Peres.
\newblock {\em {Markov Chains and Mixing Times}}.
\newblock American Mathematical Society, 2nd edition, 2017.

\bibitem{Liu2012}
Yuanyuan Liu.
\newblock {Perturbation bounds for the stationary distributions of Markov
  chains}.
\newblock {\em SIAM Journal on Matrix Analysis and Applications},
  33(4):1057--1074, 2012.
\newblock \href {https://doi.org/10.1137/110838753}
  {\path{doi:10.1137/110838753}}.

\bibitem{Liu2020}
Yuanyuan Liu and Fangfang Lyu.
\newblock {Kemeny's constant for countable Markov chains}.
\newblock {\em Linear Algebra and Its Applications}, 604:425--440, 2020.
\newblock \href {https://doi.org/10.1016/j.laa.2020.07.001}
  {\path{doi:10.1016/j.laa.2020.07.001}}.

\bibitem{Lu1994}
Steve Lu, Deepa Ramaswamy, and P.~R. Kumar.
\newblock {Efficient Scheduling Policies to Reduce Mean and Variance of
  Cycle-Time in Semiconductor Manufacturing Plants}.
\newblock {\em IEEE Transactions on Semiconductor Manufacturing},
  7(3):374--388, 1994.
\newblock \href {https://doi.org/10.1109/66.311341}
  {\path{doi:10.1109/66.311341}}.

\bibitem{Lund1996}
Robert~B. Lund and Richard~L. Tweedie.
\newblock {Geometric convergence rates for stochastically ordered Markov
  chains}.
\newblock {\em Mathematics of Operations Research}, 21(1):182--194, feb 1996.
\newblock \href {https://doi.org/10.1287/moor.21.1.182}
  {\path{doi:10.1287/moor.21.1.182}}.

\bibitem{Luong2019}
Nguyen~Cong Luong, Dinh~Thai Hoang, Shimin Gong, Dusit Niyato, Ping Wang,
  Ying~Chang Liang, and Dong~In Kim.
\newblock {Applications of deep reinforcement learning in communications and
  networking: a survey}.
\newblock {\em IEEE Communications Surveys and Tutorials}, 21(4):3133--3174,
  2019.
\newblock \href {http://arxiv.org/abs/1810.07862} {\path{arXiv:1810.07862}},
  \href {https://doi.org/10.1109/COMST.2019.2916583}
  {\path{doi:10.1109/COMST.2019.2916583}}.

\bibitem{Maglaras2000}
Constantinos Maglaras.
\newblock {Discrete-review policies for scheduling stochastic networks:
  trajectory tracking and fluid-scale asymptotic optimality}.
\newblock {\em The Annals of Applied Probability}, 10(3):897--929, 2000.
\newblock \href {https://doi.org/10.1214/aoap/1019487513}
  {\path{doi:10.1214/aoap/1019487513}}.

\bibitem{Maguluri2012}
Siva~Theja Maguluri, R.~Srikant, and Lei Ying.
\newblock {Stochastic models of load balancing and scheduling in cloud
  computing clusters}.
\newblock In {\em Proceedings - IEEE INFOCOM}, pages 702--710, 2012.
\newblock \href {https://doi.org/10.1109/INFCOM.2012.6195815}
  {\path{doi:10.1109/INFCOM.2012.6195815}}.

\bibitem{Mao2016}
Hongzi Mao, Mohammad Alizadeh, Ishai Menache, and Srikanth Kandula.
\newblock {Resource management with deep reinforcement learning}.
\newblock In {\em HotNets 2016 - Proceedings of the 15th ACM Workshop on Hot
  Topics in Networks}, pages 50--56, New York, USA, 2016. ACM Press.
\newblock \href {https://doi.org/10.1145/3005745.3005750}
  {\path{doi:10.1145/3005745.3005750}}.

\bibitem{Marbach2001}
Peter Marbach and John~N. Tsitsiklis.
\newblock {Simulation-based optimization of Markov reward processes}.
\newblock {\em IEEE Transactions on Automatic Control}, 46(2):191--209, 2001.
\newblock \href {https://doi.org/10.1109/9.905687}
  {\path{doi:10.1109/9.905687}}.

\bibitem{Martins1996}
L.~F. Martins, S.~E. Shreve, and H.~M. Soner.
\newblock {Heavy traffic convergence of a controlled, multiclass queueing
  system}.
\newblock {\em SIAM Journal on Control and Optimization}, 34(6):2133--2171,
  1996.
\newblock \href {https://doi.org/10.1137/S0363012994265882}
  {\path{doi:10.1137/S0363012994265882}}.

\bibitem{McKeown1999}
Nick McKeown, Adisak Mekkittikul, Venkat Anantharam, and Jean Walrand.
\newblock {Achieving 100\% throughput in an input-queued switch}.
\newblock {\em IEEE Transactions on Communications}, 47(8):1260--1267, 1999.
\newblock \href {https://doi.org/10.1109/26.780463}
  {\path{doi:10.1109/26.780463}}.

\bibitem{Meyer1975}
Carl~D. Meyer.
\newblock {The role of the group generalized inverse in the theory of finite
  Markov chains}.
\newblock {\em SIAM Review}, 17(3):443--464, 1975.
\newblock \href {https://doi.org/10.1137/1017044} {\path{doi:10.1137/1017044}}.

\bibitem{Meyer1980}
Carl~D. Meyer.
\newblock {The condition of a finite Markov chain and perturbation bounds for
  the limiting probabilities}.
\newblock {\em SIAM Journal on Algebraic Discrete Methods}, 1(3):273--283,
  1980.
\newblock \href {https://doi.org/10.1137/0601031} {\path{doi:10.1137/0601031}}.

\bibitem{Meyn1997}
Sean Meyn.
\newblock {Stability and optimization of queueing networks and their fluid
  models}.
\newblock In G.~George Yin and Qing Zhang, editors, {\em Mathematics of
  Stochastic Manufacturing Systems}, pages 175--199. American Mathematical
  Society, Providence, RI, 1997.
\newblock \href {https://doi.org/10.1239/jap/1421763321}
  {\path{doi:10.1239/jap/1421763321}}.

\bibitem{Meyn2007}
Sean Meyn.
\newblock {\em {Control Techniques for Complex Networks}}.
\newblock Cambridge University Press, Cambridge, 2007.
\newblock \href {https://doi.org/10.1017/CBO9780511804410}
  {\path{doi:10.1017/CBO9780511804410}}.

\bibitem{Meyn2009}
Sean Meyn and Richard~L. Tweedie.
\newblock {\em {Markov Chains and Stochastic Stability}}.
\newblock Cambridge University Press, Cambridge, 2nd edition, 2009.
\newblock \href {https://doi.org/10.1017/CBO9780511626630}
  {\path{doi:10.1017/CBO9780511626630}}.

\bibitem{Mnih2015}
Volodymyr Mnih, Koray Kavukcuoglu, David Silver, Andrei~A. Rusu, Joel Veness,
  Marc~G. Bellemare, Alex Graves, Martin Riedmiller, Andreas~K. Fidjeland,
  Georg Ostrovski, Stig Petersen, Charles Beattie, Amir Sadik, Ioannis
  Antonoglou, Helen King, Dharshan Kumaran, Daan Wierstra, Shane Legg, and
  Demis Hassabis.
\newblock {Human-level control through deep reinforcement learning}.
\newblock {\em Nature}, 518(7540):529--533, 2015.
\newblock \href {https://doi.org/10.1038/nature14236}
  {\path{doi:10.1038/nature14236}}.

\bibitem{MoalKumaVanR2008}
Ciamac Moallemi, Sunil Kumar, and Benjamin {Van Roy}.
\newblock Approximate and data-driven dynamic programming for queueing
  networks.
\newblock Preprint, 2008.

\bibitem{Moritz2018}
Philipp Moritz, Robert Nishihara, Stephanie Wang, Alexey Tumanov, Richard Liaw,
  Eric Liang, Melih Elibol, Zongheng Yang, William Paul, Michael~I. Jordan, and
  Ion Stoica.
\newblock {Ray: a distributed framework for emerging AI applications}.
\newblock In {\em 13th USENIX Symposium on Operating Systems Design and
  Implementation (OSDI '18)}, 2018.
\newblock \href {http://arxiv.org/abs/1712.05889} {\path{arXiv:1712.05889}}.

\bibitem{Mouhoubi2010}
Zahir Mouhoubi and Djamil A{\"{i}}ssani.
\newblock {New perturbation bounds for denumerable Markov chains}.
\newblock {\em Linear Algebra and its Applications}, 432(7):1627--1649, 2010.
\newblock \href {https://doi.org/10.1016/J.LAA.2009.11.020}
  {\path{doi:10.1016/J.LAA.2009.11.020}}.

\bibitem{Negrea2021}
Jeffrey Negrea and Jeffrey~S. Rosenthal.
\newblock {Approximations of geometrically ergodic reversible markov chains}.
\newblock {\em Advances in Applied Probability}, 53(4):981--1022, 2021.
\newblock \href {https://doi.org/10.1017/APR.2021.10}
  {\path{doi:10.1017/APR.2021.10}}.

\bibitem{Nelson1989}
Barry~L. Nelson.
\newblock {Batch size effects on the efficiency of control variates in
  simulation}.
\newblock {\em European Journal of Operational Research}, 43(2):184--196, 1989.
\newblock \href {https://doi.org/10.1016/0377-2217(89)90212-9}
  {\path{doi:10.1016/0377-2217(89)90212-9}}.

\bibitem{Oda2018}
Takuma Oda and Carlee Joe-Wong.
\newblock {MOVI: a model-free approach to dynamic fleet management}.
\newblock In {\em IEEE Conference on Computer Communications}, pages
  2708--2716, 2018.
\newblock \href {http://arxiv.org/abs/1804.04758} {\path{arXiv:1804.04758}},
  \href {https://doi.org/10.1109/INFOCOM.2018.8485988}
  {\path{doi:10.1109/INFOCOM.2018.8485988}}.

\bibitem{OpenAI2019a}
OpenAI.
\newblock {Dota 2 with large scale deep reinforcement learning}.
\newblock 2019.
\newblock \href {http://arxiv.org/abs/1912.06680} {\path{arXiv:1912.06680}}.

\bibitem{Akkaya2019}
OpenAI, Ilge Akkaya, Marcin Andrychowicz, Maciek Chociej, Mateusz Litwin, Bob
  McGrew, Arthur Petron, Alex Paino, Matthias Plappert, Glenn Powell, Raphael
  Ribas, Jonas Schneider, Nikolas Tezak, Jerry Tworek, Peter Welinder, Lilian
  Weng, Qiming Yuan, Wojciech Zaremba, and Lei Zhang.
\newblock {Solving Rubik's cube with a robot hand}.
\newblock 2019.
\newblock \href {http://arxiv.org/abs/1910.07113} {\path{arXiv:1910.07113}}.

\bibitem{Ozkan2020}
Erhun Ozkan and Amy~R. Ward.
\newblock {Dynamic matching for real-time ride sharing}.
\newblock {\em Stochastic Systems}, 10(1):29--70, 2020.
\newblock \href {https://doi.org/10.1287/STSY.2019.0037}
  {\path{doi:10.1287/STSY.2019.0037}}.

\bibitem{Paschalidis2004}
I.C. Paschalidis, C.~Su, and M.C. Caramanis.
\newblock {Target-pursuing scheduling and routing policies for multiclass
  queueing networks}.
\newblock {\em IEEE Transactions on Automatic Control}, 49(10):1709--1722,
  2004.
\newblock \href {https://doi.org/10.1109/TAC.2004.835389}
  {\path{doi:10.1109/TAC.2004.835389}}.

\bibitem{Perkins1989}
J.~R. Perkins and P.~R. Kumar.
\newblock {Stable, distributed, real-time scheduling of flexible
  manufacturing/assembly/disassembly systems}.
\newblock {\em IEEE Transactions on Automatic Control}, 34(2):139--148, 1989.
\newblock \href {https://doi.org/10.1109/9.21085} {\path{doi:10.1109/9.21085}}.

\bibitem{Peters2008}
Jan Peters and Stefan Schaal.
\newblock {Reinforcement learning of motor skills with policy gradients}.
\newblock {\em Neural Networks}, 21(4):682--697, 2008.
\newblock \href {https://doi.org/10.1016/j.neunet.2008.02.003}
  {\path{doi:10.1016/j.neunet.2008.02.003}}.

\bibitem{Psaraftis2016}
Harilaos~N. Psaraftis, Min Wen, and Christos~A. Kontovas.
\newblock {Dynamic vehicle routing problems: three decades and counting}.
\newblock {\em Networks}, 67(1):3--31, 2016.
\newblock \href {https://doi.org/10.1002/NET.21628}
  {\path{doi:10.1002/NET.21628}}.

\bibitem{Puterman2005}
Martin~L. Puterman.
\newblock {\em {Markov Decision Processes: Discrete Stochastic Dynamic
  Programming}}.
\newblock Wiley-Interscience, 2005.

\bibitem{Qin2020a}
Zhiwei Qin, Xiaocheng Tang, Yan Jiao, Fan Zhang, Zhe Xu, Hongtu Zhu, and
  Jieping Ye.
\newblock {Ride-hailing order dispatching at DiDi via reinforcement learning}.
\newblock {\em INFORMS Journal on Applied Analytics}, 50(5):272--286, 2020.
\newblock \href {https://doi.org/10.1287/INTE.2020.1047}
  {\path{doi:10.1287/INTE.2020.1047}}.

\bibitem{Qin2022a}
Zhiwei Qin, Hongtu Zhu, and Jieping Ye.
\newblock {Reinforcement learning for ridesharing: an extended survey}.
\newblock 2022.
\newblock \href {http://arxiv.org/abs/2105.01099} {\path{arXiv:2105.01099}}.

\bibitem{Ramirez-Hernandez2007}
Jose~A. Ramirez-Hernandez and Emmanuel Fernandez.
\newblock {An approximate dynamic programming approach for job releasing and
  sequencing in a reentrant manufacturing line}.
\newblock In {\em 2007 IEEE International Symposium on Approximate Dynamic
  Programming and Reinforcement Learning}, pages 201--208, 2007.
\newblock \href {https://doi.org/10.1109/ADPRL.2007.368189}
  {\path{doi:10.1109/ADPRL.2007.368189}}.

\bibitem{Roberts1998}
Gareth~O. Roberts, Jeffrey~S. Rosenthal, and Peter~O. Schwartz.
\newblock {Convergence properties of perturbed Markov chains}.
\newblock {\em Journal of Applied Probability}, 35(1):1--11, 1998.
\newblock \href {https://doi.org/10.1239/jap/1032192546}
  {\path{doi:10.1239/jap/1032192546}}.

\bibitem{Rosenthal2007}
Jeffrey~S. Rosenthal.
\newblock {Rates of convergence for data augmentation on finite sample spaces}.
\newblock {\em The Annals of Applied Probability}, 3(3):819--839, 2007.
\newblock \href {https://doi.org/10.1214/aoap/1177005366}
  {\path{doi:10.1214/aoap/1177005366}}.

\bibitem{Schlobach2018}
M.~Schlobach and S.~Retzer.
\newblock {Didi Chuxing - How China's ride-hailing leader aims to transform the
  future of mobility}, 2018.
\newblock URL: \url{https://www.sustainabletransport.org/archives/6317}.

\bibitem{Schulman2017a}
John Schulman, Oleg Klimov, Filip Wolski, Prafulla Dhariwal, and Alec Radford.
\newblock {Proximal Policy Optimization}, 2017.
\newblock URL: \url{https://openai.com/blog/openai-baselines-ppo/}.

\bibitem{Schulman2017}
John Schulman, Oleg Klimov, Filip Wolski, Prafulla Dhariwal, Alec Radford, and
  Oleg Klimov.
\newblock {Proximal policy optimization}.
\newblock 2017.
\newblock \href {http://arxiv.org/abs/1707.06347} {\path{arXiv:1707.06347}}.

\bibitem{Schulman2015}
John Schulman, Sergey Levine, Philipp Moritz, Michael~I. Jordan, and Pieter
  Abbeel.
\newblock {Trust region policy optimization}.
\newblock In {\em Proceedings of ICML'15}, pages 1889--1897, 2015.
\newblock \href {http://arxiv.org/abs/1502.05477} {\path{arXiv:1502.05477}}.

\bibitem{Schulman2016}
John Schulman, Philipp Moritz, Sergey Levine, Michael~I. Jordan, and Pieter
  Abbeel.
\newblock {High-dimensional continuous control using generalized advantage
  estimation}.
\newblock In {\em Procedings of ICLR'16}, 2016.
\newblock \href {http://arxiv.org/abs/1506.02438} {\path{arXiv:1506.02438}}.

\bibitem{Seneta1991}
E.~Seneta.
\newblock {Sensitivity analysis, ergodicity coefficients, and rank-one updates
  for finite Markov chains}.
\newblock In W.J. Stewart, editor, {\em Numerical Solution of Markov Chains},
  pages 121--129. Marcel Dekker, New York, 1991.

\bibitem{Seneta1993}
E.~Seneta.
\newblock {Sensitivity of finite Markov chains under perturbation}.
\newblock {\em Statistics \& Probability Letters}, 17(2):163--168, 1993.
\newblock \href {https://doi.org/10.1016/0167-7152(93)90011-7}
  {\path{doi:10.1016/0167-7152(93)90011-7}}.

\bibitem{Serfozo1979}
Richard~F. Serfozo.
\newblock {Technical note—An equivalence between continuous and discrete time
  Markov decision processes}.
\newblock {\em Operations Research}, 27(3):616--620, 1979.
\newblock \href {https://doi.org/10.1287/opre.27.3.616}
  {\path{doi:10.1287/opre.27.3.616}}.

\bibitem{Shi2020}
Jie Shi, Yuanqi Gao, Wei Wang, Nanpeng Yu, and Petros~A. Ioannou.
\newblock {Operating electric vehicle fleet for ride-hailing services with
  reinforcement learning}.
\newblock {\em IEEE Transactions on Intelligent Transportation Systems},
  21(11):4822--4834, 2020.
\newblock \href {https://doi.org/10.1109/TITS.2019.2947408}
  {\path{doi:10.1109/TITS.2019.2947408}}.

\bibitem{Silver2017}
David Silver, Julian Schrittwieser, Karen Simonyan, Ioannis Antonoglou, Aja
  Huang, Arthur Guez, Thomas Hubert, Lucas Baker, Matthew Lai, Adrian Bolton,
  Yutian Chen, Timothy Lillicrap, Fan Hui, Laurent Sifre, George van~den
  Driessche, Thore Graepel, and Demis Hassabis.
\newblock {Mastering the game of Go without human knowledge}.
\newblock {\em Nature}, 550(7676):354--359, 2017.
\newblock \href {https://doi.org/10.1038/nature24270}
  {\path{doi:10.1038/nature24270}}.

\bibitem{Simm2020}
Gregor N~C Simm, Robert Pinsler, and Jos{\'{e}}~Miguel Hern{\'{a}}ndez-Lobato.
\newblock {Reinforcement learning for molecular design guided by quantum
  mechanics}.
\newblock In {\em Proceedings of the 37th International Conference on Machine
  Learning}, pages 8959--8969, 2020.
\newblock \href {http://arxiv.org/abs/2002.07717} {\path{arXiv:2002.07717}}.

\bibitem{Spivey2004}
Michael~Z. Spivey and Warren~B. Powell.
\newblock {The dynamic assignment problem}.
\newblock {\em Transportation Science}, 38(4):399--419, 2004.
\newblock \href {https://doi.org/10.1287/trsc.1030.0073}
  {\path{doi:10.1287/trsc.1030.0073}}.

\bibitem{SrikYing2014}
R.~Srikant and Lei Ying.
\newblock {\em Communication Networks: An Optimization, Control and Stochastic
  Networks Perspective}.
\newblock Cambridge University Press, Cambridge, UK, 2014.

\bibitem{Sutton2018}
Richard~S. Sutton and Andrew~G. Barto.
\newblock {\em {Reinforcement Learning: An Introduction}}.
\newblock MIT press, 2nd edition, 2018.

\bibitem{Tang2019}
Xiaocheng Tang, Zhiwei Qin, Fan Zhang, Zhaodong Wang, Zhe Xu, Yintai Ma, Hongtu
  Zhu, and Jieping Ye.
\newblock {A deep value-network based approach for multi-driver order
  dispatching}.
\newblock In {\em The 25th ACM SIGKDD International Conference on Knowledge
  Discovery and Data Mining}, pages 1780--1790. Association for Computing
  Machinery, 2019.
\newblock \href {http://arxiv.org/abs/2106.04493} {\path{arXiv:2106.04493}},
  \href {https://doi.org/10.1145/3292500.3330724}
  {\path{doi:10.1145/3292500.3330724}}.

\bibitem{Thomas2014}
Philip~S Thomas.
\newblock {Bias in natural actor-critic algorithms}.
\newblock In {\em Proceedings of the 31st International Conference on Machine
  Learning}, 2014.

\bibitem{Toth2014}
Paolo Toth and Daniele Vigo.
\newblock {\em {Vehicle Routing: Problems, Methods, and Applications}}.
\newblock Society for Industrial and Applied Mathematics, Philadelphia, 2nd
  edition, 2014.
\newblock \href {https://doi.org/10.1137/1.9781611973594}
  {\path{doi:10.1137/1.9781611973594}}.

\bibitem{Veatch2015}
Michael~H. Veatch.
\newblock {Approximate linear programming for networks: average cost bounds}.
\newblock {\em Computers $\&$ Operations Research}, 63:32--45, 2015.
\newblock \href {https://doi.org/10.1016/j.cor.2015.04.014}
  {\path{doi:10.1016/j.cor.2015.04.014}}.

\bibitem{Vinyals2019}
Oriol Vinyals, Igor Babuschkin, Wojciech~M. Czarnecki, Micha{\"{e}}l Mathieu,
  Andrew Dudzik, Junyoung Chung, David~H. Choi, Richard Powell, Timo Ewalds,
  Petko Georgiev, Junhyuk Oh, Dan Horgan, Manuel Kroiss, Ivo Danihelka, Aja
  Huang, Laurent Sifre, Trevor Cai, John~P. Agapiou, Max Jaderberg,
  Alexander~S. Vezhnevets, R{\'{e}}mi Leblond, Tobias Pohlen, Valentin
  Dalibard, David Budden, Yury Sulsky, James Molloy, Tom~L. Paine, Caglar
  Gulcehre, Ziyu Wang, Tobias Pfaff, Yuhuai Wu, Roman Ring, Dani Yogatama,
  Dario W{\"{u}}nsch, Katrina McKinney, Oliver Smith, Tom Schaul, Timothy
  Lillicrap, Koray Kavukcuoglu, Demis Hassabis, Chris Apps, and David Silver.
\newblock {Grandmaster level in StarCraft II using multi-agent reinforcement
  learning}.
\newblock {\em Nature}, 575(7782):350--354, nov 2019.
\newblock \href {https://doi.org/10.1038/s41586-019-1724-z}
  {\path{doi:10.1038/s41586-019-1724-z}}.

\bibitem{Wagner1975}
Harvey~M. Wagner.
\newblock {\em {Principles of Operations Research: with applications to
  managerial decisions}}.
\newblock Englewood Cliffs, N.J. : Prentice-Hall, 2nd edition, 1975.

\bibitem{Wang2019}
Yuhui Wang, Hao He, Xiaoyang Tan, and Yaozhong Gan.
\newblock {Trust region-guided proximal policy optimization}.
\newblock In {\em 33rd Conference on Neural Information Processing Systems},
  volume~32, pages 626--636, 2019.
\newblock \href {http://arxiv.org/abs/1901.10314} {\path{arXiv:1901.10314}}.

\bibitem{Wang2018}
Zhaodong Wang, Zhiwei Qin, Xiaocheng Tang, Jieping Ye, and Hongtu Zhu.
\newblock {Deep reinforcement learning with knowledge transfer for online rides
  order dispatching}.
\newblock In {\em IEEE International Conference on Data Mining}, volume
  2018-Novem, pages 617--626, 2018.
\newblock \href {https://doi.org/10.1109/ICDM.2018.00077}
  {\path{doi:10.1109/ICDM.2018.00077}}.

\bibitem{Wang2016}
Ziyu Wang, Victor Bapst, Nicolas Heess, Volodymyr Mnih, Remi Munos, Koray
  Kavukcuoglu, and Nando de~Freitas.
\newblock {Sample efficient actor-critic with experience replay}.
\newblock In {\em ICLR 2017}.
\newblock \href {http://arxiv.org/abs/1611.01224} {\path{arXiv:1611.01224}}.

\bibitem{Will1998}
R.~J. Willams.
\newblock Some recent developments for queueing networks.
\newblock In L.~Accardi and C.~C. Heyde, editors, {\em Probability Towards
  2000}, pages 340--456. Springer, 1998.

\bibitem{Wu2017}
Yuhuai Wu, Elman Mansimov, Shun Liao, Roger Grosse, and Jimmy Ba.
\newblock {Scalable trust-region method for deep reinforcement learning using
  Kronecker-factored approximation}.
\newblock In {\em Proceedings of the 31st International Conference on Neural
  Information Processing Systems}, pages 5285--5294, 2017.
\newblock \href {http://arxiv.org/abs/1708.05144} {\path{arXiv:1708.05144}}.

\bibitem{Xu2018a}
Zhe Xu, Zhixin Li, Qingwen Guan, Dingshui Zhang, Qiang Li, Junxiao Nan,
  Chunyang Liu, Wei Bian, and Jieping Ye.
\newblock {Large-scale order dispatch in on-demand ride-hailing platforms: A
  learning and planning approach}.
\newblock {\em Proceedings of the ACM SIGKDD International Conference on
  Knowledge Discovery and Data Mining}, pages 905--913, 2018.
\newblock \href {https://doi.org/10.1145/3219819.3219824}
  {\path{doi:10.1145/3219819.3219824}}.

\bibitem{Zhang2021}
Yiming Zhang and Keith~W Ross.
\newblock {On-policy deep reinforcement learning for the average-reward
  criterion}.
\newblock In {\em Proceedings of ICML'21}, 2021.
\newblock \href {http://arxiv.org/abs/2106.07329} {\path{arXiv:2106.07329}}.

\bibitem{Zoph2018}
Barret Zoph, Google Brain, Vijay Vasudevan, Jonathon Shlens, and Quoc~V {Le
  Google Brain}.
\newblock {Learning transferable architectures for scalable image recognition}.
\newblock In {\em Proceedings of the IEEE Conference on Computer Vision and
  Pattern Recognition}, pages 8697--8710, 2018.
\newblock \href {http://arxiv.org/abs/1707.07012} {\path{arXiv:1707.07012}}.

\end{thebibliography}

\end{document}